\documentclass[a4paper,centertags,oneside,11pt]{amsart}
\usepackage{mathrsfs}
\usepackage{appendix}
\usepackage{amssymb}
\usepackage{fancyhdr}
\usepackage{charter}
\usepackage{typearea}
\usepackage{pdfsync}
\usepackage{mathrsfs}
\usepackage[a4paper,top=3cm,bottom=3cm,left=2cm,right=2cm]{geometry}
\usepackage{enumitem}
\usepackage{color}

\setlist[1]{itemsep=5pt}
\newcommand{\comment}[1]{}
\makeatletter

\def\@setcopyright{}

\def\serieslogo@{}
      \makeatother

\newcommand{\ol}{\overline}
\newcommand{\Td}{\widetilde}

\newcommand{\abs}[1]{\left\vert#1\right\vert}

\newcommand{\set}[1]{\left\{#1\right\}}
\newcommand{\To}{\rightarrow}

\theoremstyle{plain}
\newtheorem{thm}{Theorem}[section]
\newtheorem{cor}[thm]{Corollary}

\newtheorem{lem}[thm]{Lemma}
\newtheorem{slem}[thm]{Sublemma}
\newtheorem{prop}[thm]{Proposition}

\theoremstyle{definition}
\newtheorem{defn}[thm]{Definition}
\theoremstyle{remark}
\newtheorem{rem}[thm]{Remark}

\newtheorem{que}[thm]{Question}

\numberwithin{equation}{section}

\begin{document}
\title[Heat kernel asymptotics, local index theorem and trace integrals for CR manifolds with $S^1$ action]{Heat kernel asymptotics, local index theorem and trace integrals \\for CR manifolds with $S^1$ action}
\author{Jih-Hsin Cheng}
\address{Institute of Mathematics, Academia Sinica and National Center for
Theoretical Sciences, 6F, Astronomy-Mathematics Building, No.1, Sec.4,
Roosevelt Road, Taipei 10617, Taiwan}
\email{cheng@math.sinica.edu.tw}
\author{Chin-Yu Hsiao}
\address{Institute of Mathematics, Academia Sinica, 6F,
Astronomy-Mathematics Building, No.1, Sec.4, Roosevelt Road, Taipei 10617,
Taiwan}
\email{chsiao@math.sinica.edu.tw or chinyu.hsiao@gmail.com}
\author{I-Hsun Tsai}
\address{Department of Mathematics, National Taiwan University, Taipei 10617,
Taiwan}
\email{ihtsai@math.ntu.edu.tw}

\begin{abstract}
Among those transversally elliptic operators initiated by Atiyah and Singer,
Kohn's $\Box_b$ operator on CR manifolds with $S^1$ action is a natural one
of geometric significance for complex analysts. Our first main result
establishes an asymptotic expansion for the heat kernel of such an operator
with values in its Fourier components, which involves an unprecedented
contribution in terms of a distance function from lower dimensional strata
of the $S^1$-action. Our second main result computes a local index density,
in terms of \emph{tangential} characteristic forms, on such manifolds
including \emph{Sasakian manifolds} of interest in String Theory, by showing
that certain non-trivial contributions from strata in the heat kernel expansion
will eventually cancel out by applying Getzler's rescaling technique to
off-diagonal estimates. This leads to a local result which can be thought of
as a type of local index theorem on these CR
manifolds. As applications of our CR index theorem we can prove a CR version
of Grauert-Riemenschneider criterion, and produce many CR functions on a
weakly pseudoconvex CR manifold with transversal $S^1$ action and many CR
sections on some class of CR manifolds, answering (on this class of manifolds) some long-standing
questions in several complex variables and CR geometry. We give examples of these
CR manifolds, some of which arise from Brieskorn manifolds.
Moreover in some cases, without use of equivariant cohomology method nor
keeping contributions arising from lower dimensional strata as done
in previous works,
we can reinterpret Kawasaki's Hirzebruch-Riemann-Roch formula for a complex
orbifold with an orbifold holomorphic line bundle, as an index theorem
obtained by a single integral over a smooth CR manifold which is essentially the
circle bundle of this line bundle.  By contrast, if one computes the trace integral (instead of 
{\it supertrace} as in the case of index theorems) of our heat kernel, then the contributions arising 
from the stratification of the $S^1$ action necessarily occur in a nontrivial way.   Some explicit expressions
about these {\it corrections} are obtained in this paper.   

In short, besides certain applications our paper treats three 
major topics: i) an asymptotic expansion of a (transversal) heat kernel related to Kohn Laplacian (Theorem~\ref{t-gue160114});
ii) a formulation of a local CR index theorem for the case of locally free $S^1$ action (Corollary~\ref{c-gue150508I}); 
iii) study of this
heat kernel trace integral 
in terms of some explicit invariants as reflections upon the geometry of the $S^1$ stratification inside the CR manifold 
(Theorems~\ref{t-gue160416}, ~\ref{t-gue160416d} and \ref{t-gue160416dc}).  Among the three topics, 
the first topic is foundational.   The third topic focuses on 
the role of the {\it Gaussian} part of the heat kernel (which is boiled down to a
Dirac type delta function on the $S^1$ stratification) while the second topic does mainly on the {\it non-Gaussian} part. 
Jointly, the three topics explore and integrate the separate aspects of this class of CR manifolds in our study.  

\end{abstract}

\maketitle
\bigskip

\begin{center}
\large{Contents}
\end{center}

\begin{eqnarray*}
&& 1.\text{ Introduction and statement of the results} \\
&& 1.1.\text{ Introduction and Motivation} \\
&& 1.2.\text{ Main theorems} \\
&& 1.3.\text{ Applications} \\
&& 1.4.\text{ Kawasaki's Hirzebruch-Riemann-Roch and Grauert-Riemenschneider criterion} \\
&& \ \ \ \ \ \text{ for orbifold line bundles} \\
&& 1.5.\text{ Examples} \\
&& 1.6.\text{ Proof of Theorem 1.2} \\
&& 1.7.\text{ The idea of the proofs of Theorem 1.3, Theorem 1.10 and Corollary 1.13} 
\end{eqnarray*}

\bigskip

\begin{center}
\large{Part I: Preparatory foundations}
\end{center}

\begin{eqnarray*}
&& 2.\text{ Preliminaries} \\
&& 2.1.\text{ Some standard notations} \\
&& 2.2.\text{ Set up and terminology} \\
&& 2.3.\text{ Tangential de Rham cohomology group, Tangential Chern character and} \\ 
&& \ \ \ \ \ \text{ Tangential Todd class} \\
&& 2.4.\text{ BRT trivializations and rigid geometric objects} \\
&& 3.\text{ A Hodge theory for } \Box^{(q)}_{b,m} \\
&& 4.\text{ Modified Kohn Laplacian }(\mathrm{Spin}^c\text{ Kohn Laplacian}) \\
&& 5.\text{ Asymptotic expansions for the heat kernels of the modified Kohn Laplacians} \\
&& 5.1.\text{ Heat kernels of the modified Kodaira Laplacians on BRT trivializations} \\
&& 5.2.\text{ Heat kernels of the modified Kohn Laplacians }(\mathrm{Spin}^c\text{ Kohn Laplacians})
\end{eqnarray*}
\bigskip

\begin{center}
\large{Part II: Proofs of main theorems}
\end{center}

\begin{eqnarray*}
&& 6.\text{ Proofs of Theorems 1.3 and 1.10} \\
&& 7.\text{ Trace integrals and proof of Theorem 1.14}  \\
&& 7.1.\text{ A setup, including a comparison with recent developments}\\%
&& 7.2.\text{ Local angular integral}\\
&& 7.3.\text{ Global angular integral} \\
&& 7.4.\text{ Patching up angular integrals over X; proof for the simple type} \\
&& 7.5.\text{ Types for }S^1\text{ stratifications; proof for the general type}\\
&& \text{References}
\end{eqnarray*}

\bigskip

\section{Introduction and statement of the results}

\label{s-gue150507}

\subsection{Introduction and Motivation}

Let $(X,T^{1,0}X)$ be a compact (with no boundary) CR manifold of dimension $2n+1$ and let $%
\overline{\partial }_{b}:\Omega ^{0,q}(X)\rightarrow \Omega ^{0,q+1}(X)$ be
the tangential Cauchy-Riemann operator. For a smooth function $u$, we say $u$
is CR if $\overline{\partial }_{b}u=0$. In CR geometry, it is crucial to be
able to produce many CR functions. When $X$ is strongly pseudoconvex and the 
dimension of $X$ is greater than or equal to five, it is well-known that the
space of global smooth CR functions $H_{b}^{0}(X)$ is infinite dimensional.
When $X$ is weakly pseudoconvex or the Levi form of $X$ has negative
eigenvalues, the space of global CR functions could be trivial. In general,
it is very difficult to determine when the space $H_{b}^{0}(X)$ is large.

A clue to the above phenomenon arises from the following.  By $\overline{\partial }_{b}^{2}=0$, one has the 
 $\overline{\partial }_{b}$-complex: $\cdots \rightarrow \Omega ^{0,q-1}(X)\rightarrow \Omega
^{0,q}(X)\rightarrow \Omega ^{0,q+1}(X)\rightarrow \cdots $ and the Kohn-Rossi cohomology group: $H_{b}^{q}(X):=\frac{\mathrm{Ker\,}\overline{%
\partial }_{b}:\Omega ^{0,q}(X)\rightarrow \Omega ^{0,q+1}(X)}{\mathrm{Im\,}%
\overline{\partial }_{b}:\Omega ^{0,q-1}(X)\rightarrow \Omega ^{0,q}(X)}$.
As in complex geometry, to understand the space $H_{b}^{0}(X)$, one could try
to establish, in the CR case, a Hirzebruch-Riemann-Roch theorem for
$\sum\limits_{j=0}^{n}(-1)^{j}\mathrm{dim}H_{b}^{j}(X)$, an analogue of the Euler characteristic, and 
to prove vanishing
theorems for $H_{b}^{j}(X)$, $j\geq 1$.   The first difficulty with 
such an approach lies in the 
fact that $\mathrm{dim}H_{b}^{j}(X)$ could be infinite for some $j$.  Let's
say more about this in the following.  

If $X$ is strongly pseudoconvex and of dimension $\geq 5$%
, it is well-known that $\overline{\partial }_{b}:\Omega
^{0,0}(X)\rightarrow \Omega ^{0,1}(X)$ is not hypoelliptic and for $q\geq 1$, $\overline{%
\partial }_{b}:\Omega ^{0,q}(X)\rightarrow \Omega ^{0,q+1}(X)$ is
hypoelliptic so that $\mathrm{dim\,}H_{b}^{0}(X)=\infty $
and $\mathrm{dim\,}H_{b}^{q}(X)<\infty $ for $q\geq 1$.  In other cases if the Levi form of
$X$ has exactly one negative, $n-1$ positive eigenvalues on $X$ and $n>3$,
it is well-known that $\mathrm{dim\,}H_{b}^{1}(X)=\infty $, $\mathrm{dim\,}%
H_{b}^{n-1}(X)=\infty $ and $\mathrm{dim\,}H_{b}^{q}(X)<\infty $, for $%
q\notin \left\{ 1,n-1\right\} $.    With these possibly infinite dimensional spaces, we have the trouble with defining the
Euler characteristic $\sum\limits_{j=0}^{n}(-1)^{j}\mathrm{dim}H_{b}^{j}(X)$ properly%
.   

Another line of thought lies in the fact that the space $H_{b}^{q}(X)$ is related to the Kohn
Laplacian $\Box _{b}^{(q)}=\overline{\partial }_{b}^{\ast }\,\overline{%
\partial }_{b}+\overline{\partial }_{b}\,\overline{\partial }_{b}^{\ast
}:\Omega ^{0,q}(X)\rightarrow \Omega ^{0,q}(X)$. 
One can try to define the associated heat operator $e^{-t\Box _{b}^{(q)}}$ by
using spectral theory and then the small $t$ behavior of $e^{-t\Box _{b}^{(q)}}$
is closely related to the dimension of $H_{b}^{q}(X)$.   Unfortunately without any Levi
curvature assumption, $\Box _{b}^{(q)}$ is not hypoelliptic in general and
it is unclear how to determine the small $t$ behavior of $e^{-t\Box
_{b}^{(q)}}$.   Even if $\Box _{b}^{(q)}$ is hypoelliptic, it is
still difficult to calculate the local density. 

We are led to ask the following question.

\begin{que}
\label{q-gue150508} {\it Can we establish some kind of heat kernel asymptotic
expansions for Kohn Laplacian and obtain a CR Hirzebruch-Riemann-Roch theorem (not necessarily 
the usual ones) on some class of CR manifolds?}
\end{que}

It turns out that 
the class of CR manifolds with $S^{1}$
action is a natural choice for the above question. 
On this class of CR manifolds, the geometrical significance 
of Kohn's $\Box _{b}$ operator in connection with transversally
elliptic operators initiated by Atiyah and Singer ~\cite{A74}, 
has been mentioned in the seminal work of Folland and Kohn (\cite{FK72}, p.113).   Three
dimensional (strongly pseudoconvex) CR manifolds with $S^{1}$ action have
been intensively studied back to 1990s in relation to the CR embeddability
problem. We refer the reader to the works \cite{Ep92} and~\cite{Lem92} of
Charles Epstein and Laszlo Lempert respectively. (see more comments on this in
Section~\ref{s-gue150704}). Another related paper is about CR structure on
Seifert manifolds by Kamishima and Tsuboi \cite{KT91} (cf. Remark \ref{r-gue150516}).  
In our present paper the CR manifold
with $S^1$ action is not restricted to the three dimensional case.

To motivate our approach, let's first look at the case which can be reduced to complex geometry.
Consider a compact complex manifold $M$ of dimension $n$ and let $%
(L,h^{L})\rightarrow M$ be a holomorpic line bundle over $M$, where $h^{L}$
denotes a Hermitian fiber metric of $L$. Let $(L^{\ast },h^{L^{\ast
}})\rightarrow M$ be the dual bundle of $(L,h^{L})$ and put $X=\left\{ v\in
L^{\ast };\,\left\vert v\right\vert _{h^{L^{\ast }}}^{2}=1\right\} $. We
call $X$ the circle bundle of $(L^{\ast },h^{L^{\ast }})$. It is clear that $%
X$ is a compact CR manifold of dimension $2n+1$.   Clearly $X$ is 
equipped with a natural (globally free) $S^1$ action (by acting on the circular fiber). 

Let $T\in C^{\infty }(X,TX)$
be the real vector field induced by the $S^{1}$ action, that is, $Tu=\frac{%
\partial }{\partial \theta }(u(e^{-i\theta }\circ x))|_{\theta =0}$, $u\in
C^{\infty }(X)$.    This $S^{1}$ action is {\it CR and
transversal}, i.e. $[T,C^{\infty }(X,T^{1,0}X)]\subset
C^{\infty }(X,T^{1,0}X)$ and $\mathbb{C}T(x)\oplus T_{x}^{1,0}X\oplus
T_{x}^{0,1}X=\mathbb{C}T_{x}X$ respectively.  
For each $m\in \mathbb{Z}$ and $q=0,1,2,\ldots ,n$, put
\begin{equation*}
\begin{split}
\Omega _{m}^{0,q}(X):& =\left\{ u\in \Omega ^{0,q}(X);\,Tu=-imu\right\}  \\
& =\left\{ u\in \Omega ^{0,q}(X);\,u(e^{-i\theta }\circ x)=e^{-im\theta
}u(x),\forall \theta \in \lbrack 0,2\pi \lbrack \right\} .
\end{split}%
\end{equation*}%
Since $\overline{\partial }_{b}T=T\overline{\partial }_{b}$, we have $%
\overline{\partial }_{b,m}=\overline{\partial }_{b}:\Omega
_{m}^{0,q}(X)\rightarrow \Omega _{m}^{0,q+1}(X)$.  We consider the
cohomology group: $H_{b,m}^{q}(X):=\frac{\mathrm{Ker\,}\overline{\partial }%
_{b,m}:\Omega _{m}^{0,q}(X)\rightarrow \Omega _{m}^{0,q+1}(X)}{\mathrm{Im\,}%
\overline{\partial }_{b,m}:\Omega _{m}^{0,q-1}(X)\rightarrow \Omega
_{m}^{0,q}(X)}$, and call it the 
$m$-th $S^1$ Fourier component of the 
Kohn-Rossi cohomology group.  

The following result can be viewed as the starting point
of this paper.   Note $\Omega^{0,q}(M,L^m)$ denotes the space of smooth
sections of $(0,q)$ forms on $M$ with values in $L^m$ ($m$-th power of $L$)
and $H^q(M,L^m)$ the $q$-th $\overline\partial$-Dolbeault cohomology group with values in $%
L^m$.  

\begin{thm}
\label{t-gue150508} For every $q=0,1,2,\ldots,n$,
and every $m\in\mathbb{Z}$, there is a bijective map $A^{(q)}_m:%
\Omega^{0,q}_m(X)\rightarrow\Omega^{0,q}(M,L^m)$ such that $%
A^{(q+1)}_m\overline\partial_{b,m}=\overline\partial A^{(q)}_m$ on $%
\Omega^{0,q}_m(X)$. Hence, $\Omega^{0,q}_m(X)\cong\Omega^{0,q}(M,L^m)$ and $%
H^q_{b,m}(X)\cong H^q(M,L^m)$. In particular $\mathrm{dim\,}%
H^q_{b,m}(X)<\infty$ and $\sum%
\limits^n_{j=0}(-1)^j\mathrm{dim\,}H^j_{b,m}(X)=\sum\limits^n_{j=0}(-1)^j%
\mathrm{dim\,}H^j(M,L^m)$.
\end{thm}

Theorem~\ref{t-gue150508} is probably known to the experts.   As a precise
reference is not easily available (see, however, Folland and Kohn \cite{FK72} p.113),
we will give a proof of Theorem~%
\ref{t-gue150508} in Section~\ref{s-gue150509} for the convenience of the
reader.

In this paper by {\it Kodaira Laplacian}
we mean the Laplacian $\Box^{(q)}_m$ on $L^m$-valued $(0,q)$ forms (on $M$) 
associated with the $\overline\partial$ operator,
a term we borrow from the work of Ma and Marinescu \cite{MM07}.
Let $e^{-t\Box^{(q)}_m}$ be the associated heat operator. It is
well-known that $e^{-t\Box^{(q)}_m}$ admits an asymptotic expansion as $%
t\rightarrow0^+$. Consider $B_m(t):=(A^{(q)}_m)^{-1}\circ
e^{-t\Box^{(q)}_m}\circ A^{(q)}_m$ ($A^{(q)}_m$ 
as in the theorem above). Let $\Box^{(q)}_{b,m}$ be the Kohn
Laplacian (on $X$) acting on (the $m$-th $S^1$ Fourier
component of) $(0,q)$ forms, with $e^{-t\Box^{(q)}_{b,m}}$ the associated heat operator.

A word of caution is in order.  We made no use of metrics for stating Theorem~\ref{t-gue150508}.
However, to define those Laplacians above an appropriate choice of metrics is needed 
(for {\it adjoint} of an operator) 
so that $A_m^{(q)}$ of Theorem~\ref{t-gue150508} also preserves the chosen metrics.  
With this set up it is fundamental that (cf. Proposition~\ref{l-gue150606}) 
\begin{equation}  \label{e-gue150923bi}
e^{-t\Box^{(q)}_{b,m}}=((A^{(q)}_m)^{-1}\circ
e^{-t\Box^{(q)}_m}\circ A^{(q)}_m)\circ Q_m=B_m(t)\circ Q_m=Q_m\circ B_m(t)\circ Q_m,
\end{equation}
where $Q_m:\Omega^{0,q}(X)\rightarrow\Omega^{0,q}_m(X)$ is the orthogonal
projection.  Hence the asymptotic expansion of $e^{-t\Box^{(q)}_m}$ and %
\eqref{e-gue150923bi} lead to an asymptotic expansion
\begin{equation}  \label{e-gue151108}
e^{-t\Box^{(q)}_{b,m}}(x,x)\sim
t^{-n}a^{(q)}_n(x)+t^{-n+1}a^{(q)}_{n-1}(x)+\cdots.
\end{equation}

One goal of this work is to establish a formula similar to \eqref{e-gue151108} (which is however not exactly 
of this form) on any CR manifold with $S^1$ action.   More precisely, 
due to the assumption that the $S^1$ action is only locally free, it turns out  
that $e^{-t\Box^{(q)}_{b,m}}(x,x)$ cannot have the standard asymptotic expansion as \eqref{e-gue151108}.  
Rather, our asymptotic expansion involves an
unprecedented contribution in terms of a {\it distance function} from lower
dimensional strata of the $S^{1}$ action. (See \eqref{e-gue160114} in Theorem %
\ref{t-gue160114} for details and for our first main result.)  It should be emphasized that no
pseudoconvexity condition is assumed.  

Roughly speaking, on the regular part of $X$ we have 
\begin{equation}
e^{-t\Box _{b,m}^{(q)}}(x,x)\sim
t^{-n}a_{n}^{(q)}(x)+t^{-n+1}a_{n-1}^{(q)}(x)+\cdots \mod O\Bigr(t^{-n}e^{-%
\frac{\varepsilon_0\hat{d}(x,X_{\mathrm{sing\,}})^{2}}{t}}\Bigr).  \label{e-gue151108I}
\end{equation}%

On the whole $X$ we have, however, 
\begin{equation}
e^{-t\Box _{b,m}^{(q)}}(x,x)\sim
t^{-n}A_{n}^{(q)}(t, x)+t^{-n+1}A_{n-1}^{(q)}(t, x)+\cdots.  \label{e-gue151108II}
\end{equation}

The difference between \eqref{e-gue151108II} and \eqref{e-gue151108I} lies in that 
$A^{(q)}_s(t,x)$ in \eqref{e-gue151108II} cannot be $t$-independent for all $s$ and are not 
canonically determined (by our method) while $a^{(q)}_s(x)$ in \eqref{e-gue151108I} are $t$-independent for all $s$ 
and are canonically determined.   This ($t$-dependence nature so introduced) presents a great distinction between 
our asymptotic expansion and those in the previous literature, and this distinct point of departure appears to have a big 
influence on the formulation and proof of the relevant index theorems and trace integrals.  
See Section~\ref{s-gue160416} for more comments.  

In addition to the introduction of a distance function
$\hat d$ in \eqref{e-gue151108I} our generalization has another
feature, which is pertinent to the third topic of this paper,  as follows.   
A heat kernel result for orbifolds obtained in 2008
by Dryden, Gordon, Greenwald and Webb for the case
of Laplacian on functions (see \eqref{e-gue160416} and \cite{DGGW08})
and independently by Richardson (\cite{Ri10}) 
seems to suggest that integrating \eqref{e-gue151108I} over $X$ is basically a 
power series in $t^{\frac{1}{2}}$.  See \eqref{e-gue160416} for more.  
To see such a possible connection, one considers $X$ as a fiber space over 
$X/S^1$ which is then an orbifold, and presumes boldly an analogy with ``\eqref{e-gue151108}
for the orbifold case".    Then by the above result  \cite{DGGW08}, 
integrating \eqref{e-gue151108I} over $X$ might give an 
asymptotic expansion which is a power series at most in the {\it fractional}
power $t^{\frac{1}{2}}$ of $t$ (cf. Theorem \ref{t-gue160416})
(while for the case where the $S^{1}$ action is globally free, such as
in the circle bundle above, the
asymptotic expansion is expressed in the {\it integral} power of $t$).
However, our further study shows that the coefficients of $t^j$ for $j$ being 
half-integral necessarily vanish in our present case (irrespective of the local or global freeness of the $S^1$ action).   
Despite that there is no nontrivial fractional power in the $t$-expansion, 
the {\it corrections}/contributions associated with the stratification of the locally free $S^1$ action do arise 
nontrivially in a proper sense.    Some explicit computations about these extra terms are worked out in the 
main result of the final section (Section~\ref{s-gue160416}) regarded as the third topic of this paper.  

As far as the asymptotic expansion is concerned, we remark that
the approach of using Kodaira Laplacian on $M$ (downstairs) as done above
is no longer applicable to the general CR case, as the contribution of a distance function on $X$
involved in our expansion cannot be easily forseen by use of objects in the space
downstairs  (an orbifold in general).   (However, for {\it trace integrals} on invariant functions, 
cf. Section~\ref{s-gue160416}, 
like $\sum_me^{-t\lambda_m}$ denoted by $I(t)$ in certain Riemannian cases, 
$I(t)$ has been studied asymptotically with the help of 
the underlying/quotient manifold/orbifold, cf. \cite[p. 2316-2317]{Ri10}.
See also Proposition~\ref{l-gue150606}, Remark~\ref{r-gue150606a}.)  
We must work on the entire $X$ from scratch with the operator being
only transversally elliptic (on $X$).  (See HRR theorem below for another
instance of this idea.)  Furthermore, as we make no assumption on (strong) pseudoconvexity of $X$,
this renders the techniques usually useful in this direction by previous work (e.g. \cite{BGS})
hardly adequate in our case.    Our current approach is essentially independent of
the previous methods.  This technicality partly accounts for the length of the present paper
(see Section~\ref{s-gue150920} for an outline of proof and Section~\ref{s-gue160416} for 
a comparison with the previous work).

We expect that the coefficients $a^{(q)}_s(x)$ in \eqref{e-gue151108I} 
are related to some geometric quantities. For $q$ $=$ $0$,
function case with strong pseudoconvexity, we refer the
reader to the paper of Beals, Greiner, and Stanton \cite{BGS}. In this
regard, Chern-Moser invariants (see \cite{CM}, and \cite{BFG} for a related question 
posed in the end of the paper) or Tanaka-Webster invariants
(see \cite{Ta} or \cite{We}) should be used to express these coefficients.
In our present situation (without assumptions on pseudoconvexity) it is however more natural to use geometric quantities
adapted to the $S^1$ invariance property, so that a notion of {\it tangential curvature}
arises (with the associated {\it tangential} characteristic forms, cf. Section
\ref{s-gue150508d}) and enters into the coefficients of our asymptotic expansion.  It 
essentially comes back to the Tanaka-Webster curvature
in the strongly pseudoconvex case  (cf. Remark \ref{r-gue160413a}).

The mathematics (existence, asymptotics etc.) of equivariant/transversal heat kernels in the Riemannian situation 
(including that of Riemannian foliations) have been studied in recent years and last decades.   
For a comparison between these developments and our results, we postpone the survey, 
together with that of trace integrals, until Section~\ref{s-gue160416}.  
  
Back to the special case of the circle bundle $X$ over a compact complex manifold $M$%
,  the Hirzebruch-Riemann-Roch Theorem or Atiyah-Singer index Theorem,
together with Theorem~\ref{t-gue150508},  tells us
that
\begin{equation}
\sum_{j=0}^{n}(-1)^{j}\mathrm{dim\,}H_{b,m}^{j}(X)=\sum_{j=0}^{n}(-1)^{j}%
\mathrm{dim\,}H^{j}(M,L^{m})=\int_{M}\mathrm{Td\,}(T^{1,0}M)\mathrm{ch\,}%
(L^{m}),  \label{e-gue150508}
\end{equation}%
in terms of standard characteristic classes on $M$.  
Let's reformulate \eqref{e-gue150508} in geometric terms
on $X$ rather than on $M$:   
\begin{equation}
\sum_{j=0}^{n}(-1)^{j}\mathrm{dim\,}H_{b,m}^{j}(X)=\frac{1}{2\pi }\int_{X}%
\mathrm{Td_{b}\,}(T^{1,0}X)\wedge e^{-m\frac{d\omega _{0}}{2\pi }}\wedge
\omega _{0} \label{e-gue150508I}
\end{equation}%
where $\mathrm{Td_{b}\,}(T^{1,0}X)$ denotes the {\it tangential} Todd class of $%
T^{1,0}X$ and $e^{-m\frac{d\omega _{0}}{2\pi }}$ denotes the Chern
polynomial of the Levi curvature and $\omega _{0}$ is the uniquely
determined global real $1$-form (see Section~\ref{s-gue150508bI} and Section~%
\ref{s-gue150508d} for the precise definitions).

Our second main result turns to any (abstract) CR manifold $X$ with (locally free) $S^{1}$ action (but with no
assumption on pseudoconvexity); we see that the above Euler characteristic has an index
interpretation related to $\overline{\partial }_{b}+\overline{\partial }%
_{b}^{\ast }$ on $X$ (see 
\eqref{e-gue150524f} and \eqref{e-gue150524fI}).   We 
are able to establish \eqref{e-gue150508I} (cf. Corollary~\ref{c-gue150508I})
on such $X$ based on our
asymptotic expansion for the heat kernel $e^{-t\Box _{b,m}^{(q)}}(x,x)$ and a type of 
\emph{McKean-Singer} formula on $X$ (see Corollaries \ref{t-gue150603} and \ref{t-gue150630w}).

As an application to complex (orbifold) geometry, it is worth noting a comparison between the present result and
a result of Kawasaki on Hirzebruch-Riemann-Roch
theorem (HRR theorem for short) on a complex orbifold $N$ (\cite{Ka79}) (which plays the role
as our $M$ above).  Our formula
simplifies in the sense that in the original formula of Kawasaki, his part involving the dependence on the
(lower dimensional) strata
of the orbifold $M$ entirely drops out here, at least for the class of orbifolds and orbifold line bundles that fit into
our assumption (see Theorem ~\ref{t-gue150802},  remarks following it and Subsection~\ref{s-gue150723sub1}
for examples).    In our view this simplification
does not appear obvious at all within the original approach of Kawasaki because by
his approach the contributions from the (lower dimensional) strata of the orbifold cannot be avoided 
(unless it is proved to be vanishing) even if
the total space of the (orbifold) circle bundle is smooth.   Conceptually speaking one may attribute such
a simplification to one's working on the entire (smooth) $X$ rather than on the downstairs $M$ (as Kawasaki), a
strategy already employed for the asymptotic expansion above and proving useful again in this 
context of (CR) index theorem.
We remark that the vanishing of the contribution of strata also occurs
in a related context studied by these works \cite{PV}, \cite{F} (see also discussions after Theorem ~\ref{t-gue150802}).

In short our second main result (Theorem \ref{t-gue160307}) computes a
local index density in terms of \emph{tangential} characteristic forms, which is to 
show that certain non-trivial contributions (cf. $t^{-n}e^{-%
\frac{\varepsilon_0\hat{d}(x,X_{\mathrm{sing\,}})^{2}}{t}}$ of \eqref{e-gue151108I}) in the heat kernel
expansion will eventually cancel out in the index density computation.  
We can do this by applying Getzler's
rescaling technique to the {\it off-diagonal estimate}
(not needed in the classical index theorems).   As, to the best of our knowledge, an appropriate term
for such a result about the local density 
hasn't appeared in the literature yet, we shall follow the classical
cases and call it a {\it local index
theorem} on these CR manifolds (Corollary~\ref{c-gue150508I}), including \emph{%
Sasakian manifolds} of interest in String Theory.

With reference to the questions in the beginning of this Introduction, 
for further application of our results to CR geometry 
it is important to produce many CR functions or CR sections.
Namely we hope to know when $H_{b}^{0}(X,E)$ or $H_{b,m}^{0}(X,E)$ is large
(see Questions~\ref{q-gue150704}, ~\ref{q-gue150704I} and ~%
\ref{q-gue150704II} in Section~\ref{s-gue150704}).
Progress towards this circle of questions seems limited (Section ~\ref{s-gue150704}).
We can now develop a tool for tackling some of these questions.
The idea here is to combine our version of CR
index theorem with a sort of vanishing theorem for higher cohomology groups, which
is intimately related to a version of Grauert-Riemenschneider criterion adapted to
the CR case.  This methodology turns out to be effective for those CR manifolds studied in this paper.

In Section~\ref{s-gue150704} we apply our CR index theorem to prove a CR
version of Grauert-Riemenschneider criterion, and produce many CR functions
on a weakly pseudoconvex CR manifold with transversal $S^1$ action and many
CR sections on some class of CR manifolds, which give answers to some long-standing
questions in several complex variables and CR geometry. In Section~\ref%
{s-gue150723II} we provide an abundance of examples of those CR manifolds studied in the
present paper, some of which arise
from Brieskorn manifolds (generalized Hopf manifolds).  

There is another index theory of geometric significance, developed by
Charles Epstein. He studied the so called relative index of a pair of embeddable
CR structures through their Szeg\"o projectors in a series of papers (see~%
\cite{Ep98I},~\cite{Ep98II}, ~\cite{Eps06},~\cite{Ep08} and~\cite{Eps12}).
On the other hand, Erik van Erp derived an index formula for subelliptic
operators on a contact manifold (see~\cite{Erp1},~\cite{Erp2}). Moreover,
recent work of Paradan and Vergne~\cite{PV} gave an expression for the index
of transversally elliptic operators which is an integral of compactly
supported equivariant form on the cotangent bundle; see also Fitzpatrick~%
\cite{F} for related directions. Br\"uning, Kamber and Richardson~\cite{BKR}, \cite{BKR08}
computed the index of an equivariant transversally elliptic operator as a
sum of integrals over blow ups of the strata of the group action.

Finally it is natural to ask for a generalization from the action of $S^1$ to that of other Lie groups
or even to foliations (cf. Subsection~\ref{s-gue160416s0} for references).  
As we will discuss in Section~\ref{s-gue160416}, the asymptotic expansion in 
the form \eqref{e-gue151108II} as indicated there a sort of remedy for \eqref{e-gue151108I}
by involving a ``distance function" $\hat d$ seems to be best illustrated in the $S^1$ case.
It appears also conceivable that these features shall be preserved for generalization in a certain 
(as yet unknown) way.    This paper may be presented or read in token of a prototype for 
further study into much more complicated, diversified situations.   

\bigskip

\subsection{Main theorems}

\label{s-gue150508a}

We shall now formulate the main results. We refer to Section~\ref{s-gue150508bI}
and Section~\ref{s-gue150508d} for some notations and terminologies used here.
After the background material, we will discuss in the sequel i) asymptotic expansions, ii) a local index theorem
and iii) trace integrals.  

\subsubsection{Background}
\label{s-gue150508s1}
Let $(X, T^{1,0}X)$ be a compact connected CR manifold with a transversal CR
locally free $S^1$ action $e^{-i\theta}$, where $T^{1,0}X$ is a CR structure
of $X$.  $X$ is of dimension $2n+1$ throughout this paper. 

Let $T\in C^\infty(X,TX)$ be the real vector field induced by the $S^1$
action and let $\omega_0\in C^\infty(X,T^*X)$ be the global real one form
determined by $\langle\,\omega_0\,,\,T\,\rangle=1$, $\langle\,\omega_0\,,\,u%
\,\rangle=0$, for every $u\in T^{1,0}X\oplus T^{0,1}X$.

Associated with the $S^1$ action of $X$ it is natural to consider various geometric objects 
admitting an $S^1$ action.   In the following,  to streamline the exposition we shall freely 
use the notion of {\it rigid} objects: 
``rigid bundles", ``rigid metrics" etc., and refer to Definitions~\ref{d-gue50508d}, ~\ref{d-gue150508dI}
and ~\ref%
{d-gue150514f} for the precise meanings.  (See also the work of 
Baouendi-Rothschild-Treves \cite[Definition II.2] {BRT85} for a similar use of 
this term.)  It suffices to say here that this notion of rigid objects 
is nothing but an equivalent way (by using metric)
to consider objects (originally defined without assumption on metric) 
which admit (compatible) $S^1$ actions (or $S^1$ invariance,
subject to the proper context)
provided one starts with a CR manifold with an $S^1$ action (cf. Theorem~\ref{t-gue150514fa}).

Henceforth let $E$ be a rigid
CR vector bundle over $X$, equipped with a
rigid Hermitian metric $\langle\,\cdot\,|\,\cdot\,\rangle_E$.
We note that $T^{1,0}X$ is known to be a rigid
complex vector bundle (see the work of Baouendi-Rothschild-Treves~\cite{BRT85})
with a rigid
Hermitian metric $\langle\,\cdot\,|\,\cdot\,\rangle$ satisfying 
extra properties (not specified here, cf. \eqref{e-gue22a1}).
Let $\langle\,\cdot\,|\,\cdot\,\rangle_E$
be the Hermitian metric on $T^{*0,\bullet}X\otimes E$ induced by those
on $E$ and $\mathbb{C }TX$.   Denoting by $dv_X=dv_X(x)$
the volume form on $X$ induced by the Hermitian metric $\langle\,\cdot\,|\,%
\cdot\,\rangle$ on $\mathbb{C }TX$ we get the natural global $L^2$ inner
product $(\,\cdot\,|\,\cdot\,)_{E}$ on $\Omega^{0,\bullet}(X,E)$.

As remarked in Introduction, for $u\in\Omega^{0,\bullet}(X,E)$, $Tu\in\Omega^{0,%
\bullet}(X,E)$ is defined and $T\overline\partial_b=\overline\partial_bT$. For $%
m\in\mathbb{Z}$, put
\begin{equation*}
\begin{split}
\Omega^{0,\bullet}_m(X,E):&=\left\{u\in\Omega^{0,\bullet}(X,E);\,
Tu=-imu\right\} \\
&=\left\{u\in\Omega^{0,\bullet}(X,E);\, (e^{-i\theta})^*u=e^{-im\theta}u,\ \
\forall\theta\in[0,2\pi[\right\},
\end{split}%
\end{equation*}
where $(e^{-i\theta})^*$ denotes the pull-back by 
the map $e^{-i\theta}: X\to X$ of $S^1$ action.  

Write $L^{2}(X,T^{*0,\bullet}X\otimes E)$ (resp.  $%
L^{2}_m(X,T^{*0,\bullet}X\otimes E)$)
for the $L_2$-completion of $%
\Omega^{0,\bullet}(X,E)$ (resp.  $\Omega^{0,%
\bullet}_m(X,E)$) with respect to $(\,\cdot\,|\,\cdot\,)_{E}$.

By $T\overline\partial_b=\overline\partial_bT$ one defines
$\overline\partial_{b,m}:\Omega^{0,\bullet}_m(X,E)%
\rightarrow\Omega^{0,\bullet}_m(X,E)$ as the restriction 
of $\overline\partial_b$ on $\Omega_m^{0,\bullet}$. 
Write
\begin{equation*}
\overline{\partial}^{*}_b:\Omega^{0,\bullet}(X,E)\rightarrow\Omega^{0,%
\bullet}(X,E),\quad \mbox{resp.\,\,\,$\overline\partial_{b,m}^*:\Omega^{0,\bullet}_m(X,E)%
\rightarrow\Omega^{0,\bullet}_m(X,E)$}
\end{equation*}
for the formal adjoint of $\overline\partial_b$ (under $%
(\,\cdot\,|\,\cdot\,)_E$), resp. $\overline\partial_{b,m}$. Since $\langle\,\cdot\,|\,\cdot\,\rangle_E$ and $%
\langle\,\cdot\,|\,\cdot\,\rangle$ are rigid, one sees 
\begin{equation}  \label{e-gue150517i}
\begin{split}
&T\overline\partial^{*}_b=\overline\partial^{*}_bT\ \ \mbox{on
$\Omega^{0,\bullet}(X,E)$}, \\
&\overline\partial^{*}_{b,m}=\overline\partial^{*}_b|_{\Omega^{0,\bullet}_m}:\Omega^{0,%
\bullet}_m(X,E)\rightarrow\Omega^{0,\bullet}_m(X,E),\ \ \forall m\in\mathbb{Z%
}.
\end{split}%
\end{equation}

Let $A_m:\Omega^{0,\bullet}_m(X,E)\rightarrow\Omega^{0,\bullet}_m(X,E)$ be a certain 
smooth zeroth order operator with $TA_m=A_mT$ and $A_m:\Omega^{0,\pm}_m(X,E)%
\rightarrow\Omega^{0,\mp}_m(X,E)$ (arising from a CR version of 
$\mathrm{Spin}^c$ Dirac operator, cf. Definition~\ref{d-gue150524}). Put
\begin{equation}  \label{e-gue151113}
\widetilde D_{b,m}:=\overline\partial_{b,m}+\overline{\partial}%
^*_{b,m}+A_m:\Omega^{0,\bullet}(X,E)\rightarrow\Omega^{0,\bullet}(X,E)
\end{equation}
and let
\begin{equation}  \label{e-suVIIIa}
\widetilde
D^{*}_{b,m}:\Omega^{0,\bullet}(X,E)\rightarrow\Omega^{0,\bullet}(X,E)
\end{equation}
be the formal adjoint of $\widetilde D_{b,m}$ (with respect to $%
(\,\cdot\,|\,\cdot\,)_E$). 

We have $\widetilde\Box_{b,m}$, given by
\begin{equation}  \label{e-gue151113y}
\widetilde\Box_{b,m}:=\widetilde D^*_{b,m}\widetilde
D_{b,m}:\Omega^{0,\bullet}(X,E)\rightarrow\Omega^{0,\bullet}(X,E)
\end{equation}
which denotes the $m$-th {\it modified} Kohn Laplacian, thought of as $\mathrm{Spin}^c$ Kohn Laplacian 
(cf. Definition~\ref{d-gue150524} and the paragraph below it).  
We extend $\widetilde\Box_{b,m}$ by
\begin{equation}  \label{e-gue151113yI}
\widetilde\Box_{b,m}:\mathrm{Dom\,}\widetilde\Box_{b,m}\,(\subset
L^{2}_m(X,T^{*0,\bullet}X\otimes E))\rightarrow
L^{2}_m(X,T^{*0,\bullet}X\otimes E)\,,
\end{equation}
with $\mathrm{Dom\,}\widetilde\Box_{b,m}:=\{u\in
L^{2}_m(X,T^{*0,\bullet}X\otimes E);\, \widetilde\Box_{b,m}u\in
L^{2}_m(X,T^{*0,\bullet}X\otimes E)\}$ in which $\widetilde\Box_{b,m}u$ is defined in
the sense of distribution.

We will show in Section~\ref{s-gue150517} that $\widetilde\Box_{b,m}$ is
self-adjoint, $\mathrm{Spec\,}\widetilde\Box_{b,m}$ is a discrete subset of
$[0,\infty[$ and for $\nu\in\mathrm{Spec\,}\widetilde\Box_{b,m}$, $\nu$
is an eigenvalue of $\widetilde\Box_{b,m}$ with finite multiplicities $d_\nu<\infty$.  
Let $%
\left\{f^\nu_1,\ldots,f^\nu_{d_{\nu}}\right\}$ be an orthonormal frame for
the eigenspace of $\widetilde\Box_{b,m}$ with eigenvalue $\nu$. The (smooth) heat
kernel $e^{-t\widetilde\Box_{b,m}}(x,y)$ can be given by
\begin{equation}  \label{e-gue151023a}
e^{-t\widetilde\Box_{b,m}}(x,y)=\sum_{\nu\in\mathrm{Spec\,}%
\widetilde\Box_{b,m}}\sum^{d_{\nu}}_{j=1}e^{-\nu
t}f^\nu_j(x)\wedge(f^\nu_j(y))^\dagger,
\end{equation}
where $f^\nu_j(x)\wedge(f^\nu_j(y))^\dagger$ denotes the linear map:
\begin{equation*}
\begin{split}
f^\nu_j(x)\wedge(f^\nu_j(y))^\dagger:T^{*0,\bullet}_yX\otimes
E_y&\rightarrow T^{*0,\bullet}_xX\otimes E_x, \\
u(y)\in T^{*0,\bullet}_yX\otimes E_y&\rightarrow
f^\nu_j(x)\langle\,u(y)\,|\,f^\nu_j(y)\,\rangle_E\in
T^{*0,\bullet}_xX\otimes E_x.
\end{split}%
\end{equation*}
Let $e^{-t\widetilde\Box_{b,m}}:L^2(X,T^{*0,\bullet}X\otimes E)\rightarrow
L^2_m(X,T^{*0,\bullet}X\otimes E)$ be the (continuous) operator associated with the 
distribution kernel $e^{-t\widetilde\Box_{b,m}}(x,y)$.

Let $e_1(x),\ldots,e_d(x)$
be an orthonormal frame of $T^{*0,q}_xX\otimes E_x$ ($q=0,1,\ldots,n$), and $A\in\mathrm{End\,}%
(T^{*0,\bullet}_xX\otimes E_x)$. Put $\mathrm{Tr^{(q)}\,}A:=\sum^d_{j=1}%
\langle\, Ae_j\,|\,e_j\,\rangle_E$ and set
\begin{equation}  \label{e-gue160114bI}
\begin{split}
&\mathrm{Tr\,}A:=\sum^n_{j=0}\mathrm{Tr^{(j)}\,}A, \\
&\mathrm{STr\,}A:=\sum^n_{j=0}(-1)^j\mathrm{Tr^{(j)}\,}A.
\end{split}%
\end{equation}

Let $\nabla^{TX}$ be the Levi-Civita connection on $TX$ (with respect to $%
\langle\,\cdot\,|\,\cdot\,\rangle$).    Then $T^{1,0}X$ is equipped with a 
connection  $%
\nabla^{T^{1,0}X}:=P_{T^{1,0}X}\nabla^{TX}$ where $P_{T^{1,0}X}$ be the 
projection from $\mathbb{C }TX$ onto $T^{1,0}X$.

Let $\nabla^E$ be the connection on $E$ induced by $\langle\,\cdot\,|\,\cdot%
\,\rangle_E$ (see Theorem~\ref{t-gue150515}).  Let $\mathrm{Td_b\,}(\nabla^{T^{1,0}X},T^{1,0}X)$ denote the
representative of the tangential Todd class of $T^{1,0}X$ and $\mathrm{ch_b\,}(\nabla^{E},E)$ the
representative of the tangential Chern character of $E$ (see Section~\ref{s-gue150508d} for {\it tangential} classes).


In what follows we aim to define a {\it distance function} $\hat d$
which plays an important role (for the asymptotic expansion) in this paper.
For $x\in X$, we say that the period of $x$ is $\frac{2\pi}{\ell}$, $\ell\in%
\mathbb{N}$ provided that $e^{-i\theta}\circ x\neq x$ for every $0<\theta<\frac{2\pi}{%
\ell}$ and $e^{-i\frac{2\pi}{\ell}}\circ x=x$.  Put, for each $\ell\in\mathbb{N}$,
\begin{equation}  \label{e-gue150802b}
X_\ell=\left\{x\in X;\, \mbox{the period of $x$ is $\frac{2\pi}{\ell}$}%
\right\}
\end{equation}
and let $p=\min\left\{\ell\in\mathbb{N};\, X_\ell\neq\emptyset\right\}$.
We call $X_p=X_{p_1}$ the {\it principal stratum}.   It
is well-known that if $X$ is connected, then $X_p$ is an open and dense
subset of $X$ (see Proposition 1.24 in Meinrenken~\cite{Me03} and
Duistermaat-Heckman~\cite{Du82}). Assume $X=X_{p_1}\bigcup
X_{p_2}\bigcup\cdots\bigcup X_{p_k}$, $p=:p_1<p_2<\cdots<p_k$. Put $X_{%
\mathrm{sing\,}}=X^1_{\mathrm{sing\,}}:=\bigcup^k_{j=2}X_{p_j}$, $X^{r}_{%
\mathrm{sing\,}}:=\bigcup^k_{j=r+1}X_{p_j}$, $k-1\geq r\geq 1$. Set $X^{k}_{%
\mathrm{sing\,}}:=\emptyset$.  Note $p_1|p_j$ for $1\le j\le k$ (cf. Remark~\ref{r-gue150804}).  

Let $d(\cdot, \cdot)$ denotes the
standard Riemannian distance with respect to the given
Hermitian metric. Take $\zeta$
\begin{equation}  \label{e-gue160327}
0<\zeta<\inf\left\{\frac{2\pi}{p_k},\left\vert \frac{2\pi}{p_r}-\frac{2\pi}{
p_{r+1}}\right\vert, r=1,\ldots, k-1\right\}.
\end{equation}
Set, for $x\in X$ and $r=1,2,\ldots,k$,
\begin{equation}
\label{e-gue160327aIe1}
\hat d_\zeta(x,X^r_{\mathrm{sing\,}}):=\inf\left\{d(x,e^{-i\theta}x);\,
\zeta\leq\theta\leq\frac{2\pi}{p_r}-\zeta\right\}.
\end{equation}

This notation reflects the fact that $\hat d_{\zeta}(x, X^r_{\mathrm{sing\,}})$
is equivalent to the ordinary distance $d(x, X^r_{\mathrm{sing\,}})$
(see below).  
Note by definition $\hat d_\zeta(x,X^k_{\mathrm{sing\,}})\,(=\hat d_\zeta(x,\emptyset))> 0$ for 
all $x\in X$.   We remark 
that for any $0<\zeta, \zeta_1$ satisfying \eqref{e-gue160327}, $\hat
d_\zeta(x,X^r_{\mathrm{sing\,}})$ and $\hat d_{\zeta_1}(x,X^r_{\mathrm{sing\,%
}})$ are equivalent (as far as the estimate in Theorem~\ref{t-gue160114} below
is concerned). We shall denote $\hat d(x,X^r_{\mathrm{sing\,}}):=\hat
d_\zeta(x,X^r_{\mathrm{sing\,}})$. 

Remark that, 
by examining the definition $\hat d(x,X^r_{\mathrm{sing\,}})=0
$ if and only if $x\in X^r_{\mathrm{sing\,}}$.  Further, for $\varepsilon>0$
there is a $\delta>0$ such that $\hat d(x,X^r_{\mathrm{sing\,}})\geq\delta$
provided $x\in X$ satisfies (the ordinary distance) $d(x,X^r_{\mathrm{sing\,}})\geq\varepsilon$.
It is thus convenient to think of $\hat d(x,X^r_{\mathrm{sing\,}})$ as a distance function 
from $x$ to $X^r_{\mathrm{sing\,}}$.   

Indeed in Theorem~\ref{t-gue160413} 
for a strongly pseudoconvex $X$ there is a constant $C\geq 1$ such that
\begin{equation*}
\frac{1}{C}d(x,X^r_{\mathrm{sing\,}})\leq\hat d(x,X^r_{\mathrm{sing\,}})\leq
Cd(x,X^r_{\mathrm{sing\,}}),\ \ \forall x\in X.  
\end{equation*}

\subsubsection{Asymptotic expansion of the heat kernel $e^{-t\widetilde\Box_{b,m}}(x,x)$}
\label{s-gue150508s2}

With the distance function $\hat d$, we state the first main result of this paper (see Section~\ref{s-gue150702} for a proof).

\begin{thm}
\label{t-gue160114} Suppose $(X, T^{1,0}X)$ is a compact, connected CR
manifold (of dimension $2n+1$) with a transversal CR locally free $S^1$ action. With the notations
above, there exist $a_s(t,x)\,(=a_{s,m}(t, x))\,\in C^\infty(\mathbb{R}_+\times X,\mathrm{End\,}%
(T^{*,\bullet}X\otimes E))$ with $\left\vert a_s(t,x)\right\vert\leq C$
on $\mathbb{R}_+\times X$ where $C>0$ is independent of $t$, $%
s=n, n-1,\ldots$, such that
\begin{equation}  \label{e-gue160306}
e^{-t\widetilde\Box_{b,m}}(x,x)\sim
t^{-n}a_n(t,x)+t^{-n+1}a_{n-1}(t,x)+\cdots\ \ \mbox{as $t\To0^+$}.
\end{equation}
(See Definition~\ref{d-gue150608} for ``$\sim$".)

Moreover, there exist $\alpha_s(x)\, (=\alpha_{s, m}(x))\,\in C^\infty(X,\mathrm{End\,}%
(T^{*,\bullet}X\otimes E))$, $s=n, n-1,\ldots$, satisfying the following 
property.   Given any differential operator $P_\ell:C^\infty(X,T^{*,\bullet}X%
\otimes E)\rightarrow C^\infty(X,T^{*,\bullet}X\otimes E)$ of order $\ell\in%
\mathbb{N}_0$, there exist $\varepsilon_0>0$ and $C_0>0
$ such that
\begin{equation}  \label{e-gue160114}
\left\vert P_\ell\Bigr(a_s(t,x)-\bigl(\sum\limits^{p_r}_{s=1}e^{\frac{2\pi(s-1)}{%
p_r}mi}\bigr)\alpha_{s}(x)\Bigr)\right\vert\leq C_0t^{-\frac{\ell}{2}}e^{-\frac{%
\varepsilon_0\hat d(x,X^{r}_{\mathrm{sing\,}})^2}{t}},\ \forall t\in\mathbb{R%
}_+,\ \ \forall x\in X_{p_r}
\end{equation}
$r=1,\ldots,k$. 

\end{thm}

The following is immediate from the proof of Theorem~\ref{t-gue160114}.

\begin{cor}
\label{c-gue160306} Suppose $(X, T^{1,0}X)$ is a compact, connected CR
manifold with a transversal CR locally free $S^1$ action. With the notations
above, for any $r=1,\ldots,k$, any differential operator $%
P_\ell:C^\infty(X,T^{*,\bullet}X\otimes E)\rightarrow
C^\infty(X,T^{*,\bullet}X\otimes E)$ of order $\ell\in\mathbb{N}_0$, every $%
N_0\in\mathbb{N}$ with $N_0\ge N_0(n)$ for some $N_0(n)$, there are $\varepsilon_0>0$, $\delta>0$ and $C_{N_0}>0$
such that
\begin{equation}  \label{e-gue160306I}
\begin{split}
&\left\vert P_\ell\Bigr(e^{-t\widetilde\Box_{b,m}}(x,x)-\bigl(\sum%
\limits^{p_r}_{s=1}e^{\frac{2\pi(s-1)}{p_r}mi}\bigr)\sum^{N_0}_{j=0}t^{-n+j}%
\alpha_{n-j}(x)\Bigr)\right\vert \\
&\leq C_{N_0}\Bigr(t^{-n+N_0+1-\frac{\ell}{2}}+t^{-n-\frac{\ell}{2}}e^{-\frac{\varepsilon_0\hat
d(x,X^{r}_{\mathrm{sing\,}})^2}{t}}\Bigr),\ \forall x\in X_{p_r},\ \forall\,
0<t<\delta.
\end{split}%
\end{equation}
\end{cor}

In the following we supplement these results with a number of remarks before going further. 

\begin{rem}
\label{r-gue150508} $(%
\sum\limits^{p_r}_{s=1}e^{\frac{2\pi(s-1)}{p_r}mi})=p_r$ if 
$p_r|m$; $(\sum\limits^{p_r}_{s=1}e^{\frac{%
2\pi(s-1)}{p_r}mi})=0$ if $p_r\not\vert m$.  
\end{rem}

\begin{rem}
\label{r-gue150508I}
We shall now see that if one wants an asymptotic expansion of $e^{-t\widetilde\Box_{b,m}}(x,x)$ to be valid 
around each $x\in X$ (cf. Definition~\ref{d-gue150608}), then 
\eqref{e-gue160306} is basically optimal (i.e. in general, $a_s(t, x)$ cannot be $t$-independent
for all $s$).  
For $U$ open with $\overline U\subset X_{p_1}$, a tradition-like formula (assuming $p_1=1$ for simplicity)
\begin{equation}
\label{e-gue150508eI}
e^{-t\widetilde\Box_{b,m}}(x,x)\sim C\big(
t^{-n}\alpha_n(x)+t^{-n+1}\alpha_{n-1}(x)+\cdots\big)
\end{equation}
is valid for $x\in U$ and $C=1$ (as follows from \eqref{e-gue160114} for $l=0$) whereas for $x\in X_{p_r}$, $r\ge 2$, 
an asymptotic expansion (for $p_r|m$) with $C=p_r$
is valid around an open subset ($\ni x$) of the stratum $X_{p_r}$. 
Since $e^{-t\widetilde\Box_{b,m}}(x,y)$ is going to be a well defined smooth kernel, 
it is easily seen that those functions $\alpha_s(x)$ ($s=n, n-1, \cdots$) satisfying Theorem~\ref{t-gue160114} are unique
(if they exist).  (We notice that $a_s(t, x)$ in \eqref{e-gue160306} are not 
canonically defined by our method which is subject to choice of 
BRT trivializations, cf. \eqref{e-gue150626fIII} and Subsection~\ref{s-gue150514}.) 
In short, the above suggests that an asymptotic expansion of the form as \eqref{e-gue150508eI} 
can only be true {\it in the piecewise sense} with respect to strata.   See also Subsection~\ref{s-gue160416s0}.

To confirm this, one uses Theorem~\ref{t-gue160416} (see Theorems~\ref{t-gue160416d} 
and ~\ref{t-gue160416dc}
for a more precise version) by noting 
$\int_X \mathrm{Tr}\alpha^+_s(x)dv_X(x)=S_{1, s}^+$ in Theorem~\ref{t-gue160416d}
(which is taking the ``even" part of the Laplacian).
Hence one can interpret the trace integral result 
(obtained by integrating $\mathrm{Tr}e^{-t\widetilde\Box_{b,m}}(x,x)$ over $X$) 
as one that gives extra nonzero {\it correction terms}, cf. the second line in \eqref{e-gue160416II} or
the third line in \eqref{e-gue160417wII}.  



It follows that if there exists a global asymptotic expansion (not just in the piecewise sense) such as \eqref{e-gue160306},
then not all of $a_s(t, x)$ can be independent of $t$.   Otherwise, if all $a_s(t,x)$ are independent of $t$, 
it would be of the form \eqref{e-gue150508eI}
globally by assumption ($C=1$ if $p_1=1$), so 
by integrating the trace over $X$, there would be no correction terms as discussed above.  
To say more, $e^{-t\widetilde\Box_{b,m}}(x,x)$ cannot have any asymptotic expansion 
of the form $t^{m_1}\beta_{m_1}(x)+t^{m_2}\beta_{m_2}(x)+\cdots$ (globally)
$m_1<m_2<,\cdots\in \mathbb{R}, \beta_{m_1}(x), \beta_{m_2}(x),\cdots$ continuous functions on $X$. 
Otherwise by equating it to \eqref{e-gue160306}, 
each $a_s(t, x)$ would be rendered independent of $t$, absurd as just remarked
(see the next remark for argument independent of Theorems~\ref{t-gue160416}, \ref{t-gue160416d}).  

The next remark shows that $a_s(t, x)$ for the particular $s=n$ 
must be dependent on $t$ (nontrivially).    This part will not use 
Theorem~\ref{t-gue160416}.  
\end{rem}

\begin{rem}
\label{r-gue160414a}   In the above remark a certain discontinuity 
in the form \eqref{e-gue150508eI} for, say $x\in X_{p_1}$ and 
$x\in X_{p_2}$ seems to appear.   We shall now explore it.  If the (Gaussian-like) term 
to the right of \eqref{e-gue160114} is examined, it  
arises from a precise integral (see \eqref{e-gue160125V}).
To show that this integral is generally nontrivial, regardless of whether 
our estimate given by \eqref{e-gue160114} is a fine or crude one, 
we are actually going to show that the term for $s=n$ in \eqref{e-gue160114}
\begin{equation*}
a_n(t,x)-\sum\limits^{p_r}_{s=1}e^{\frac{2\pi(s-1)}{p_r}mi}\alpha_{n}(x)
\end{equation*}
is nontrivial. For the sake of illustration we assume that $X=X_1\bigcup X_2$, that is $%
p_1=1$ and $p_2=2$, and take $m$ to be an even number.  For $x\in X_1$, 
by \eqref{e-gue160114} (for $l$=0) and $p_1=1$, we see that
$a_n(t,x)=\alpha_n(x)+r_n(t,x)$ 
\begin{equation}
\label{e-gue160414c}
\left\vert
r_n(t,x)\right\vert\lesssim e^{-\frac{\varepsilon_0\hat d(x,X_{\mathrm{sing\,%
}})^2}{t}}.
\end{equation}  As our $\alpha_n(x)$ essentially arises from a 
local Kodaira Laplacian (see \eqref{e-gue160224I}, similar to discussion after
Theorem~\ref{t-gue150508}), it is well known that $\alpha_n(x)$, as the coefficient of the leading term 
(in the $t$-expansion of the heat kernel for Kodaira Laplacian), is constant in $x$ 
with $\mathrm{Tr}\,\alpha_n>0$ (cf. \cite[Lemma 4.1.4 and Section 4.4] {Gi95}).      
By continuity  ($a_s$ and $\alpha_s$ being globally continuous functions) 
\begin{equation}
\label{e-gue160414bb}
a_n(t,x)=\alpha_n(x)+r_n(t,x)
\end{equation}
remains true on $X_2$.    For $x_0\in X_2$, the estimate of \eqref{e-gue160114} is given by %
(with $p_2=2$ and discussion after \eqref{e-gue160327aIe1} for $\hat d(x_0, \emptyset)>0$)
\begin{equation}  \label{e-gue160414b}
a_n(t,x_0)=2\alpha_n(x_0)+O(t^{\infty}).
\end{equation}
By \eqref{e-gue160414bb} and \eqref{e-gue160414b} it follows 
$r_n(t,x_0)=\alpha_n(x_0)+O(t^{\infty})$ so $r_n(t,x)\approx \alpha_n(x)$
around $x_0$ as $t\to 0$, giving $|r_n(t, x)|\ge \epsilon>0$ nearby $x_0$
for some constant $\epsilon$ independent of $x$ and $t$.   But this would be absurd by \eqref{e-gue160414c}
if $r_n$ were independent of $t$ (taking $x\,(\in X_1,\ne x_0)$ near $x_0$ so that $|r_n(t, x)|\ge \epsilon$
and letting $t\to 0$ in \eqref{e-gue160414c}). 
Hence $a_n(t, x)$ cannot be independent of $t$ either, as desired.  
\end{rem}

\begin{rem}
\label{r-gue160414ar} To discuss the estimate \eqref{e-gue160306I},
let's take $\hat d$ in \eqref{e-gue160306I} to be $d$ for convenience (as remarked previously $\hat d$ is equivalent to the ordinary distance function $d$ at least in the strongly pseudoconvex case, cf. Theorem~\ref{t-gue160413}).  
Take $P_l=\mathrm{id}$ (so $l=0$).  The term to the rightmost of \eqref{e-gue160306I} appears as a Gaussianlike term. 
As $t\to 0$, this term tends to a sort of Dirac delta function supported along the strata $X^r_{\mathrm{sing}}$
(with an extra singular factor $t^{-\frac{a-1}{2}}$, $a=\mathrm{dim}X^r_{\mathrm{sing}}$).  
This may conceptually explain the piecewise continuity nature just discussed in Remarks~\ref{r-gue150508I} 
and \ref{r-gue160414a} if the asymptotic expansion is to be expressed in something, without $t$-dependence,
such as $\alpha_s(x)$.    Conversely, the estimate as \eqref{e-gue160306I} involving a type of Dirac delta function
is conceptually reasonable under the piecewise continuity phenomenon 
in terms of $\alpha_s(x)$.   For more about this, some quantitative information may be available by Theorems~\ref{t-gue160416}, ~\ref{t-gue160416d} and \ref{t-gue160416dc}.  
\end{rem}

\begin{rem}
\label{r-gue160413a} We make a short comment on the coefficients $a_s(t, x)$ or $\alpha_j(x)$
in \eqref{e-gue160114} (the difference between $a_s(t, x)$ and 
$p_1\alpha_s(x)$ (at a given $x\in X_{p_1}$) is $O(t^{\infty})$ by \eqref{e-gue160114}; this is partly explained 
conceptually right below).  
For the standard (elliptic) case (of Dirac type)  it is well-known that
the coefficients of a heat kernel along the diagonal (by taking trace) are expressible in terms of the curvature and
its covariant derivatives (e.g. \cite{Gi95}).    In our transversally elliptic case (without bundle
$E$ for simplicity)
if $S^1$ action is globally free, it follows from the standard case above (cf. \eqref{e-gue150923bi}-\eqref{e-gue151108})
that these coefficients of the (transversal) heat kernel
are expressible in terms of the {\it tangential} curvature (and its derivatives) (cf. Section ~\ref{s-gue150508d}).    
In the locally free case the same results can be
achieved in view of the proof of Theorem~\ref{t-gue160114}, which basically arises from
a procedure of patching and successive approximations based on the local (transversal) heat kernels that give
the asymptotic approximations of the final (transversal) heat kernel (see Section~\ref{s-gue150920}
for details of an outline).
Since the local kernels can be so expressed as just said (at least on
the principal stratum), it follows from the asymptotic approximation (e.g.  Theorems 2.23 and 2.30
of \cite{BGV92} or Theorem~\ref{t-gue150630I} in our case) that the same (expression in tangential curvature and its 
derivatives) can be said for
the global kernel (on the principal stratum then followed by continuous extension
of this global kernel on $X$).   It is also of interest to consider the integral version of 
these coefficients, which is the topic of Section~\ref{s-gue160416} of this paper.  

\end{rem}

\subsubsection{A local index theorem for CR manifolds with $S^1$ action} 
\label{s-gue150508s3}

Here we discuss issues related to the index theorem we will prove.  
We recall that the term to the left of the inequality in \eqref{e-gue160114}
is basically nontrivial by Remark~\ref{r-gue160414a}.  In our formulation of index theorems, 
the contribution arising from such a term is expected to be removed.  
This can be done when
$\widetilde\Box_{b,m}$ is the $\mathrm{Spin}^c$ Kohn Laplacian (cf. \eqref{e-gue150525aII}).  
In this case, we show that 
taking supertrace in \eqref{e-gue160306I} ($P_l=\mathrm{id}$) and 
applying Getzler's rescaling technique to the off-diagonal estimate (see Subsection~\ref{s-gue151025}
for more) yield that the singular part $t^{-n}$ to the
rightmost of \eqref{e-gue160306I} can be removed
(see Subsection~\ref{s-gue151025f} and Section~\ref{s-gue150702} for a proof).
More precisely

\begin{thm}
\label{t-gue160307} Suppose $(X, T^{1,0}X)$ is a compact, connected CR
manifold with a transversal CR locally free $S^1$ action. With the notations
above, if $\widetilde\Box_{b,m}$ is the $\mathrm{Spin}^c$ Kohn Laplacian
(see \eqref{e-gue150525aII}), then for $r=1,\ldots,k$ and every $N_0\in
\mathbb{N}$ with $N_0\ge N_0(n)$ for some $N_0(n)$, 
there are $\varepsilon_0>0$, $\delta>0$ and $C_{N_0}>0$ such
that ($\mathrm{STr}$ denoting supertrace, cf. \eqref{e-gue160114bI})
\begin{equation}  \label{e-gue160224}
\begin{split}
&\left\vert \mathrm{STr\,}e^{-t\widetilde\Box_{b,m}}(x,x)-\bigl(\sum%
\limits^{p_r}_{s=1}e^{\frac{2\pi(s-1)}{p_r}mi}\bigr)\sum^{N_0}_{j=0}t^{-n+j}%
\mathrm{STr\,}\alpha_{n-j}(x)\Bigr)\right\vert \\
&\leq C_{N_0}\Bigr(t^{-n+N_0+1}+e^{-\frac{\varepsilon_0\hat d(x,X^r_{\mathrm{%
sing\,}})^2}{t}}\Bigr),\ \ \forall \,0<t<\delta,\ \ \forall x\in X_{p_r},
\end{split}%
\end{equation}
and
\begin{equation}  \label{e-gue160114I}
\begin{split}
&\sum^n_{\ell=0}t^{-\ell}\mathrm{STr\,}\alpha_\ell(x)dv_X(x) \\
&=\frac{1}{2\pi}\left[\mathrm{Td_b\,}(\nabla^{T^{1,0}X},T^{1,0}X)\wedge%
\mathrm{ch_b\,}(\nabla^{E},E)\wedge e^{-m\frac{d\omega_0}{2\pi}%
}\wedge\omega_0\right]_{2n+1}(x)
\end{split}%
\end{equation}
where $[...]|_{2n+1}$ to the right denotes the part of $(2n+1)$-form.  
\end{thm}

As $\mathrm{Spin}^c$ objects can be simplified in the
K\"ahler case, so can the  $\mathrm{Spin}^c$ Kohn Laplacian in the
{\it CR K\"ahler case}, to which we turn now.

\begin{defn}
\label{d-gue160308} We say that $X$ is CR K\"ahler if there is a closed form
$\Theta\in C^\infty(X,T^{*1,1}X)$ such that $\Theta(Z,\overline Z)>0$, for
all $Z\in C^\infty(X,T^{1,0}X)$. We call $\Theta$ a CR K\"ahler form on $X$.
\end{defn}

When $X$ is a strongly pseudoconvex CR manifold with a transversal CR locally
free $S^1$ action, the closed form $d\omega_0$ satisfies $%
d\omega_0(Z,\overline Z)>0$, for all $Z\in C^\infty(X,T^{1,0}X)$.  Hence $X
$ is CR K\"ahler. 

A {\it quasi-regular Sasakian} manifold is also a CR
K\"ahler manifold.  We recall that for a compact smooth manifold $X$ of $%
\mathrm{dim} X=2n+1, n\geq1$, the triple $(X, g,\alpha)$ where $g$ is a
Riemannian metric and $\alpha$ is a real 1-form is called a Sasakian manifold if
the cone $\mathcal{C}(X)=\{(x,t)\in X\times \mathbb{R}_{>0}\}$ is a K\"ahler
manifold with complex struture $J$ and K\"ahler form $t^2d\alpha+2tdt\wedge%
\alpha$ compatible with the metric $t^2g +dt\otimes dt$ (see \cite%
{Bl76}, \cite{BG08}, \cite{OV07}). As a consequence, $X$ is a compact
strongly pseudoconvex CR manifold and the Reeb vector field $\xi$, defined
by $\alpha(\cdot)=g(\xi,\cdot)$, induces a transversal CR $\mathbb{R}$ action 
on $X$. If the orbits of this $\mathbb{R}$ action are compact, the
Sasakian structure is called quasi-regular. In this case, the Reeb vector
field generates a locally free transversal CR $S^1$ action on $X$. We 
can thus identify a compact quasi-regular Sasakian manifold with a compact
strongly pseudoconvex CR manifold $(X,T^{1,0}X)$ equipped with a transversal CR
locally free $S^1$ action such that the induced vector field of the $S^1$ action
coincides with the Reeb vector field on $X$ (see \cite{OV06}, \cite{OV07}).

Let $X$ be a CR K\"ahler manifold with a transversal CR locally free $S^1$
action. If $\langle\,\cdot\,|\,\cdot\,\rangle$ is induced by a CR K\"ahler
form on $X$, then $\Box_{b,m}$ {\it is equal to} the $\mathrm{Spin}^c$ Kohn
Laplacian.   By Theorem~\ref{t-gue160307}, we immediately obtain a 
version of {\it local index theorem} on CR K\"ahler manifolds with transversal CR locally
free $S^1$ action (which include the compact quasi-regular Sasakian manifolds
as a special case by above).   These results are discussed below. 

For a proof of the following, see the beginning of Subsection~\ref{s-gue151025f}
and the discussion leading to Proposition~\ref{t-gue150627}):

\begin{cor}
\label{t-gue160308} (CR K\"ahler case of Theorem~\ref{t-gue160307}) 
Suppose $(X, T^{1,0}X)$ is a compact, connected CR
K\"ahler manifold with a transversal CR locally free $S^1$ action and assume
that $\langle\,\cdot\,|\,\cdot\,\rangle$ is induced by a CR K\"ahler form on
$X$. With the notations above, for $r=1,\ldots,k$ and every $N_0\in%
\mathbb{N}$ with $N_0\ge N_0(n)$ for some $N_0(n)$, there are $\varepsilon_0>0$, $\delta>0$ and $C_{N_0}>0$ such
that
\begin{equation}  \label{e-gue160308}
\begin{split}
&\left\vert \mathrm{STr\,}e^{-t\Box_{b,m}}(x,x)-\sum\limits^{p_r}_{s=1}e^{%
\frac{2\pi(s-1)}{p_r}mi}\sum^{N_0}_{j=0}t^{-n+j}\mathrm{STr\,}\alpha_{n-j}(x)%
\Bigr)\right\vert \\
&\leq C_{N_0}\Bigr(t^{-n+N_0+1}+e^{-\frac{\varepsilon_0\hat d(x,X^r_{\mathrm{%
sing\,}})^2}{t}}\Bigr),\ \ \forall \,0<t<\delta,\ \ \forall x\in X_{p_r},
\end{split}%
\end{equation}
and
\begin{equation}  \label{e-gue160308I}
\begin{split}
&\sum^n_{\ell=0}t^{-\ell}\mathrm{STr\,}\alpha_\ell(x)dv_X(x) \\
&=\frac{1}{2\pi}\left[\mathrm{Td_b\,}(\nabla^{T^{1,0}X},T^{1,0}X)\wedge%
\mathrm{ch_b\,}(\nabla^{E},E)\wedge e^{-m\frac{d\omega_0}{2\pi}%
}\wedge\omega_0\right]_{2n+1}(x).
\end{split}%
\end{equation}
\end{cor}


We are in a position to state an index theorem
(including a local index theorem in the CR K\"ahler case).
Recall 
$\overline\partial_{b,m}:=\overline\partial_b:\Omega^{0,q}_m(X,E)\rightarrow%
\Omega^{0,q+1}_m(X,E), m\in\mathbb{Z}$, 
and a $\overline\partial_{b,m}$-complex:
\begin{equation*}
\overline\partial_{b,m}:\cdots\rightarrow\Omega^{0,q-1}_m(X,E)\rightarrow%
\Omega^{0,q}_m(X,E)\rightarrow\Omega^{0,q+1}_m(X,E)\rightarrow\cdots.
\end{equation*}
The $q$-th $\overline\partial_{b,m}$ Kohn-Rossi cohomology
group (regarded as the $m$-th Fourier compoment of the ordinary $q$-th Kohn-Rossi cohomology group)
is
\begin{equation*}
H^q_{b,m}(X,E):=\frac{\mathrm{Ker\,}\overline\partial_{b,m}:%
\Omega^{0,q}_m(X,E)\rightarrow\Omega^{0,q+1}_m(X,E)}{\mathrm{Im\,}%
\overline\partial_{b,m}:\Omega^{0,q-1}_m(X,E)\rightarrow\Omega^{0,q}_m(X,E)}.
\end{equation*}

We will prove in Theorem~\ref{t-gue150517II} that there holds $\mathrm{dim\,}H^q_{b,m}(X,E)<\infty$ (for each $m\in\mathbb{Z}$
and $q=0,1,2,\ldots,n$) without any Levi curvature
assumption.  

In Corollary~\ref{t-gue150603} (see
also Remark~\ref{r-gue151003}) we have a \emph{McKean-Singer type
formula} in our CR case: for every $t>0$,
\begin{equation}  \label{e-gue160114rq}
\sum^n_{j=0}(-1)^j\mathrm{dim\,}H^j_{b,m}(X,E)=\int_X\mathrm{STr\,}e^{-t%
\widetilde{\Box}_{b,m}}(x,x)dv_X.
\end{equation}



Combining  \eqref{e-gue160114rq}, \eqref{e-gue160224} and \eqref{e-gue160114I}
and noting $e^{-\frac{\varepsilon_0\hat d(x,X_{\mathrm{sing\,}})^2}{t}}$
is bounded by $1$ and rapidly decays to $0$ for $x$ in the principal stratum as $t\to 0$, 
we conclude the following form of an index theorem
on our CR manifolds (see  Section~\ref{s-gue150508d} for the precise meanings
of $\mathrm{Td_b\,}(T^{1,0}X)$ and $\mathrm{ch_b\,}(E)$ below):

\begin{cor}
\label{c-gue150508I}  (CR Index Theorem, cf. Corollary~\ref{t-gue160226a}) Suppose $(X, T^{1,0}X)$ is a compact, connected CR
manifold with a transversal CR locally free $S^1$ action. Then
\begin{equation}  \label{e-gue150508b}
\begin{split}
&\sum^n_{j=0}(-1)^j\mathrm{dim\,}H^j_{b,m}(X,E) \\
&=(\sum^{p(=p_1)}_{s=1}e^{\frac{2\pi(s-1)}{p}mi})\frac{1}{2\pi}\int_X\mathrm{Td_b\,}%
(T^{1,0}X)\wedge\mathrm{ch_b\,}(E)\wedge e^{-m\frac{d\omega_0}{2\pi}%
}\wedge\omega_0,
\end{split}%
\end{equation}
where $\mathrm{Td_b\,}(T^{1,0}X)$ denotes the tangential Todd class of $%
T^{1,0}X$ and $\mathrm{ch_b\,}(E)$ denotes the tangential Chern character of
$E$.
\end{cor}

For a connection with other works on index theorems by different formulations and methods, 
we refer to comments that come after Theorem~\ref{t-gue150802} and to the sixth paragraph in 
Subsection~\ref{s-gue160416s0}.  
 
\subsubsection{Trace integrals in terms of geometry of the $S^1$ stratification}
\label{s-gue150508s4}
This is the third and the last topic of this paper.  
For some applications (e.g. see a natural connection 
with Remark~\ref{r-gue150508I}), one studies the asymptotic behavior of $%
\int_X\mathrm{Tr}\,e^{-t\widetilde\Box_{b,m}}(x,x)dv_X(x)$. 
More comments on the historical respect come in the beginning of Section~\ref{s-gue160416}. 
Suppose $M$ is an orbifold of
(real) dimension $k$ and $H(t,x,x)$ is the associated heat kernel on the
diagonal for the standard Laplacian on $M$. It is known in 2008 ~\cite{DGGW08} (for Laplacian 
on functions; see also Richardson~\cite[Theorem 3.5]{Ri10}) that
\begin{equation}  \label{e-gue160416}
\int_MH(t,x,x)dv_X(x)\sim t^{-\frac{k}{2}}a_k+t^{-\frac{k}{2}+\frac{1}{2}%
}a_{k-\frac{1}{2}}+t^{-\frac{k}{2}+1}a_{k-1}+t^{-\frac{k}{2}+\frac{3}{2}%
}a_{k-\frac{3}{2}}+\cdots,
\end{equation}
where $a_s\in\mathbb{R}$ is independent of $t$, $s=k,k-\frac{1}{2},k-1,\ldots
$.    A novelty is that apart from the overall $t^{-\frac{k}{2}}$ the expansion is a power series 
in $t^{\frac{1}{2}}$.  

By a strategy partly in connection with the proof of Theorem~\ref{t-gue160114}, 
we obtain an expansion of the trace integral similar to 
\eqref{e-gue160416} in spirit.  We find that in our case, the expansion 
is a power series still in {\it integral} power of $t$.   However, 
there appear various {\it corrections} (depending on $m$) supported on each stratum (cf. \eqref{e-gue160416cde13}
and \eqref{e-gue160416r2e2}) in contrast to the expansion in the globally free case (of $S^1$ action).

More precisely, we have (see 
Theorems~\ref{t-gue160416d} and ~\ref{t-gue160416dc} for more 
information and proof):

\begin{thm} (cf. Theorems~\ref{t-gue160416d}, ~\ref{t-gue160416dc})
\label{t-gue160416} With notations in Theorem~\ref{t-gue160114} and assumption that 
the $S^1$ action is locally free but not globally free,
let $e$ be the number (which is even) defined to be 
the minimum of the (real) codimensions of connected components $M$ of $X_{p_\ell}$ for all $\ell\ge 2$.  
For $s=n,n-1,\ldots$, we have 
\begin{equation}  \label{e-gue160416I}
\int_X\mathrm{Tr}\,a_{s,m}(t,x)dv_X(x)\sim q_{s,0}
+tq_{s,1}
+t^2q_{s, 2}\ldots\ \ \mbox{as $t\To0^+$},
\end{equation}
where $a_{s,m}(t,x)\,(=a_{s}(t, x))\,$ is as in \eqref{e-gue160306} and $q_{s, j}\in\mathbb{R}$ is
independent of $t$ (dependent on $m$ though), $j=0,1, 2,\cdots$.   
Similarly, as $t\To0^+$, 
\begin{equation}  \label{e-gue160416II}\begin{split}
\int_X\mathrm{Tr}\,e^{-t\widetilde\Box_{b,m}}(x,x)dv_X(x)\sim 
(\sum\limits^{p_1}_{s=1}e^{\frac{i2\pi(s-1)}{p_1}m})\big(t^{-n}c_{n}
&+t^{-n+1}c_{n-1}
+t^{-n+2}c_{n-2}+\cdots\big)\\
&\,\, +t^{-n+\frac{e}{2}}\tilde c_{n-\frac{e}{2}}+O(t^{-n+\frac{e}{2}+1}). 
\end{split}
\end{equation}
These coefficients satisfy the following.  For an $\ell\ge 2$, write $\{M_{\ell, \gamma_\ell}\}_{\gamma_\ell}$ 
(possibly empty for some $\ell$) for those connected 
components $M_{\ell,\gamma_\ell}$ of $X_{p_\ell}$ with the codimension 
$\mathrm{codim}\,M_{\ell, \gamma_\ell}=e$.
Set $S_{\ell, \gamma_\ell, s, m}=\int_{M_{\ell, \gamma_\ell}}\mathrm{Tr}\,\alpha_{s,m}dv_{M_{\ell, \gamma_\ell}}$ where
$\alpha_{s,m}\,(=\alpha_{s})$ is as in \eqref{e-gue160114} and the numerical factor 
\begin{equation}\label{e-gue160416III}D_{\ell, m}=(\sqrt{\pi})^{e}
\sum_{\scriptstyle c, h\in \mathbb{N}, (h, c)=1\atop\scriptstyle c>1, c|p_\ell,  c\not\\ | p_1}
\frac{e^{-i\frac{2\pi h}{c}m}}{\left\vert e^{i\frac{2\pi h}{c}p_1}-1\right\vert^e}\quad\mbox{\,\,$(>0$ \, if \, $p_\ell |m)$}.
\end{equation}

i) $q_{s, 1}=q_{s, 2}=\cdots=q_{s, \frac{e}{2}-1}=0$, 
$q_{s, 0}=(\sum\limits^{p_1}_{s=1}e^{\frac{i2\pi(s-1)}{p_1}m})\int_X\mathrm{Tr}\,\alpha_{s, m}dv_X$ $(s=n,,n-1,n-2,\ldots)$.

ii) $q_{s, \frac{e}{2}}$ 
is (a finite sum) of the form $\sum_{\ell,\gamma_\ell}D_{\ell,m}S_{\ell,\gamma_\ell,s, m}$ $(s=n,,n-1,n-2,\ldots)$.  

iii) $c_{s}=\int_X\mathrm{Tr}\,\alpha_{s,m}dv_X$ $(s=n,,n-1,n-2,\ldots)$.  

iv) $\tilde c_{n-\frac{e}{2}}=(2\pi)^{-(n+1)}\sum_{\ell,\gamma_\ell}D_{\ell,m}\mathrm{vol}\,(M_{\ell,\gamma_\ell})$,
($\mathrm{vol}=\mathrm{volume}$), which is $>0$ if $p_\ell |m$ for each $\ell$ here.  
\end{thm}

The Laplacian in the work~\cite{DGGW08} is limited to the Laplacian acting on
functions while ours above is not.   We remark that in \cite[Theorem 3.5]{Ri10} the 
nontrivial fractional power in $t^{\frac{1}{2}}$ does occur.   This is however due partly to 
a fixed point set of codimension $1$ under a reflection isometry ({\it loc. cit.}, p. 2315).   In our CR case, 
all of the various fixed point submanifolds are of even (real) codimension, cf. i) of Remark~\ref{r-gue160416r2}
or \cite[p. 2324]{Ri10}.  See Section~\ref{s-gue160416} for a comparison of these methods and results.  


It will be of interest to study the geometrical significance of the various coefficients in 
\eqref{e-gue160416I} and \eqref{e-gue160416II} as usually studied in the standard heat kernel case.   
Explicit expressions for more in this regard are available by our treatment, e.g. \eqref{e-gue160416cde13}, 
\eqref{e-gue160416r2e2} and Theorems~\ref{t-gue160416d}, ~\ref{t-gue160416dc}.  

Remark that the above results essentially deal with the Gaussian part of the heat kernel, 
which behaves as a Dirac type delta function supported on (each) stratum.  By contrast, the CR local 
index theorem as Corollary~\ref{c-gue150508I} is derived by exploring the non-Gaussian parts of the heat kernel
such as the off-diagonal estimate in Theorem~\ref{t-gue150627g}.    In spirit, the two approaches 
are complementary to each other in the present paper, and jointly enhance the understanding 
of heat kernels for this special class of CR manifolds.  

Two more remarks go as follows.  

\begin{rem}
\label{r-gue150516} We note that the topological obstruction exists for a CR
manifold to admit a transversal CR $S^1$ action. For instance, a compact
strictly pseudoconvex CR $3$-manifold must have even first Betti number if
admitting a transversal CR $S^1$ action. The reason is that such a manifold
must be pseudohermitian torsion free (see~\cite{Lee}), and
this vanishing pseudohermitian torsion implies even
first Betti number as shown by Alan Weinstein (see the Appendix in~\cite{CH}). In this paper, we only
consider the $S^1$ action that is transversal and locally free.   Here are two examples: \newline
Example I: $X=\left\{(z_1,z_2,z_3)\in\mathbb{C}^3;\, \left\vert
z_1\right\vert^2+\left\vert z_2\right\vert^2+\left\vert
z_3\right\vert^2+\left\vert z^2_1+z_2\right\vert^4+\left\vert
z^3_2+z_3\right\vert^6=1\right\}$. Then $X$ admits a transversal CR locally
free $S^1$ action: $e^{-i\theta}\circ
(z_1,z_2,z_3)=(e^{-i\theta}z_1,e^{-2i\theta}z_2,e^{-6i\theta}z_3)$. It is
clear that this $S^1$ action is not globally free. \newline
Example II: Let $X$ be a compact orientable Seifert $3$-manifold. Kamishima
and Tsuboi~\cite{KT91} proved that $X$ is a compact CR manifold with a
transversal CR locally free $S^1$ action.  $X$ is $S^1$-fibered over 
a possibly singular base (an orbifold).  
\end{rem}

In Section~\ref{s-gue150723II}, we collect more examples.


\begin{rem}
\label{r-gue150804} The $S^1$ action might admit a reduction
to a simpler one as ${\rm Hom}(S^1, S^1)\ne {\rm id}$.
Recall $p_1=p<p_2<p_3<\cdots<p_k$, associated with periods of $X$
under the given $S^1$ action $(e^{-i\theta},x)\rightarrow e^{-i\theta}\circ x$.  
Then $p_1=p$ divides each $p_j$, $j>1$.   For,
the isotropy subgroup $\mathbb{Z}_p$ ($={\mathbb{Z}}/ p{\mathbb{Z}}$) $\subset S^1$
acts trivially on the principal stratum, which is
dense and open, hence on the whole $X$ by continuity.   The isotropy subgroups $\mathbb{%
Z}_{p_j}$, $j=2,\ldots,k$, on any other stratum must contain $\mathbb{Z}%
_p$, giving $\frac{p_j}{%
p}\in\mathbb{N}$.    

One renormalizes the given $S^1$ action
by the new $S^1$ action satisfying $p_1=1$. More precisely, define
\begin{equation*}
\begin{split}
S^1\times X&\rightarrow X, \\
(e^{-i\theta},x)&\rightarrow e^{-i\theta}\diamond x:=e^{-i\frac{\theta}{p}%
}\circ x.
\end{split}%
\end{equation*}
The new $S^1$-action $(e^{-i\theta},\diamond)$ has $p_1=1$. Let $%
\widetilde\omega_0$ be the global real one form with respect to $%
(e^{-i\theta},\diamond)$ and let $\widetilde
H^q_{b,m}(X,E)$ be the corrsponding cohomology group with respect to $(e^{-i\theta},\diamond)$. 
One sees 
\begin{equation}  \label{e-gue150804}
\begin{split}
&\widetilde\omega_0=p\omega_0, \\
&\widetilde H^q_{b,m}(X,E)=H^q_{b,pm}(X,E),\ \ \forall m\in\mathbb{Z},\ \
\forall q=0,1,2,\ldots,n.
\end{split}%
\end{equation}
Examining \eqref{e-gue150804} and Corollary~\ref{c-gue150508I} yields that the
index formulas in both cases can be transformed to each other.
\end{rem}

\subsection{Applications}

\label{s-gue150704}

\subsubsection{Applications in CR geometry}

\label{s-gue150723}

In CR geometry, it has been an important issue to produce many CR functions or CR
sections. Put
\begin{equation*}
H^0_b(X,E)=\left\{u\in C^\infty(X,E);\, \overline\partial_bu=0\right\}.
\end{equation*}
The following belongs to one of the standard questions in this respect.

\begin{que}
\label{q-gue150704} Let $X$ be a compact weakly pseudoconvex CR manifold.
When is the space $H^0_b(X,E)$ large?  (Pseudoconvex CR manifolds will be
briefly reviewed following Definition~\ref{d-gue150508f}.)
\end{que}

In~\cite{Lem92} Lempert proved that a three dimensional compact strongly
pseudoconvex CR manifold $X$ with a transversal CR locally free $S^1 $action
can be CR embedded into $\mathbb{C}^N$. In~\cite{Ep92} Epstein proved that
a three dimensional compact strongly pseudoconvex CR manifold $X$ with a 
transversal CR globally free $S^1$ action can be embedded into $\mathbb{C}^N$
by the positive Fourier components.

The embeddability of $X$ by positive
Fourier coefficients is related to the behavior of the $S^1$ action on $X$.
For example, suppose for $f_1,\ldots,f_{d_m}\in H^0_{b,m}(X)$ 
and $g_1\ldots,g_{h_{l}}\in H^0_{b,l}(X)$
the map
\begin{equation*}
\Phi_{m, l}:x\in X\to
(f_1(x),\ldots,f_{d_m}(x),g_1(x),\ldots,g_{h_{l}}(x))\in\mathbb{C}%
^{d_m+h_{l}}
\end{equation*}
is a CR embedding. Then, the $S^1$ action on $X$ naturally induces an $S^1$ action
on $\Phi_{m, l}(X)$, 
given by the following:
\begin{equation}
\label{e-gue150704e1} 
e^{-i\theta}\circ(z_1,\ldots,z_{d_m},z_{d_m+1},%
\ldots,z_{d_m+h_{l}})=(e^{-im\theta}z_1,\ldots,e^{-im%
\theta}z_{d_m},e^{-il\theta}z_{d_m+1},\ldots,e^{-il%
\theta}z_{d_m+h_{l}}).
\end{equation}

In short, under a CR embedding by positive Fourier components, one
can describe the $S^1$ action explicitly.   Conversely, to study the embedding theorem of
those CR manifolds by positive Fourier components, it becomes important
to know

\begin{que}
\label{q-gue150704I} When is $\mathrm{dim}H^0_{b, m}(X,E)\approx m^{n}$ for $%
m$ large?
\end{que}

We shall answer, combining our index theorems with some vanishing theorems (see below),
Question~\ref{q-gue150704} and Question~\ref{q-gue150704I} for CR manifolds
with transversal CR locally free $S^1$ action.

Firstly it follows from Corollary~\ref%
{c-gue150508I} (by extracting the leading coefficient
of the term $m^n$)

\begin{cor}
\label{c-gue150704} In the same assumption as in Corollary~\ref%
{c-gue150508I}, one has
\begin{equation}  \label{e-gue150417I}
\begin{split}
&\sum^n_{j=0}(-1)^j\mathrm{dim\,}H^j_{b,m}(X,E) \\
&=r(\sum^{p}_{s=1}e^{\frac{2\pi(s-1)}{p}mi})\frac{m^{n}}{n!(2\pi)^{n+1}}%
\int_X(-d\omega_0)^n\wedge\omega_0+O(m^{n-1}),
\end{split}%
\end{equation}
where $r$ denotes the complex rank of the vector bundle $E$.
\end{cor}

For a vanishing theorem we can repeat the proof of Theorem 2.1 in~\cite{HsiaoLi15} with minor change
and get

\begin{prop}
\label{t-gue150704h} In the same assumption as in Corollary~\ref%
{c-gue150508I}, suppose further that $X$ is weakly pseudoconvex. Then, 
for $m\gg 1$ $\mathrm{%
dim\,}H^j_{b,m}(X,E)=o(m^n)$, for every $j=1,2,\ldots,n$.
\end{prop}

Combining Corollary~\ref{c-gue150704} and Proposition~\ref{t-gue150704h} one has 

\begin{cor}
\label{t-gue150704hI} In the same assumption as in Proposition~\ref%
{t-gue150704h} (with $X$ being weakly pseudoconvex).  One has, for $m\gg1$,
\begin{equation*}
\begin{split}
&\mathrm{dim\,}H^0_{b,m}(X,E) \\
&=r(\sum^p_{s=1}e^{\frac{2\pi(s-1)}{p}mi})\frac{m^{n}}{n!(2\pi)^{n+1}}%
\int_X(-d\omega_0)^n\wedge\omega_0+o(m^{n}),
\end{split}%
\end{equation*}
where $r$ denotes the complex rank of the vector bundle $E$.
In particular, if the Levi form is strongly pseudoconvex at some point of $X$%
, then $\mathrm{dim\,}H_{b,pm}^{0}(X)\approx m^{n}$ for $m\gg 1$, and hence 
$\mathrm{dim\,}H_{b}^{0}(X,E)=\infty $.
\end{cor}

These results have provided answers to Question~\ref{q-gue150704}
and Question~\ref{q-gue150704I} (for our class of CR manifolds).

For another application, it is of great interest in CR geometry to study whether and when
a CR manifold $X$ can be CR embedded into a complex space.
It is a classical theorem of L. Boutet de Monvel~\cite%
{BdM1:74b} which asserts that $X$ can be globally CR embedded into $\mathbb{C}^N$
for some $N\in\mathbb{N}$ provided that $X$ is compact (with no boundary),
strongly pseudoconvex, and of dimension greater
than or equal to five.

When $X$ is not strongly pseudoconvex, the space
of global CR functions could even be trivial.  As many interesting
examples live in the projective space (e.\,g.\ the quadric $\{[z]\in\mathbb{C%
}\mathbb{P}^{N-1};\, |z_1|^2+\ldots+|z_q|^2-|z_{q+1}|^2-\ldots-|z_N|^2=0\}$%
), it is natural to consider a setting analogous to the Kodaira
embedding theorem and ask if $X$ can be embedded into the projective space
by means of CR sections of a CR line bundle $L\rightarrow X$ or its $k$-th power $L^k$.

For a study into the above question it is natural to seek the case where 
the dimension of the space $H^0_b(X,L^k)$ of CR sections of $L^k$ is large as
$k\rightarrow\infty$ (so one may hopefully find many CR sections to 
carry out the embedding).   In this regard the following question is asked by Henkin and Marinescu~%
\cite[p.47-48]{Ma96}.

\begin{que}
\label{q-gue150704II} When is $\mathrm{dim\,}H_{b}^{0}(X,L^{k})\approx
k^{n+1}$ for $k$ large?
\end{que}


Assume that $L$ is a {\it rigid} CR line bundle with a {\it rigid} Hermitian
fiber metric $h^L$ (i.e. $L$ a CR line bundle admitting a compatible $S^1$ action, 
cf. the beginning of Section~\ref{s-gue150508a}).   Let $\mathcal{R}^L\in \Omega^2_{0}(X)$ be the
curvature of $L$ associated to $h^L$.   For a local trivializing ($S^1$-invariant) section $s$ of
$L$, $\left\vert
s(x)\right\vert^2_{h^L}=e^{-2\phi(x)}$ with $T\phi=0$. Then $\mathcal{R}%
^L=2\partial_b\overline\partial_b\phi\in\Omega^2_{0}(X)$.   $%
(L^k,h^{L^k})$ denotes the $k$-th power of $(L,h^L)$.

With Corollary~\ref{c-gue150508I} one can show

\begin{prop}
\label{t-gue150704fr} With the notations above, for $k$ large we have
\begin{equation}  \label{e-gue150704fr}
\begin{split}
&\sum_{m\in\mathbb{Z},\left\vert m\right\vert\leq
k\delta}\sum^{n}_{j=0}(-1)^{j}\mathrm{dim\,}H^j_{b,m}(X,L^k\otimes E) \\
&=r(2\pi)^{-n-1}\frac{1}{n!}k^{n+1}\int_X\int_{[-\delta,\delta]}(i\mathcal{R}%
^L_x-sd\omega_0(x))^{n}\wedge\omega_0(x)ds+o(k^{n+1}),
\end{split}%
\end{equation}
where $\delta>0$ and $r$ denotes the complex rank of the vector bundle $E$.
\end{prop}

\begin{proof}
By Corollary~\ref{c-gue150508I} one can check
\begin{equation}  \label{e-gue150704w}
\begin{split}
&\sum_{m\in\mathbb{Z},\left\vert m\right\vert\leq
k\delta}\sum^{n}_{j=0}(-1)^{j}\mathrm{dim\,}H^j_{b,m}(X,L^k\otimes E) \\
&=r(2\pi)^{-n-1}\sum_{m\in\mathbb{Z},\left\vert m\right\vert\leq
k\delta}(\sum^p_{s=1}e^{\frac{2\pi(s-1)}{p}mi})\frac{1}{n!}\int_X(ik\mathcal{R}%
^L_x-md\omega_0(x))^{n}\wedge\omega_0(x)+o(k^{n}) \\
&=r(2\pi)^{-n-1}\sum_{m\in\mathbb{Z},\left\vert m\right\vert\leq
k\delta}(\sum^p_{s=1}e^{\frac{2\pi(s-1)}{p}mi})\frac{k^n}{n!}\int_X(i\mathcal{R%
}^L_x-\frac{m}{k}d\omega_0(x))^{n}\wedge\omega_0(x)+o(k^{n}).
\end{split}%
\end{equation}
Note $\sum^p_{s=1}e^{\frac{2\pi(s-1)}{p}mi}=p$ if $p|m$, 
and $0$ otherwise.
By this and \eqref{e-gue150704w} we get
\begin{equation}  \label{e-gue150704wI}
\begin{split}
&\sum_{m\in\mathbb{Z},\left\vert m\right\vert\leq
k\delta}\sum^{n}_{j=0}(-1)^{j}\mathrm{dim\,}H^j_{b,m}(X,L^k\otimes E) \\
&=r(2\pi)^{-n-1}\sum_{\ell\in\mathbb{Z},\left\vert p\ell\right\vert\leq
k\delta}\frac{p}{n!}k^{n}\int_X(i\mathcal{R}^L_x-\frac{p\ell}{k}%
d\omega_0(x))^{n}\wedge\omega_0(x)+o(k^{n}).
\end{split}%
\end{equation}
It is clear that the (Riemann) sum $\sum_{\ell\in\mathbb{Z},\left\vert
p\ell\right\vert\leq k\delta}\frac{p}{k}\int_X(i\mathcal{R}^L_x-\frac{p\ell}{%
k}d\omega_0(x))^{n}\wedge\omega_0(x)$ converges to
\begin{equation*}
\int_X\int_{[-\delta,\delta]}(i\mathcal{R}^L_x-sd\omega_0(x))^{n}\wedge%
\omega_0(x)ds
\end{equation*}
as $k\rightarrow\infty$.  Hence 
\begin{equation}  \label{e-gue150704wII}
\begin{split}
&\sum_{\ell\in\mathbb{Z},\left\vert p\ell\right\vert\leq k\delta}\lbrace\frac{p}{n!}%
k^{n}\int_X(i\mathcal{R}^L_x-\frac{p\ell}{k}d\omega_0(x))^{n}\wedge%
\omega_0(x) \rbrace+o(k^{n})\\
&=\lbrace\frac{1}{n!}k^{n+1}\int_X\int_{[-\delta,\delta]}(i\mathcal{R}%
^L_x-sd\omega_0(x))^{n}\wedge\omega_0(x)ds\rbrace+o(k^{n+1}).
\end{split}%
\end{equation}
Combining \eqref{e-gue150704wII} with \eqref{e-gue150704wI} we have %
\eqref{e-gue150704fr}.
\end{proof}

The following two results may be viewed as
a companion of the Grauert-Riemenschneider criterion
in the CR case (with $S^1$ action).  To start with

\begin{defn}
\label{d-gue150705} We say that $(L,h^L)$ is positive at $p\in X$ if the curvature $%
\mathcal{R}_p^L$ is a positive Hermitian quadratic form over $T_p^{1,0}X$.
We say that $(L,h^L)$ is semipositive if for any $x\in X$ there exists a
constant $\delta>0$ such that $\mathcal{R}_x^L-sid\omega_0(x)$ is a semipositive Hermitian quadratic form over $T_x^{1,0}X$ for any $|s|<\delta.$
\end{defn}

We can repeat the proof of Theorem 1.24 in~\cite{HsiaoLi15a} with minor
change and get

\begin{prop}
\label{t-gue150705}  {\rm (Asymptotical vanishing)} Assume that $(L,h^L)$ is a semi-positive CR line bundle
over $X$. Then, for $\delta>0$, $\delta$ small, we have
\begin{equation*}
\sum_{m\in\mathbb{Z},\left\vert m\right\vert\leq k\delta}\mathrm{dim\,}%
H^j_{b,m}(X,L^k\otimes E)=o(k^{n+1}),\ \ j=1,2,\ldots,n.
\end{equation*}
\end{prop}

Combining Proposition~\ref{t-gue150704fr} and Proposition~\ref{t-gue150705}, we get

\begin{cor}
\label{t-gue150705I}  {\rm (Bigness)} Assume that $(L,h^L)$ is semi-positive. Then, for $%
\delta>0$, $\delta$ small, we have
\begin{equation}  \label{e-gue150705}
\begin{split}
&\sum_{m\in\mathbb{Z},\left\vert m\right\vert\leq k\delta}\mathrm{dim\,}%
H^0_{b,m}(X,L^k\otimes E) \\
&=r(2\pi)^{-n-1}\frac{1}{n!}k^{n+1}\int_X\int_{[-\delta,\delta]}(i\mathcal{R}%
^L_x-sd\omega_0(x))^{n}\wedge\omega_0(x)ds+o(k^{n+1}),
\end{split}%
\end{equation}
where $r$ denotes the complex rank of the vector bundle $E$.
In particular, if $(L,h^{L})$ is positive at some point of $X$, then 
\begin{equation*}
\mathrm{dim\,}H_{b,m}^{0}(X,L^{k}\otimes E)\approx k^{n+1}.
\end{equation*}
\end{cor}

The above result yields an answer to Question~\ref%
{q-gue150704II} in the case pertinent to our class of CR manifolds.  

\subsection{Kawasaki's Hirzebruch-Riemann-Roch and Grauert-Riemenschneider
criterion for orbifold line bundles}

\label{s-gue150723I}
There is a link between our CR result and a complex geometry result 
of Kawasaki on Hirzebruch-Riemann-Roch formula
over complex orbifolds ~\cite{Ka79}.  
Compared to Kawasaki's, we get a simpler Hirzebruch-Riemann-Roch formula for
some class of orbifold line bundles using our second main result Corollary~\ref%
{c-gue150508I}. Moreover, from Corollary~\ref{t-gue150704hI} it follows 
a Grauert-Riemenschneider criterion for orbifold line bundles.

To the aim we shortly review the orbifold geometry and also set up notations. 
 Let $M$ be a manifold and $G$ a compact Lie group. Assume that $M$
admits a $G$-action:
\begin{equation*}
\begin{split}
G\times M&\rightarrow M, \\
(g,x)&\rightarrow g\circ x.
\end{split}%
\end{equation*}
We suppose that the action $G$ on $M$ is locally free, that is, for every
point $x\in M$, the stabilizer group $G_x=\left\{g\in G;\, g\circ x=x\right\}
$ of $x$ is a finite subgroup of $G$.  In this case the
quotient space
\begin{equation}  \label{e-gue150801}
M/G
\end{equation}
is known to be an orbifold. A remark of Kawasaki~\cite[p. 76]{Ka78} discusses the validity of
a converse statement about when a space has
a presentation of the form \eqref{e-gue150801}.
(As is well-known, these spaces are actually called V-manifolds by Satake \cite{Sa56}, a slightly
restrictive class of orbifolds.)  

We assume now that $M$ is a compact connected complex manifold with complex
structure $T^{1,0}M$. $G$ induces an action on $\mathbb{C }TM$:
\begin{equation}  \label{e-gue150801I}
\begin{split}
G\times\mathbb{C }TM&\rightarrow \mathbb{C }TM, \\
(g,u)&\rightarrow g^*u,
\end{split}%
\end{equation}
where $g^*=(g^{-1})_*$ the push-forward by $g^{-1}$ on
$\mathbb{C }TM$. 
Suppose $G$ acts holomorphically, that is $g^*(T^{1,0}M)\subset T^{1,0}M$
for $g\in G$.  Put $\mathbb{C }%
T(M/G):=\mathbb{C }TM/G$, $T^{0,1}(M/G):=T^{0,1}M/G$ and $%
T^{1,0}(M/G):=T^{1,0}M/G$.
Assume that $T^{0,1}(M/G)\bigcap T^{1,0}(M/G)=\left\{0\right\}$. Then, $%
T^{1,0}(M/G)$ gives a complex structure on $M/G$ and $M/G$ becomes a complex
orbifold. Suppose $\mathrm{dim\,_{\mathbb{C}}}T^{1.0}(M/G)=n$.  

Let $L$ be a $G$-invariant holomorphc line bundle over $M$, that is, there exists a choice of 
transition functions $h$ (defined on open charts $U$) of $L$ such that $%
h(g\circ x)=h(x)$ for every $g\in G$, $x\in U$ with $g\circ x\in U$.
Suppose that $L$ admits a locally free $G$-action compatible with that on $M$, i.e. an action
$(g,v)\,(\in G\times L)\rightarrow g\circ v\in L$ with the property $\pi(g\circ v)=g\circ(\pi(v))$ ($g$ linearly acts on 
fibers of $L$), $%
\pi:L\rightarrow M$ the projection.   Then, $L/G$ is
an orbifold holomorphic line bundle over $M/G$ (the fiber is not necessarily a vector space). 

The above construction induced by  (locally free) $G$-action on $L$
naturally extends to $L^m$, the $m$-th power of $L$, 
and $L^*$, the dual line bundle of $L$.  
Thus $L^m/G$ and $L^*/G$ are also 
orbifold holomorphic line bundles over $M/G$. 
Put  ($q=0,1,2,\ldots,n$) 
\begin{equation}  \label{e-gue150802I}
\Omega^{0,q}(M/G,L^m/G):=\left\{u\in\Omega^{0,q}(M,L^m);\, g^*u=u,\ \
\forall g\in G\right\}. 
\end{equation}
The Cauchy-Riemann operator $\overline\partial:%
\Omega^{0,q}(M,L^m)\rightarrow\Omega^{0,q+1}(M,L^m)$ is $G$-invariant, hence 
gives a $\overline\partial$-complex $(\overline\partial, \Omega^{0,\bullet}(M/G,L^m/G))$ and 
the $q$-th Dolbeault cohomology group:
\begin{equation*}
H^q(M/G,L^m/G):=\frac{\mathrm{Ker\,}\overline\partial:%
\Omega^{0,q}(M/G,L^m/G)\rightarrow\Omega^{0,q+1}(M/G,L^m/G)}{\mathrm{Im\,}%
\overline\partial:\Omega^{0,q-1}(M/G,L^m/G)\rightarrow\Omega^{0,q}(M/G,L^m/G)%
}.
\end{equation*}

Let $\mathrm{Tot\,}(L^*)$ be
the space of all non-zero vectors of $L^*$.
Assume that $\mathrm{Tot\,}%
(L^*)/G$ is a smooth manifold. Take any $G$-invariant Hermitian fiber metric
$h^{L^*}$ on $L^*$, set $\widetilde X=\left\{v\in L^*;\, \left\vert
v\right\vert_{h^{L^*}}=1\right\}$ and put $X=\widetilde X/G$. Since $\mathrm{%
Tot\,}(L^*)/G$ is a smooth manifold by the foregoing assumption, $X=\widetilde X/G$ is a smooth manifold.
The natural $S^1$ action on $\widetilde X$ induces a locally free $S^1$
action $e^{-i\theta}$ on $X$.   One can check that $X$ is a CR
manifold and the $S^1$ action on $X$ is CR and transversal.

In a similar vein as the proof of Theorem~\ref
{t-gue150508} one can show (for $q=0,1,2,\ldots,n$
 and $m\in\mathbb{Z}$) 
\begin{equation}  \label{e-gue150802a}
H^q(M/G,L^m/G))\cong H^q_{b,m}(X). 
\end{equation}

We pause and introduce some notations.
For every $x\in\mathrm{Tot\,}(L^*)$ and $g\in G$, put $N(g,x)=1$ if $g\notin
G_x$ and $N(g,x)=\inf\left\{\ell\in\mathbb{N};\, g^\ell=\mathrm{id}\right\}$ if $g\in
G_x$. Set
\begin{equation}
\label{e-gue150802ae1} 
p=\inf\left\{N(g,x);\, x\in\mathrm{Tot\,}(L^*),\ \ g\in G,\ \ g\neq
\mathrm{id}\right\}. 
\end{equation}

Putting together Corollary~\ref{c-gue150508I} and %
\eqref{e-gue150802a} gives 

\begin{thm}
\label{t-gue150802} With the notations above, recall that we work with
assumptions that $M$ is connected and $\mathrm{Tot\,}(L^*)/G$ is smooth.
Then (for every $m\in\mathbb{Z}$) 
\begin{equation}  \label{e-gue150802g}
\begin{split}
&\sum^n_{j=0}(-1)^j\mathrm{dim\,}H^j(M/G,L^m/G) \\
&=(\sum^{p}_{s=1}e^{\frac{2\pi(s-1)}{p}mi})\frac{1}{2\pi}\int_{X}\mathrm{Td_b\,%
}(T^{1,0}X)\wedge e^{-m\frac{d\omega_0}{2\pi}}\wedge\omega_0.
\end{split}%
\end{equation}
\end{thm}

To compare this result with that of Kawasaki (\cite{Ka79}) we assume
$p=1$ for simplicity.  Note $X$ is smooth (yet $M/G$ could be singular).  
The above integral \eqref{e-gue150802g} 
reduces to an integral over
the principal stratum of $M/G$ (by integrating $\omega_0$ along the
fiber $S^1$, which gives $1$).   It is thus the same as to say that the contributions
from the lower dimensional strata {\it sum up to zero}.   As remarked in
Introduction a vanishing result as such is not readily available in the formula of 
\cite{Ka79}.  Note that the notion of ``orbifold" in \cite{Ka79} is 
slightly more general than that of Satake (on which the present section is based).
As this generality does no real harm to the reasoning above, we omit the details in this regard.

The above result on the vanishing of the contributions from strata may also
be reflected in the index formula of the works \cite{PV}, \cite{F}
which study the index of transversally elliptic operators on
a smooth compact manifold with the action of a compact Lie group $G$, by using
the framework of equivariant cohomology theory.   A remarkable point
is that they define the index as a {\it generalized function} (on $G$) (also discussed
in Atiyah \cite{A74}, p.9-17).
In fact it is not difficult to verify that for the case of the $S^1$ action,
our present $m$-th index is basically the
$m$-th Fourier component of the corresponding index (in the sense of 
generalized functions) of theirs
(for the case $g={\rm id}\in G$ in \cite{F}, \cite{PV}).

The consistency of our result
with those works above 
helps to shape
our own view towards the asymptotic expansion
of a (transversal) heat kernel conceived in this subject.

For examples that satisfy Theorem~\ref{t-gue150802} we refer to 
Section~\ref{s-gue150723II}.   There, we construct, among
others, an orbifold holomorphic line
bundle over a {\it singular} complex orbifold such that the assumptions of Theorem~%
\ref{t-gue150802} and Corollary~\ref{t-gue150802I} are fulfilled (see Corollary~\ref{c-gue150726}
and Subsection~\ref{s-gue150723sub1} below).  Indeed there are ample examples in this respect.

As promised in the beginning of this section,  we obtain now
a Grauert-Riemenschneider criterion for orbifold line bundles,
upon combining Corollary~\ref{t-gue150704hI} and \eqref{e-gue150802a}.

\begin{cor}
\label{t-gue150802I} With the notations and assumptions above, suppose 
that there is a $G$-invariant Hermitian fiber metric $h^{L}$ on $L$ such
that the associated curvature $R^{L}$ is semipositive and positive at some
point of $M$. Then $\mathrm{dim\,}H^{0}(M/G,L^{pm}/G)\approx m^{n}$ for $%
m\gg 1$ and $p$ as in \eqref{e-gue150802ae1}.  
\end{cor}

\subsection{Examples}

\label{s-gue150723II}

In this subsection, some examples  of CR manifolds with locally free $S^1$
action (including those fitting Theorem~\ref{t-gue150802} above) are collected.

We first review the construction of generalized Hopf manifolds introduced by
Brieskorn and Van de Ven~\cite{BV68}.  

\subsubsection{Generalized Hopf manifolds}
\label{s-gue150723aII}
Let $a=(a_1,\ldots,a_{n+2})\in\mathbb{N}^{n+2}$, let $z=(z_1,\ldots,z_{n+2})$
be the standard coordinates of $\mathbb{C}^{n+2}$ and let $M(a)$ be the
affine algebraic variety given by the equation
\begin{equation*}
\sum^{n+2}_{j=1}z^{a_j}_j=0.
\end{equation*}
If some $a_j=1$, the variety $M(a)$ is non-singular. Otherwise $M(a)$ has
exactly one singular point, namely $0=(0,\ldots,0)$. Put $\widetilde{M(a)}%
:=M(a)-\left\{0\right\}$. Now we define a holomorphic $\mathbb{C}$-action on
$\widetilde{M(a)}$ by
\begin{equation*}
t\circ(z_1,\ldots,z_{n+2})=(e^{\frac{t}{a_1}}z_1,\ldots,e^{\frac{t}{a_{n+2}}%
}z_{n+2}),\ \ t\in\mathbb{C}, \ \ (z_1,\ldots,z_{n+2})\in\widetilde{M(a)}.
\end{equation*}
It is easy to see that the $\mathbb{Z}$-action on $\widetilde{M(a)}$ is
globally free. The equivalence class of $(z_1,\ldots,z_{n+2})\in\mathbb{C}%
^{n+2}$ with respect to the $\mathbb{Z}$-action is denoted by $%
(z_1,\ldots,z_{n+2})+\mathbb{Z}$ and hence
\begin{equation*}
H(a):=\widetilde{M(a)}/\mathbb{Z}=\left\{(z_1,\ldots,z_{n+2})+\mathbb{Z}%
;\,(z_1,\ldots,z_{n+2})\in\widetilde{M(a)}\right\}
\end{equation*}
is a compact complex manifold of complex dimension $n+1$. We call $H(a)$ a
{\it (generalized) Hopf manifold}.

Let $\Gamma_a$ be the discrete subgroup of $\mathbb{C}$,
generated by $1$ and $2\pi\alpha i$, where $\alpha$ is the least common
multiple of $a_1,a_2,\ldots,a_{n+2}$. Consider the complex $1$-torus $T_a=%
\mathbb{C}/\Gamma_a$. $H(a)$ admits a natural $T_a$-action. Put $V(a):=H(a)/
T_a$. By Holmann~\cite{Hol60}, $V(a)$ is a complex orbifold. Let $%
\pi_a:H(a)\rightarrow V(a)$ be the natural projection.

The following is well-known (see the discussion before Proposition 4 in~\cite%
{BV68}).

\begin{thm}
\label{t-gue150726} Let $p=(z_1,\ldots,z_{n+2})+\mathbb{Z}\in H(a)$. Assume
that there are exactly $k$ coordinates $z_{j_1},\ldots,z_{j_k}$ all
different from zero, $k\geq2$. Then, $V(a)$ is non-singular at $\pi_a(p)$ if
and only if
\begin{equation*}
\frac{[a_1,\ldots,a_{n+2}]}{[a_{j_1},\cdots,a_{j_k}]}=\prod_{\ell\notin\left%
\{j_1,\ldots,j_k\right\}}\frac{[a_1,\ldots,a_{n+2}]}{[a_1,\ldots,a_{%
\ell-1},a_{\ell+1},\cdots,a_{n+2}]}.
\end{equation*}
where $[m_1,\ldots,m_d]$ denotes the least common
multiple of $m_1,\ldots,m_d\,\in \mathbb{N}$.  
\end{thm}


It follows readily 


\begin{cor}
\label{c-gue150726} Assume $n\geq2$ and $(a_1,a_2,%
\ldots,a_{n+2})=(4b_1,4b_2,2b_3,2b_4,\ldots,2b_{n+2})$, where $b_j\in\mathbb{%
Z}$, $b_j$ is odd, $j=1,\ldots,n+2$. Let $p=(0,0,1,i,0,0,\ldots,0)+\mathbb{Z}%
\in H(a)$. Then, $V(a)$ is singular at $\pi_a(p)$.
\end{cor}


The ideas in the next two (sub)subsections are heavily based on Theorem~\ref{t-gue150726} and Corollary~\ref{c-gue150726}.

\subsubsection{Smooth orbifold circle bundle over a singular orbifold.}
\label{s-gue150723sub1}
Put
\begin{equation}  \label{e-gue150726}
\begin{split}
X:=&\{(z_1,\ldots,z_{n+2})\in\mathbb{C}^{n+2};\,
z^{a_1}_1+z^{a_2}_2+\cdots+z^{a_{n+2}}_{n+2}=0, \\
&\quad\quad\quad \left\vert z_1\right\vert^{2a_1}+\left\vert
z_2\right\vert^{2a_2}+\left\vert z_3\right\vert^{2a_3}+\cdots+\left\vert
z_{n+2}\right\vert^{2a_{n+2}}=1\}.
\end{split}%
\end{equation}
It can be checked that $X$ is a compact {\it weakly pseudoconvex} CR manifold of
dimension $2n+1$ with CR structure $T^{1,0}X:=T^{1,0}\mathbb{C}^{n+2}\bigcap%
\mathbb{C }TX$, where $T^{1,0}\mathbb{C}^{n+2}$ denotes the standard complex
structure on $\mathbb{C}^{n+2}$. 

Let $\alpha$ be the least common multiple
of $a_1,\ldots,a_{n+2}$. Consider the following $S^1$ action on $X$:
\begin{equation}  \label{e-gue150725}
\begin{split}
S^1\times X&\rightarrow X, \\
e^{-i\theta}\circ(z_1,\ldots,z_{n+2})&\rightarrow(e^{-i\frac{\alpha}{a_1}%
\theta}z_1,\ldots,e^{-i\frac{\alpha}{a_{{n+2}}}\theta}z_{n+2}).
\end{split}%
\end{equation}
One sees that the $S^1$ action is well-defined, locally
free, CR and transversal. Moreover one has that the
quotient $X/S^1$ is equal to $V(a)$, $a=(a_1,a_2,\cdots,a_{n+2})$. Hence, $X/S^1$ is a complex orbifold.

One sees, by using Corollary~\ref{c-gue150726},  that the above $X/S^1$ is {\it singular}
if $n\geq2$ and $(a_1,\ldots,a_{n+2})=(4b_1,4b_2,2b_3,2b_4,\ldots,2b_{n+2})$%
, where $b_j\in\mathbb{Z}$, $b_j$ is odd, $j=1,2,\ldots,n+2$. 

We now show
that $(X,T^{1,0}X)$ is CR-isomorphic to the (orbifold) circle bundle associated with an
orbifold line bundle over $X/ S^1=V(a)$.

To see this and to construct the circle bundle in the first place, let $L=(\widetilde{M(a)}\times%
\mathbb{C})/_\equiv$\,, where $(z_1,\ldots,z_{n+2},
\lambda)\equiv(\widetilde z_1,\ldots,\widetilde z_{n+2}, \widetilde\lambda)$
if
\begin{equation*}
\begin{split}
&\widetilde z_j=e^{\frac{m}{a_j}}z_j,\ \ j=1,\ldots,n+2, \\
&\widetilde\lambda=e^{m}\lambda,
\end{split}%
\end{equation*}
where $m\in\mathbb{Z}$. We can check that $\equiv$ is an equivalence
relation and $L$ is a holomorphic line bundle over $H(a)$. The equivalence
class of $(z_1,\ldots,z_{n+2},\lambda)\in\widetilde{M(a)}\times\mathbb{C}$
is denoted by $[(z_1,\ldots,z_{n+2},\lambda)]$. The complex $1$-torus $T_a$
action on $L$ is given by the following:
\begin{equation}  \label{e-gue150727}
\begin{split}
T_a\times L&\rightarrow L, \\
(t+i\theta)\circ[(z_1,\ldots,z_{n+2},\lambda)]&\rightarrow[(e^{\frac{%
t+i\theta}{a_1}}z_1,\ldots,e^{\frac{t+i\theta}{a_{n+2}}}z_{n+2},e^{t-i\frac{%
\theta}{\alpha}}\lambda)],
\end{split}%
\end{equation}
where $\alpha$ is the least common multiple of $a_1,\ldots,a_{n+2}$. 
One has that the torus action \eqref{e-gue150727} is well-defined and $%
L/T_a$ is an orbifold line bundle over $H(a)/T_a=V(a)$. 

Let $%
\tau:L\rightarrow L/T_a$ be the natural projection and for $%
[(z_1,\ldots,z_{n+2},\lambda)]\in L$, we write $\tau([z_1,\ldots,z_{n+2},%
\lambda])=[z_1,\ldots,z_{n+2},\lambda]+T_a$.   One sees that the 
pointwise norm
\begin{equation*}
\big\lvert\lbrack (z_1,\ldots,z_{n+2}, \lambda)]+T_a\big\rvert%
^2_{h^{L/T_a}}:=\left\vert \lambda\right\vert^2\Bigr(\left\vert
z_1\right\vert^{2a_1}+\left\vert z_2\right\vert^{2a_2}+\left\vert
z_3\right\vert^{2a_3}+\cdots+\left\vert z_{2n+2}\right\vert^{2a_{n+2}}\Bigr)%
^{-1}
\end{equation*}
is well-defined as a Hermitian fiber metric on $L/T_a$. The (orbifold) circle bundle $%
C(L/T_a)$ with respect to $(L/T_a,h^{L/T_a})$ is given by
\begin{equation}  \label{e-gue150727I}
\begin{split}
C(L/T_a):&=\left\{v\in L/T_a;\, \left\vert
v\right\vert^2_{h^{L/T_a}}=1\right\} \\
&=\left\{[(z_1,\ldots,z_{n+2},\lambda)]+T_a;\, \left\vert
\lambda\right\vert^2=\left\vert z_1\right\vert^{2a_1}+\left\vert
z_2\right\vert^{2a_2}+\left\vert z_3\right\vert^{2a_3}+\cdots+\left\vert
z_{2n+2}\right\vert^{2a_{n+2}}\right\}.
\end{split}%
\end{equation}
One sees  that $C(L/T_a)$ is a smooth CR manifold with the CR structure
\begin{equation*}
T^{1,0}C(L/T_a):=T^{1,0}L/T_a\bigcap\mathbb{C }TC(L/T_a),
\end{equation*}
where $T^{1,0}C(L/T_a)$ denotes the complex structure on $L/T_a$. Moreover,
the orbifold line bundle $L/T_a\rightarrow V(a)$ satisfies a similar situation as 
in Theorem~\ref{t-gue150802}  (i.e. the space $X/S^1=V(a)$ here as the $M/G$ there, 
is singular and $C(L/T_a)$ as a (orbifold) circle bundle over $M/G$ 
is smooth).  

We are ready to give an CR isomorphism of 
$X$ and the (orbifold) circle bundle $C(L/T_a)$.  Note $C(L/T_a)$ admits a nature $S^1$ action:
\begin{equation*}
e^{-i\theta}\circ([(z_1,\ldots,z_{n+2},\lambda)]+T_a)=[(z_1,%
\ldots,z_{n+2},e^{-i\theta}\lambda)]+T_a,
\end{equation*}
Let $\Phi:C(L/T_a)\rightarrow X$ be the smooth map defined as follows. For
every $[(z_1,\ldots,z_{n+2},\lambda)]+T_a\in C(L/T_a)$, there is a unique $%
(\hat z_1,\ldots,\hat z_{n+2})\in X$ such that
\begin{equation*}
[(z_1,\ldots,z_{n+2},\lambda)]+T_a=[(\hat z_1,\ldots,\hat z_{n+2},1)]+T_a.
\end{equation*}
Then, $\Phi([(z_1,\ldots,z_{n+2},\lambda)]+T_a):=(\hat z_1,\ldots,\hat
z_{n+2})\in X$. It can be checked that $\Phi$ is a CR
embedding, globally one to one, onto and the inverse $\Phi^{-1}:X%
\rightarrow C(L/T_a)$ is also a CR embedding. Moreover
$e^{-i\theta}\circ\Phi(x)=\Phi(e^{-i\theta}\circ x),\ \ \forall x\in C(L/T_a)$.
We conclude $\Phi$ is a CR isomorphism. 

\subsubsection{Family, non-pseudoconvex cases and deformations}
\label{s-gue150723sub2}  
In the notation of Subsection~\ref{s-gue150723aII} we assume $a_1=1$, so 
\begin{equation*}
M(a)=\left\{(z_1,\ldots,z_{n+2})\in\mathbb{C}^{n+2};\,
z_1=-z^{a_2}_2-\cdots-z^{a_{n+2}}_{n+2}\right\}.
\end{equation*}
Fix a $q=2,3,\ldots,n+1$. 
Put, for $t\in \mathbb{C}$, 
\begin{equation}  \label{e-gue150723}\begin{split}
{X_{q,t}}:=\{(z_1,\ldots,z_{n+2})+\mathbb{Z}\in H(a);\, &-\left\vert
z^{a_2}_2+tz^{a_3}_3\right\vert^{2}-\left\vert
z_3\right\vert^{2a_3}-\cdots-\left\vert z_q\right\vert^{2a_q} \\
&\quad\qquad+\left\vert z_{q+1}\right\vert^{2a_{q+1}}+\cdots+\left\vert
z_{n+2}\right\vert^{2a_{n+2}}=0\}.
\end{split}
\end{equation}

One can check that for each $t$, $X_{q, t}$ is a compact CR manifold of
dimension $2n+1$ with CR structure $T^{1,0}X_{q, t}:=T^{1,0}H(a)\bigcap\mathbb{C }%
TX_{q, t}$, where $T^{1,0}H(a)$ denotes the natural complex structure inherited by $M(a)$.  
Note $X_{q, t_1}$ is diffeomorphic to $X_{q,t_2}$ for $t_1, t_2\in \mathbb{C}$
since they can be connected through a (smooth) family of compact manifolds.  

Let $\widetilde a$ be the least common multiple of $a_1,\ldots,a_{q}$.
Consider the following $S^1$ action on $X_{q,t}$:
\begin{equation}  \label{e-gue150723I}
\begin{split}
&S^1\times X_{q,t}\rightarrow X_{q,t}, \\
&e^{-i\theta}\circ((z_1,\ldots,z_{n+2})+\mathbb{Z}) \\
&\rightarrow(e^{-i\widetilde
a\theta}(-z^{a_2}_2-\cdots-z^{a_{q}}_{q})-z^{a_{q+1}}_{q+1}-%
\cdots-z^{a_{n+2}}_{n+2},e^{-i\frac{\widetilde a}{a_2}\theta}z_2,\ldots,e^{-i%
\frac{\widetilde a}{a_{q}}\theta}z_{q},z_{q+1},\ldots,z_{n+2})+\mathbb{Z}.
\end{split}%
\end{equation}
One sees that the $S^1$ action is well-defined,
locally free, CR and transversal.   This is an example for 
a family of CR manifolds admitting a transversal CR locally free $S^1$ action.  

Moreover these CR manifolds $X_{q, t}$ are not 
pseudoconvex.


Now we consider certain CR deformations of a compact CR manifold $X$ with a 
transversal CR locally free $S^1$ action. 
Let $F(x)\in C^\infty({X})$ with $TF=0$ ($T$ the global 
real vector field induced by the $S^1$
action). Let $Z_1,\ldots,Z_n\in C^\infty({X%
},T^{1,0}{X})$ be a basis for $T^{1,0}{X}$. Put
\begin{equation}\label{e-gue150723sub2e1}
H^{1,0}{X}:=\left\{Z_j+Z_j(F)T;\, j=1,2,\ldots,n\right\}.
\end{equation}
One can check that $H^{1,0}{X}$ is a CR structure and the $S^1$
action is locally free, CR and transversal with
respect to this new CR structure $H^{1,0}{X}$ (see 
\eqref{e-can} via the BRT construction).    

To see how ``new" this CR structure $H^{1,0}X$ is, let's take $X$ to be a circle bundle
associated with a holomorphic line bundle $(L, ||\cdot||)$ over a compact complex manifold $M$.  
Consider a change of metric $||\cdot||\to e^{-2f}||\cdot||$ on $L$ and the circle bundle 
$\widetilde X$ thus induced by this new metric.   By using the formula \eqref{e-gue150510} below one sees 
that $H^{1,0}X$ of \eqref{e-gue150723sub2e1} for $F=-if$ is equivalent to 
$T^{1,0}\widetilde X$.    But is $(X, T^{1,0}X)$ CR equivalent to $(\widetilde X, T^{1,0}\widetilde X)$?  
The answer is in general no.  For instance, spherical CR structures on a certain topological type 
of $X$ can be obtained by using special metrics on $L$ (cf. \cite{CT00}).  Hence an arbitrary perturbation 
of the bundle metric, say by the multiplier $e^{-2f}$,  would bring $X$ out of the spherical category.
Note that the moduli space of spherical CR structures in \cite{CT00} is finite dimensional.  
It follows that for $F$ a purely imaginary function on $X$, the CR structure $H^{1,0}X$ is in general not CR equivalent to 
$T^{1,0}X$.  

If, however, $F$ is a real function,  it is easily seen that
the change $(z, \theta)\to (z, \theta +F)$ is globally defined, hence
it gives a diffeomorhism $\phi$ of $X$.   One sees $\phi_*(T^{1,0}X)=H^{1,0}X$, 
cf. \eqref{e-gue150510} below.  So
in this case the CR structure $H^{1,0}X$ is equivalent to the original one.

\subsection{Proof of Theorem~\protect\ref{t-gue150508}}

\label{s-gue150509}

Notations as in Theorem~\ref{t-gue150508} let $s$ be a local
trivializing section of $L$ defined on some open set $U$ of $M$, $\left\vert
s\right\vert^2_{h^L}=e^{-2\phi}$. Let $z=(z_1,\ldots,z_n)$ be holomorphic
coordinates on $U$. We identify $U$ with an open set of $\mathbb{C}^{n}$ and 
have the local diffeomorphism:
\begin{equation}  \label{e-gue150509I}
\tau:U\times]-\varepsilon_0, \varepsilon_0[\rightarrow X\,,\:
(z,\theta)\mapsto e^{-\phi(z)}s^*(z)e^{-i\theta}\,, \ 0<\varepsilon_0\leq\pi.
\end{equation}

Put $D=U\times]-\varepsilon_0, \varepsilon_0[$ as a canonical coordinate patch 
with $(z,\theta)$ canonical
coordinates (with respect to the trivialization $s$) 
such that $T=\frac{\partial}{\partial\theta}$ (recall $T$ is the global real
vector field induced by the $S^1$ action). Moreover one has 
\begin{equation}  \label{e-gue150510}
\begin{split}
&T^{1,0}X=\left\{\frac{\partial}{\partial z_j}-i\frac{\partial\phi}{\partial
z_j}(z)\frac{\partial}{\partial\theta};\, j=1,2,\ldots,n\right\}, \\
&T^{0,1}X=\left\{\frac{\partial}{\partial\overline z_j}+i\frac{\partial\phi}{%
\partial\overline z_j}(z)\frac{\partial}{\partial\theta};\,
j=1,2,\ldots,n\right\},
\end{split}%
\end{equation}
and
\begin{equation}  \label{e-gue150909II}
T^{*1,0}X=\left\{dz_j;\, j=1,2,\ldots,n\right\},\ \
T^{*0,1}X=\left\{d\overline z_j;\, j=1,2,\ldots,n\right\}.
\end{equation}
See also Theorem~\ref{t-gue150514} and proof of Proposition~\ref{l-gue150524d} for 
similar formulas in the general case of $S^1$ action.  

Let $f(z)\in\Omega^{0,q}(D)$.  By \eqref{e-gue150909II} we may identify $f$
with an element in $\Omega^{0,q}(U)$.

The key object in our proof is the map $A^{(q)}_m:\Omega^{0,q}_m(X)
\rightarrow\Omega^{0,q}(M,L^m)$, to be defined as follows. 
Let $u\in\Omega^{0,q}_m(X)$.   We can write $u(z,\theta)
=e^{-im\theta}\hat u(z)$ (on $D$) for some $\hat u(z)\in\Omega^{0,q}(U)$.    Then, on $U\subset M$, 
we define

\begin{equation}\label{e-gue150909aI}
A^{(q)}_mu:=s^m(z)e^{m\phi(z)}\hat
u(z)\in\Omega^{0,q}(U,L^m).
\end{equation}
We need to check the following.

i) $A^{(q)}_m$ in \eqref{e-gue150909aI} is well-defined, hence gives rise to 
a global element $A^{(q)}_mu\in \Omega^{0,q}(M,L^m)$.

ii)  It satisfies the commutativity $\overline\partial A^{(q)}_m=A^{(q+1)}_m\overline\partial_b$
(thus induces a map on respective cohomologies).  

To check i) let $s$ and $s_1$ be local trivializing sections of $L$ on an open set $U$.
Let $(z,\theta)\in\mathbb{C}^n\times\mathbb{R}$ and $(z,\eta)\in\mathbb{C}%
^n\times\mathbb{R}$ be canonical coordinates of $D$ with respect to $s$ and $%
s_1$ respectively ($D$ as the above).    Set $\left\vert
s\right\vert^2_{h^L}=e^{-2\phi}$ and $\left\vert
s_1\right\vert^2_{h^L}=e^{-2\phi_1}$. We write (on $D$) 
\begin{equation}\label{e-gue150909aII} 
\begin{split}
&u=e^{-im\theta}\hat u(z) \\
&u=e^{-im\eta}\hat u_1(z).  
\end{split}%
\end{equation}
To check i) amounts to the following 
\begin{equation}  \label{e-gue150509III}
s^m(z)e^{m\phi(z)}\hat u(z)=s^m_1(z)e^{m\phi_1(z)}\hat u_1(z),\ \ \forall
z\in U.
\end{equation}

Let $s_1=gs$ for $g$ a unit on $U$.  To find relations 
between $\phi$ and $\phi_1$, $\hat u$ and $\hat u_1$ in terms of 
$g$,
\begin{equation*}
\left\vert s_1\right\vert^2_{h^L}=e^{-2\phi_1}=\left\vert
g\right\vert^2\left\vert s\right\vert^2_{h^L}=e^{2\log\left\vert
g\right\vert-2\phi}, 
\end{equation*}
giving 
\begin{equation}  \label{e-gue150509a}
\phi_1=\phi-\log\left\vert g\right\vert.
\end{equation}
For $\hat u$ and $\hat u_1$, we first claim the following 
($\tau$ in \eqref{e-gue150509I} for $(z, \theta)$ and $\tau_1$ the
similar one for $(z,\eta)$)
\begin{equation}  \label{e-gue150509aI}
\mbox{If $\tau (z,\theta)={\tau_1}(z,\eta)$, then
$e^{-i\theta}\bigr(\frac{g(z)}{\ol g(z)}\bigr)^{\frac{1}{2}}=e^{-i\eta}$ (with
a certain branch of the square root)}.
\end{equation}

\begin{proof}[Proof of the claim \eqref{e-gue150509aI}]
Combining \eqref{e-gue150509I} and \eqref{e-gue150509a} one sees
\begin{equation}  \label{e-gue150509f}
\begin{split}
\tau(z, \theta)=s^*(z)e^{-i\theta-\phi(z)}&=s^*_1(z)g(z)e^{-i\theta-\phi(z)} \\
&=s^*_1(z)g(z)e^{-i\theta-\phi_1(z)-\log\left\vert g(z)\right\vert} \\
&=s^*_1(z)\bigr(\frac{g(z)}{\overline g(z)}\bigr)^{\frac{1}{2}%
}e^{-i\theta-\phi_1(z)}.
\end{split}%
\end{equation}
The condition $\tau (z,\theta)={\tau_1}(z,\eta)$ is the same as to say, by \eqref{e-gue150509I},
\begin{equation}  \label{e-gue150509fI}
s^*(z)e^{-i\theta-\phi(z)}=s^*_1(z)e^{-i\eta-\phi_1(z)}.
\end{equation}
By \eqref{e-gue150509f} and \eqref{e-gue150509fI} we deduce that $%
\bigr(\frac{g(z)}{\overline g(z)}\bigr)^{\frac{1}{2}}e^{-i\theta}=e^{-i\eta}$%
, as claimed.
\end{proof}

Now that the relations \eqref{e-gue150509a} and \eqref{e-gue150509aI} have been 
found, the \eqref{e-gue150509III} follows by using \eqref{e-gue150909aII}.  Hence $A^{(q)}_m:\Omega^{0,q}_m(X)%
\rightarrow\Omega^{0,q}(M,L^m)$ is well-defined, proving i) above.   

Moreover it is easily 
checked that $A^{(q)}_m$ is 
bijective.  We omit the detail.

To prove ii) that $\overline\partial A^{(q)}_m=A^{(q+1)}_m\overline\partial_b
$, by \eqref{e-gue150510} and %
\eqref{e-gue150909II} one sees (on $D$)
\begin{equation}  \label{e-gue150510I}
\overline\partial_bu=\overline\partial_b(e^{-im\theta}\hat
u)=\sum^n_{j=1}e^{-im\theta}d\overline z_j\wedge\bigr(\frac{\partial\hat u}{%
\partial\overline z_j}(z)+m\frac{\partial\phi}{\partial\overline z_j}(z)\hat
u(z)\bigr).
\end{equation}
Hence \eqref{e-gue150510I} and \eqref{e-gue150909aI} yield 
\begin{equation}  \label{e-gue150510II}
\begin{split}
A^{(q+1)}_m(\overline\partial_bu)&=s^m(z)e^{m\phi(z)}\sum^n_{j=1}d\overline
z_j\wedge\bigr(\frac{\partial\hat u}{\partial\overline z_j}(z)+m\frac{%
\partial\phi}{\partial\overline z_j}(z)\hat u(z)\bigr) \\
&=s^m(z)\overline\partial(e^{m\phi(z)}\hat u(z))\ \ \mbox{on $U$},
\end{split}%
\end{equation}
giving $\overline\partial A^{(q)}_m=A^{(q+1)}_m\overline\partial_b$. Theorem~%
\ref{t-gue150508} follows.  

\begin{rem}\label{r-gue150508r}
The map $A^{(q)}_m$ does not depend on the metrics of the manifolds $X$ and $M$.  
In later sections we study the Kohn Laplacian and Kodaira Laplacian on 
$X$ and $M$ respectively, and try to establish a link between the two 
Laplacians (with the aim at the Kohn's).   In this regard we need 
equip $X$ and $M$ with appropriate metrics so that 
$A^{(q)}_m$ thus defined is also compatible with these metrics.  
Note a localization of this (metrical) construction (cf. Proposition~\ref{l-gue150606}) paves the way for 
our subsequent plan in this work.  
\end{rem}

Some difficulties (and ways out) for a straightforward generalization of the proof
for this special case (globally free $S^1$ action) will be discussed in the subsection below.

\subsection{The idea of the proofs of Theorem~\protect\ref{t-gue160114},
Theorem~\protect\ref{t-gue160307} and Corollary~\protect\ref{c-gue150508I}}

\label{s-gue150920}

We will give an outline of main ideas of some proofs.  For the proof 
of Theorem~\ref{t-gue160416}, some ideas are outlined in the beginning of 
Section~\ref{s-gue160416}.
We refer to Section~\ref{s-gue150508bI} and Section~\ref%
{s-gue150508d} for notations and terminologies used here. 
The main technical tool of our method lies in a construction of a heat kernel
for the Kohn Laplacian associated to the $m$-th $S^1$ Fourier component.

\subsubsection{Global difficulties}

\label{s-gue151023}

For simplicity we assume that $X$ is CR K\"ahler (cf. Definition~\ref{d-gue160308}) without $E$ and
$\langle\,\cdot\,|\,\cdot\,\rangle$ is induced by a CR K\"ahler
form $\Theta$ on $X$.   Write $\overline{\partial}^*_b$ for the adjoint of $%
\overline\partial_b$ with respect to $(\,\cdot\,|\,\cdot\,)$ and $\overline{\partial}^*_{b,m}=\overline{\partial}%
^*_b:\Omega^{0,q+1}_m(X)\rightarrow\Omega^{0,q}_m(X)$ with 
$\Omega^{0,+}_m(X)$ and $\Omega^{0,-}_m(X)$ denoting forms of even and odd degree.  
Consider
\begin{equation*}
D^{\pm}_{b,m}:=\overline\partial_{b,m}+\overline{\partial}^*_{b,m}:%
\Omega^{0,\pm}_m(X)\rightarrow\Omega^{0,\mp}_m(X), \quad m\in \mathbb{Z}
\end{equation*}
and let $
\Box^{+}_{b,m}:=D_{b,m}^-D_{b,m}^+:\Omega^{0,+}_m(X)\rightarrow\Omega^{0,+}_m(X)$
($\Box^{-}_{b,m}:=D_{b,m}^+D_{b,m}^-$ similarly).  

Extending $\Box^+_{b,m}$ and $\Box^-_{b,m}$ to $L^{2,+}_m(X)$ and $L^{2,-}_m(X)$
($L^2$-completion), respectively in the
standard way, we will show in Theorem~\ref{t-gue150517aI} that $\mathrm{%
Spec\,}\Box^{\pm}_{b,m}$ are discrete subsets
of $[0,\infty[$ and $\mathrm{Spec\,}\Box^{\pm}_{b,m}$ consist of eigenvalues of $
\Box^{\pm}_{b,m}$. 

For $\nu\in\mathrm{Spec\,}\Box^{+}_{b,m}$, let $\left\{f^\nu_1,%
\ldots,f^\nu_{d_{\nu}}\right\}$ be an orthonormal frame for the eigenspace
of $\Box^+_{b,m}$ with eigenvalue $\nu$.   Write 
$T^{*0,\bullet}X=\oplus_{0\le q\le n}T^{*0,q}X$.
$e^{-t\Box^+_{b,m}}(x,y):T_y^{*0,\bullet}X\to T_x^{*0,+}X$, said to be a heat kernel,  is given by
(cf. \eqref{e-gue151023a}) 
\begin{equation}  \label{e-gue151023}
e^{-t\Box^+_{b,m}}(x,y)=\sum_{\nu\in\mathrm{Spec\,}\Box^+_{b,m}}\sum^{d_{%
\nu}}_{j=1}e^{-\nu t}f^\nu_j(x)\wedge(f^\nu_j(y))^\dagger. 
\end{equation}
(Similarly we can define $e^{-t\Box^-_{b,m}}(x,y)$.)

We will show in Corollary~\ref{t-gue150603} (see also Remark~\ref{r-gue151003}%
) that we have a \emph{CR McKean-Singer type formula}: for $t>0$,
\begin{equation}  \label{e-gue150922}
\sum^n_{j=0}(-1)^j\mathrm{dim\,}H^j_{b,m}(X)=\int_X\Bigr(\mathrm{Tr\,}%
e^{-t\Box^+_{b,m}}(x,x)-\mathrm{Tr\,}e^{-t\Box^-_{b,m}}(x,x)\Bigr)dv_X.
\end{equation}
By this formula the proof of our index theorem (cf. Corollary~\ref%
{c-gue150508I}) is reduced to determining the small $t$ behavior of the
function $\Bigr(\mathrm{Tr\,}e^{-t\Box^+_{b,m}}(x,x)-\mathrm{Tr\,}%
e^{-t\Box^-_{b,m}}(x,x)\Bigr)$.    

With the kernel $%
e^{-t\Box^+_{b,m}}(x,y)$ there is associated an operator denoted by
$e^{-t\Box^+_{b,m}}:\Omega^{0,+}(X)%
\rightarrow\Omega^{0,+}_m(X)$.  Note the domain is set to be the
full space $\Omega^{0,+}(X)$.  
From \eqref{e-gue151023} it follows that the kernel satisfies a heat equation which
is expressed in the following operator form
\begin{equation}  \label{e-gue151024}
\frac{\partial e^{-t\Box^+_{b,m}}}{\partial t}+\Box^+_{b,m}e^{-t%
\Box^+_{b,m}}=0
\end{equation}
and
\begin{equation}  \label{e-gue151024I}
e^{-t\Box^+_{b,m}}|_{t=0}=Q^+_m,
\end{equation}
where $Q^+_m:L^{2,+}(X)\rightarrow L^{2,+}_m(X)$ is the orthogonal
projection.

The main difficulty lies in that the initial condition %
\eqref{e-gue151024I} is a \emph{projection operator} rather than an \emph{identity
operator} because we are dealing with part of the $L^2$ space (i.e. the $m$-th eigenspaces) rather
than the whole $L^2$ space (as in the usual case). 
In a similar vein, let us quote in a paper of Richardson \cite[p. 358]{Ri98}:
``A point of difficulty that often arises in this area of research is that the space...is 
not the set of all sections of any vector bundle, and therefore the usual theory 
of elliptic operators and heat kernels does not apply directly...". 
The condition \eqref{e-gue151024I} eventually 
leads to the result that the heat kernels $e^{-t\Box^{\pm}_{b,m}}(x,y)$ do not 
have the \emph{standard expansions} (as usually seen).  

For a better understanding let's assume that $X$ is a (orbifold) 
circle bundle of an orbifold line bundle $L$ over a K\"ahler orbifold $M$
(see Section~\ref{s-gue150723II} for specific examples).
As in Theorem~\ref{t-gue150508} (see Section~\ref{s-gue150509}), 
one sees bijective maps
\begin{equation*}
A^{\pm}_m:\Omega^{0,\pm}_m(X)\rightarrow\Omega^{0,\pm}(M,L^m)
\end{equation*}
such that $A^-_m\overline\partial_b=\overline\partial A^+_m$. Let $\Box^+_m$
be the Kodaira Laplacian with values in $T^{*0,+}M\otimes L^m$ and let $%
e^{-t\Box^+_m}$ be the associated heat operator. Consider $%
B_m(t):=(A^+_m)^{-1}\circ e^{-t\Box^+_m}\circ A^+_m$.  
$A_m^{\pm}$ are metric-independent (on a given $X$).   To get a link between 
$\Box^+_{b, m}$ and $\Box_m^+$ it requires, however, a compatible choice of 
metrics on $X$ and $M$.   With this done, 
one checks that $%
B^{\prime }_m(t)+\Box^+_{b,m}B_m(t)=0$ and $B_m(0)=I$ on $\Omega^{0,+}_m(X)$.

But $B_m(t)$ is \emph{not} the heat operator $e^{-t\Box^+_{b,m}}$.  A trivial 
reason is that $%
B_m(t)$ is defined on $\Omega^{0,+}_m(X)$ while $e^{-\Box^+_{b,m}}$
is on the whole $\Omega^{0,+}(X)$.   In fact one has 
\begin{equation}  \label{e-gue150923}
e^{-t\Box^+_{b,m}}=\big((A^+_m)^{-1}\circ e^{-t\Box^+_m}\circ A^+_m\big)\circ Q_m^+
=B_m(t)\circ Q^+_m\,\,(=Q^+_m\circ B_m(t)\circ Q^+_m).
\end{equation}

Let $B_m(t,x,y)$ be the distribution kernel
of $B_m(t)$.   To emphasize the role played by $Q_m^+$ in our construction, 
it is illuminating to note the following (cf. \eqref{e-gue150923}, \eqref{e-gue150510f} and \eqref{e-gue151024b})
\begin{equation}  \label{e-gue150923I}
e^{-t\Box^+_{b,m}}(x,y)=\frac{1}{2\pi}\int_{-\pi}^{\pi} B_m(t,x,e^{-iu}\circ y)e^{-imu}du.
\end{equation}
For $x\in X_p$ (the principal stratum),  from \eqref{e-gue150923I} and 
the much better known kernel $e^{-t\Box^+_m}$ (on the principal stratum of $M$)
it follows 
\begin{equation}
\label{e-gue150923II}
e^{-t\Box^+_{b,m}}(x,x)\sim t^{-n}a^+_n(x)+t^{-(n-1)}a^+_{n-1}(x)+\cdots.
\end{equation}
However for $x\notin X_p$,  by lack of the asymptotic
expansion of $B_m(t)$ (or $e^{-t\Box^+_m}$ on low dimensional strata of $M$) it is unclear how one can understand
the asymptotic behavior of $e^{-t\Box^+_{b,m}}(x,x)$ by means of \eqref{e-gue150923I}.
This presents a major deviation from proof of the globally free case (as
Theorem~\ref{t-gue150508}, cf. \eqref{e-gue150923bi}, \eqref{e-gue151108}).

To see more clearly the discrepancy between the two cases (locally free and globally free)
we note that  the expansion \eqref{e-gue150923II} converges only
\emph{locally uniformly} on $X_p$, due to a nontrivial contribution
involving a ``distance function" (see Subsection~\ref{s-gue151025} for more).
In fact the expansion of the form \eqref{e-gue150923II} which is usually seen, 
cannot hold here ({\it globally} on $X$) (cf. Remark~\ref{r-gue150508I}). 

It is thus not immediate for one to arrive at a detailed understanding of the (transversal) heat
kernel by only using the global argument.  Even in the (smooth) orbifold
circle bundle case, to understand the asymptotic behavior of the heat
operator $e^{-t\Box^+_{b,m}}$ we will still need to work directly on the CR manifold
$X$ instead of $M$.

In this paper we give a construction which is independent of the use
of orbifold geometry and is more adapted to CR geometry as our CR manifold $X$ is not assumed
to be an orbifold circle bundle of a complex orbifold.
Because of the failure of the global argument as just said, we are now led to work on
it {\it locally}.  
The framework for this is \emph{BRT trivialization} (Section~\ref{s-gue150514}) which is
first treated by Baouendi, Rothschild and Treves~\cite{BRT85} in a more
general context.

\subsubsection{Transition to local situation}

\label{s-gue151024}

Let $B:=(D,(z,\theta),\varphi)$ be a BRT trivialization (see Theorem~\ref%
{t-gue150514}). We write $D=U\times]-\varepsilon,\varepsilon[$,
where $\varepsilon>0$ and $U$ is an open set of $\mathbb{C}^n$.  Let  $%
L\rightarrow U$ be a trivial line bundle with a non-trivial Hermitian fiber
metric $\left\vert 1\right\vert^2_{h^L}=e^{-2\varphi}$ (where $\varphi\in
C^\infty(D,\mathbb{R})$ is as in Theorem~\ref{t-gue150514}) and $%
(L^m,h^{L^m})\rightarrow U$ be the $m$-th power of $(L,h^L)$. $\Theta$
(cf. Definition~\ref{d-gue160308}, recalling $X$ is CR K\"ahler as assumed
for the moment) induces a K\"ahler form $\Theta_U$ on the complex manifold $U$.
Let $%
\langle\,\cdot\,,\,\cdot\,\rangle$ be the Hermitian metric on $\mathbb{C }TU$
(associated with $\Theta_U$), inducing together with $h^{L^m}$ the $L^2$ inner
product $(\,\cdot\,,\,\cdot\,)_m$ on $\Omega^{0,*}(U,L^m)$. 

Let $\overline{\partial}^{*,m}:\Omega^{0,q+1}(U,L^m)\rightarrow%
\Omega^{0,q}(U,L^m)$ be the formal adjoint of $\overline\partial$
with respect to $(\,\cdot\,,\,\cdot\,)_m$.  
Put, as the case of $D_{b, m}$ and $\Box_{b,m}$, 
$D^{\pm}_{B,m}:=\overline\partial+\overline\partial^{*,m}: 
\Omega^{0,\pm}(U,L^m)\rightarrow\Omega^{0,\mp}(U,L^m)$
and $\Box^+_{B,m}:=D^-_{B,m}D^+_{B,m}: \Omega^{0,+}(U,L^m)\rightarrow\Omega^{0,+}(U,L^m)$. 

In Proposition~\ref{l-gue150606}, by the above choice of metrics in forming Laplacians 
on two different spaces ($D$ and $U$), we can provide a link between these Laplacians, asserting 
that
\begin{equation}  \label{e-gue150606IIIda}
e^{-m\varphi}\Box^{\pm}_{B,m}(e^{m\varphi}\widetilde u)=e^{im\theta}{\Box}%
^{\pm}_{b,m}(u)
\end{equation}
where as before, $u\in\Omega^{0,\pm}_m(X)$ can be written (on $D$) as $%
u(z,\theta)=e^{-im\theta}\widetilde u(z)$ for some $\widetilde u(z)\in\Omega^{0,\pm}(U,L^m)\subset
\Omega^{0,\pm}(D,L^m)$.

Write $x=(z, \theta)$, $y=(w,\eta)$ (on $D$).
With \eqref{e-gue150606IIIda} one expects that the heat kernel $%
e^{-t\Box^+_{b,m}}(x,y)$ locally (on $D$) should be
\begin{equation}  \label{e-gue150924a}
e^{-m\varphi(z)-im\theta}e^{-t\Box^+_{B,m}}(x,y)e^{m\varphi(w)+im\eta}.
\end{equation}
Thus one obtains \emph{local}
heat kernels on these BRT charts.  

We would like to patch them up.
Assume that $X=D_1\bigcup D_2\bigcup\cdots\bigcup D_N$ (where $D_j$ in a 
BRT trivialization $B_j:=(D_j,(z,\theta),\varphi_j)$) with $D_j=U_j\times]-\delta_j,\widetilde\delta_j[%
\subset\mathbb{C}^n\times\mathbb{R}$, $\delta_j>0$, $\widetilde\delta_j>0$, $%
U_j$ is an open set in $\mathbb{C}^n$). 

Let $\chi_j, \widetilde\chi_j\in
C^\infty_0(D_j)$ ($j=1,2,\ldots,N$). Put
\begin{equation}  \label{e-gue150924aI}
\begin{split}
&\mathcal{A}_m(t)=\sum^N_{j=1}\chi_j(x)\Bigr(e^{-m\varphi_j(z)-im%
\theta}e^{-t\Box^+_{B_j,m}}(z,w)e^{m\varphi_j(w)+im\eta}\Bigr)%
\widetilde\chi_j(y), \\
&\mathcal{P}_m(t)=\mathcal{A}_m(t)\circ Q^+_m.
\end{split}%
\end{equation}

It is hoped that $\mathcal{P}_m(0)=Q^+_m$ and $\mathcal{P}^{\prime
}_m(t)+\Box^+_{b,m}\mathcal{P}_m(t)$ is small as $t\rightarrow0^+$
for certain $\chi_j$, $\widetilde\chi_j$.  This is related to {\it asymptotic heat kernel}.  
But as we will see,  this standard patch-up construction does not quite work
out in our case.   

In short, we will see that in the {\it locally free} case the nice (pointwise) 
relation \eqref{e-gue150606IIIda} between 
Kodaira and Kohn Laplacians 
does not quite carry over to the global objects: heat kernels, whose mutual relation is to be seen below 
by more delicate analysis relevant to the presence of strata beyond the principal stratum.  

\subsubsection{Local difficulties}

\label{s-gue151025}

A necessary condition for $\mathcal{P}_m(0)=Q^+_m$ is (cf. Lemma~\ref{l-gue150626f})
\begin{equation}  \label{e-gue150924aII}
\sum^N_{j=1}\chi_j(x)\int_{-\pi}^{\pi}\widetilde\chi_j(w,\eta)|_{w=z}d\eta=1.
\end{equation}
For the cut-off
functions $\chi_j$, $\widetilde\chi_j$ above,
a reasonable choice (adapted to BRT trivializations) is the following
(for $j=1,2,\ldots,N$):

i) $%
\chi_j(z,\theta)\in C^\infty_0(D_j)$ with $%
\sum^N_{j=1}\chi_j=1$ on $X$;

ii) $\tau_j(z)\in C^\infty_0(U_j)$ with $\tau_j(z)=1$ if $(z,\theta)\in\mathrm{%
Supp\,}\chi_j$, 

iii) $\sigma_j\in
C^\infty_0(]-\delta_j,\widetilde\delta_j[)$ with $\int_{-\delta_j}^{\tilde\delta_j}\sigma_j(\eta)d\eta=1$.%

\noindent
Set $\widetilde\chi_j(y)\equiv\tau_j(w)\sigma_j(\eta)$.
Then  $\chi_j(x)$, $\widetilde\chi_j(y)$ satisfy \eqref{e-gue150924aII}.

One can check $\mathcal{%
P}_m(0)=Q^+_m$ and a little more work shows
\begin{equation}  \label{e-gue151025a}
\mathcal{P}^{\prime }_m(t)+\Box^+_{b,m}\mathcal{P}_m(t)=\mathcal{R}%
_m(t)\circ Q^+_m
\end{equation}
where for some $k$,
\begin{equation}  \label{e-gue151025ag}
\mathcal{R}_m(t)=\sum^N_{j=1}\sum^k_{\ell=1}L_{\ell,j}\Bigr(%
\chi_j(x)e^{-m\varphi_j(z)-im\theta}\Bigr)P_{\ell,j}\Bigr(%
e^{-t\Box^+_{B_j,m}}(z,w)\Bigr)e^{m\varphi_j(w)+im\eta}\widetilde\chi_j(y),
\end{equation}
$L_{\ell,j}$ is a partial differential operator of order $%
\geq1$ and  $\leq2$ (for all $\ell, j$) and $P_{\ell,j}$ is a partial differential operator
of order $1$ acting on $x$ (for all $\ell, j$). Since $\left\vert
e^{-t\Box^+_{B_j,m}}(z,w)\right\vert\sim {1\over t^n}e^{-\frac{\left\vert
z-w\right\vert^2}{t}}$, there could be terms of the form, say
\begin{equation}  \label{e-gue151025b}
P_{\ell,j}\Bigr(e^{-t\Box^+_{B_j,m}}(z,w)\Bigr)\sim {1\over t^n} e^{-\frac{\left\vert
z-w\right\vert^2}{t}}\frac{\left\vert z-w\right\vert}{t}.
\end{equation}
To require $\mathcal{P}^{\prime }_m(t)+\Box^+_{b,m}%
\mathcal{P}_m(t)$ to be small (as $t\to 0^+$) we need (by substituting \eqref{e-gue151025b} into
\eqref{e-gue151025ag} to get singular terms in powers of ${1\over t}$ smooth out):
\begin{equation}  \label{e-gue151025aI}
L_{\ell,j}\Bigr(\chi_j(x)e^{-m\varphi_j(z)-im\theta}\Bigr)%
e^{m\varphi_j(w)+im\eta}\widetilde\chi_j(y)=0\ \
\mbox{
if $z$ is close to $w$ ($\abs{z-w}\lesssim\sqrt{t}$)}.
\end{equation}

Since $\chi_j$ may not be constant on $\mathrm{Supp\,}\widetilde\chi_j$ (for some
$j$), it is hard for \eqref{e-gue151025aI} to hold.
Despite that in the usual (elliptic) case a construction of the heat kernel using cut-off functions
as above is available, in view that a distance function will appear in our asymptotic expression
(cf. \eqref{e-gue150626fIIIda} below) it is unclear whether this type of standard construction can be immediately carried out in our case.

It turns out that upon transferring to an \emph{adjoint} version of the original equation
one may bypass the aforementioned difficulty (cf. Lemmas~\ref{l-gue150626f},~\ref{l-gue150628}),
to which we turn now.

For $%
j=1,2,\ldots,N$ there exists $A_{B_j,+}(t,z,w)$ ($\in C^\infty(\mathbb{R}_+\times
U_j\times U_j,T^{*0,+}U\boxtimes (T^{*0,+}U)^*)$, cf. Theorem~\ref{t-gue150607}), regarded
as an {\it adjoint heat kernel}, such that
\begin{equation}  \label{e-gue150607abda}
\begin{split}
&\mbox{$\lim_{t\To0+}A_{B_j,+}(t)=I$ in $\mathscr D'(U,T^{*0,+}U)$}, \\
&A^{\prime }_{B_j,+}(t)u+A_{B_j,+}(t)(\Box^+_{B_j,m}u)=0,\ \ \forall
u\in\Omega^{0,+}_0(U),\ \ \forall t>0,
\end{split}%
\end{equation}
and $A_{B_j,+}(t,z,w)$ admits an asymptotic expansion as $t\rightarrow 0^+$
(see \eqref{e-gue150608a}). Put
\begin{equation}  \label{e-gue150627fda}
H_j(t,x,y)=\chi_j(x)e^{-m\varphi_j(z)-im\theta}A_{B_j,+}(t,z,w)e^{m%
\varphi_j(w)+im\eta}\widetilde\chi_j(y).
\end{equation}
Also set
\begin{equation}  \label{e-gue150627fIIda}
\Gamma(t):=\sum^N_{j=1}H_j(t)\circ
Q^+_m:\Omega^{0,+}(X)\rightarrow\Omega^{0,+}(X).
\end{equation}

By using the adjoint equation, we can avoid 
the difficulty mentioned in \eqref{e-gue151025aI} so that 
$\Gamma(t)$ gives an asymptotic (adjoint) heat kernel (see that below \eqref{e-gue150924aI}).
To get back to the kernel of the original equation, 
we can now start with the adjoint of $\Gamma(t)$.    By carrying out the (standard) method of successive approximation,
we can reach the global kernel of the adjoint of (the adjoint of) $e^{-t\Box^+_{b,m}}$ (Section~\ref{s-gue150627I}).
This yields the kernel of $e^{-t\Box^+_{b,m}}$ since $e^{-t\Box^+_{b,m}}$ is self-adjoint.
More precisely we can prove that (see Theorem~\ref{t-gue150630I} and
Theorem~\ref{t-gue160123})
\begin{equation}  \label{e-gue150626fIIIda}
\begin{split}
&\left\Vert e^{-t\Box^+_{b,m}}(x,y)-\Gamma(t,x,y)\right\Vert_{C^0(X\times
X)}\leq e^{-\frac{\varepsilon_1}{t}},\ \ \forall t\in(0,\varepsilon_2), \\
&\Gamma(t,x,x)\sim\Bigr(\sum\limits^{p}_{s=1}e^{\frac{2\pi(s-1)}{p}mi}\Bigr)%
\sum^\infty_{j=0}t^{-n+j}\alpha^+_{n-j}(x)\mod O\Bigr(t^{-n}e^{-\frac{\varepsilon_0\hat
d(x,X_{\mathrm{sing\,}})^2}{t}}\Bigr),\ \ \forall x\in X_p,
\end{split}%
\end{equation}
where $\alpha^+_s(x)\in C^\infty(X,\mathrm{End\,}(T^{*0,+}X\otimes E))$,
$s=n,n-1,\ldots$, $\varepsilon_0, \varepsilon_1, \varepsilon_2>0$ some constants and
$\hat d$ a sort of ``distance function" (discussed above Theorem~\ref{t-gue160114}).

The appearance of this distance function $\hat d$ may be attributed to the use
of projection $Q_m^+$ in \eqref{e-gue150627fIIda} (which picks up the
$m$-th Fourier component; see \eqref{e-gue150626fIII} and \eqref{e-gue160125V}).
See below for more about this point.  By the first inequality in \eqref{e-gue150626fIIIda}
one obtains the (same) asymptotic expansion 
\begin{equation}  \label{e-gue150924g}
e^{-t\Box^+_{b,m}}(x,x)\sim\Bigr(\sum\limits^{p}_{s=1}e^{\frac{2\pi(s-1)}{p}%
mi}\Bigr)\sum^\infty_{j=0}t^{-n+j}\alpha^+_{n-j}(x)\mod O\Bigr(t^{-n}e^{-%
\frac{\varepsilon_0\hat d(x,X_{\mathrm{sing\,}})^2}{t}}\Bigr)
\end{equation}
on $X_p$. Similar results hold for $e^{-t\Box^-_{b,m}}(x,x)$.

The terms involved in $O\Bigr(t^{-n}e^{-%
\frac{\varepsilon_0\hat d(x,X_{\mathrm{sing\,}})^2}{t}}\Bigr)$ of \eqref{e-gue150924g}
are singular (due to $t^{-n}$ as $x\to X_{\mathrm{sing}}$).   
Only upon taking the {\it supertrace} can these terms be (partially) cancelled
($t^{-n}$ dropping out).  That is, for $x\in X_p$
\begin{equation}  \label{e-gue160226u}
\begin{split}
&\mathrm{Tr\,}e^{-t\Box^+_{b,m}}(x,x)-\mathrm{Tr\,}e^{-t\Box^-_{b,m}}(x,x) \\
&\sim\Bigr(\sum\limits^{p}_{s=1}e^{\frac{2\pi(s-1)}{p}mi}\Bigr)%
\sum^\infty_{j=0}t^{-n+j}\Bigr(\mathrm{Tr\,}\alpha^+_{n-j}(x)-\mathrm{Tr\,}%
\alpha^-_{n-j}(x)\Bigr)\mod O\Bigr(e^{-\frac{\varepsilon_0\hat d(x,X_{\mathrm{sing\,}})^2%
}{t}}\Bigr).
\end{split}%
\end{equation}

To see this conceptually, let's take, for instance, \eqref{e-gue150923} and \eqref{e-gue150923I}
in which along the diagonal (i.e. setting $x=y$ to the left of \eqref{e-gue150923I}),
the off-diagonal contribution (in the term to the right of the same equation) 
still enters nontrivially  (unseen in the usual elliptic case) due to the projection $Q_m^+$.

To get estimates on these off-diagonal terms
our argument (cf. Theorem~\ref{t-gue150627g}) is based on the rescaling technique of Getzler and
on a supertrace identity in Berenzin integral (cf. Prop. 3.21 of \cite{BGV92}), which combine to give
the needed (partial) cancellation.  

From \eqref{e-gue150922}, \eqref{e-gue150924g} %
and \eqref{e-gue160226u} it follows 
\begin{equation}  \label{e-gue150924gII}
\sum^n_{j=0}(-1)^j\mathrm{dim\,}H^j_{b,m}(X,E)=\Bigr(\sum\limits^{p}_{s=1}e^{%
\frac{2\pi(s-1)}{p}mi}\Bigr)\lim_{t\rightarrow0^+}\int_X\sum^n_{\ell=0}t^{-%
\ell}\Bigr(\mathrm{Tr\,}\alpha^+_\ell(x)-\mathrm{Tr\,}\alpha^-_\ell(x)\Bigr)%
dv_X(x).
\end{equation}

Remark that we have had a (transversal) heat kernel which is put in the
disguise of the spectral geometry \eqref{e-gue151023}, \eqref{e-gue150525fbI}.   To our knowledge no
argument in the literature claims that (in the transversally elliptic case) the spectral heat kernel shall have
the asymptotic estimates as \eqref{e-gue150924g}. 
The somewhat lengthy part of our reconstruction of the (transversal)
heat kernel (beyond its spectral realization) becomes indispensable as far as
our purpose is concerned.

\subsubsection{Completion by evaluating local density and by using $\mathrm{%
Spin}^c$ structure}

\label{s-gue151025f}

As above we first treat the case that $X$ is CR K\"ahler (Definition~\ref{d-gue160308}).  
In view of \eqref{e-gue150924gII}, to complete the proof of our index theorem (cf. Corollary~\ref%
{c-gue150508I}) amounts to understanding the small $t$ behavior of the local
density
\begin{equation*}
\sum^n_{\ell=0}t^{-\ell}\Bigr(\mathrm{Tr\,}\alpha^+_\ell(x)-\mathrm{Tr\,}%
\alpha^-_\ell(x)\Bigr).
\end{equation*}
Let's be back to the local situation. Fix $x_0\in X_p$. Let $%
B_j=(D_j,(z,\theta),\varphi_j)$ ($j=1,2,\ldots,N$) be BRT trivializations as
before.  Assume that $x_0\in D_j$ and $x_0=(z_j,0)\in U_j\subset D_j$.

As our heat kernel (on $X$) is related to the local heat kernel
(on $U_j$), one sees (for some $N_0(n)\ge n$) 
\begin{equation}  \label{e-gue151025faIII}
\begin{split}
&\sum^{N_0(n)}_{\ell=0}t^{-\ell}\Bigr(\mathrm{Tr\,}\alpha^+_\ell(x_0)-\mathrm{Tr\,}%
\alpha^-_\ell(x_0)\Bigr) \\
&=\frac{1}{2\pi}\sum^N_{j=1}\chi_j(x_0)\Bigr(\mathrm{Tr\,}%
A_{B_j,+}(t,z_j,z_j)-\mathrm{Tr\,}A_{B_j,-}(t,z_j,z_j)\Bigr)+O(t),
\end{split}%
\end{equation}
where $A_{B_j,+}(t,z,w)$ is as in \eqref{e-gue150607abda}.

By borrowing the \emph{rescaling technique} in~\cite{BGV92} and~\cite{Du11} we
can show (in a fairly standard manner, cf. Theorem~\ref{t-gue150627} or the
second half of this section) that for each $j=1,2,\ldots,N$,
\begin{equation}  \label{e-gue151025h}
\begin{split}
&\Bigr(\mathrm{Tr\,}A_{B_j,+}(t,z,z)-\mathrm{Tr\,}A_{B_j,-}(t,z,z)\Bigr)%
dv_{U_j}(z) \\
&=[\mathrm{Td\,}(\nabla^{T^{1,0}U_j},T^{1,0}U_j)\wedge\mathrm{ch\,}%
(\nabla^{L^m},L^m)]_{2n}(z)+O(t),\ \ \forall z\in U_j,
\end{split}%
\end{equation}
($dv_{U_j}$ the induced volume form on $U_j$) where $\mathrm{Td\,}(\nabla^{T^{1,0}U_j},T^{1,0}U_j)
$  and $\mathrm{ch\,} (\nabla^{L^m},L^m)$ denote the representatives of the Todd class of $T^{1,0}U_j$
and the Chern character of $L^m$, respectively.  

A novelty here is Section~\ref{s-gue150508d} in which we
will introduce \emph{tangential characteristic classes}, \emph{tangential
Chern character} and \emph{tangential Todd class} on CR manifolds with $S^1$
action, so that
\begin{equation}  \label{e-gue151025y}
\frac{[\mathrm{Td\,}(\nabla^{T^{1,0}U_j},T^{1,0}U_j)\wedge\mathrm{ch\,}%
(\nabla^{L^m},L^m)]_{2n}(z_j)}{dv_{U_j}(z_j)}=\frac{[\mathrm{Td_b\,}%
(\nabla^{T^{1,0}X},T^{1,0}X)\wedge e^{-m\frac{d\omega_0}{2\pi}%
}\wedge\omega_0]_{2n+1}(x_0)}{dv_X(x_0)},
\end{equation}
where $\mathrm{Td_b\,}(\nabla^{T^{1,0}X},T^{1,0}X)$ denotes the
representative of the tangential Todd class of $T^{1,0}X$ (associated with the
given Hermitian metric \eqref{e-gue150607I}).
From \eqref{e-gue151025faIII}, \eqref{e-gue151025h} and \eqref{e-gue151025y}
it follows 
\begin{equation}  \label{e-gue150924gIII}
\begin{split}
&\sum^n_{\ell=0}t^{-\ell}\Bigr(\mathrm{Tr\,}\alpha^+_\ell(x)-\mathrm{Tr\,}%
\alpha^-_\ell(x)\Bigr)dv_X(x) \\
&=\frac{1}{2\pi}\left[\mathrm{Td_b\,}(\nabla^{T^{1,0}X},T^{1,0}X)\wedge e^{-m%
\frac{d\omega_0}{2\pi}}\wedge\omega_0\right]_{2n+1}(x)+O(t),\ \ \forall x\in X_p.
\end{split}%
\end{equation}
(The $O(t)$ term to the rightmost of \eqref{e-gue150924gIII} actually vanishes by using \eqref{e-gue150627}.)  

Combining \eqref{e-gue150924gIII} and \eqref{e-gue150924gII} we get our index theorem (cf. Corollary~\ref%
{c-gue150508I}) when $X$ is CR K\"ahler.

When $X$ is {\it not CR K\"ahler}, we still have \eqref{e-gue150626fIIIda}, %
\eqref{e-gue150924g} 
and \eqref{e-gue150924gII}. The ensuing 
\emph{obstackle} is more or less known: 

i) the rescaling technique does not 
quite work well as the local operator $\Box^+_{B_j,m}$ in %
\eqref{e-gue150607abda} is not going to be of Dirac type (in a strict sense); 

ii) it is obscure to
understand the small $t$ behavior of $A_{B_j,+}(t,z,z)$ in this case;

iii) \eqref{e-gue151025h} is not even true in general.

To overcome this difficulty in the CR case, we follow
the classical (yet nonK\"ahler) case and introduce some kind
of CR $\mathrm{Spin}^c$ Dirac operator on CR manifolds with $S^1$ action:
\begin{equation*}
\widetilde D_{b,m}=\overline\partial_b+\overline{\partial}^*_b+\mbox{zeroth
order term}
\end{equation*}
with modified/$\mathrm{Spin}^c$ Kohn Laplacians $\widetilde{%
\Box}^+_{b,m}=\widetilde D^*_{b,m}\widetilde D_{b,m}$, $\widetilde{\Box}%
^-_{b,m}=\widetilde D_{b,m}\widetilde D^*_{b,m}$.

A word of caution is in order.   The above adaptation of the idea of $\mathrm{Spin}^c$ structure 
to our CR case is not altogether straightforward.  Locally $X$ is realized as a (portion of a) circle 
bundle over a small piece of complex manifold (via BRT charts), so presumably there could 
arise a problem of patching up when this global $\mathrm{Spin}^c$ operator is to be formed.  
See Proposition~\ref{l-gue150524d} for more.  

We will show in Theorem~%
\ref{t-gue150530I} the \emph{homotopy
invariance} for the index of $\overline\partial_b+\overline{\partial}^*_b$, 
and in Corollary~\ref{t-gue150603} a \emph{McKean-Singer formula} 
for the modified Kohn Laplacians: for 
$t>0$,
\begin{equation}  \label{e-gue150924gf}
\sum^n_{j=0}(-1)^j\mathrm{dim\,}H^j_{b,m}(X,E)=\int_X\Bigr(\mathrm{Tr\,}e^{-t%
\widetilde{\Box}^+_{b,m}}(x,x)-\mathrm{Tr\,}e^{-t\widetilde{\Box}%
^-_{b,m}}(x,x)\Bigr)dv_X.
\end{equation}

For $u\in\Omega^{0,\pm}_m(X)$ we can write (on $D$) $%
u(z,\theta)=e^{-im\theta}\widetilde u(z)$ for some $\widetilde u(z)\in\Omega^{0,\pm}(U,L^m)$
with $D$ in a BRT trivialization $B:=(D,(z,\theta),\varphi)$.  

A fundamental relation that we will show in Proposition~\ref{l-gue150606}, based on Proposition~\ref{l-gue150524d}, 
is that
\begin{equation}  \label{e-gue151028}
e^{-m\varphi}\widetilde{\Box}^{\pm}_{B,m}(e^{m\varphi}\widetilde u)=e^{im\theta}%
\widetilde{\Box}^{\pm}_{b,m}(u)
\end{equation}
where $\widetilde{\Box}^{\pm}_{B,m}=D^*_{B,m}D_{B,m}:\Omega^{0,\pm}(U,L^m)%
\rightarrow\Omega^{0,\pm}(U,L^m)$
and  $%
D_{B,m}:\Omega^{0,\pm}(U,L^m)\rightarrow\Omega^{0,\mp}(U,L^m)$ the 
(ordinary) $\mathrm{%
Spin}^c$ Dirac operator  (cf. Definition~\ref{d-gue150825}) with respect to the Chern connection on $L^m$
(induced by $h^{L^m}$) and the 
\emph{Clifford connection} on $\Lambda(T^{*0,1}U)$
(induced by the given Hermitian metric on $\Lambda(T^{*0,1}U)$).

It is conceivable that $X$  with the
CR structure and $X/S^1=M$ with the complex structure (if defined)
are linked in some way (as Theorem~\ref{t-gue150508}).  
To say more, the result \eqref{e-gue151028} asserts a fundamental fact that not only complex/CR geometrically
can the two spaces be linked, but {\it metrically} in the sense of Laplacians they also can.  
This link is important for our $\mathrm{Spin}^c$ approach to the CR case to be possible.   

In the remaining let's give an outline with the
CR $\mathrm{Spin}^c$ Dirac operator when $X$ is not CR K\"ahler. Although 
the following ingredients mostly parallel those in the preceding Subsection~\ref{s-gue151025},
the success of this method relies on, among others, the $\mathrm{Spin}^c$ structure and the 
associated Clifford connection. For that reason and for the sake of clarity, we prefer to put
down the precise formulas despite the great similarity in expressions as above.  

As \eqref{e-gue150607abda},
there exists (modified) $\widetilde A_{B_j,+}(t,z,w)$ such that
\begin{equation}  \label{e-gue151028I}
\begin{split}
&\mbox{$\lim_{t\To0+}\Td A_{B_j,+}(t)=I$ in $\mathscr D'(U_j,T^{*0,+}U_j)$}, \\
&\widetilde A^{\prime }_{B_j,+}(t)u+\widetilde A_{B_j,+}(t)(\widetilde{\Box}%
^+_{B_j,m}u)=0,\ \ \forall u\in\Omega^{0,+}_0(U_j),\ \ \forall t>0,
\end{split}%
\end{equation}
and $\widetilde A_{B_j,+}(t,z,w)$ admits an asymptotic expansion as $%
t\rightarrow 0^+$ (see \eqref{e-gue150608a}). Put
\begin{equation}  \label{e-gue151028II}
\begin{split}
&\widetilde H_j(t,x,y)=\chi_j(x)e^{-m\varphi_j(z)-im\theta}\widetilde
A_{B_j,+}(t,z,w)e^{m\varphi_j(w)+im\eta}\widetilde\chi_j(y), \\
&\widetilde\Gamma(t)=\sum^N_{j=1}\widetilde H_j(t)\circ Q_m.
\end{split}%
\end{equation}

Similar to \eqref{e-gue150626fIIIda} and \eqref{e-gue150924g} in Subsection~\ref{s-gue151025},
one has 
\begin{equation}  \label{e-gue151028III}
\left\Vert e^{-t\widetilde{\Box}^+_{b,m}}(x,y)-\widetilde\Gamma(t,x,y)\right%
\Vert_{C^0(X\times X)}\leq e^{-\frac{\epsilon_1}{t}},\ \ \forall
t\in(0,\epsilon_2)
\end{equation}
and
\begin{equation}  \label{e-gue151028IIIu}
\widetilde\Gamma(t,x,x)\sim \Bigr(\sum\limits^{p}_{s=1}e^{\frac{2\pi(s-1)}{p}%
mi}\Bigr)\sum^\infty_{j=0}t^{-n+j}\widetilde\alpha^+_{n-j}(x)\mod O\Bigr(%
t^{-n}e^{-\frac{\varepsilon_0\hat d(x,X_{\mathrm{sing\,}})^2}{t}}\Bigr),\ \ \forall x\in
X_p,
\end{equation}
with some constants $\varepsilon_0, \varepsilon_1, \varepsilon_2>0$, 
giving 
\begin{equation}  \label{e-gue151028IIIw}
e^{-t\widetilde{\Box}^+_{b,m}}(x,x)\sim \Bigr(\sum\limits^{p}_{s=1}e^{\frac{%
2\pi(s-1)}{p}mi}\Bigr)\sum^\infty_{j=0}t^{-n+j}\widetilde\alpha^+_{n-j}(x)%
\mod O\Bigr(t^{-n}e^{-\frac{\varepsilon_0\hat d(x,X_{\mathrm{sing\,}})^2}{t}}\Bigr)
\end{equation}
on $X_p$.  Similar results hold for $e^{-t\widetilde{\Box}^-_{b,m}}(x,x)$.  

The novelty here is analogous to \eqref{e-gue160226u}.  By taking supertrace we can improve the estimates in %
\eqref{e-gue151028IIIw} 
(see Theorem~\ref%
{t-gue150701}) so that $t^{-n}$ is removed:  
\begin{equation}  \label{e-gue151028IIIwb}
\begin{split}
&\mathrm{Tr\,}e^{-t\widetilde{\Box}^+_{b,m}}(x,x)-\mathrm{Tr\,}e^{-t%
\widetilde{\Box}^-_{b,m}}(x,x) \\
&\sim \Bigr(\sum\limits^{p}_{s=1}e^{\frac{2\pi(s-1)}{p}mi}\Bigr)%
\sum^\infty_{j=0}t^{-n+j}\Bigr(\mathrm{Tr\,}\widetilde\alpha^+_{n-j}(x)-%
\mathrm{Tr\,}\widetilde\alpha^-_{n-j}(x)\Bigr)\mod O\Bigr(e^{-\frac{\hat
d(x,X_{\mathrm{sing\,}})^2}{t}}\Bigr),
\end{split}%
\end{equation}
for $x\in X_p$.
Hence \eqref{e-gue150924gf} and \eqref{e-gue151028IIIwb} give 
\begin{equation}  \label{e-gue151028g}
\sum^n_{j=0}(-1)^j\mathrm{dim\,}H^j_{b,m}(X,E)=\Bigr(\sum\limits^{p}_{s=1}e^{%
\frac{2\pi(s-1)}{p}mi}\Bigr)\lim_{t\rightarrow0^+}\int_X\sum^n_{\ell=0}t^{-%
\ell}\Bigr(\mathrm{Tr\,}\widetilde\alpha^+_\ell(x)-\mathrm{Tr\,}%
\widetilde\alpha^-_\ell(x)\Bigr)dv_X(x).
\end{equation}

A key advantage of introducing our CR $\mathrm{Spin}^c$ Dirac operator
is basically that \emph{Lichnerowicz formulas} hold for $%
\widetilde{\Box}^+_{B,m}$ (and $\widetilde{\Box}^-_{B,m}$).
This enables
us to apply the \emph{%
rescaling technique} (this part of rescaling is essentially the same as in classical cases, cf. ~\cite{BGV92} and~\cite{Du11})
and to obtain that for each $%
j=1,2,\ldots,N$,
\begin{equation}  \label{e-gue151028gI}
\begin{split}
&\Bigr(\mathrm{Tr\,}\widetilde A_{B_j,+}(z,z)-\mathrm{Tr\,}\widetilde
A_{B_j,-}(z,z)\Bigr)dv_{U_j}(z) \\
&=[\mathrm{Td\,}(\nabla^{T^{1,0}U_j},T^{1,0}U_j)\wedge\mathrm{ch\,}%
(\nabla^{L^m},L^m)]_{2n}(z)+O(t),\ \ \forall z\in U_j.
\end{split}%
\end{equation}
Rewriting \eqref{e-gue151028gI} in tangential forms, one has 
\begin{equation}  \label{e-gue151028gII}
\begin{split}
&\sum^n_{\ell=0}t^{-\ell}\Bigr(\mathrm{Tr\,}\widetilde\alpha^+_\ell(x)-%
\mathrm{Tr\,}\widetilde\alpha^-_\ell(x)\Bigr)dv_X(x) \\
&=\frac{1}{2\pi}\left[\mathrm{Td_b\,}(\nabla^{T^{1,0}X},T^{1,0}X)\wedge e^{-m%
\frac{d\omega_0}{2\pi}}\wedge\omega_0\right]_{2n+1}(x)+O(t)
\end{split}%
\end{equation}
for $t>0$ and $x\in X_p$.

Theorem~\ref{t-gue160114}, Theorem~\ref%
{t-gue160307} and Corollary~\ref{c-gue150508I}
follows from \eqref{e-gue151028IIIw}, %
\eqref{e-gue151028IIIwb}, \eqref{e-gue151028g} and %
\eqref{e-gue151028gII}.  

The layout of this paper is as follows. In Section~\ref{s-gue150508b} and
Section~\ref{s-gue150508bI}, we collect some notations, definitions,
terminologies and statements we use throughout. In Section~\ref{s-gue150508d},
we introduce the tangential de Rham cohomology group, tangential Chern character
and tangential Todd class on CR manifolds with $S^1$ action. In Section~\ref%
{s-gue150514}, we recall a classical result of Baouendi-Rothschild-Treves~\cite%
{BRT85} which plays an important role in our construction of the heat
kernel. We also prove that for a rigid vector bundle $F$ over $X$ there
exist rigid Hermitian metric and rigid connection on $F$. In Section~\ref%
{s-gue150517}, we establish a Hodge theory for Kohn Laplacian in the $L^2$
space of the $m$-th $S^1$ Fourier component. In Section~\ref{s-gue150524}, we
introduce our CR $\mathrm{Spin}^c$ Dirac operator $\widetilde D_{b,m}$, modified/$\mathrm{Spin}^c$
Kohn Laplacians $\widetilde{\Box}^{\pm}_{b,m}$  and
prove \eqref{e-gue150924gf}. In Section~\ref{s-gue150606}, we construct
approximate heat kernels for the operators $e^{-t\widetilde{\Box}^{\pm}_{b,m}}$ and prove that 
$e^{-t\widetilde{\Box}%
^{\pm}_{b,m}}(x,y)$ admit asymptotic
expansions in the sense as \eqref{e-gue151028III}. In Section~\ref{s-gue150702}, we prove %
\eqref{e-gue151028IIIu}, \eqref{e-gue151028IIIwb}, \eqref{e-gue151028gII}
and finish the proofs of Theorem~\ref{t-gue160114}, Theorem~\ref%
{t-gue160307} and Corollary~\ref{c-gue150508I}. In Section~\ref{s-gue160416}
we prove Theorem~\ref{t-gue160416}.

\bigskip

\begin{center}
\large{Part I: Preparatory foundations} 
\end{center}

\section{Preliminaries}

\label{s:prelim}

\subsection{Some standard notations}

\label{s-gue150508b} We use the following notations: $\mathbb{N}%
=\left\{1,2,\ldots\right\}$, $\mathbb{N}_0=\mathbb{N}\cup\left\{0\right\}$,
$\mathbb{R}$ is the set of real numbers,
$\overline{\mathbb{R}}_{+}:=\left\{x\in\mathbb{R};\, x\geq0\right\}$. For a multiindex $\alpha=(\alpha_1,\ldots,%
\alpha_n)\in\mathbb{N}_0^n$ we set $\left\vert
\alpha\right\vert=\alpha_1+\ldots+\alpha_n$. For $x=(x_1,\ldots,x_n)$ we
write
\begin{equation*}
\begin{split}
&x^\alpha=x_1^{\alpha_1}\ldots x^{\alpha_n}_n,\quad \partial_{x_j}=\frac{%
\partial}{\partial x_j}\,,\quad
\partial^\alpha_x=\partial^{\alpha_1}_{x_1}\ldots\partial^{\alpha_n}_{x_n}=%
\frac{\partial^{\left\vert \alpha\right\vert}}{\partial x^\alpha}\,, \\
&D_{x_j}=\frac{1}{i}\partial_{x_j}\,,\quad
D^\alpha_x=D^{\alpha_1}_{x_1}\ldots D^{\alpha_n}_{x_n}\,, \quad D_x=\frac{1}{%
i}\partial_x\,.
\end{split}
\end{equation*}
Let $z=(z_1,\ldots,z_n)$, $z_j=x_{2j-1}+ix_{2j}$, $j=1,\ldots,n$, be
coordinates of $\mathbb{C}^n$. We write
\begin{equation*}
\begin{split}
&z^\alpha=z_1^{\alpha_1}\ldots z^{\alpha_n}_n\,,\quad\overline
z^\alpha=\overline z_1^{\alpha_1}\ldots\overline z^{\alpha_n}_n\,, \\
&\partial_{z_j}=\frac{\partial}{\partial z_j}= \frac{1}{2}\Big(\frac{\partial%
}{\partial x_{2j-1}}-i\frac{\partial}{\partial x_{2j}}\Big)%
\,,\quad\partial_{\overline z_j}= \frac{\partial}{\partial\overline z_j}=%
\frac{1}{2}\Big(\frac{\partial}{\partial x_{2j-1}}+i\frac{\partial}{\partial
x_{2j}}\Big), \\
&\partial^\alpha_z=\partial^{\alpha_1}_{z_1}\ldots\partial^{\alpha_n}_{z_n}=%
\frac{\partial^{\left\vert \alpha\right\vert}}{\partial z^\alpha}\,,\quad
\partial^\alpha_{\overline z}=\partial^{\alpha_1}_{\overline
z_1}\ldots\partial^{\alpha_n}_{\overline z_n}= \frac{\partial^{\left\vert
\alpha\right\vert}}{\partial\overline z^\alpha}\,.
\end{split}
\end{equation*}

Let $X$ be a $C^\infty$ orientable paracompact manifold. We denote the tangent and cotangent bundle of $X$
by $TX$ and $T^*X$ respectively, and the complexified tangent and
cotangent bundle by $\mathbb{C }TX$ and $\mathbb{C }T^*X$.  
We write $\langle\,\cdot\,,\cdot\,\rangle$ to denote the
pointwise pairing between $T^*X$ and $TX$ and extend $\langle\,\cdot\,,\cdot%
\,\rangle$ bilinearly to $\mathbb{C }T^*X\times\mathbb{C }TX$. 

Let $E$, $F$ be $C^\infty$ vector bundles over $X$.  We
write $F\boxtimes E^*$ for the vector bundle over $X\times X$ with
fiber over $(x, y)\in X\times X$ consisting of linear maps from $E_y$ to
$F_x$.

Let $Y\subset X$ be an open subset. The spaces of smooth sections 
and {\it distribution sections} of $E$ over $Y$ will be denoted by $C^\infty(Y,
E)$ and $\mathscr D^{\prime }(Y, E)$ respectively. Let $\mathscr E^{\prime
}(Y, E)$ be the subspace of $\mathscr D^{\prime }(Y, E)$ whose elements are of 
{\it compact support} in $Y$. For $m\in\mathbb{R}$, we let $H^m(Y, E)$ denote the
Sobolev space of order $m$ for 
sections of $E$ over $Y$. Put
\begin{gather*}
H^m_{\mathrm{loc\,}}(Y, E)=\big\{u\in\mathscr D^{\prime }(Y, E); \,\varphi u\in H^m(Y, E),
\,\varphi\in C^\infty_0(Y)\big\}\,, \\
H^m_{\mathrm{comp\,}}(Y, E)=H^m_{\mathrm{loc}}(Y, E)\cap\mathscr E^{\prime
}(Y, E)\,.
\end{gather*}

\subsection{Set up and terminology}

\label{s-gue150508bI}

Let $(X, T^{1,0}X)$ be a compact CR manifold of dimension $2n+1$, $n\geq 1$,
where $T^{1,0}X$ is a CR structure of $X$. That is $T^{1,0}X$ is a subbundle
of rank $n$ of the complexified tangent bundle $\mathbb{C}TX$, satisfying $%
T^{1,0}X\cap T^{0,1}X=\{0\}$, where $T^{0,1}X=\overline{T^{1,0}X}$, and $[%
\mathcal{V},\mathcal{V}]\subset\mathcal{V}$, $\mathcal{V}=C^\infty(X,
T^{1,0}X)$. 

We assume that $X$ admits an $S^1$ action: $S^1\times
X\rightarrow X$. We write $e^{-i\theta}$ to denote the $S^1$ action. Let $%
T\in C^\infty(X, TX)$ be the global real vector field induced by the $S^1$
action given by $(Tu)(x)=\frac{\partial}{\partial\theta}\left(u(e^{-i\theta}%
\circ x)\right)|_{\theta=0}$ for $u\in C^\infty(X)$.

\begin{defn}\label{d-gue150508d1}
We say that the $S^1$ action $e^{-i\theta}$ is CR if $[T, C^\infty(X,
T^{1,0}X)]\subset C^\infty(X, T^{1,0}X)$ and the $S^1$ action is {\it transversal}
if for each $x\in X$, $\mathbb{C }T(x)\oplus T_x^{1,0}X\oplus T_x^{0,1}X=%
\mathbb{C}T_xX$. Moreover, we say that the $S^1$ action is {\it locally free} if $%
T\neq0$ everywhere.
\end{defn}

We assume throughout that $(X, T^{1,0}X)$ is a compact CR manifold with a
transversal CR locally free $S^1$ action $e^{-i\theta}$ with $T$ 
the global vector field induced by the $S^1$ action.  Let $\omega_0\in
C^\infty(X,T^*X)$ be the global real one form determined by $%
\langle\,\omega_0\,,\,u\,\rangle=0$ for all $u\in T^{1,0}X\oplus T^{0,1}X$,
and $\langle\,\omega_0\,,\,T\,\rangle=1$.

\begin{defn}
\label{d-gue150508f} For $p\in X$, the {\it Levi form} $\mathcal{L}_p$ is the
Hermitian quadratic form on $T^{1,0}_pX$ given by $\mathcal{L}_p(U,\overline
V)=-\frac{1}{2i}\langle\,d\omega_0(p)\,,\,U\wedge\overline V\,\rangle$, $U,
V\in T^{1,0}_pX$.
\end{defn}

If the Levi form $\mathcal{L}_p$ is semi-positive definite (resp. positive
definite), we say that $X$ is weakly pseudoconvex (resp. strongly pseudoconvex) at
$p$. If the Levi form is semi-positive definite (resp. positive definite) at every
point of $X$, we say that $X$is weakly pseudoconvex (resp. strongly pseudoconvex).

Denote by $T^{*1,0}X$ and $T^{*0,1}X$ the dual bundles of $T^{1,0}X$ and $%
T^{0,1}X$ respectively. Define the vector bundle of $(p,q)$ forms by $%
T^{*p,q}X=\Lambda^p(T^{*1,0}X)\wedge\Lambda^q(T^{*0,1}X)$. 

Let $D\subset X$ be an open subset and $E$ be a complex vector bundle over $D$.
Denote by $\Omega^{p,q}(D, E)$ (resp. $\Omega^{p,q}(D)$)
the space of smooth sections of $T^{*p,q}X\otimes
E$ (resp. $T^{*p,q}X$)) over $D$ and by $\Omega_0^{p,q}(D, E)$ (resp. $\Omega_0^{p,q}(D)$)
those elements of {\it compact support} in $D$.

Put
\begin{equation*}
\begin{split}
&T^{*0,\bullet}X:=\oplus_{j\in\left\{0,1,\ldots,n\right\}}T^{*0,j}X, \\
&T^{*0,+}X:=\oplus_{j\in\left\{0,1,\ldots,n\right\},\mbox{$j$ is even}%
}T^{*0,j}X, \\
&T^{*0,-}X:=\oplus_{j\in\left\{0,1,\ldots,n\right\},\mbox{$j$ is odd}%
}T^{*0,j}X.
\end{split}%
\end{equation*}
Put $\Omega^{0,\bullet}(X,E)$, $\Omega^{0,+}(X,E)$ and $ \Omega^{0,-}(X,E)$
in a similar way as above.


Fix $\theta_0\in]-\pi, \pi[$. 
Let $(e^{-i\theta_0})^*:\Lambda^r(\mathbb{C }T^*X)\rightarrow\Lambda^r(%
\mathbb{C }T^*X)$ be the pull-back map, 
$(e^{-i\theta_0})^*:T^{*p,q}_{e^{-i\theta_0}\circ x}X\rightarrow T^{*p,q}_{x}X$.  
Define for $u\in\Omega^{p,q}(X)$
\begin{equation}  \label{e-gue150508faII}
Tu:=\frac{\partial}{\partial\theta}\bigr((e^{-i\theta})^*u\bigr)%
|_{\theta=0}\in\Omega^{p,q}(X).
\end{equation}
(See also \eqref{lI}.) 

We shall write $u(e^{-i\theta}\circ
x):=(e^{-i\theta})^*u(x)$ for $u\in\Omega^{p,q}(X)$.   Clearly 
\begin{equation}  \label{e-gue150510f}
u(x)=\sum_{m\in\mathbb{Z}}\frac{1}{2\pi}\int^{\pi}_{-\pi}u(e^{-i\theta}\circ
x)e^{im\theta}d\theta.
\end{equation}

Let $\overline\partial_b:\Omega^{0,q}(X)\rightarrow\Omega^{0,q+1}(X)$ be the
tangential Cauchy-Riemann operator. From the CR property of the $S^1$
action it follows that (see also \eqref{e-gue150514f})
\begin{equation*}
T\overline\partial_b=\overline\partial_bT\ \ \mbox{on $\Omega^{0,q}(X)$}.
\end{equation*}

Naturally associated with the $S^1$ action are the so-called {\it rigid} objects.  See also 
\cite{BRT85} for a similar use of this term (cf. Definition II.2 of {\it loc.cit.}).  

\begin{defn}
\label{d-gue50508d} Let $D\subset X$ be an open set and $u\in C^\infty(D)$.
We say that $u$ is {\it rigid} if $Tu=0$, $u$ is Cauchy-Riemann (CR for short) 
if $\overline\partial_bu=0$ and $u$ is a rigid CR function if $\overline\partial_bu=0$ and $Tu=0$.
\end{defn}

\begin{defn}
\label{d-gue150508dI} Let $F$ be a complex vector bundle of rank $r$ over $X$. We say
that $F$ is {\it rigid} (resp. CR) if $X$ can be covered by open subsets $U_j$ with
trivializing frames $\{f^1_j,f^2_j,\dots,f^r_j\}$ such that the corresponding transition functions 
are rigid (resp. CR) (in the sense of the
preceding definition).  In this case the frames
$\{f^1_j,f^2_j,\dots,f^r_j\}$are called {\it rigid
frames} (resp. CR frames).
\end{defn}

Let $F$ be a rigid complex vector bundle over $X$ in the sense of Definition
\ref{d-gue150508dI}.

\begin{defn}
\label{d-gue150514f} Let $\langle\,\cdot\,|\,\cdot\,\rangle_F$ be a
Hermitian metric on $F$. We say that $\langle\,\cdot\,|\,\cdot\,\rangle_F$
is a {\it rigid Hermitian metric} if for every rigid local frames $\{f_1,\ldots, f_r\}$
of $F$, we have $T\langle\,f_j\,|\,f_k\,\rangle_F=0$, for $%
j,k=1,2,\ldots,r$.
\end{defn}

The condition of being rigid is not a severe restriction as far as the $S^1$ action is concerned.
See Theorems~\ref{t-gue150514fa} and~\ref{t-gue150515} which we shall prove
within the framework of BRT trivializations in the next section.

Henceforth let $E$ be a rigid CR vector bundle over $X$.  Write 
$\overline\partial_b:\Omega^{0,q}(X, E)\rightarrow\Omega^{0,q+1}(X,E)$
for the tangential Cauchy-Riemann operator.  
Since $E$ is rigid, we can define $Tu$ for $u\in\Omega^{0,q}(X,E)$ (cf. Theorem~\ref{t-gue150514fa}) 
and have 
\begin{equation}  \label{e-gue150508d}
T\overline\partial_b=\overline\partial_bT\ \ \mbox{on $\Omega^{0,q}(X,E)$}.
\end{equation}

For $m\in\mathbb{Z}$, let
\begin{equation}  \label{e-gue150508dI}
\Omega^{0,q}_m(X,E):=\left\{u\in\Omega^{0,q}(X,E);\, Tu=-imu\right\}
\end{equation}
and put $\Omega_m^{0,\bullet}(X,E)$, $\Omega_m^{0,+}(X,E)$ and $ \Omega_m^{0,-}(X,E)$
in a similar way as above. 

Put $\overline\partial_{b,m}:=\overline\partial_b:\Omega^{0,q}_m(X,E)%
\rightarrow\Omega^{0,q+1}_m(X,E)$ with a $\overline\partial_{b,m}$-complex:
\begin{equation*} \mbox{$\overline\partial_{b,m}$:}\,\,\,
\cdots\rightarrow\Omega^{0,q-1}_m(X,E)%
\rightarrow\Omega^{0,q}_m(X,E)\rightarrow\Omega^{0,q+1}_m(X,E)\rightarrow%
\cdots.
\end{equation*}
\noindent
Define
\begin{equation*}
H^q_{b,m}(X,E):=\frac{\mathrm{Ker\,}\overline\partial_{b,m}:%
\Omega^{0,q}_m(X,E)\rightarrow\Omega^{0,q+1}_m(X,E)}{\mathrm{Im\,}%
\overline\partial_{b,m}:\Omega^{0,q-1}_m(X,E)\rightarrow\Omega^{0,q}_m(X,E)}.
\end{equation*}
It is instructive to think of $H^q_{b,m}(X,E)$ as the $m$-th $S^1$ Fourier component of the $q$-th $%
\overline\partial_{b}$ Kohn-Rossi cohomology group. 

We will prove in Theorem~%
\ref{t-gue150517II} that $\mathrm{dim\,}H^q_{b,m}(X,E)<\infty$, for $%
m\in\mathbb{Z}$ and $q=0,1,2,\ldots,n$.  

We take a rigid Hermitian metric $\langle\,\cdot\,|\,\cdot\,%
\rangle_E$ on $E$ (in the sense of Definition~\ref{d-gue150514f}),
and a rigid Hermitian metric $\langle\,\cdot\,|\,%
\cdot\,\rangle$ on $\mathbb{C }TX$ such that 
\begin{equation}
\label{e-gue22a1}
T\perp (T^{1,0}X\oplus T^{0,1}X), \quad \langle\,T\,|\,T\,\rangle=1
\end{equation}
(and $T^{1,0}X\perp T^{0,1}X$).  (This is
always possible; see Theorem~\ref{t-gue150514fa} and Theorem 9.2 in~\cite%
{Hsiao14}.)

The Hermitian metric $\langle\,\cdot\,|\,\cdot\,\rangle$
on $\mathbb{C}TX$ induces by duality a Hermitian metric on $\mathbb{C}T^*X$
and on the bundles of $(0,q)$ forms $T^{*0,q}X$ ($q=0,1\cdots,n$),
to be denoted by $\langle\,\cdot\,|\,\cdot\,%
\rangle$ too.    A Hermitian metric denoted by $\langle\,\cdot\,|\,%
\cdot\,\rangle_{E}$ on $T^{*0,\bullet}X%
\otimes E$ is induced by those on $T^{*0,\bullet}X$ and $E$.
Let the lnear map $A(x,y)\in%
(T^{*,\bullet}X\otimes E)\boxtimes (T^{*,\bullet}X\otimes E)^*|_{(x, y)}$. We
write $\left\vert A(x,y)\right\vert$ to denote the natural matrix norm of $%
A(x,y)$ induced by $\langle\,\cdot\,|\,\cdot\,\rangle_{E}$.

We denote by $%
dv_X=dv_X(x)$ the induced volume form, and form the 
global $L^2$ inner products $(\,\cdot\,|\,\cdot\,)_{E}$ and $%
(\,\cdot\,|\,\cdot\,)$ on $\Omega^{0,\bullet}(X,E)$ and $\Omega^{0,%
\bullet}(X)$ respectively, with $L^2$-completion $L^2(X,T^{*0,q}X \otimes E)$ and $%
L^2(X,T^{*0,q}X)$.  Similar notation applies to $%
L^2_m(X,T^{*0,q}X\otimes E)$ and $L^2_m(X,T^{*0,q}X)$ (the completions of $%
\Omega^{0,q}_m(X,E)$ and $\Omega^{0,q}_m(X)$ with respect to $%
(\,\cdot\,|\,\cdot\,)_{E}$ and $(\,\cdot\,|\,\cdot\,)$).

Put $L^{2}(X,T^{*0,\bullet}X\otimes E)$, $L^{2,+}(X, E)$ and $L^{2,-}(X, E)$
in a similar way as above, and $L_m^{2}(X,T^{*0,\bullet}X\otimes E)$, $L_m^{2,+}(X, E)$ and $L_m^{2,-}(X, E)$
too.  

\subsection{Tangential de Rham cohomology group, Tangential Chern character
and Tangential Todd class}

\label{s-gue150508d}

In this section it is convenient to put $\Omega^r_{0}(X)=\left\{u\in
\oplus_{p+q=r}\Omega^{p,q}(X);\, Tu=0\right\}$ for $r=0,1,2,\ldots,2n$ (without the danger of
confusion with $\Omega_0^{p,q}$ in the preceding section) and set $\Omega^%
\bullet_{0}(X)=\oplus_{r=0}^{2n}\Omega^r_{0}(X)$. Since $Td=dT$ (see %
\eqref{e-gue150508d}), we have $d$-complex:
\begin{equation*}
d:\cdots\rightarrow\Omega^{r-1}_{0}(X)\rightarrow\Omega^{r}_{0}(X)%
\rightarrow\Omega^{r+1}_{0}(X)\rightarrow\cdots
\end{equation*}
Define the $r$-th tangential de Rham cohomology group:
\begin{equation*}
\mathcal{H}%
^r_{b,0}(X):=\frac{\mathrm{Ker\,}d:\Omega^{r}_{0}(X)\rightarrow%
\Omega^{r+1}_{0}(X)}{\mathrm{Im\,}d:\Omega^{r-1}_{0}(X)\rightarrow%
\Omega^{r}_{0}(X)}.
\end{equation*}
Put $\mathcal{H}^\bullet_{b,0}(X)=\oplus_{r=0}^{2n}%
\mathcal{H}^r_{b,0}(X)$.

Let a complex vector bundle $F$ over $X$ of rank $r$ be {\it rigid} as in Definition
\ref{d-gue150508dI}. We will show
in Theorem~\ref{t-gue150515} that there exists a connection $\nabla$ on $F$ such
that for any rigid local frame $f=(f_1,f_2,\ldots,f_r)$ of $F$ on an open
set $D\subset X$, the connection matrix $\theta(\nabla,f)=\left(\theta_{j,k}%
\right)^r_{j,k=1}$ satisfies
\begin{equation*}
\theta_{j,k}\in\Omega^1_{0}(D),
\end{equation*}
for $j,k=1,\ldots,r$. We call $\nabla$ as such a {\it rigid connection} on $F$. Let $%
\Theta(\nabla,F)\in C^\infty(X,\Lambda^2(\mathbb{C }T^*X)\otimes\mathrm{End\,%
}(F))$ be the associated {\it tangential curvature}.

Let $h(z)=\sum^\infty_{j=0}a_jz^j$ be a
real power series on $z\in\mathbb{C}$. Set
\begin{equation*}
H(\Theta(\nabla,F))=\mathrm{Tr\,}\Bigr(h\bigr(\frac{i}{2\pi}\Theta(\nabla,F)%
\bigr)\Bigr).
\end{equation*}
It is clear that $H(\Theta(\nabla,F))\in\Omega^{*}_{0}(X)$.

The following is
well-known (see Theorem B.5.1 in Ma-Marinescu~\cite{MM07}).

\begin{thm}
\label{t-gue150516} $H(\Theta(\nabla,F))$ is a closed differential form.
\end{thm}

That the tangential de Rham cohomology class
\begin{equation*}
[H(\Theta(\nabla,F))]\in\mathcal{H}^\bullet_{b,0}(X)
\end{equation*}
does not depend on the choice of rigid connections $\nabla$
is given by

\begin{thm}
\label{t-gue150516I} Let $\nabla^{\prime }$ be another rigid connection on $F
$. Then, $H(\Theta(\nabla,F))-H(\Theta(\nabla^{\prime },F))=dA$, for some $%
A\in\Omega^{*}_{0}(X)$.
\end{thm}

\begin{proof}
The idea of the proof is standard.
For each $t\in[0,1]$, put $\nabla_t=(1-t)\nabla+t\nabla^{\prime }$ which is a rigid connection on $F$.
Set
\begin{equation*}
Q_t=\frac{i}{2\pi}\mathrm{Tr\,}\Bigr(\frac{\partial \nabla_t}{\partial t}%
h^{\prime }\bigr(\frac{i}{2\pi}\Theta(\nabla_t,F)\bigr)\Bigr).
\end{equation*}
Since $\nabla_t$ is rigid, it is easily seen that
\begin{equation}  \label{e-gue150516}
Q_t\in\Omega^\bullet_{0}(X).
\end{equation}
It is well-known that (see Remark B.5.2 in Ma-Marinescu~\cite{MM07})
\begin{equation}  \label{e-gue150516I}
H(\Theta(\nabla,F))-H(\Theta(\nabla',F))=d\int^1_0Q_tdt.
\end{equation}
From \eqref{e-gue150516} and \eqref{e-gue150516I}, the theorem follows.
\end{proof}

For  $h(z)=e^{z}$ put
\begin{equation}  \label{e-gue160607}
\mathrm{ch_b\,}(\nabla,F):=H(\Theta(\nabla,F))\in\Omega^\bullet_{0}(X),
\end{equation}
and for
$h(z)=\log (\frac{z}{1-e^{-z}})$ set
\begin{equation}  \label{e-gue150607I}
\mathrm{Td_b\,}(\nabla,F):=e^{H(\Theta(\nabla,F))}\in\Omega^\bullet_{0}(X).
\end{equation}

We can now introduce tangential Todd
class and tangential Chern character.

\begin{defn}
\label{d-gue150516} The {\it tangential Chern character} of $F$ is given by
\begin{equation*}
\mathrm{ch_b\,}(F):=[\mathrm{ch_b\,}(\nabla,F)]\in\mathcal{H}%
^\bullet_{b,0}(X)
\end{equation*}
and the {\it tangential Todd class} of $F$ is given by
\begin{equation*}
\mathrm{Td_b\,}(F)=[\mathrm{Td_b\,}(\nabla,F)]\in\mathcal{H}%
^\bullet_{b,0}(X).
\end{equation*}
\end{defn}

Baouendi-Rothschild-Treves~\cite{BRT85} proved that $T^{1,0}X$ is a rigid
complex vector bundle over $X$ (cf. the first part of Theorem~\ref{t-gue150514fa} below). The tangential Todd class of
$T^{1,0}X$ and tangential Chern character of $T^{1,0}X$ are thus well defined.

The tangential Chern classes can be defined similarly.  
Put $\det(\frac{i\Theta(\nabla,F)}{2\pi}t+I)=\sum\limits^r_{j=0}%
\hat c_j(\nabla,F)t^j$.  Thus $\hat
c_j(\nabla,F)\in\Omega^{2j}_{0}(D)$.   By the matrix identity $\det
A=e^{\mathrm{Tr\,}(\log A)}$ and taking $h(z)=\log(1+z)$, one sees 
$\hat c_j(\nabla,F)$ ($j=0,1,\ldots,r$) is a closed differential
form on $X$ and its tangential de Rham cohomology class $[\hat
c_j(\nabla,F)]\in\mathcal{H}^{2j}_{b,0}(X)$ is independent of the choice of
rigid connections $\nabla$.  Put $\hat c_j(F)=[\hat c_j(\nabla,F)]\in\mathcal{H}%
^{2j}_{b,0}(X)$.  We call $\hat c_j(F)$ the $j$-th {\it tangential Chern class} of $F$,
and $\hat c(F)=1+\sum\limits^r_{j=1}\hat c_j(F)\in\mathcal{H}%
^\bullet_{b,0}(X)$ the {\it tangential total Chern class} of $F$.

\subsection{BRT trivializations and rigid geometric objects}

\label{s-gue150514}

In this paper, much of our strategy is heavily based on the following result
thanks to Baouendi-Rothschild-Treves~
\cite[Proposition I.2]{BRT85}.     Note in the following, $Z_j$ corresponds to 
$\overline L_j$ in their proposition.  Some geometrical significance related to a certain circle 
bundle structure will be discussed in the proof of Proposition~\ref{l-gue150524d}.  

\begin{thm}
\label{t-gue150514} For every point $x_0\in X$ there exist local
coordinates $x=(x_1,\cdots,x_{2n+1})=(z,\theta)=(z_1,\cdots,z_{n},\theta),
z_j=x_{2j-1}+ix_{2j},j=1,\cdots,n, x_{2n+1}=\theta$, defined in some small
neighborhood $D=\{(z, \theta): \left\vert z\right\vert<\delta,
-\varepsilon_0<\theta<\varepsilon_0\}$ of $x_0$, $\delta>0$, $%
0<\varepsilon_0<\pi$, such that $(z(x_0),\theta(x_0))=(0,0)$ and
\begin{equation}  \label{e-can}
\begin{split}
&T=\frac{\partial}{\partial\theta} \\
&Z_j=\frac{\partial}{\partial z_j}-i\frac{\partial\varphi}{\partial z_j}(z)%
\frac{\partial}{\partial\theta}, \quad j=1,\cdots,n
\end{split}%
\end{equation}
where $Z_j(x)$, $j=1,\cdots, n$, form a basis of $T_x^{1,0}X$ for each $x\in D
$ and $\varphi(z)\in C^\infty(D,\mathbb{R})$ is independent of $\theta$.
We summarize these data by the notation $(D,(z,\theta),\varphi)$.  

Furthermore, let $(D,(z,\theta),\varphi)$ and $(\widetilde
D,(w,\eta),\widetilde\varphi)$ be two such data on $D$.  
Then the coordinate transformation between them (on $D\cap\widetilde D$) can be given such that if 
$$w=(w_1,\ldots,w_n)=H(z)=(H_1(z),\ldots,H_n(z))$$
then 
\begin{equation}  \label{e-gue150524dIe1}
\begin{split}
&H_j(z)\in C^\infty(|z|<\delta),\ \
\overline\partial H_j(z)=0,\ \ \forall \,j\\
&\eta=\theta+\mathrm{arg}\,g(z)\quad\mbox{$(\mathrm{mod}\,2\pi)$\,\, where 
$\mathrm{arg}\,g(z)=\mathrm{Im}\log g(z)$}\\
&\widetilde\varphi(H(z), \overline {H(z)})=\varphi(z, \overline z)+\log|g(z)|
\end{split}%
\end{equation}
for some nowhere vanishing holomorphic function $g(z)$ on $|z|<\delta$.  
\end{thm}

\begin{rem}
\label{r-gue150514}
The relation between $\tilde\varphi$ and $\varphi$ in \eqref{e-gue150524dIe1} 
is a corrected version of a similar formula in \cite[the line below (I.31)]{BRT85}. 
See the proof of Proposition~\ref{l-gue150524d} for a derivation.  
\end{rem}

There exist examples that $H$ is not necessarily one to one.  Nevertheless, it can be shown 
that after shrinking $D$ and $\widetilde D$ properly, it is one to one, hence a biholomorphism.    

We call the above triple $(D,(z,\theta),\varphi)$ a {\it BRT trivialization}.
Note for $(z, \theta)\in D$ and $-\pi<\alpha<\pi$, 
 $e^{-i\alpha}\circ (z, \theta)=(z, \theta+\alpha)$ 
if $\{e^{-it\alpha}\circ (z, \theta)\}_{0\le t\le 1}\subset D$.

By using BRT trivializations some operations simplify, as follows.
Under the BRT triple $(D,(z,\theta),\varphi)$
it is clear that
\begin{equation*}
\{d\overline{z_{j_1}}\wedge\cdots\wedge d\overline{z_{j_q}}, 1\leq
j_1<\cdots<j_q\leq n\}
\end{equation*}
is a basis for $T^{\ast0,q}_xX$ for every $x\in D$.   For $u\in\Omega^{0,q}(X)
$, on $D$ we write
\begin{equation}  \label{e-gue150524fb}
u=\sum\limits_{j_1<\cdots<j_q}u_{j_1\cdots j_q}d\overline{z_{j_1}}%
\wedge\cdots\wedge d\overline{z_{j_q}}.
\end{equation}
Recall $T$ is the vector field associated with the $S^1$ action.
We have
\begin{equation}  \label{lI}
Tu=\sum\limits_{j_1<\cdots<j_q}(Tu_{j_1\cdots j_q})d\overline{z_{j_1}}%
\wedge\cdots\wedge d\overline{z_{j_q}}
\end{equation}
and $Tu$ is independent of the choice of BRT trivializations.

For $\overline\partial_b$ on
the BRT triple $(D,(z,\theta),\varphi)$ we have
\begin{equation}  \label{e-gue150514f}
\overline\partial_b=\sum^n_{j=1}d\overline z_j\wedge(\frac{\partial}{%
\partial\overline z_j}+i\frac{\partial\varphi}{\partial\overline z_j}(z)%
\frac{\partial}{\partial\theta}).
\end{equation}

The rigid objects (discussed in the preceding section) are natural geometric objects pertinent to the $S^1$ action.
In the following $X$ is again a compact connected CR manifold with a transversally CR locally free $S^1$ action.

\begin{thm}
\label{t-gue150514fa} Suppose $F$ is a
complex vector bundle over $X$ (not necessarily a CR bundle)
and admits an $S^1$ action compatible with that on $X$.
Then $F$ is actually a rigid vector bundle (in the sense of Definition~\ref{d-gue150508dI}).
Moreover there is a rigid Hermitian metric $\langle\,\cdot\,|\,%
\cdot\,\rangle_F$ on $F$.   Conversely if $F$ is a rigid vector bundle, then $F$ admits
a compatible $S^1$ action.
\end{thm}

\begin{proof}  We first work on the existence of a rigid Hermitian metric
(assuming $F$ is rigid).
Fix $p\in X$ and let $(D,(z,\theta),\varphi)$ be a BRT trivialization
around $p$ such that $(z(p),\theta(p))=(0,0)$, $(z, \theta)\in\{z\in\mathbb{C}^{n-1}:\,
\left\vert z\right\vert<\delta\}\times\{\theta\in\mathbb{R}:\, \left\vert
\theta\right\vert<\delta\}$ for some $\delta>0$.   Put
\begin{equation*}
\begin{split}
A:=&\{\lambda\in[-\pi,\pi]:\, \mbox{there is a local rigid trivializing
frame (l.r.t. frame for short ) defined on} \\
&\quad\quad\quad\mbox{$\set{e^{-i\theta}\circ(z,0);\,\, \abs{z}<\varepsilon,
\theta\in[-\pi,\lambda+\varepsilon)}$ for some $0<\varepsilon<\delta$}\}.
\end{split}%
\end{equation*}
Clearly  $A$ is a non-empty open set in $[-\pi,\pi]$.  We claim 
$A=[-\pi,\pi]$.   (Remark that the l.r.t. frame above is closely related to
the {\it canonical basis} in \cite[Definition I.3 without (I.29a)]{BRT85} when 
$E$ is $T^{1,0}X$.) 

It suffices to prove $A$ is closed. Let $\lambda_0$ be a limit point of $A$. 
For some small $\varepsilon_1>0$, there is a l.r.t. frame $\hat f=(\hat
f_1,\ldots,\hat f_r)$ defined on $\left\{e^{-i\theta}\circ(z,0);\,
\left\vert z\right\vert<\varepsilon_1,
\lambda_0-\varepsilon_1<\theta_1<\lambda_0+\varepsilon_1\right\}$.   By assumption $%
\lambda_0\in \overline A$ there exists a l.r.t.
frame $\widetilde f=(\widetilde f_1,\ldots,\widetilde f_r)$ defined around 
$\{e^{-i\theta}\circ (z,0)\}$ in which $\left\vert z\right\vert<\varepsilon_2,
\theta\in[-\pi,\lambda_0-\frac{\varepsilon_1}{2})$ for some $%
\varepsilon_2>0$. Now $\widetilde f=g\hat f$ on
\begin{equation*}
\left\{e^{-i\theta}\circ(z,0);\, \left\vert z\right\vert<\varepsilon_0,
\theta\in(\lambda_0-\varepsilon_1,\lambda_0-\frac{\varepsilon_1}{2})\right\},\quad 
\varepsilon_0=\min\left\{%
\varepsilon_1,\varepsilon_2\right\}
\end{equation*}
for some {\it rigid} $r\times r$ matrix $g$. 

We now patch up the frames.  Put , for $\theta\in[-\pi,\lambda_0-\frac{\varepsilon_1}{2})$,
$f=\widetilde f$ (on $\{e^{-i\theta}\circ(z,0)\}$)
and for $\theta\in[\lambda_0-\frac{\varepsilon_1}{2},\lambda_0+\varepsilon_1)$,
$f=g\hat f$ because $g$ is independent of $\theta$.  By $\widetilde f=g\hat f$ on the overlapping, 
$f$ is well-defined as a l.r.t. frame on
\begin{equation*}
\left\{e^{-i\theta}\circ(z,0);\, \left\vert z\right\vert<\varepsilon_0,
\theta\in[-\pi,\lambda_0+\varepsilon_1)\right\}.
\end{equation*}
extending $\theta=\lambda_0$.  
Thus $A$ is closed as desired.  

By the discussion above we can actually
find local rigid
trivializations $W_1,\ldots,W_N$ such that $X=\bigcup^N_{j=1}W_j$ and each $%
W_j\supset\bigcup_{-\pi\leq\theta\leq\pi}e^{-i\theta}W_j$ (i.e. $W_j$ is $S^1$ invariant).
Take any Hermitian metric $\langle\,\cdot\,,\,\cdot\,\rangle_F$ on $F$. Let $%
\langle\,\cdot\,|\,\cdot\,\rangle_F$ be the Hermitian metric on $F$ defined
as follows. For each $j=1,2,\ldots,N$, let $h^1_j,\ldots,h^r_j$ be local
rigid trivializing frames on $W_j$. Put $\langle\,h^s_j(x)\,|\,h^t_j(x)\,%
\rangle_F=\frac{1}{2\pi}\int^{\pi}_{-\pi}\langle\,h^s_j(e^{-iu}\circ
x)\,,\,h^t_j(e^{-iu}\circ x)\,\rangle_Fdu$, $s,t=1,2,\ldots,r$. 
One sees that $\langle\,\cdot\,|\,\cdot\,\rangle_F$ is
well-defined as a rigid Hermitian metric on $F$.

By examining the above reasoning we have also proved that
if $F$ is rigid, then it admits a natural $S^1$ action (by declaring the 
l.r.t. frames as $S^1$ invariant frames) compatible
with that on $X$.

For the reverse direction if $F$ admits a compatible $S^1$ action,
by using BRT trivializations one can construct $S^1$ invariant local frames.
These invariant local frames can be easily verified to be
local rigid frames, or equivalently the transition functions between
them are annihilated by $T$ due to the $S^1$ invariant property.   Hence
$F$ is rigid by definition.
\end{proof}

We shall now prove

\begin{thm}
\label{t-gue150515} Assume the complex vector bundle $F$ is rigid (on $X$).  There exists a rigid connection on $F$.
And if $F$ is equipped with a rigid Hermitian metric, there exists a rigid connection compatible
with this Hermitian metric.  Suppose $F$ is furthermore CR (and rigid) equipped with a rigid Hermitian metric $h$.
Then there exists a unique rigid connection (see the second paragraph of Subsection~\ref{s-gue150508d}) $\nabla^F$
compatible with $h$ such that $\nabla^F$ induces a  
Chern connection on $U_j$ (of any given BRT chart), and 
that the $S^1$ invariant sections are its parallel sections along $S^1$ orbits of $X$.   
\end{thm}

\begin{proof}
Let $\nabla$ be a connection on $F$.   For any $g\in S^1$ considering $g^*\nabla$
on $g^*F$ which is $F$ by using Theorem~\ref{t-gue150514fa} and summing over $g$ (in analogy with the
construction of a rigid Hermitian metric above), one obtains a rigid connection on $F$.
Suppose $F$ has a rigid Hermitian metric $\langle\,\cdot\,|\,\cdot\,\rangle_F$ 
and a connection $\nabla$ compatible with $\langle\,\cdot\,|\,\cdot\,\rangle_F$. 
One readily sees that the rigid connection resulting from the preceding procedure of summation, is still compatible with $%
\langle\,\cdot\,|\,\cdot\,\rangle_F$.    For the last statement, note that i) given a rigid CR bundle $F$ with a rigid Hermitian 
metric $h$ and any BRT chart $D_j=U_j\times ]-\varepsilon, \varepsilon[$, $(F, h)$ can descend to $U_j$ as a holomorphic vector bundle with the inherited metric, and ii) for a holomorphic vector bundle with a Hermitian metric, the Chern connection
is canonically defined.   Combining i) with ii) and using the $S^1$ invariant local frames (cf. proof of 
Theorem~\ref{t-gue150514fa}), one can construct a canonical connection $\nabla^F_j$ on $U_j$.  Then 
by using \eqref{e-can}, \eqref{e-gue150524dIe1} and the canonical property of the 
Chern connection, one sees that these $\nabla^F_j$ patch up to form a global 
connection $\nabla^F$ on $X$, satisfying the property as stated in the proposition.  
\end{proof}

\section{A Hodge theory for $\Box^{(q)}_{b,m}$}

\label{s-gue150517}

A Hodge decomposition theorem for $\overline\partial_b$ on
pseudoconvex CR manifolds has been well developed.  See \cite[Section 9.4]{CS01}
for a nice presentation in some respects; see also \cite{Ta}. Our goal of this section is to
develop an analogous theory for $\Box^{(q)}_{b,m}$ on CR manifolds with transversal CR locally free $S^1$ action (irrespective of pseudoconvexity).  
Much of what follows appears to parallel the corresponding part of
Hodge theory in complex geometry.  

Besides the relevance to 
the index theorem on CR manifolds, the present theory has an application to our proof of 
homotopy invariance of index (Theorem~\ref{t-gue150530I}).  

As before, $X$ is a compact CR manifold with a transversal CR locally
free $S^1$ action.   Let $
\overline{\partial}^{*}_b:\Omega^{0,q+1}(X,E)\rightarrow\Omega^{0,q}(X,E)
$ ($q=0,1,2,\ldots,n$) 
be the formal adjoint of $\overline\partial_b$ with respect to $%
(\,\cdot\,|\,\cdot\,)_E$.   Put
$
\Box^{(q)}_b:=\overline\partial_b\overline{\partial}^{*}_b+\overline{\partial%
}^{*}_b\overline\partial_b:\Omega^{0,q}(X,E)\rightarrow\Omega^{0,q}(X,E)$.
$T$ is the vector field on $X$ induced by the $S^1$ action, $T\overline\partial_b=
\overline\partial_b T$ and $\overline\partial_{b,m}:=\overline\partial_b|_{\Omega_{m}^{0,q}}
:\Omega^{0,q}_m(X,E)\rightarrow
\Omega^{0,q+1}_m(X,E)$ on eigenspaces of the $S^1$ acton $(\forall m\in\mathbb{Z})$.  

Recall $\langle\,\cdot\,|\,\cdot\,\rangle_E$ is
rigid.  One sees $T\overline\partial^{*}_b=\overline\partial^{*}_bT$ so that 
$\overline\partial^{*}_b|_{\Omega_m^{0,q+1}}:\Omega^{0,q+1}_m(X,E)\rightarrow
\Omega^{0,q}_m(X,E)$
is the same as the formal adjoint $\overline\partial^*_{b,m}$
of $\overline\partial_{b,m}$.   Form $\Box^{(q)}_{b,m}=
\overline\partial_{b,m}\overline\partial^*_{b,m}+\overline\partial^*_{b,m}\overline\partial_{b,m}:\Omega^{0,q}_m(X,E)\rightarrow\Omega^{0,q}_m(X,E)$. 
We have $\Box^{(q)}_{b,m}=\Box^{(q)}_b|_{\Omega^{0,q}_m(X,E)}$.  



On a general compact CR manifold, there is a fundamental result that follows from Kohn's $L^2$ estimates.
(See~\cite[Theorem 8.4.2]%
{CS01}).   Adapting it to our present situation, we can state the result as follows. 

\begin{thm}
\label{t-gue150517} For every $s\in\mathbb{N}_0$, there is a constant $C_s>0$
such that
\begin{equation*}  \label{e-gue150517IV}
\left\Vert u\right\Vert_{s+1}\leq C_s\Bigr(\left\Vert
\Box^{(q)}_bu\right\Vert_s+\left\Vert Tu\right\Vert_s+\left\Vert
u\right\Vert_s\Bigr),\ \ \forall u\in\Omega^{0,q}(X,E),
\end{equation*}
where $\left\Vert \cdot\right\Vert_s$ denotes the usual Sobolev norm of
order $s$ on $X$.
\end{thm}

Theorem~\ref{t-gue150517} restricted to $\Omega^{0,q}_m(X,E)$ yields 

\begin{cor}
\label{t-gue150517I} Fix $m\in\mathbb{Z}$. For every $s\in\mathbb{N}_0$,
there is a constant $C_s>0$ such that
\begin{equation*}  \label{e-gue150517V}
\left\Vert u\right\Vert_{s+1}\leq C_s\Bigr(\left\Vert
\Box^{(q)}_{b,m}u\right\Vert_s+\left\Vert u\right\Vert_s\Bigr),\ \ \forall
u\in\Omega^{0,q}_m(X,E).
\end{equation*}
\end{cor}

This suggests that a good regularity theory might exist on our $X$.  
Observe that $\Box^{(q)}_{b}-T^2$  is {\it elliptic} on $X$ while $\Box^{(q)}_{b}$ is not, 
which on $\Omega_m^{0,q}(X, E)$ is 
$\Box^{(q)}_{b,m}+m^2$.    In fact, without 
using the above theorems all of the following results are essentially proven 
by standard results in elliptic theory.  

Write ${\rm Dom\,}\Box^{(q)}_{b,m}:=\{u\in L^2_m(X,T^{*0,q}X\otimes E);\, \Box^{(q)}_{b,m}u\in L^2_m(X,T^{*0,q}X\otimes E)\}$ where $\Box^{(q)}_{b,m}u$ is defined in the sense of distribution.
$\Box^{(q)}_{b,m}$ is extended by
\begin{equation}  \label{e-gue150517III}
\Box^{(q)}_{b,m}:\mathrm{Dom\,}\Box^{(q)}_{b,m}\,(\subset
L^2_m(X,T^{*0,q}X\otimes E))\rightarrow L^2_m(X,T^{*0,q}X\otimes E).
\end{equation}

\begin{lem}
\label{l-gue150517} We have $\mathrm{Dom\,}%
\Box^{(q)}_{b,m}=L^2_m(X,T^{*0,q}X\otimes E)\bigcap H^2(X,T^{*0,q}X\otimes E)
$.
\end{lem}

\begin{proof}
For the inclusion put $%
v=\Box^{(q)}_{b,m}u\in L^2_m(X,T^{*0,q}X\otimes E)$.   Then $%
(\Box^{(q)}_{b,m}-T^2)u=v+m^2u\in L^2_m(X,T^{*0,q}X\otimes E)$. Since $%
(\Box^{(q)}_{b}-T^2)$ is elliptic, we conclude $u\in
H^2(X,T^{*0,q}X\otimes E)$.   The reverse inclusion is clear.  
\end{proof}

\begin{lem}
\label{t-gue150517a} $\Box^{(q)}_{b,m}:\mathrm{Dom\,}\Box^{(q)}_{b,m}\,(\subset
L^2_m(X,T^{*0,q}X\otimes E))\rightarrow L^2_m(X,T^{*0,q}X\otimes E)$ is
self-adjoint.
\end{lem}

\begin{proof}  Since the similar extension of 
$\Box^{(q)}_{b}$ on $L^2(X,T^{*0,q}X\otimes E)$ is self-adjoint 
and its restriction to (an invariant subspace) $L^2_m(X,T^{*0,q}X\otimes E)$ gives 
$\Box^{(q)}_{b,m}$, $\Box^{(q)}_{b,m}$ is also self-adjoint. 
\end{proof}


Let $\mathrm{Spec\,}\Box^{(q)}_{b,m}\subset[0,\infty[$ denote the spectrum
of $\Box^{(q)}_{b,m}$ (Davies~\cite{Dav95}). 

\begin{prop}
\label{t-gue150517aI} $\mathrm{Spec\,}\Box^{(q)}_{b,m}$ is a discrete subset
of $[0,\infty[$.   For any $\nu\in\mathrm{Spec\,}\Box^{(q)}_{b,m}$, $\nu$ is an
eigenvalue of $\Box^{(q)}_{b,m}$ and the eigenspace
\begin{equation*}
\mathcal{E}^q_{m,\nu}(X,E):=\big\{u\in\mathrm{Dom\,}\Box^{(q)}_{b,m};\,
\Box^{(q)}_{b,m}u=\nu u\big\}
\end{equation*}
is finite dimensional with $\mathcal{E}^q_{m,\nu}(X,E)\subset%
\Omega^{0,q}_m(X,E)$.
\end{prop}

\begin{proof}
 $\Box^{(q)}_{b}-T^2\equiv \Delta$ is a second order elliptic operator. 
By standard elliptic theory, $\Delta$ and hence $\Delta+m^2$, satisfy the statement of the proposition on 
the (invariant) subspace $L^2_m(X,T^{*0,q}X\otimes E)\supset \Omega_m^{0,q}(X, E)$.  
On it $T^2$ acts as 
$-m^2$, the proposition follows for 
$\Box^{(q)}_{b,m}$ which is $\Delta+m^2$ on $L^2_m(X,T^{*0,q}X\otimes E)$.  
\end{proof}

A role analogous to the Green's operator in the ordinary Hodge
theory is given as follows.  Let
\begin{equation*}
N^{(q)}_m:L^2_m(X,T^{*0,q}X\otimes E)\rightarrow\mathrm{Dom\,}%
\Box^{(q)}_{b,m}
\end{equation*}
be the partial inverse of $\Box^{(q)}_{b,m}$ and let
\begin{equation*}
\Pi^{(q)}_m:L^2_m(X,T^{*0,q}X\otimes E)\rightarrow\mathrm{Ker\,}%
\Box^{(q)}_{b,m}
\end{equation*}
be the orthogonal projection. We have
\begin{equation}  \label{e-gue150520III}
\begin{split}
&\Box^{(q)}_{b,m}N^{(q)}_m+\Pi^{(q)}_m=I\ \ \mbox{on
$L^2_m(X,T^{*0,q}X\otimes E)$}, \\
&N^{(q)}_m\Box^{(q)}_{b,m}+\Pi^{(q)}_m=I\ \ \mbox{on ${\rm
Dom\,}\Box^{(q)}_{b,m}$}.
\end{split}%
\end{equation}

\begin{lem}
\label{l-gue150524} We have $N^{(q)}_m:\Omega^{0,q}_m(X,E)\rightarrow%
\Omega^{0,q}_m(X,E)$.
\end{lem}

\begin{proof} A slight variant of the standard argument 
applies as $\Box^{(q)}_{b}$ is almost elliptic.  
Let $u\in\Omega^{0,q}_m(X,E)$ and put $N^{(q)}_mu=v\in
L^2_m(X,T^{*0,q}X\otimes E)$.   By \eqref{e-gue150520III},
$(I-\Pi^{(q)}_m)u=\Box^{(q)}_{b,m}v$, giving
\begin{equation}  \label{e-gue150524}
(\Box^{(q)}_{b,m}-T^2)v=(I-\Pi^{(q)}_m)u+m^2v.
\end{equation}
By Proposition~\ref{t-gue150517aI}, $\mathrm{Ker\,}\Box^{(q)}_{b,m}$
consists of smooth sections, so $\Pi^{(q)}_mu$ is smooth and
\begin{equation}  \label{e-gue150524I}
(I-\Pi^{(q)}_m)u\in\Omega^{0,q}_m(X,E).
\end{equation}
By combining 
\eqref{e-gue150524} and \eqref{e-gue150524I} and noting 
$\Box^{(q)}_{b}-T^2$ is elliptic, the standard
technique in elliptic regularity applies to give $%
v\in\Omega^{0,q}_m(X,E)$.  
\end{proof}

The following is a version of ``harmonic realization" of
cohomology.

\begin{thm}
\label{t-gue150517II} For every $q\in\left\{0,1,2,\ldots,n\right\}$ and
every $m\in\mathbb{Z}$, we have
\begin{equation}  \label{e-gue150517VI}
\mathrm{Ker\,}\Box^{(q)}_{b,m}=\mathcal{E}^q_{m,0}(X,E)\cong H^q_{b,m}(X,E).
\end{equation}

As a consequence $\mathrm{dim\,}H^q_{b,m}(X, E)<\infty$ by Proposition~\ref{t-gue150517aI}.
\end{thm}

\begin{proof}  The argument is mostly standard (although $\Box^{(q)}_{b}$ 
is not elliptic).    Consider the map
\begin{equation}\label{e-gue150517IIe1}
\begin{split}
\tau^q_m:\mathrm{Ker\,}\overline\partial_{b,m}\bigcap\Omega^{0,q}_m(X,E)&%
\rightarrow\mathrm{Ker\,}\Box^{(q)}_{b,m}, \\
u&\rightarrow\Pi^{(q)}_mu.
\end{split}%
\end{equation}
Clearly $\tau^q_m$ is surjective.  
Put $M^q_m:=\left\{\overline%
\partial_{b,m}u;\, u\in\Omega^{0,q-1}_m(X,E)\right\}$.
The theorem follows if one shows 
\begin{equation}  \label{e-gue150524a}
\mathrm{Ker\,}\tau^q_m=M^q_m.
\end{equation}
It is easily seen $M^q_m\subset\mathrm{Ker\,}\tau^q_m$ since
$M^q_m\perp \mathrm{Ker}\,\Box^{(q)}_{b, m}$.  
For the reverse let $u\in\mathrm{Ker\,%
}\tau^q_m$ so $\Pi^{(q)}_mu=0$. From \eqref{e-gue150520III} we have
\begin{equation}  \label{e-gue150524aI}
\begin{split}
u&=\Box^{(q)}_{b,m}N^{(q)}_mu+\Pi^{(q)}_mu \\
&=(\overline\partial_b\,\overline{\partial}^*_b+\overline{\partial}%
^*_b\,\overline\partial_b)N^{(q)}_mu.
\end{split}%
\end{equation}
We claim that
\begin{equation}  \label{e-gue150524aII}
\overline{\partial}^*_b\,\overline\partial_bN^{(q)}_mu=0.
\end{equation}
One sees, by using $\overline\partial_b^2=0$ for the first equality below,
\begin{equation}  \label{e-gue150524aIII}
\begin{split}
(\,\overline{\partial}^*_b\,\overline\partial_bN^{(q)}_mu\,|\,\overline{%
\partial}^*_b\,\overline\partial_bN^{(q)}_mu\,)_E&=(\,\overline{\partial}%
^*_b\,\overline\partial_b\Box^{(q)}_mN^{(q)}_mu\,|\,N^{(q)}_mu\,)_E=(\,%
\overline{\partial}^*_b\,\overline\partial_b(I-\Pi^{(q)}_m)u\,|%
\,N^{(q)}_mu)_E \\
&=(\,\overline{\partial}^*_b\,\overline\partial_bu\,|\,N^{(q)}_mu)_E
\end{split}%
\end{equation}
which is zero because $u\in\mathrm{Ker}\,\overline\partial_{b,m}$ by 
\eqref{e-gue150517IIe1}, giving the claim \eqref{e-gue150524aII}.  
By \eqref{e-gue150524aI} and \eqref{e-gue150524aII}, 
\begin{equation}  \label{e-gue150524aIV}
u=\overline\partial_b\,\overline{\partial}^*_bN^{(q)}_mu,
\end{equation}
with $\overline{\partial}^*_bN^{(q)}_mu\in\Omega^{0,q-1}_m(X,E)$ by Lemma~\ref{l-gue150524}.
By \eqref{e-gue150524aIV}, $u\in M^q_m$, yielding the desired inclusion 
$\mathrm{Ker\,}\tau^q_m\subset M^q_m$.   Hence \eqref{e-gue150524a}.

\end{proof}

Let $D_{b,m}:=\overline\partial_b+\overline{\partial}^*_b:
\Omega^{0,+}_m(X,E)\rightarrow\Omega^{0,-}_m(X,E)$ 
with extension $D_{b,m}$
\begin{equation*}
D_{b,m}:\mathrm{Dom\,}D_{b,m}\,(\subset L^{2,+}_m(X,E))\rightarrow
L^{2,-}_m(X,E),
\end{equation*}
$\mathrm{Dom\,}D_{b,m}=\left\{u\in L^{2,+}_m(X,E);\, \mbox{distribution $D_{b,m}u\in L^{2,+}_m(X,E)$}\right\}$.
The Hilbert space adjoint of $D_{b,m}$ with respect to $(\,\cdot\,|\,
\cdot\,)_E$ is given by $D^*_{b,m}: \mathrm{Dom\,}D^*_{b,m}\,(\subset L^{2,-}_m(X,E))\rightarrow
L^{2,+}_m(X,E)$.  

Combining Proposition~\ref{t-gue150517aI} and Theorem~\ref{t-gue150517II}, 
one can verify (as in standard Hodge theory) 

\begin{thm}
\label{t-gue150524f} In the notation above 
\begin{equation}  \label{e-gue150524f}
\begin{split}
&\mathrm{Ker\,}D_{b,m}=\bigoplus_{\scriptstyle q\in\left\{0,1,\ldots,n\right\}\atop\scriptstyle\ q\,\,\rm{even}}\mathrm{Ker\,}\Box^{(q)}_{b,m}\,\,(\subset\Omega^{0,+}_m(X,E)), \\
&\mathrm{Ker\,}D^*_{b,m}=\bigoplus_{\scriptstyle q\in\left\{0,1,\ldots,n\right\}\atop\scriptstyle\ q\,\,\rm{odd}}\mathrm{Ker\,}\Box^{(q)}_{b,m}\,\,(\subset\Omega^{0,-}_m(X,E)).
\end{split}%
\end{equation}
Put $\mathrm{ind\,}D_{b,m}:=\mathrm{dim\,}\mathrm{Ker\,}D_{b,m}-\mathrm{dim\,%
}\mathrm{Ker\,}D^*_{b,m}$.  Hence, together with Theorem~\ref{t-gue150517II}, 
\begin{equation}  \label{e-gue150524fI}
\sum^n_{j=0}(-1)^j\mathrm{dim\,}H^j_{b,m}(X,E)=\mathrm{ind\,}D_{b,m}.
\end{equation}
\end{thm}

\section{Modified Kohn Laplacian ($\mathrm{Spin}^c$ Kohn Laplacian)}

\label{s-gue150524}

We are prepared by Theorem~\ref{t-gue150524f} above to see 
that to calculate $\sum^n_{j=0}(-1)^j\mathrm{dim\,}%
H^j_{b,m}(X,E)$ is the same as to calculate the index 
$\mathrm{ind\,}D_{b,m}$.  
To do so effectively we need to modify the Dirac type operator
$D_{b, m}$ hence the standard Kohn Laplacian because
the modified versions $\widetilde D_{b,m}$, 
$\widetilde\Box_{b,m}$ will have a manageable heat kernel that suits our purpose better
for the CR non-K\"ahler case (cf. Remark~\ref{r-gue151003}).
Lastly we shall give an argument for the homotopy invariance,
and obtain $\mathrm{ind\,}D_{b,m}=\mathrm{ind\,}\widetilde D_{b,m}$.  


The main idea here is borrowed from that of classical cases. 
But as the CR manifold $X$ is not assumed to be a (orbifold) circle bundle globally,
there could arise the problem of patching (from local constructions
to the global one).   Part of the technicality in the beginning of this section 
lies in a careful treatment in this regard.

We recall some basics of Clifford
connection and $\mathrm{Spin}^c$ Dirac operator. For more details 
we refer to Chapter 1 in~%
\cite{MM07} and~\cite{Du11}.

Let $B:=(D,(z,\theta),\varphi)$ be a BRT trivialization with $%
D=U\times]-\varepsilon,\varepsilon[$ where $\varepsilon>0$ and $U$ is an
open set of $\mathbb{C}^n$.   Using $\varphi$ in $B$, we let 
$%
\langle\,\cdot\,,\,\cdot\,\rangle$ be the Hermitian metric on $\mathbb{C }TU$
induced by that on $D$
\begin{equation}\label{e-gue150823e1}
\langle\,\frac{\partial}{\partial z_j}\,,\,\frac{\partial}{\partial z_k}%
\,\rangle=\langle\,\frac{\partial}{\partial z_j}-i\frac{\partial\varphi}{%
\partial z_j}(z)\frac{\partial}{\partial\theta}\,|\,\frac{\partial}{\partial
z_k}-i\frac{\partial\varphi}{\partial z_k}(z)\frac{\partial}{\partial\theta}%
\,\rangle,\ \ j,k=1,2,\ldots,n
\end{equation}
(cf. Theorem~\ref{t-gue150514}).  
By \eqref{e-can} and Theorem~\ref{t-gue150514}, 
the above metric is actually intrinsically defined.  

The $\langle\,\cdot\,,\,\cdot\,\rangle$ induces Hermitian metrics on $T^{*0,q}U$
still denoted by $\langle\,\cdot\,,\,\cdot\,\rangle$
and a Riemannian metric $g^{TU}$ on $TU$.

For any $v\in TU$ with decomposition $v=v^{(1,0)}+v^{(0,1)}\in
T^{1,0}U\oplus T^{0,1}U$, let $\overline v^{(1,0),*}\in T^{*0,1}U$ be the
metric dual of $v^{(1,0)}$ with respect to $\langle\,\cdot\,,\,\cdot\,\rangle
$. That is, 
$\overline v^{(1,0),*}(u)=\langle\,v^{(1,0)}\,,\,\overline u\,\rangle$ 
for all $u\in T^{0,1}U$.  

The Clifford action $v$ on $\Lambda(T^{*0,1}U):=%
\oplus^n_{q=0}T^{*0,q}U$ is defined by 
\begin{equation}  \label{e-gue150823I}
c(v)(\cdot)=\sqrt{2}(\overline v^{(1,0),*}\wedge(\cdot)-i_{v^{(0,1)}}(\cdot))
\end{equation}
(where $\wedge$ and $i$ denote the exterior and
interior product respectively).

Let $\left\{w_j\right\}^n_{j=1}$ be a local orthonormal frame of $T^{1,0}U$
with respect to $\langle\,\cdot\,,\,\cdot\,\rangle$ with dual frame $%
\left\{w^j\right\}^n_{j=1}$. Write 
\begin{equation}  \label{e-gue150825}
\mbox{$e_{2j-1}=\frac{1}{\sqrt{2}}(w_j+\ol w_j)$ and
$e_{2j}=\frac{i}{\sqrt{2}}(w_j-\ol w_j)$, $j=1,2,\ldots,n$},
\end{equation}
for an orthonormal frame of $TU$.
Let $\nabla^{TU}$ be the Levi-Civita connection on $TU$ (with respect to $%
g^{TU}$), and $\nabla^{\mathrm{det\,}}$ be the Chern connection on the determinant line bundle $\mathrm{%
det\,}(T^{1,0}U)$ (with $\langle\,\cdot\,,\,\cdot\,\rangle$),
with connection forms  $\Gamma^{TU}$ and $\Gamma^{\mathrm{det\,}}$
associated to the frames $\left\{e_j\right\}^{2n}_{j=1}$ and $%
w_1\wedge\cdots\wedge w_n$.
That is,
\begin{equation}  \label{e-gue150825I}
\begin{split}
&\nabla^{TU}_{e_j}e_{\ell}=\Gamma^{TU}(e_j)e_{\ell},\ \ j,\ell=1,2,\ldots,2n,
\\
&\nabla^{\mathrm{det\,}}(w_1\wedge\cdots\wedge w_n)=\Gamma^{\mathrm{det\,}%
}w_1\wedge\cdots\wedge w_n.
\end{split}%
\end{equation}

The \emph{Clifford connection} $\nabla^{\mathrm{Cl\,}}$ on $%
\Lambda(T^{*0,1}U)$ is defined for the frame
\begin{equation*}
\left\{\overline w^{j_1}\wedge\cdots\wedge\overline w^{j_q};\, 1\leq
j_1<\cdots<j_q\leq n\right\}
\end{equation*}
by the local formula
\begin{equation}  \label{e-gue150825II}
\nabla^{\mathrm{Cl\,}}=d+\frac{1}{4}\sum^{2n}_{j,\ell=1}\langle\,%
\Gamma^{TU}e_j,e_{\ell}\,\rangle c(e_j)c(e_\ell)+\frac{1}{2}\Gamma^{\mathrm{%
det\,}}.
\end{equation}

In general a Levi-Civita connection $\nabla$ cannot be compatible with the 
complex structure unless a certain extra condition is imposed such as 
K\"ahler condition on the metric.  Or else one takes the orthogonal projection 
$P_{T^{1,0}X}\nabla$ to produce a connection on $T^{1,0}X$.  
One key point above is that the Clifford connection $\nabla^{\mathrm{Cl\,}}$ 
(regardless of K\"ahler condition nor orthogonal projection) defines a
{\it Hermitian connection} (connection compatible with
the underlying Hermitian metric) on $\Lambda(T^{*0,1}U)$ (see Proposition 1.3.1 in~\cite{MM07}).

Let's be back to the CR case.   In the same notation
as before, $\Omega^{0,q}(U,E)$ denotes the space of $(0,q)$ forms on $U$ with values in $%
E$, $\Omega^{0,+}(U,E)$ the even part and $\Omega^{0,-}(U,E)$ the odd part
of $\Omega^{0,*}(U,E)$ etc.

Assume $X$ is equipped with a CR bundle $E$ which is rigid.  
Being rigid $E$ can descend as a holomorphic vector bundle over $%
U$.  We may assume that $E$ is (holomorphically) trivial on $U$ (possibly after shrinking $U$).
A rigid Hermitian (fiber) metric $\langle\,\cdot\,|\,\cdot\,\rangle_E$
descends to a Hermitian (fiber) metric $\langle\,\cdot\,|\,\cdot\,\rangle_E$ on $E$ over $U$. Let $\nabla^{E}$
be the Chern connection on $E$ associated with $\langle\,\cdot\,|\,\cdot\,%
\rangle_E$ (over $U$). 

We still denote by $\nabla^{\mathrm{Cl\,}}$ the connection on $%
\Lambda(T^{*0,1}U)\otimes E$ induced by $\nabla^{\mathrm{Cl\,}}$ and $%
\nabla^{E}$.


\begin{defn}
\label{d-gue150825} The $\mathrm{Spin}^c$ Dirac operator $D_{B}$ is defined
by
\begin{equation}  \label{e-gue150825a}
D_B=\frac{1}{\sqrt{2}}\sum^{2n}_{j=1}c(e_j)\nabla^{\mathrm{Cl\,}%
}_{e_j}:\Omega^{0,*}(U,E)\rightarrow\Omega^{0,*}(U,E).
\end{equation}
\end{defn}

It is well-known that $D_B$ is formally self-adjoint 
(see Proposition 1.3.1 and equation (1.3.1) in~\cite{MM07})
and $D_B:\Omega^{0,\pm}(U,E)\rightarrow\Omega^{0,\mp}(U,E)$.

Write $\overline{\partial}^*:\Omega^{0,q+1}(U,E)\rightarrow\Omega^{0,q}(U,E)$ 
for the adjoint of $\overline\partial:\Omega^{0,q}(U,E)\rightarrow\Omega^{0,q+1}(U, E)$
with respect to the $L^2$
inner product on $\Omega^{0,q}(U,E)$ induced by $\langle\,\cdot\,,\,\cdot\,%
\rangle$ and $\langle\,\cdot\,|\,\cdot\,\rangle_E$ ($q=0,1,2,\ldots,n-1$). 
Then, by Theorem 1.4.5 in~\cite{MM07} 
\begin{equation}  \label{e-gue150524faI}
D_B=\overline\partial+\overline{\partial}^*+A_B:\Omega^{0,\pm}(U,E)%
\rightarrow\Omega^{0,\mp}(U,E)
\end{equation}
where $A_B: \Omega^{0,\pm}(U,E)\rightarrow\Omega^{0,\mp}(U,E)$
is a smooth zeroth order operator (and $A_B=A_B(z)$, independent of $\theta$).
Note that $A_B$ as an operator $\Omega^{0,*}(U,E)\rightarrow\Omega^{0,*}(U,E)$
is self-adjoint because both $D_B$ and $\overline\partial+\overline{\partial}^*$ are so.  

The following is instrumental in forming a global operator from local ones,
whose proof is based on canonical coordinates of BRT trivializations.
Note for $u\in\Omega^{0,q}(D,E)$ ($q=0,1,2,\ldots,n$) with
$u=u(z)$, i.e. $u$ is independent of $\theta$,  we may identify such 
$u$ with an element in $%
\Omega^{0,q}(U,E)$ by using %
\eqref{e-gue150524fb} (and vice versa).

\begin{prop}
\label{l-gue150524d} Let $B=(D,(z,\theta),\varphi)$ and $\widetilde
B=(D,(w,\eta),\widetilde\varphi)$ be two BRT trivializations with 
$D=U\times]-\varepsilon,\varepsilon[$ for $\varepsilon>0$ and an
open set $U$ of $\mathbb{C}^n$. Let $A_B, A_{\widetilde
B}:\Omega^{0,\pm}(U,E)\rightarrow\Omega^{0,\mp}(U,E)$ be the operators given
by \eqref{e-gue150524faI}. Fix an $m\in\mathbb{Z}$.   For $u\in\Omega^{0,%
\pm}_m(X,E)$ we can write $u(=u|_D)=e^{-im\theta}v(z)=e^{-im\eta}\widetilde v(w)$ 
for some $v(z), \widetilde v(w)\in\Omega^{0,\pm}(U,E)$. Then
\begin{equation}  \label{e-gue150524d}
e^{-im\theta}A_B(v(z))=e^{-im\eta}A_{\widetilde B}(\widetilde v(w))\ \ %
\mbox{on $D$}.
\end{equation}
\end{prop}

\begin{proof}
Although we shall only use part of the coordinate transformation of BRT trivializations \eqref{e-gue150524dIe1}
\begin{equation}  \label{e-gue150524dI}
\begin{split}
&w=H(z),\ \ \overline\partial H(z)=0;\quad \eta=\theta+\mathrm{arg}\,g(z); 
\quad\widetilde\varphi(H(z), \overline {H(z)})=\varphi(z, \overline z)+\log|g(z)|,
\end{split}%
\end{equation}
let's give a geometrical interpretation of how the above can be obtained for an independent interest.
This complements the treatment of \cite{BRT85}.   See Subsection~\ref{s-gue150627} (cf. 
Remark~\ref{r-gue150606}) for its use 
in the construction of a modified Kodaira Laplacian as well as in the proof of Lemma~\ref{l-gue160417lem2}.

To see \eqref{e-gue150524dI}, we are going to realize $D$ (possibly after shrinking it) as
(part of) the total space of a circle bundle associated with a trivial holomorphic line bundle $L$
over a complex manifold $U\subset \mathbb{C}^n$.  More precisely suppose $L$ is equipped with a Hermitian 
metric such that a local basis $1$ has $|| 1||=e^{-2\phi(z)}$ and $Y=\{(z, \lambda)\subset\mathbb{C}^{n+1}; |\lambda|^2e^{2\phi}=1\}$ 
is the circle bundle inside the $L^*$.   Write $\rho=|\lambda|^2e^{2\phi}-1$ and $\lambda|_Y=e^{-\phi-i\theta}$. 
One has $T^{0,1}Y=(\mathrm{Ker}\,\overline\partial\rho)\cap T^{0,1}\mathbb{C}^{n+1}$.  
In terms of $(z, \theta)$ coordinates on $Y$, one has 
$T^{0,1}Y=\{\frac{\partial}{\partial\overline z_j}+i\frac{\partial\phi}{
\partial\overline z_j}(z)\frac{\partial}{\partial\theta};\,
j=1,2,\ldots,n\}$ (and $T^{1,0}Y=\overline{T^{0,1}Y}$) because 
the RHS is checked to be contained in $\mathrm{Ker}\,\overline\partial\rho$
and it has the correct dimension.   In view of Theorem~\ref{t-gue150514} by 
taking the above $\phi$ to be $\varphi$ of $Z_j$ in \eqref{e-can} and mapping 
$(z, \theta)$ of $D$ to $(z, e^{-\phi(z)-i\theta})$ of $Y$, 
it will be seen that $D$ is realized in this way as a portion of $Y$.  

For this realization, we compare the above with \cite[(0.1) or Theorem I.2]{BRT85} 
and set up the (holomorphic) transformation in coordinates: 
setting our $\varphi$, $\lambda$, $\theta$ and $z$ (on $U$) above to be $\phi$, $e^{iw}$, $-s_1$ and $z$ respectively in 
\cite[(0.1) and Theorem I.2]{BRT85}).  Then our $Z_j$ in \eqref{e-can} corresponds to 
$\overline L_j$ in \cite[Proposition I.2]{BRT85}.   One sees that 
the ambient complex space of \cite{BRT85} is (locally) biholomorphic to part of the above $L^*$.  
For the next purpose, let us give an instrinsic formulation of this complex space from another point of view.   
Let $\mathbb{R}_{>0}$ be the set of positive real numbers.  Consider $D\times \mathbb{R}_{>0}$ 
and equip it with the complex structure $J$ defined as follows.   $J|_{TU}$ is set to be the complex structure 
of $U$; it suffices to define $J$ on $]-\varepsilon, \varepsilon[\times \mathbb{R}_{>0}$ with coordinates $s, r$: 
$J(\partial/\partial r)=-\frac{1}{r}\partial/\partial s$, $J(\frac{1}{r}\partial/\partial s)=\partial/\partial r$.  
Note this definition of $J$ is independent of choice of BRT trivializations (since $Z_j\subset T^{1,0}Y$ identified 
with $T^{1,0}U$, gives a local basis of 
$T^{1,0}D$; see Theorem~\ref{t-gue150514}, and $\{z\}\times ]-\varepsilon, \varepsilon[$ for each $z\in U$ is 
(part of) an $S^1$ orbit in $X$).     $J$ is seen to be (equivalent to) the complex structure on $L^*$ (with 
$(z, s, r)\in D\times\mathbb{R}_{>0}$ and $(z, re^{is})\in L^*$ in correspondence), hence an integrable 
complex structure.

Let now $\widetilde B=(D,(w,\eta),\widetilde\varphi)$ be any other BRT chart.  Correspondingly we will denote the associated
objects by the same notation as in $B$ but topped with a tilde.   Note 
that in $\widetilde B$ the set defined by $w=w_0$ for a fixed $w_0$ is part of an $S^1$ orbit in $X$; 
the same can be said with $B$.    Conversely, any $S^1$ orbit of $X$ is described (locally)
by the $\theta$ parameter in any BRT charts with $z$-coordinates being fixed.  
By using $(D\times\mathbb{R}_{>0}, J)$ above, one has 
$L^*\cong\tilde L^*$ (locally) by a biholomorphism $F$ that preserves respective fibers 
(since these are $S^1$ orbits, described by $\theta$ parameters in each chart and hence must 
be in correspondence via $D\times\mathbb{R}_{>0}$ by the property of BRT charts as just remarked).  
Further, one sees that $F$ restricted to fibers has to be linear, hence $F$ is a bundle isomorphism.  
Geometrically this picture is essentially the same as a local change of holomorphic coordinates on the base manifold $U$ 
and a change of a local basis of $L^*$ by $\tilde e^*(z)=g(z) e^*(z)$ with $||e^*||^2=e^{2\varphi}, 
||\tilde e^*||^2=e^{2\tilde\varphi}$ for some nowhere vanishing holomorphic function $g(z)$ on $U$.  
The above transformation formula \eqref{e-gue150524dI} easily follows from this concrete realization.  

Using the above transformation \eqref{e-gue150524dI}, one claims
\begin{equation}  \label{e-gue150524dII}
\begin{split}
&D_B=D_{\widetilde B}\ \ \mbox{on $\Omega^{0,\pm}(U,E)$}, \\
&A_B=A_{\widetilde B}\ \ \mbox{on $\Omega^{0,\pm}(U,E)$}.
\end{split}%
\end{equation}



To see this for the case without $E$, note $D$ is just realized as
(part of) the total space of a circle bundle of a holomorphic line bundle.  
Clearly $\overline\partial_U+\overline\partial_U^*$ does not depend on choice of holomorphic coordinates on $U$;
that is, $\overline\partial_U+\overline\partial_U^*$ is an intrinsic object (cf. the Hermitian metric 
used for $\overline\partial_U^*$ is intrinsic, \eqref{e-gue150823e1}).  The same idea can be applied to $D_B$
which is defined above as an intrinsic object too.  
Therefore \eqref{e-gue150524dII} holds with the change of coordinates in \eqref{e-gue150524dI} (using only 
$w=H(z)$).   Now with $E$ resumed, the reasoning is basically unchanged.   Hence the claim \eqref{e-gue150524dII}.  

By
\eqref{e-gue150524dI} $v(z)=e^{-imG(z)}\widetilde v(w)$ ($G(z)=\mathrm{arg}\,g(z)$), hence by \eqref{e-gue150524dII}
\begin{equation*}
e^{-im\theta}A_B(v(z))=e^{-im\theta}A_B(e^{-imG(z)}\widetilde
v(w))=e^{-im\theta}A_{\widetilde B}(e^{-imG(z)}\widetilde
v(w))=e^{-im\eta}A_{\widetilde B}(\widetilde v(w)),
\end{equation*}
proving the Proposition. 
\end{proof}

We are now ready to introduce a global operator:

\begin{defn}
\label{d-gue150524} For every $m\in\mathbb{Z}$, let $A_m:\Omega^{0,%
\pm}_m(X,E)\rightarrow\Omega^{0,\mp}_m(X,E)$ be the linear operator defined
as follows. Let $u\in\Omega^{0,\pm}_m(X,E)$. Then, $v:=A_mu$ is an element
in $\Omega^{0,\mp}_m(X,E)$ such that for every BRT trivialization $%
B:=(D,(z,\theta),\varphi)$ ($D=U\times]-\varepsilon,\varepsilon[$, $%
\varepsilon>0$, $U$ an open set in $\mathbb{C}^n$) we have $%
v|_D=e^{-im\theta}A_B(\widetilde u)(z)$ where $u=e^{-im\theta}\widetilde
u(z)$ on $D$ for some $\widetilde u\in\Omega^{0,\pm}(U,E)$ and $A_B$ is given in %
\eqref{e-gue150524faI}.
\end{defn}

In view of Proposition~\ref{l-gue150524d}, Definition~\ref{d-gue150524}
is well-defined.

We are now in a position to define the modified Kohn Laplacian ($\mathrm{Spin}^c$
Kohn Laplacian) including 
a type of CR $\mathrm{Spin}^c$ Dirac operator  $\widetilde
D_{b,m}$.  One goal of this part is
to express the index of $\widetilde
D_{b,m}$ in an integral form of the heat kernel density (cf. Proposition~\ref{t-gue150530}).

The treatment below mostly follows traditional cases except
the use of the projection operator $Q^{\pm}_m$ together with
its explicit expression in integral (see \eqref{e-gue150525dhII} and \eqref{e-gue151024b}).

By using $A_m$ in Definition~\ref{d-gue150524} we consider
\begin{equation}  \label{e-gue150525a}
\begin{split}
&\widetilde D_{b,m}=\overline\partial_b+\overline{\partial}%
^*_b+A_m:\Omega^{0,\bullet}_m(X,E)\rightarrow\Omega^{0,\bullet}_m(X,E), \\
&\widetilde D^{\pm}_{b,m}=\overline\partial_b+\overline{\partial}%
^*_b+A_m:\Omega^{0,\pm}_m(X,E)\rightarrow\Omega^{0,\mp}_m(X,E)
\end{split}%
\end{equation}
with the formal adjoint $\widetilde D^*_{b,m}$ on $\Omega^{0,\bullet}_m(X,E)$. 

We remark that $\widetilde D_{b,m}^*=\widetilde
D_{b,m}$ on $\Omega^{0,\bullet}_m(X,E)$.   For, 
by \eqref{e-gue22a1} the $L^2$ inner product on $\Omega^{0,q}_m(D, E)$ 
is clearly $2\varepsilon(\,\cdot\,,\,\cdot\,)$ 
with the $L^2$ inner product $(\,\cdot\,,\,\cdot\,)$ on $\Omega^{0,q}(U,E)$.  
Now that $A_B$ is self-adjoint on $\Omega^{0,*}(U, E)$ as aforementioned, 
it follows that $A_m$ is self-adjoint on $\Omega_m^{0,\bullet}(X, E)$.   
That $\widetilde D_{b,m}$ is self-adjoint follows as $\overline\partial_b+\overline{\partial}^*_b$ 
is also self-adjoint.

The {\it modified/$\mathrm{Spin}^c$ Kohn Laplacian} is given by
\begin{equation}  \label{e-gue150525aII}
\begin{split}
&\widetilde\Box_{b,m}:= \widetilde D^*_{b,m}\widetilde D_{b,m}:
 \Omega^{0,\bullet}_m(X,E)\rightarrow\Omega^{0,\bullet}_m(X,E)\\
&\widetilde{\Box}^{+}_{b,m}:=\widetilde D^*_{b,m}\widetilde D_{b,m}=\widetilde
D^-_{b,m}\widetilde D_{b,m}^+: 
\Omega^{0,+}_m(X,E)\rightarrow\Omega^{0,+}_m(X,E)\quad 
(\widetilde{\Box}^{-}_{b,m}:=\widetilde
D^+_{b,m}\widetilde D_{b,m}^-). 
\end{split}%
\end{equation}
We extend $\widetilde{\Box}^+_{b,m}$ and $\widetilde{\Box}^-_{b,m}$ by
\begin{equation}  \label{e-gue150525aIII}
\begin{split}
&\widetilde{\Box}^{\pm}_{b,m}:\mathrm{Dom\,}\widetilde{\Box}^{\pm}_{b,m}\,(\subset
L^{2,\pm}_m(X,E))\rightarrow L^{2,\pm}_m(X,E)
\end{split}%
\end{equation}
where
$\mathrm{Dom\,}\widetilde{\Box}^{\pm}_{b,m}:=\{u\in L^{2,\pm}_m(X,E);\, \mbox{distribution
 $\widetilde{%
\Box}^{\pm}_{b,m}u\in L^{2,\pm}_m(X,E)$}\}$.  

Clearly $\widetilde\Box_{b,m}-T^2$ 
(where $T$ is the real vector field induced by the $S^1$ action) is (the restriction of) an elliptic operator on $X$
since $(\overline\partial_b+\overline{\partial}^*_b)^2-T^2$ is.  

Put as usual,  the Sobolov spaces (cf. Subsection~\ref{s-gue150508b})
$H^{s,+}(X,E)$, $H^{s,-}(X,E)$ the even and odd part of $H^{s}(X, T^{*,\bullet}\otimes E)$.  
In the same vein as Lemma~\ref{l-gue150517} and Lemma~\ref%
{t-gue150517a} one has 
\begin{equation}  \label{e-gue150525aIV}
\begin{split}
&\mathrm{Dom\,}\widetilde{\Box}^{\pm}_{b,m}=L^{2,\pm}_m(X,E)\bigcap H^{2,\pm}(X,E),
\\
&\mbox{$\Td{\Box}^+_{b,m}$ and $\Td{\Box}^-_{b,m}$ are self-adjoint}.
\end{split}%
\end{equation}

Further, for the spectrum $\mathrm{Spec\,}\widetilde{\Box}^+_{b,m}\subset[0,\infty[$ (resp. $\mathrm{%
Spec\,}\widetilde{\Box}^-_{b,m}\subset[0,\infty[$) one has (similar to Proposition~\ref{t-gue150517aI})

\begin{prop}
\label{t-gue150525} $\mathrm{Spec\,}\widetilde{\Box}^+_{b,m}$ is a discrete subset of $[0,\infty[$. For
any $\mu\in\mathrm{Spec\,}\widetilde{\Box}^+_{b,m}$, $\mu$ is an eigenvalue
of $\widetilde{\Box}^{+}_{b,m}$ and the eigenspace
\begin{equation*}
\widetilde{\mathcal{E}}^+_{m,\nu}(X,E):=\big\{u\in\mathrm{Dom\,}\widetilde{%
\Box}^{+}_{b,m};\, \widetilde{\Box}^{+}_{b,m}u=\nu u\big\}
\end{equation*}
is finite dimensional with $\widetilde{\mathcal{E}}^+_{m,\nu}(X,E)\subset%
\Omega^{0,+}_m(X,E)$.  Similar results hold for the case of ${%
\Box}^-_{b,m}$.
\end{prop}

The following can be proved by standard argument. 

\begin{lem}
\label{t-gue150525I} We have $\mathrm{Spec\,}\widetilde{\Box}%
^+_{b,m}\bigcap]0,\infty[=\mathrm{Spec\,}\widetilde{\Box}^-_{b,m}\bigcap]%
0,\infty[$, and for every $0\ne\mu\in\mathrm{Spec\,}\widetilde{\Box}^+_{b,m}$, $\mathrm{dim\,}\widetilde{\mathcal{E}}^+_{m,\mu}(X,E)=%
\mathrm{dim\,}\widetilde{\mathcal{E}}^-_{m,\mu}(X,E)$.
\end{lem}

We are going to introduce a {\it McKean-Singer} type formula (Corollary~\ref{t-gue150603}).  
Let $F$ be a complex vector bundle
over $X$ of rank $r$ with a Hermitian metric $\langle\,\cdot\,|\,\cdot\,%
\rangle_F$. Let $A(x,y)\in C^\infty(X\times X, F\boxtimes F^*)$.  
For every $u\in C^\infty(X,F)$, $\int_X A(x,y)u(y)dv_X(y)\in C^\infty(X,F)$
is defined in a fairly standard manner.

Much of what follows parallels the classical cases except 
that $Q_m^+$ is introduced in our case.  For  $\nu\in\mathrm{Spec\,}\widetilde{\Box}^{\pm}_{b,m}$ let
$P^{\pm}_{m,\nu}:L^{2,\pm}(X,E)\rightarrow\widetilde{\mathcal{E}}^{\pm}_{m,\nu}(X,E)$ 
be the orthogonal projections (with respect to $(\,\cdot\,|\,\cdot\,)_E$), and
$P^{\pm}_{m,\nu}(x,y)$ ($\in C^\infty(X\times X, (T^{*0,\pm}X\otimes E)\boxtimes (T^{*0,\pm}X\otimes E)^*)$) 
the distribution kernels of $P^{\pm}_{m,\nu}$.

The heat kernels of $\widetilde{\Box}^{+}_{b,m}$ and $\widetilde{\Box}%
^-_{b,m}$ are given by
\begin{equation}  \label{e-gue150525fbI}
e^{-t\widetilde{\Box}^{\pm}_{b,m}}(x,y)=P^{\pm}_{m,0}(x,y)+\sum_{\nu\in\mathrm{%
Spec\,}\widetilde{\Box}^{\pm}_{b,m},\nu>0}e^{-\nu t}P^{\pm}_{m,\nu}(x,y)
\end{equation}
with the associated continuous operators $e^{-t\widetilde{\Box}^{\pm}_{b,m}}:\Omega^{0,\pm}(X,E)\rightarrow%
\Omega^{0,\pm}_m(X,E)\subset  \Omega^{0,\pm}(X,E)$.  $e^{-t\widetilde{\Box}^{\pm}_{b,m}}$ is self-adjoint 
on $\Omega^{0,\pm}(X,E)$.

Remark that the heat kernels \eqref{e-gue150525fbI} are smooth.   For, the eigenfunctions involved (in 
the equivalent form as \eqref{e-gue151023}) are
still eigenfunctions of $(\widetilde{\Box}_{b,m}-T^2)$ hence eigenfunctions of an elliptic operator.  
In the elliptic case, one has the G\"arding type inequality which estimates the various Sobolev norms of the eigenfunctions, 
and hence mainly by Sobolev embeddings, gives eventually the smoothness of the heat kernels
(cf. \cite[Lemmas 1.6.3 and 1.6.5]{Gi95}). 


An important operator is given by the orthogonal projection 
\begin{equation}  \label{e-gue150525dhII}
Q^{\pm}_m:L^{2,\pm}(X,E)\rightarrow L^{2,\pm}_m(X,E)
\end{equation}
(for the $m$-th Fourier component).  Fourier analysis with %
\eqref{e-gue150510f} gives 
\begin{equation}  \label{e-gue151024b}
Q^{\pm}_mu=\frac{1}{2\pi}\int^\pi_{-\pi}u(e^{-i\theta}\circ
x)e^{im\theta}d\theta,\ \ \forall u\in\Omega^{0,\pm}(X,E)
\end{equation}
(where $u(e^{-i\theta}\circ x)$ stands for the pull back  $(e^{-i\theta})^*u$,
cf. \eqref{e-gue150510f}).  
The explicit expression \eqref{e-gue151024b} 
turns out to be crucial to many (unconventional) estimates later.

It is fairly standard (note $Q_m^+$ in the second line below) to obtain
(by \eqref{e-gue150525fbI})
\begin{equation}  \label{e-gue150525dhIII}
\begin{split}
&(\frac{\partial}{\partial t}+\widetilde{\Box}^{\pm}_{b,m})(e^{-t\widetilde{%
\Box}^{\pm}_{b,m}}u)=0,\ \ \forall u\in\Omega^{0,\pm}(X,E),\ \ \forall t>0, \\
&\lim_{t\rightarrow0^+}(e^{-t\widetilde{\Box}^{\pm}_{b,m}}u)=Q^{\pm}_mu,\ \
\forall u\in\Omega^{0,\pm}(X,E).
\end{split}
\end{equation}

For  $\nu\in\mathrm{Spec\,}\widetilde{\Box}^{+}_{b,m}$,
let $\left\{f^\nu_1,\ldots,f^%
\nu_{d_{\nu}}\right\}$ be an orthonormal basis for $\widetilde{\mathcal{E}}^+_{m,\nu}(X,E)$.
Define
\begin{equation}  \label{e-gue150525fbII}
\mathrm{Tr\,}P^{+}_{m,\nu}(x,x):=\sum^{d_{\nu}}_{j=1}\left\vert
f^\nu_j(x)\right\vert^2_E\in C^\infty(X)
\end{equation}
which is equal to 
$\mathrm{Tr\,}P^{+}_{m,\nu}(x,x)=\sum^{d}_{j=1}\langle\,P^+_{m,%
\nu}(x,x)e_j(x)\,|\,e_j(x)\,\rangle_E$
where $\{e_j(x)\}_j$ is any orthonormal
basis of $T^{*0,+}_xX\otimes E_x$. 
Define $\mathrm{Tr\,}P^{-}_{m,\mu}(x,x)$ similarly.

Clearly $d_{\nu}^{\pm}=\int_X\mathrm{Tr\,}%
P^{\pm}_{m,\nu}(x,x)dv_X(x)$.

Put $\mathrm{Tr\,}e^{-t\widetilde{\Box}^{\pm}_{b,m}}(x,x):=\mathrm{Tr\,}%
P^{\pm}_{m,0}(x,x)+\sum_{\nu\in\mathrm{Spec\,}\widetilde{\Box}%
^{\pm}_{b,m},\nu>0}e^{-\nu t}\mathrm{Tr\,}P^{\pm}_{m,\nu}(x,x)$, so  
for  $t>0$
\begin{equation}  \label{e-gue150525fbV}
\int_X\mathrm{Tr\,}e^{-t\widetilde{\Box}^{\pm}_{b,m}}(x,x)dv_X(x)=\mathrm{dim\,%
}\widetilde{\mathcal{E}}^{\pm}_{m,0}(X,E)+\sum_{\nu\in\mathrm{Spec\,}\widetilde{%
\Box}^{\pm}_{b,m},\nu>0}e^{-\nu t}\mathrm{dim\,}\widetilde{\mathcal{E}}%
^{\pm}_{m,\nu}(X,E).
\end{equation}
Combining Lemma~\ref{t-gue150525I} and \eqref{e-gue150525fbV} gives 
\begin{equation}  \label{e-gue150530}
\int_X\Bigr(\mathrm{Tr\,}e^{-t\widetilde{\Box}^{+}_{b,m}}(x,x)-\mathrm{Tr\,}%
e^{-t\widetilde{\Box}^{-}_{b,m}}(x,x)\Bigr)dv_X(x) =\mathrm{dim\,}\mathrm{%
Ker\,}\widetilde{\Box}^+_{b,m}-\mathrm{dim\,}\mathrm{Ker\,}\widetilde{\Box}%
^-_{b,m}.
\end{equation}

As in Theorem~\ref{t-gue150524f}
one has
\begin{equation}  \label{e-gue150530I}
\mathrm{Ker\,}\widetilde D_{b,m}=\mathrm{Ker\,}\widetilde{\Box}%
^+_{b,m}\subset\Omega^{0,+}_m(X,E), \quad
\mathrm{Ker\,}\widetilde D^*_{b,m}=\mathrm{Ker\,}\widetilde{\Box}%
^-_{b,m}\subset\Omega^{0,-}_m(X,E).
\end{equation}

Put $\mathrm{ind\,}\widetilde D_{b,m}:=\mathrm{dim\,}\mathrm{Ker\,}%
\widetilde D_{b,m}-\mathrm{dim\,}\mathrm{Ker\,}\widetilde D^*_{b,m}$. 
We express the index (by \eqref{e-gue150530I} and \eqref{e-gue150530})
as 

\begin{prop}
\label{t-gue150530} For every $t>0$, we have
\begin{equation}  \label{e-gue150530II}
\mathrm{ind\,}\widetilde D_{b,m}=\int_X\Bigr(\mathrm{Tr\,}e^{-t\widetilde{%
\Box}^{+}_{b,m}}(x,x)-\mathrm{Tr\,}e^{-t\widetilde{\Box}^{-}_{b,m}}(x,x)%
\Bigr)dv_X(x).
\end{equation}
\end{prop}

The invariance of the index is expressed by the following 
(some aspects on $\mathrm{ind\,}D_{b,m}$ refer to 
Theorem~\ref{t-gue150524f}).

\begin{thm}
\label{t-gue150530I} {\rm (Homotopy invariance)} We have $\mathrm{ind\,}D_{b,m}=\mathrm{ind\,}\widetilde
D_{b,m}$.
\end{thm}

To summarize (with Theorem~\ref{t-gue150530I}, Proposition~\ref{t-gue150530} and %
\eqref{e-gue150524fI}) we have a McKean-Singer formula 
(cf.  Corollary~\ref{t-gue150630w} for McKean-Singer (II)).  

\begin{cor}
\label{t-gue150603} {\rm (McKean-Singer (I))} Fix $m\in\mathbb{Z}$. For $t>0$, we have
\begin{equation}  \label{e-gue150603b}
\sum^n_{j=0}(-1)^j\mathrm{dim\,}H^j_{b,m}(X,E)=\int_X\Bigr(\mathrm{Tr\,}e^{-t%
\widetilde{\Box}^{+}_{b,m}}(x,x)-\mathrm{Tr\,}e^{-t\widetilde{\Box}%
^{-}_{b,m}}(x,x)\Bigr)dv_X(x).
\end{equation}
\end{cor}

\begin{rem}
\label{r-gue151003} To compare with the original Kohn Laplacian,  a similar formula
(as Corollary~\ref{t-gue150603})
\begin{equation*}
\sum^n_{j=0}(-1)^j\mathrm{dim\,}H^j_{b,m}(X,E)=\int_X\Bigr(\mathrm{Tr\,}%
e^{-t\Box^{+}_{b,m}}(x,x)-\mathrm{Tr\,}e^{-t\Box^{-}_{b,m}}(x,x)\Bigr)%
dv_X(x)
\end{equation*}
holds. 
When $X$ is not CR K\"ahler, it is obscure, by the experience from classical cases,
to calculate the density $\mathrm{Tr\,}%
e^{-t\Box^{+}_{b,m}}(x,x)-\mathrm{Tr\,}e^{-t\Box^{-}_{b,m}}(x,x)$ with 
the original Kohn Laplacian.  
The introduction of the modified Kohn Laplacians replacing $\Box^{\pm}_{b,m}$ by
$\widetilde{\Box}^{\pm}_{b,m}$ 
is expected to facilitate this
calculation.  But because of the unconventional 
asymptotic expansion of $e^{-t\widetilde \Box^{\pm}_{b,m}}(x,x)$ some novelty beyond the classical cases 
shows up (as mentioned in Introducton).    It should be noted that when $X$ is CR
K\"ahler, $\widetilde{\Box}^{\pm}_{b,m}=\Box^{\pm}_{b,m}$. 
\end{rem}

To prove Theorem~\ref{t-gue150530I} in the remaining of this section, observe that it is nothing but a statement of 
{\it homotopy invariance} of index.  For, with $A_m$ a global operator (see Definition~\ref{d-gue150524}), putting $L_t=\overline\partial_{b,m}+\overline{\partial}%
^*_{b,m}+tA_m:\Omega^{0,+}_m(X,E)\rightarrow\Omega^{0,-}_m(X,E)$ for $t\in [0,1]$, gives 
the homotopy between $L_0=D_{b,m}$ and $L_1=\widetilde D_{b,m}$.   

Remark that there have been proofs for results of this type; for instance, see \cite{BGV92}
using heat kernel method and \cite{BB} using functional analysis method (both not exactly formulated 
in the above form though).   
To make it accessible
to a wider readership, we include a (comparatively) self-contained and short proof.
It is amusing to note that the Hodge theory in Section~\ref{s-gue150517} is useful at certain points 
of our proof.  

Some preparations are in order.
We extend $L_t$ by setting 
$$\mathrm{Dom\,}L_t=\{u\in L^{2,+}_m(X,E);\, \mbox{distribution $L_tu\in
L^{2,-}_m(X,E)$}\}$$ 
so that 
$L_t:\mathrm{Dom\,}L_t\, (\subset L^{2,+}_m(X,E))\rightarrow L^{2,-}_m(X,E)$.
Write $L^*_t$ for the Hilbert space
adjoint of $L_t$.  

Let $H^{1,+}_m(X,E)$ be the completion of $\Omega^{0,+}_m(X,E)$ with
respect to the Hermitian inner product
\begin{equation*}
Q(u,v)=(\,u\,|\,v\,)_E+(\,\overline\partial_bu\,|\,\overline\partial_bv%
\,)_E+(\,\overline{\partial}^*_bu\,|\,\overline{\partial}^*_bv\,)_E.
\end{equation*}
Clearly $H^{1,+}_m(X,E)\subset \mathrm{Dom\,}L_t, \forall \,t\in\mathbb{R}$.  
One can show that $H^{1,+}_m(X,E)=\mathrm{Dom\,}L_t$.
Let $t=0$.  Assume $L_0f=u$ with $f, u\in L_m^{2, +}$ and also assume 
$f\perp \mathrm{Ker}\,\Box_{b,m}^+$ since for any smooth $g$, $f+g\in H_m^{1,+}$
iff $f\in H_m^{1,+}$.  Using the partial 
inverse $N_m$ in \eqref{e-gue150520III} of our Hodge theory 
in Section~\ref{s-gue150517}, we have $L_0N_mf=N_mL_0f=N_mu$
since $L_0$ commutes with $N_m$ as in the ordinary Hodge theory.  
Now $L_0^*L_0N_mf=\Box_{b,m}^{+}N_mf=f$ by \eqref{e-gue150520III},
one has $f=L_0^*N_mu$.   But $N_mu$ increases the (Sobolev) order of regularity of $u$ by $2$
and then $L_0^*N_mu$ decreases by 1, the regularity of $f$ is of order 1.  
By localization, with a partition of unity, on an open subset $D$ in place of $X$ and 
by the formula \eqref{e-gue150514f} in BRT charts $D=U\times \rbrack -\delta, \delta\lbrack$
for $\overline\partial_b$, it follows from the standard G\"arding's inequality (e.g. \cite[p. 93] {GH})
that the above $Q(\cdot,\cdot)$ is equivalent to the Sobolev norm of order one (on the $m$-th component).   
Hence $f\in H_m^{1,+}$.   For $t\ne 0$, since $L_t=L_0+tA_m$ with $A_m$ a smooth zeroth order operator, it follows 
$\mathrm{Dom}\,L_t=\mathrm{Dom}\,L_0$.   



Consider $H_0:=H^{1,+}_m(X,E)\oplus\mathrm{Ker\,}L^*_0$ 
and $H_1=L^{2,-}_m(X,E)\oplus\mathrm{Ker\,}L_0$. Let $(\,\cdot\,|\,\cdot\,)_{H_0}$
and $(\,\cdot\,|\,\cdot\,)_{H_1}$ be inner products on $H_0$ and $H_1$
respectively, given by
\begin{equation*}
\begin{split}
(\,(f_1,g_1)\,|\,(f_2,g_2)\,)_{H_0}=Q(f_1,f_2)+(\,g_1\,|\,g_2\,)_E,\ \
(\,(\widetilde f_1,\widetilde g_1)\,|\,(\widetilde f_2,\widetilde
g_2)\,)_{H_1}=(\,\widetilde f_1\,|\,\widetilde f_2\,)_E+(\,\widetilde
g_1\,|\,\widetilde g_2\,)_E.  
\end{split}
\end{equation*}
Let $P_{%
\mathrm{Ker\,}L_0}$ denote the orthogonal projection onto $\mathrm{Ker\,}L_0$
with respect to $(\,\cdot\,|\,\cdot\,)_E$.

Let $A_t:H_0\rightarrow H_1$ be
the (continous) linear map defined as follows.   For  $(u,v)\in H_0$, 
$$A_t(u,v)=(L_tu+v,P_{\mathrm{Ker\,}L_0}u)\in H_1.$$

\begin{lem}
\label{l-gue150531} There is a $r>0$ such that $A_t:H_0\rightarrow H_1$ is
invertible, for every $0\leq t\leq r$.
\end{lem}

\begin{proof}
We first claim that
\begin{equation}  \label{e-gue150531a}
\mbox{$A_0$ is invertible}.
\end{equation}
If $A_0(u,v)=0$ for some $(u,v)\in H_0$, then
\begin{equation}  \label{e-gue150531aI}
i)\,\,L_0u=-v\in\mathrm{Ker\,}L^*_0,\quad ii)\,\,P_{\mathrm{Ker\,}L_0}u=0.
\end{equation}
By \eqref{e-gue150531aI}
\begin{equation}  \label{e-gue150531aIII}
(\,L_0u\,|\,L_0u\,)_E=-(\,L_0u\,|\,v\,)_E=-(\,u\,|\,L^*_0v\,)_E=0,
\end{equation}
giving $u\in\mathrm{Ker\,}L_0$.   Hence by ii) of \eqref{e-gue150531aI} we
obtain $u=0$, giving also $v=0$ by i) of \eqref{e-gue150531aI}. We have proved that $A_0$ is injective. 

We shall now
prove that $A_0$ is surjective. Let $(a,b)\in H_1$.   First we note $L_0: \mathrm{Dom\,}L_0\rightarrow L^{2,-}_m(X,E)$
has an $L^2$ closed range, so
\begin{equation}  \label{e-gue150531b}
a=L_0\alpha+\beta,\ \ \alpha\in H^{1,+}_m(X,E),\ \ \alpha\perp\mathrm{Ker\,}%
L_0,\ \ \beta\in(\mathrm{Rang\,}L_0)^\perp=\mathrm{Ker\,}L^*_0.
\end{equation}
Another way to see \eqref{e-gue150531b} is to use $\Box^{-}_{b,m}N_m^-+\Pi_m^-=I$ (on $L_m^-$) of 
\eqref{e-gue150520III} (for the ``$-$" case) of Hodge theory in Section~\ref{s-gue150517}, 
and obtain $a=L_0L_0^*N_m^-a+\Pi_m^- a$ where $\Pi_m^- a\in \mathrm{Ker}\,L_0^*$ (cf. Theorem~\ref{t-gue150524f})
and $L_0^*N_m^-a\in H^{1,+}_m(X,E)$ as mentioned above this lemma.  
In either way, by \eqref{e-gue150531b} one sees
\begin{equation*}
A_0(\alpha+b,\beta)=(L_0(\alpha+b)+\beta,P_{\mathrm{Ker\,}%
L_0}(\alpha+b))=(L_0\alpha+\beta,b)=(a,b).
\end{equation*}
Thus $A_0$ is surjective. The claim \eqref{e-gue150531a} follows.

Let $A^{-1}_0:H_1\rightarrow H_0$ be the inverse of $A_0$.  It follows from open mapping
theorem that $A^{-1}_0$ is continuous.

To finish the proof the following arguement based on geometric series is standard.
Write $A_t=A_0+R_t$, where $R_t:H_0\rightarrow H_1$ is continuous
and there is a constant $c>0$ such that $\left\Vert
R_tu\right\Vert_{H_1}\leq ct\left\Vert u\right\Vert_{H_0}$, for $u\in
H_0$. Put
\begin{equation*}
\begin{split}
&H_t=I-A^{-1}_0R_t+(A^{-1}_0R_t)^2-(A^{-1}_0R_t)^3+\cdots, \\
&\widetilde H_t=I-R_tA^{-1}_0+(R_tA^{-1}_0)^2-(R_tA^{-1}_0)^3+\cdots.
\end{split}%
\end{equation*}
Since $A^{-1}_0$ is continuous, $H_t:H_0\rightarrow H_0$ and $\widetilde
H_t:H_1\rightarrow H_1$ are well-defined as continuous maps for small $t\geq0$.
Moreover $A_t\circ (H_t\circ A^{-1}_0)=I$ (on $H_1
$) and $(A^{-1}_0\circ\widetilde H_t)\circ A_t=I$ on ($H_0$), giving 
right and left inverses of $A_t$ for small $t\geq0$.
Hence the lemma.  
\end{proof}

For $t\in[0,1]$ write
$L^*_t:\mathrm{Dom\,}L^*_t\,(\subset L^{2,-}_m(X,E))\rightarrow L^{2,+}_m(X,E)
$
for the adjoint of $L_t$ with respect to $(\,\cdot\,|\,\cdot\,)_E$%
.   Similar to $L_0$ and $L_1$, one has $\mathrm{dim\,}\mathrm{Ker\,}L_t<\infty$ and $\mathrm{%
dim\,}\mathrm{Ker\,}L^*_t<\infty$ (with $\mathrm{Ker\,}%
L_t\subset\Omega^{0,+}_m(X,E)$, $\mathrm{Ker\,}L^*_t\subset%
\Omega^{0,-}_m(X,E)$).  

Put $\mathrm{ind\,}L_t:=\mathrm{dim\,}%
\mathrm{Ker\,}L_t-\mathrm{dim\,}\mathrm{Ker\,}L^*_t$.

\begin{lem}
\label{l-gue150603} There is a $r_0>0$ such that $\mathrm{ind\,}L_t=\mathrm{%
ind\,}L_0$, for every $0\leq t\leq r_0$.
\end{lem}

\begin{proof}
Let $r>0$ be as in Lemma~\ref{l-gue150531}. We first show that
\begin{equation}  \label{e-gue150603}
\mathrm{ind\,}L_0\leq\mathrm{ind\,}L_t,\ \ \forall\,0\leq t\leq r.
\end{equation}

Fix $0\leq t\leq r$.  We define
\begin{equation*}
B:\mathrm{Ker\,}L^*_t\oplus\mathrm{Ker\,}L_0\rightarrow\mathrm{Ker\,}%
L_t\oplus\mathrm{Ker\,}L^*_0
\end{equation*}
as follows. Let $(a,b)\in\mathrm{Ker\,}%
L^*_t\oplus\mathrm{Ker\,}L_0$. By Lemma~\ref{l-gue150531},
\begin{equation*}
A_t:H^{1,+}_m(X,E)\oplus\mathrm{Ker\,}L^*_0\rightarrow L^{2,-}_m(X,E)\oplus%
\mathrm{Ker\,}L_0
\end{equation*}
is invertible. There is a unique $(u,v)\in H^{1,+}_m(X,E)\oplus\mathrm{Ker\,}%
L^*_0=\mathrm{Dom\,}L_t\oplus\mathrm{Ker\,}L^*_0$ such that $A_t(u,v)=(a,b)$%
. Let $P_{\mathrm{Ker\,}L_t}:L^{2,+}_m(X,E)\rightarrow\mathrm{Ker\,}L_t$ be
the orthogonal projection with respect to $(\,\cdot\,|\,\cdot\,)_E$. 
Then the above map $B$ is defined by $$%
B(a,b):=(P_{\mathrm{Ker\,}L_t}u,v)\in\mathrm{Ker\,}L_t\oplus\mathrm{Ker\,}%
L^*_0.$$

We claim that $B$ is injective.  If so, then 
\begin{equation}  \label{e-gue150603I}
\mathrm{dim\,}\mathrm{Ker\,}L^*_t+\mathrm{dim\,}\mathrm{Ker\,}L_0\leq\mathrm{%
dim\,}\mathrm{Ker\,}L_t+\mathrm{dim\,}\mathrm{Ker\,}L^*_0,
\end{equation}
i.e. $\mathrm{dim\,}\mathrm{Ker\,}L_0-\mathrm{dim\,}\mathrm{Ker\,}L^*_0\leq\mathrm{%
dim\,}\mathrm{Ker\,}L_t-\mathrm{dim\,}\mathrm{Ker\,}L^*_t$, 
yielding the desired \eqref{e-gue150603}.

For the claim that $B$ is injective, 
if $B(a,b)=(0,0)$ for some $%
(a,b)\in\mathrm{Ker\,}L^*_t\oplus\mathrm{Ker\,}L_0$, write $(u,v)$ ($\in
H^{1,+}_m(X,E)\oplus\mathrm{Ker\,}L^*_0$) such that $A_t(u,v)=(a,b)$.   As $%
(0,0)=B(a,b)=(P_{\mathrm{Ker\,}L_t}u,v)$, $P_{\mathrm{Ker\,}L_t}u=0$
and $v=0$.  Using the definition of $A_t$, one has 
\begin{equation}\label{e-gue150603e1}
A_t(u,v)=A_t(u,0)=(L_tu,P_{\mathrm{Ker\,}L_0}u)=(a,b)\in\mathrm{Ker\,}%
L^*_t\oplus\mathrm{Ker\,}L_0
\end{equation}
to give $a=L_tu\in\mathrm{Ker\,}L^*_t$, hence $(\,a\,|\,a\,)_E=(\,a\,|\,L_tu%
\,)_E=(\,L^*_ta\,|\,u\,)_E=0$ gives $L_tu=a=0$ so that 
$u\in\mathrm{Ker\,}L_t$, i.e. $u=P_{\mathrm{Ker\,}L_t}u$ by definition. 
It follows that $u=0$ since $P_{\mathrm{Ker\,}L_t}u=0$ as just seen.  
With $u=0$ and \eqref{e-gue150603e1}
one sees $(a,b)=(L_tu, P_{\mathrm{Ker\,}%
L_0}u)=(0,0)$, giving the injectivity of $B$.

By the same argument, 
$\mathrm{ind\,}L^*_0\leq\mathrm{ind\,}L^*_t$ for small $t$. 
By $\mathrm{ind\,}L^*_t=-\mathrm{ind\,}L_t$,
$\mathrm{ind\,}L_0\geq\mathrm{ind\,}L_t$ holds.  
This and \eqref{e-gue150603} prove the lemma.  
\end{proof}

\begin{proof}[Proof of Theorem~\protect\ref{t-gue150530I}]
Let
\begin{equation*}
I_0:=\left\{r\in[0,1];\, \mbox{there is an $\varepsilon>0$ such that ${\rm
ind\,}L_t={\rm ind\,L_0}$, $\forall
t\in(r-\varepsilon,r+\varepsilon)\bigcap[0,1]$}\right\}.
\end{equation*}
$%
I_0\ne\emptyset$ is open by
Lemma~\ref{l-gue150603}.  Around a limit point $r_{\infty}$ of $I_0$, 
by the same type of argument in the proof of Lemma~\ref{l-gue150603} and
Lemma~\ref{l-gue150531} (replacing $t=0$ by $t=r_\infty$ in 
$H_0$, $H_1$ and $A_0$), 
one finds  $\mathrm{ind\,}L_t=\mathrm{ind\,}L_{r_{\infty}}$ for 
$t\in(r_\infty-\varepsilon_0,r_\infty+\varepsilon_0)$ with some $\epsilon_0>0$.   
This implies $I_0$ is closed in $[0,1]$, so $I_0=[0,1]$.  
\end{proof}

\section{Asymptotic expansions for the heat kernels of the modified Kohn
Laplacians}

\label{s-gue150606}
In view of the McKean-Singer formula (Corollary~\ref{t-gue150603}), one of the goals is to 
calculate the local density (i.e. the term to the right of \eqref{e-gue150603b}).   
It consists in obtaining an asymptotic expansion for the heat kernel of the modified Kohn Laplacian
($\mathrm{Spin}^c$ Kohn Laplacian), to which 
we base our approach on two main steps.   While the first step is motivated by the globally free
case (see Theorem~\ref{t-gue150508}), it will be replaced by a local version within
the framework of BRT trivializations (Section~\ref{s-gue150514}).   A crucial off-diagonal
estimate is going to be done in this subsection (cf. Theorem~\ref{t-gue150627g}).  In the second step
we use the {\it adjoint} version of the heat equation to construct a global heat kernel with an asymptotic 
expansion related to {\it local kernels}.  

\subsection{Heat kernels of the modified Kodaira Laplacians on BRT
trivializations}

\label{s-gue150627}
This subsection is motivated by the globally free case
(cf. Theorem~\ref{t-gue150508}).   Here the emphasis is made 
on the localization of the argument including the $\mathrm{Spin}^c$ structure (which is needed for explicit 
local formulas of the heat kernel density).  An important heat kernel estimate, 
termed as {\it off-diagonal estimate}, will be established in Theorem~\ref{t-gue150627g}.  

It is worth remarking that in the statement and proof of Theorem~\ref{t-gue150508}, 
we make no use of harmonic theory.    In the locally free case, by contrast, it will be an important step to 
relate the (modified) Kohn Laplacian to (modified) Kodaira Laplacian (see
discussion after that theorem).  Since these Laplacians 
are defined via certain {\it adjoints}, suitable matching of metrics involved in both Laplacians must be 
done as an essential step.  

We will use the same notations as in Section~\ref{s-gue150524}.
Let $B:=(D,(z,\theta),\varphi)$ be a
BRT trivialization (with $D=U\times]-\varepsilon,\varepsilon[$%
, $\varepsilon>0$ and $U$ an open subset of $\mathbb{C}^n$, 
cf. Subsection~\ref{s-gue150514}).   For $x\in D$ wrtie 
$z=z(x)$ and $\theta=\theta(x)$.  
Since $E$
is rigid and CR, equipped with a rigid Hermitian (fiber) metric $\langle\,\cdot\,|\,\cdot\,\rangle_E$,
(as in Section~\ref{s-gue150524}) $E$ descends as a
(holomorphically trivial) vector bundle over $U$ (possibly after shrinking $U$)
and $\langle\,\cdot\,|\,\cdot\,\rangle_E$ as a Hermitian (fiber) metric on $E$ over $U$.

Let $%
L\rightarrow U$ be a trivial (complex) line bundle with a non-trivial Hermitian fiber
metric $\left\vert 1\right\vert^2_{h^L}=e^{-2\varphi}$ ($\varphi$ as in the above BRT triple $B$).   
Write $%
(L^m,h^{L^m})\rightarrow U$ for the $m$-th power of $(L,h^L)$. Let $%
\Omega^{0,q}(U,E\otimes L^m)$ be the space of $(0,q)$ forms on $U$ with
values in $E\otimes L^m$ ($q=0,1,2,\ldots,n$).   As usual,
$\Omega^{0,+}(U,E\otimes L^m)$ and $\Omega^{0,-}(U,E\otimes L^m)$
denote forms of even and odd degree.

To start with the matching of the metrics we 
let $\langle\,\cdot\,,\,\cdot\,%
\rangle$ be the Hermitian metric on $\mathbb{C }TU$ given by
(cf. \eqref{e-gue150823e1})
\label{e-gue150627e1} 
\begin{equation}
\langle\,\frac{\partial}{\partial z_j}\,,\,\frac{\partial}{\partial z_k}%
\,\rangle=\langle\,\frac{\partial}{\partial z_j}-i\frac{\partial\varphi}{%
\partial z_j}(z)\frac{\partial}{\partial\theta}\,|\,\frac{\partial}{\partial
z_k}-i\frac{\partial\varphi}{\partial z_k}(z)\frac{\partial}{\partial\theta}%
\,\rangle,\ \ j,k=1,2,\ldots,n.
\end{equation}
$\langle\,\cdot\,,\,\cdot\,\rangle$ induces Hermitian metrics on $T^{*0,q}U$
(bundle of $(0,q)$ forms on $U$), denoted also by $\langle\,\cdot\,,\,\cdot\,\rangle$.
These metrics induce Hermitian metrics on $T^{*0,q}U\otimes
E$, still denoted by $%
\langle\,\cdot\,|\,\cdot\,\rangle_{E}$.

Let $(\,\cdot\,,\,\cdot\,)$ be the $%
L^2$ inner product on $\Omega^{0,q}(U,E)$ induced by $\langle\,\cdot\,,\,%
\cdot\,\rangle$, $\langle\,\cdot\,|\,\cdot\,\rangle_E$,
and similarly $(\,\cdot\,,\,\cdot\,)_m$ the $L^2$ inner product on $%
\Omega^{0,q}(U,E\otimes L^m)$ induced by $\langle\,\cdot\,,\,\cdot\,\rangle$%
, $\langle\,\cdot\,|\,\cdot\,\rangle_E$ and $h^{L^m}$.

Let $\overline\partial_{L^m}:\Omega^{0,q}(U,E\otimes
L^m)\rightarrow\Omega^{0,q+1}(U,E\otimes L^m),\,(q=0,1,2,\ldots,n-1),$
be the Cauchy-Riemann operator.  Let
\begin{equation*}
\overline{\partial}^*_{L^m}:\Omega^{0,q+1}(U,E\otimes
L^m)\rightarrow\Omega^{0,q}(U,E\otimes L^m)
\end{equation*}
be the formal adjoint of $\overline\partial_{L^m}$ with respect to $%
(\,\cdot\,,\,\cdot\,)_m$.  

An essential operator that enters our picture 
is the following one (of Dirac type).   
\begin{equation}  \label{e-gue150606}
D_{B,m}\,(=D^{+-}_{B,m})\,:=\overline\partial_{L^m}+\overline{\partial}_{L^m}^{*}+A_B:
\Omega^{0,+}(U,E\otimes L^m)\rightarrow\Omega^{0,-}(U,E\otimes L^m)
\end{equation}
where $A_B:\Omega^{0,+}(U,E\otimes L^m)\rightarrow\Omega^{0,-}(U,E\otimes
L^m)$ is as in \eqref{e-gue150524faI} (replacing $E$ there by $E\otimes L^m$ here) and 
\begin{equation}  \label{e-gue150606I}
D^*_{B,m}: \Omega^{0,-}(U,E\otimes L^m)\rightarrow\Omega^{0,+}(U,E\otimes
L^m)
\end{equation}
the formal adjoint of $D_{B,m}$ with respect to $(\,\cdot\,,\,\cdot\,)_m$. 
(Note $D_{B,m}$ on the full $\Omega^{0,*}=\Omega^{0,+}\oplus\Omega^{0,-}$
is self-adjoint; see the line below \eqref{e-gue150825a}.  But we prefer 
to use the above $D^*_{B,m}$ in the present context.)   Note also $L$ with the metric $h^L$ depends 
on the choice of a BRT trivialization.    However, $A_B$ is indeed an intrinsic object; 
we refer to Remark~\ref{r-gue150606} in this regard.  

One has the {\it modified/$\mathrm{Spin}^c$ Kodaira Laplacian}:
\begin{equation}  \label{e-gue150606II}
\widetilde\Box^{+}_{B,m}:=D^*_{B,m}D_{B,m}: \Omega^{0,+}(U,E\otimes
L^m)\rightarrow\Omega^{0,+}(U,E\otimes L^m).
\end{equation}

One may define $\widetilde\Box^{-}_{B,m}: \Omega^{0,-}(U,E\otimes
L^m)\rightarrow\Omega^{0,-}(U,E\otimes L^m)$ analogously 
(by starting with $D^{- +}_{B,m}$ or $D^*_{B,m}$). 

The following fact appears fundamental in itself.  It is instrumental to our construction of a heat kernel 
(cf. \eqref{e-gue150628gI}) (See, however, remarks after its proof). 

\begin{prop}
\label{l-gue150606}  In notations above let $u\in\Omega^{0,\pm}_m(X,E)$.  On $D$ we can write $u(z,\theta)=e^{-im\theta}\widetilde u(z)$
for some $\widetilde
u(z)\in\Omega^{0,\pm}(U,E)$.  Recall the modified Kohn Laplacian $\widetilde{\Box}^{\pm}_{b,m}$ in \eqref{e-gue150525aII}.   We write $s$ for the local basis $1^m$ of $L^m$.  
Then 
\begin{equation}  \label{e-gue150606III}
e^{-m\varphi}\widetilde\Box^{\pm}_{B,m}(e^{m\varphi}\widetilde u\otimes s)=(e^{im\theta}%
\widetilde{\Box}^{\pm}_{b,m}(u))\otimes s.
\end{equation}
Without the danger of confusion we may write 
\begin{equation}
\label{e-gue150606IV}
e^{-m\varphi}\widetilde\Box^{\pm}_{B,m}(e^{m\varphi}\widetilde u)=e^{im\theta}\widetilde{\Box}^{\pm}_{b,m}(u).
\end{equation}
\end{prop}

\begin{proof}  One may work out this result by explicit computations. 
The following gives a somewhat conceptual proof.  
The idea is that one continues to match the objects 
on $U$ and on $D\,(\subset X)$.  (In this way it turns out 
that no explicit computations of these Laplacians in local coordinates are needed.)  

We define $\chi\,(=\chi_q)\,:\Omega^{0,q}(U,E)\to 
\Omega^{0,q}(U,E\otimes L^m)$ ($q=0,1,2,\ldots,n$) by $\tilde v(z) \to 
\tilde v(z)e^{m\varphi(z)}\otimes s(z)$ for $\tilde v\in \Omega^{0,q}(U, E)$.  
Note $\chi$ preserves the (pointwise) norms.   Equivalently $\chi(e^{-m\varphi}\tilde v)=
\tilde v\otimes s$.  

We define $\delta \tilde v=\overline\partial \tilde v+ m(\overline\partial\varphi) \wedge\tilde v$ for 
$\tilde v\in \Omega^{0,q}(U,E)$ where $\overline\partial: \Omega^{0,q}(U,E)\to \Omega^{0,q+1}(U,E)$. 
One may verify 
\begin{equation}
\label{e-gue150606le1}
\overline\partial_{L^m}\circ\chi
=\chi\circ \delta\quad \mbox{on\,\, $\Omega^{0,q}(U,E)$}.
\end{equation}

Indeed, by $\chi(e^{-m\varphi}\overline\partial \tilde u)=\overline\partial\tilde u\otimes s
=\overline\partial_{L^m}(\tilde u\otimes s)$, 
one sees the term to the left of \eqref{e-gue150606le1}:  $\overline\partial_{L^m}\circ\chi(e^{-m\varphi}\tilde u)\,(=
\overline\partial_{L^m}(\tilde u\otimes s))\,=\chi(e^{-m\varphi}\overline\partial \tilde u)$.
Further, by using definition of $\delta$ one computes $e^{m\varphi}\delta(e^{-m\varphi}\tilde u)=\overline\partial\tilde u$.
Then  
$\chi(e^{-m\varphi}\overline\partial \tilde u)=\chi(e^{-m\varphi}(e^{m\varphi}\delta(e^{-m\varphi}\tilde u)))$
which is $\chi(\delta(e^{-m\varphi}\tilde u))$, 
giving the term to the right of \eqref{e-gue150606le1} and proving \eqref{e-gue150606le1}.  

Since $\chi$ is norm-preserving, we have also 
\begin{equation}
\label{e-gue150606le2}
\overline\partial_{L^m}^*\circ \chi=\chi\circ \delta^*
\end{equation}
between respective adjoints.  Combining \eqref{e-gue150606le1}
and \eqref{e-gue150606le2} gives for $\Box_{B,m}\equiv(\overline\partial_{L^m}^*+\overline\partial_{L^m})^2$
and $\Delta\equiv (\delta^*+\delta)^2$
\begin{equation}
\label{e-gue150606le3}
\Box_{B,m}\circ \chi
=\chi\circ\Delta.
\end{equation}

By \eqref{e-gue150514f} for $\overline\partial_b$, one computes, for $g=e^{-im\theta}\tilde g\in
\Omega_m^{0,q}(U\times]-\varepsilon,\varepsilon[,E)=\Omega_m^{0,q}(D,E)$
\begin{equation}  \label{e-gue150606a}
e^{im\theta}\overline\partial_b(e^{-im\theta}\tilde g(z))=\delta(\tilde g(z)).
\end{equation}
Write the map $\chi_1: \Omega^{0,q}_m(D, E)\to \Omega^{0,q}(U, E)$ 
for $\chi_1(g)=\chi_1(e^{-im\theta}\tilde g)=\tilde g$, equivalently, $\chi_1(g)=e^{im\theta}g$.  
Note $\chi_1$ preserves the 
respective (pointwise) 
norms (cf. \eqref{e-gue150627e1}).   

By \eqref{e-gue150606a} one sees (with $\overline\partial_{b,m}=\overline\partial_{b}|_{\Omega^{0,q}_m}$) 
\begin{equation}
\label{e-gue150606le4}
\chi_1\circ \overline\partial_{b,m}=\delta\circ \chi_1\quad\mbox{on\,\,\,$\Omega^{0,q}_m(D, E)$}.
\end{equation}

By \eqref{e-gue22a1} the $L^2$ inner product on $\Omega^{0,q}_m(D, E)$ is clearly $2\varepsilon(\,\cdot\,,\,\cdot\,)$ 
with the $L^2$ inner product $(\,\cdot\,,\,\cdot\,)$ on $\Omega^{0,q}(U,E)$.  
Thus, in the same way as \eqref{e-gue150606le3} by using \eqref{e-gue150606le4} we have for $\Box_{b,m}\equiv (\overline\partial_{b,m}^*+\overline\partial_{b,m})^2$ (and $\Delta\equiv  
(\delta^*+\delta)^2$ as above) 
\begin{equation}
\label{e-gue150606le5}
\chi_1\circ\Box_{b,m}=\Delta\circ \chi_1\quad\mbox{on\,\,\,$\Omega^{0,q}_m(D, E)$}. 
\end{equation}

Combining \eqref{e-gue150606le5} and \eqref{e-gue150606le3} yields 
\begin{equation}
\label{e-gue150606le6}\begin{split}
\Box_{b,m}&=(\chi\chi_1)^{-1}\circ\Box_{B,m}\circ(\chi\chi_1)
\end{split}
\end{equation}
By $\chi\chi_1(e^{-im\theta}\tilde u)=e^{m\varphi}\tilde u\otimes s$ and $(\chi\chi_1)^{-1}(\tilde v\otimes s)=
e^{-im\theta}e^{-m\varphi}\tilde v$, one obtains 
\begin{equation}
\label{e-gue150606le7}
(\Box_{b,m}u)\otimes s=e^{-im\theta}e^{-m\varphi}(\Box_{B,m}e^{m\varphi}(\tilde u\otimes s))
\quad\mbox{for\,\,\,$u=e^{-im\theta}\tilde u\in \Omega^{0,q}_m(D, E)$}, 
\end{equation}
giving $e^{-m\varphi}\Box_{B,m}(e^{m\varphi}\widetilde u)=e^{im\theta}{\Box}_{b,m}(u)$
in notation similar to \eqref{e-gue150606IV}.

For modified Laplacians, from the definition of the zeroth order operator $%
A_m:\Omega^{0,+}_m(X,E)\rightarrow\Omega^{0,-}_m(X,E)$ (see Definition~\ref%
{d-gue150524}), it is clear that (in notation similar to \eqref{e-gue150606IV}) 
\begin{equation}  \label{e-gue150606aVIII}
e^{-m\varphi}A_B(e^{m\varphi}\widetilde u)=e^{im\theta}A_m(u).
\end{equation}

In a way similar to \eqref{e-gue150606le7} it follows by using 
\eqref{e-gue150606aVIII} that 
\begin{equation*}
e^{-m\varphi}D_{B,m}(e^{m\varphi}\widetilde u)=e^{im\theta}\widetilde
D_{b,m}(u)
\end{equation*}
hence easily that 
\begin{equation*}
e^{-m\varphi}\widetilde\Box^+_{B,m}(e^{m\varphi}\widetilde u)=e^{im\theta}%
\widetilde{\Box}^+_{b,m}(u)
\end{equation*}
proving the proposition.
\end{proof}

Remark that one might be led by Proposition~\ref{l-gue150606} to 
reduce the study of Kohn $\widetilde\Box^{\pm}_{b,m}$ to that of 
Kodaira $\widetilde\Box_{B, m}$.  Indeed such a reduction 
works quite well in the globally free case (see discussion following Theorem~\ref{t-gue150508} in Introduction).  
In the locally free case (of $S^1$ action),  however, a naive 
thought of using the Kodaira Laplacian and its associated (local) heat kernels for a 
better understanding of the heat kernel in Kohn's case is not directly accessible (see remarks
following proof of Theorem~\ref{t-gue150630I}).  Namely 
the associated heat kernels of the two Laplacians cannot be easily linked 
as \eqref{e-gue150606III} seems to suggest.  
This reflects the fact that the associated heat kernels,
rather than Laplacians themselves,  
are objects which are more global in nature.  More in this regard will be 
pursued in the coming Subsection~\ref{s-gue150627I} and Section~\ref{s-gue150702}. 

\begin{rem}
\label{r-gue150606}
The definition of $A_B$ in \eqref{e-gue150606} 
depends on a BRT triple, and the same can be said with Proposition~\ref{l-gue150606}.
To see that $A_B$ has an intrinsic meaning, one uses the transformation of BRT coordinates as shown in 
the proof of Proposition~\ref{l-gue150524d}.    The geometrical construction given there 
shows that locally $X$ is part of a circle bundle inside the $L^*$ (with metric induced by that of $L$) over $U$,
and the quantities such as $\varphi$, $z$ and $\theta$ in a BRT triple are associated with 
geometric ones as metric for a local basis (of $L$), coordinates on the base $U$ and (part of) a holomorphic coordinate on 
fibers (of $L^*$) respectively.    The transformation in these quantities with another choice of a BRT chart
is nothing more than a change of holomorphic coordinates of the same line bundle.    
It follows that $A_B$ is intrinsic in a proper sense.   
A similar explanation can be given to Proposition~\ref{l-gue150606} too (although we do not 
strictly need this intrinsic property in what follows).  
\end{rem}

\begin{rem}
\label{r-gue150606a}  In the case of certain Riemannian foliations, it is known that the 
Laplacian downstairs and Laplacian upstairs (in a suitable generalized sense) 
can be related in spirit similar to that in our proposition above.  See \cite[p. 2310-2311]{Ri10}.
\end{rem}

As remarked in Subsection~\ref{s-gue151025}, to suit our purpose
we will actually be considering {\it adjoint} heat equation
and {\it adjoint} heat kernel first.  

To proceed further, some notations are in order.  Let $M$ be a $C^\infty$ orientable paracompact
manifold with a vector bundle $F$ over it.

\begin{defn}
\label{d-gue150608} Let $A(t,x)\in C^\infty(\mathbb{R}_+\times M,F)$. We
write
\begin{equation*}
\begin{split}
&A(t,x)\sim t^kb_{-k}(x,t)+t^{k+1}b_{-k+1}(x,t)+t^{k+2}b_{-k+2}(x,t)+\cdots\
\ \mbox{as $t\To0^+$}, \\
&\mbox{$b_s(x,t)\in C^\infty(M,F)$ a possibly $t$-dependent smooth
function},\ \ s=-k,-k+1,-k+2,\ldots,
\end{split}%
\end{equation*}
for $k\in\mathbb{Z}$, provided that for every compact set $K\Subset M$, every $\ell, M_0\in%
\mathbb{N}_0$ with $M_0\ge M_0(m)$ for some $M_0(m)$ ($m=\mathrm{dim}\,M$), 
there are $C_{\ell,K,M_0}>0$, $\varepsilon_0>0
$ and $M_1(m,\ell)$ (independent of $t$) such that
\begin{equation*}
\left\vert
A(t,x)-\sum^{M_0}_{j=0}t^{k+j}b_{-k+j}(x,t)\right\vert_{C^\ell(K)}\leq
C_{\ell,K,M_0}t^{M_0-M_1(m,\ell)},\ \ \forall \,0<t<\varepsilon_0.
\end{equation*}
\end{defn}

\begin{rem}
\label{r-gue150608}
In the important case of the heat kernel $p_t(x, y)$ of a generalized Laplacian on a compact Riemannian manifold 
$B$ of dimension $\beta$, $M=B\times B$ is of dimension $m=2\beta$ and $k=-\frac{\beta}{2}$.
One can take $M_0(m)=[\frac{\beta}{2}]+1$
and $M_1(m,\ell)=\frac{\beta+\ell}{2}$.  See \cite[Theorem 2.30]{BGV92}.
In this case $b_s(t, x)$ for all $s$ can be taken to be independent of $t$.  
\end{rem}

The novelty above is that $b_s$ could be nontrivially dependent on $t$ 
(in contrast to the conventional case of an asymptotic expansion for heat kernels).   

Let $T^{*0,+}U$ and $T^{*0,-}U$ denote forms of even degree
and odd degree in $T^{*0,\bullet}U$, respectively as before.  
If $T(z,w)\in (T^{*0,+}U\otimes E)\boxtimes (T^{*0,+}U\otimes E)^*|_{(z,w)}$,
write $\left\vert T(z,w)\right\vert$ 
for the standard pointwise matrix norm of $T(z,w)$ induced by $%
\langle\,\cdot\,,\,\cdot\,\rangle$ and $\langle\,\cdot\,|\,\cdot\,\rangle_E$.

Suppose $G(t,z,w)\in C^\infty(\mathbb{R}_+\times U\times
U, (T^{*0,+}U\otimes E)\boxtimes (T^{*0,+}U\otimes E)^*)$. 
As usual, we denote $G(t)
:\Omega^{0,+}_0(U,E)\rightarrow\Omega^{0,+}(U,E)$ (resp.  $G^{\prime }(t)$)
the continuous operator
associated with the kernel $G(t, z, w)$ (resp. $\frac{\partial G(t,z,w)}{\partial t}$)
($\Omega^{0,+}_0(U,E)$ denotes elements of compact
support in $U$).  

We are now ready to consider the heat operators associated with $\widetilde\Box^+_{B,m}$ and $%
\widetilde\Box^-_{B,m}$ in an {\it adjoint} version.   By using the Dirichlet heat kernel construction
(see~\cite{G8} or \cite{Cha}) we can obtain the theorem stated in the following form.

\begin{prop}
\label{t-gue150607} There exists an $A_{B,+,m}(t,z,w)=:A_{B,+}(t,z,w)\in C^\infty(%
\mathbb{R}_+\times U\times U, (T^{*0,+}U\otimes E)\boxtimes (T^{*0,+}U\otimes E)^*)$
such that
\begin{equation}  \label{e-gue150607ab}
\begin{split}
&\mbox{$\lim_{t\To0+}A_{B,+}(t)=I$ in $\mathscr D'(U,T^{*0,+}U\otimes E)$},
\\
&A^{\prime }_{B,+}(t)u+A_{B,+}(t)(\widetilde\Box^+_{B,m}u)=0,\ \ \forall
u\in\Omega^{0,+}_0(U,E),\ \ \forall t>0,
\end{split}%
\end{equation}
and $A_{B,+}(t,z,w)$ satisfies the following: (I) For every compact set $%
K\Subset U$ and every $\alpha_1, \alpha_2, \beta_1, \beta_2\in\mathbb{N}^n_0$%
, every $\gamma\in\mathbb{N}_0$, there are constants $C_{\gamma,\alpha_1,%
\alpha_2,\beta_1,\beta_2,K}>0$, $\varepsilon_0>0$ and $P\in\mathbb{N}$
independent of $t$ such that
\begin{equation}  \label{e-gue150607abI}
\left\vert
\partial^\gamma_t\partial^{\alpha_1}_z\partial^{\alpha_2}_{\overline
z}\partial^{\beta_1}_w\partial^{\beta_2}_{\overline
w}A_{B,+}(t,z,w)\right\vert\leq
C_{\gamma,\alpha_1,\alpha_2,\beta_1,\beta_2,K} t^{-P}e^{-\varepsilon_0\frac{%
\left\vert z-w\right\vert^2}{t}},\ \ \forall (t,z,w)\in\mathbb{R}_+\times
K\times K.
\end{equation}
(II) Let $g\in\Omega^{0,+}_0(U,E)$. For every $\alpha_1, \alpha_2\in\mathbb{N%
}^n_0$ and every compact set $K\Subset U$ , there is a $C_{\alpha_1,%
\alpha_2,K}>0$ independent of $t$ such that
\begin{equation}  \label{e-gue150607abII}
\begin{split}
&\sup\left\{\left\vert \partial^{\alpha_1}_z\partial^{\alpha_2}_{\overline
z}(A_{B,+}(t)g)(z)\right\vert;\, z\in K\right\} \\
&\leq C_{\alpha_1,\alpha_2,K}\sum_{\beta_1, \beta_2\in\mathbb{N}^n_0,
\left\vert \beta_1\right\vert+\left\vert \beta_2\right\vert\leq\left\vert
\alpha_1\right\vert+\left\vert \alpha_2\right\vert}\sup\left\{\left\vert
\partial^{\beta_1}_z\partial^{\beta_2}_{\overline z}g(z)\right\vert;\, z\in
U\right\}.
\end{split}%
\end{equation}
(III) $A_{B,+}(t,z,w)$ admits an asymptotic expansion in the following sense (see 
Definition~\ref{d-gue150608} for $\sim$).  
For some $K_{B,+}(t,z,w)$
\begin{equation}  \label{e-gue150608a}
\begin{split}
&A_{B,+}(t,z,w)=e^{-\frac{h_+(z,w)}{t}}K_{B,+}(t,z,w), \\
&K_{B,+}(t,z,w)\sim
t^{-n}b^+_n(z,w)+t^{-n+1}b^+_{n-1}(z,w)+\cdots+b^+_0(z,w)+tb^+_{-1}(z,w)+%
\cdots\ \ \mbox{as $t\To0^+$}, \\
&b^+_s(z,w)\,(=b^+_{s, m}(z,w))\,\in C^\infty(U\times U, (T^{*0,+}U\otimes E)\boxtimes (T^{*0,+}U\otimes E)^*), 
\ \ s=n,n-1,n-2,\ldots,
\end{split}%
\end{equation}
where $h_+(z,w)\in C^\infty(U\times U,\mathbb{R}_+)$ with $h_+(z,z)=0$ for
every $z\in U$ and for every compact set $K\Subset U$, there is a constant $%
C_K>1$ such that $\frac{1}{C_k}\left\vert z-w\right\vert^2\leq h_+(z,w)\leq
C_K\left\vert z-w\right\vert^2$.

In \eqref{e-gue150607ab} with $\widetilde\Box^-_{B,m}$ in place of $\widetilde\Box^+_{B,m}$, 
corresponding statements (with $A_{B,-}, K_{B,-}$ etc.) for $\widetilde\Box^-_{B,m}$ hold true 
as well. 
\end{prop}

\begin{rem}
\label{r-gue150607}
One may use a (weaker) version in the sense of {\it asymptotic heat kernel}
(cf. \cite{Du11}, p. 96 for an informal explanation)
with the property that this kernel (without uniqueness) satisfies \eqref{e-gue150607abI}-\eqref{e-gue150608a}
(so that it admits the same asymptotic expansion as the above kernel) and also satisfies 
\eqref{e-gue150607ab} {\it asymptotically} (cf. Lemma~\ref{l-gue150628} with \eqref{e-gue150628g}).    
Although the asymptotic heat kernels are not unique, 
the ``formal" heat kernel given (in above notation) by this form 
$$e^{-\frac{h_+(z,w)}{t}}(t^{-n}b^+_n(z,w)+t^{-n+1}b^+_{n-1}(z,w)+\cdots+b^+_0(z,w)+tb^+_{-1}(z,w)+%
\cdots\ )$$ is unique.  

\end{rem}


We are interested in calculating $\mathrm{Tr\,}b^+_s(z,z)-\mathrm{Tr\,}b^-_s(z,z)$ ($s=n,n-1,\ldots,0$)
(where $\mathrm{Tr\,}b^{\pm}_s(z,z)=\sum_{j}\langle\,b^+_s(z,z)e_j\,|\,\,e_j\,
\rangle_{E}$ for any orthonormal frame $e_j$ of $T^{*0,+}_zU\otimes E_z$).  
The idea relies on {\it Lichnerowicz
formulas} for (modified/$\mathrm{Spin}^c$ Kodaira Laplacians)
$\widetilde\Box^+_{B,m}$ and $\widetilde\Box^-_{B,m}$
(cf. Theorem 1.3.5 and Theorem 1.4.5 in~\cite{MM07}) so that
 the (by now standard) rescaling technique in~\cite{BGV92} and~\cite{Du11}
can apply well.  

To state the result precisely, we introduce some notations. Let $%
\nabla^{TU}$ be the Levi-Civita connection on $\mathbb{C }TU$ with respect
to $\langle\,\cdot\,,\,\cdot\,\rangle$. Let $P_{T^{1,0}U}$ be the natural
projection from $\mathbb{C }TU$ onto $T^{1,0}U$.   $%
\nabla^{T^{1,0}U}:=P_{T^{1,0}U}\nabla^{TU}$ is a connection on $T^{1,0}U$.
Let $\nabla^{E\otimes L^m}$ be the (Chern) connection on $E\otimes L^m\rightarrow U$
(induced by $\langle\,\cdot\,,\,\cdot\,\rangle_E$ and $h^{L^m}$, see Theorem~\ref{t-gue150515}). Let $%
\Theta(\nabla^{T^{1,0}U},T^{1,0}U)\, (\in C^\infty(U,\Lambda^2(\mathbb{C }%
T^*U)\otimes\mathrm{End\,}(T^{1,0}U)))$ and $\Theta(\nabla^{E\otimes
L^m},E\otimes L^m)\,(\in C^\infty(U,\Lambda^2(\mathbb{C }T^*U)\otimes\mathrm{%
End\,}(E\otimes L^m)))$ be the associated curvatures. As in complex geometry, put
\begin{equation*}
\begin{split}
&\mathrm{Td\,}(\nabla^{T^{1,0}U},T^{1,0}U)=e^{\mathrm{Tr\,}(h(\frac{i}{2\pi}%
\Theta(\nabla^{T^{1,0}U},T^{1,0}U)))},\ \ h(z)=\log(\frac{z}{1-e^{-z}}), \\
&\mathrm{ch\,}(\nabla^{E\otimes L^m},E\otimes L^m)=\mathrm{Tr\,}(\widetilde
h(\frac{i}{2\pi}\Theta(\nabla^{E\otimes L^m},E\otimes L^m))),\ \ \widetilde
h(z)=e^{z}.
\end{split}%
\end{equation*}

Then the above calculation leads to the following.
\begin{equation}  \label{e-gue150627}
\begin{split}
&\Bigr(\mathrm{Tr\,}b^+_s(z,z)-\mathrm{Tr\,}b^-_s(z,z)\Bigr)=0,\ \
s=n,n-1,\ldots,1, \\
&\Bigr(\mathrm{Tr\,}b^+_0(z,z)-\mathrm{Tr\,}b^-_0(z,z)\Bigr)dv_U(z)=\left[%
\mathrm{Td\,}(\nabla^{T^{1,0}U},T^{1,0}U)\wedge\mathrm{ch\,}%
(\nabla^{E\otimes L^m},E\otimes L^m)\right]_{2n}(z),\ \ \forall z\in U,
\end{split}%
\end{equation}
where $[\cdots]|_{2n}$ 
denotes the $2n$-form part.

As the calculation to be performed here is almost entirely the same as in the standard case,
we omit the detail.  

Let $\nabla^{TX}$ be the Levi-Civita connection on $TX$ with respect to $%
\langle\,\cdot\,|\,\cdot\,\rangle$ and $\nabla^E$ the connection on $E$ associated 
with $\langle\,\cdot\,|\,\cdot%
\,\rangle_E$ (cf. Theorem~\ref{t-gue150515}).
In similar notation as above  $%
\nabla^{T^{1,0}X}:=P_{T^{1,0}X}\nabla^{TX}$ is a connection on $T^{1,0}X$.

Since $\nabla^{T^{1,0}X}$ and $\nabla^E$ are rigid, 
in view of compatibility of metrics (and connections) in \eqref{e-gue150627e1} and Theorem~\ref{t-gue150515},
one sees that $\forall\, (z,\theta)\in D$ (for $\omega_0$ see lines below Definition~\ref{d-gue150508d1}): 
\begin{equation}  \label{e-gue150627I}
\begin{split}
&\mathrm{Td\,}(\nabla^{T^{1,0}U},T^{1,0}U)(z)=\mathrm{Td_b\,}%
(\nabla^{T^{1,0}X},T^{1,0}X)(z,\theta) \\
&\mathrm{ch\,}(\nabla^{E\otimes L^m},E\otimes L^m)(z)=\Bigr(\mathrm{ch_b\,}%
(\nabla^{E},E)\wedge e^{-m\frac{d\omega_0}{2\pi}}\Bigr)(z,\theta)
\end{split}%
\end{equation}
and
\begin{equation}  \label{e-gue150627II}
\begin{split}
&\left[\mathrm{Td\,}(\nabla^{T^{1,0}U},T^{1,0}U)\wedge\mathrm{ch\,}%
(\nabla^{E\otimes L^m},E\otimes L^m)\right]_{2n}(z)\wedge d\theta \\
&=\left[\mathrm{Td_b\,}(\nabla^{T^{1,0}X},T^{1,0}X)\wedge\mathrm{ch_b\,}%
(\nabla^{E},E)\wedge e^{-m\frac{d\omega_0}{2\pi}}\wedge\omega_0\right]%
_{2n+1}(z,\theta). 
\end{split}%
\end{equation}

To sum up we arrive at the following (by \eqref{e-gue150627}, \eqref{e-gue150627II}
and $dv_U\wedge d\theta=dv_X$ on $D$ cf. \eqref{e-gue22a1}) 

\begin{prop}
\label{t-gue150627} With the notations above, we have
\begin{equation}  \label{e-gue150627ab}
\begin{split}
&\Bigr(\mathrm{Tr\,}b^+_0(z,z)-\mathrm{Tr\,}b^-_0(z,z)\Bigr)dv_X(z,\theta) \\
&=\left[\mathrm{Td_b\,}(\nabla^{T^{1,0}X},T^{1,0}X)\wedge\mathrm{ch_b\,}%
(\nabla^{E},E)\wedge e^{-m\frac{d\omega_0}{2\pi}}\wedge\omega_0\right]%
_{2n+1}(z,\theta),\ \ \forall(z,\theta)\in D.
\end{split}%
\end{equation}
\end{prop}

To state the final technical result of this subsection, we first identify $%
T^{*0,\bullet}_{z_2}U\otimes E_{z_2}$ with $T^{*0,\bullet}_{z_1}U\otimes
E_{z_1}$ (by parallel transport along geodesics joining $z_1, z_2\in U$), so 
we can identify $T\in (T^{*0,\bullet}U\otimes E)\boxtimes
(T^{*0,\bullet}U\otimes E)^*|_{(z_1,z_2)}$
with an
element in $\mathrm{End\,}(T^{*0,\bullet}_{z_1}U\otimes E_{z_1})$.
With this identification, write
\begin{equation}  \label{e-gue160309w}
\mathrm{Tr_{z_1}\,}T:=\sum^d_{j=1}\langle\,Te_j\,|\,\,e_j\,\rangle_{E},
\end{equation}
where $e_1,\ldots,e_d$ is an orthonormal frame of $T^{*0,\bullet}_{z_1}U%
\otimes E_{z_1}$.

In the proof of Theorem~\ref{t-gue160307} (see Theorem~\ref{t-gue150701}),
somewhat surprisingly, as deviated from the classical case, we need to estimate the
off-diagonal terms $\mathrm{Tr_z\,}b^+_s(z,w)-\mathrm{Tr_z\,}b^-_s(z,w)$
for each $s$.  
For this, the following can be considered as another application of the rescaling technique
(and an identity in Berenzin integral as usual).

\begin{thm} {\rm(Off-diagonal estimate)}
\label{t-gue150627g} With the notations above, we have
\begin{equation}  \label{e-gue150627abg}
\begin{split}
\mathrm{Tr_z\,}b^+_s(z,w)-\mathrm{Tr_z\,}b^-_s(z,w)=O(\left\vert
z-w\right\vert^{2s})\ \ \mbox{locally uniformly on $U\times U$},\ \
s=n,n-1,\ldots,1.
\end{split}%
\end{equation}
\end{thm}

\begin{proof}
Recall $E$ is (holomorphically) trivial on $U$.
Let $e_1,\ldots,e_{2n}$ be an orthonormal basis for $T^*_0U$.
For $f\in T^*_0U$, let $c(f)\in\mathrm{End\,}(T^{*0,\bullet}_0U)$ be the
natural Clifford action of $f$ (see \eqref{e-gue150823I} or \cite{BGV92}).   As usual, for every strictly
increasing multi-index $J=(j_1,\ldots,j_q)$ we set $\left\vert
J\right\vert:= q$, $e_J:=e_{j_1}\wedge\cdots e_{j_q}$ and $%
c(e_J)=c(e_{j_1})\cdots c(e_{j_q})$. For $T\in\mathrm{End\,}%
(T^{*0,\bullet}_0U)$, we can always write $T=\sideset{}{'}\sum_{\left\vert
J\right\vert\leq 2n}c(e_J)T_J$ ($T_J\in \mathbb{C}$), where $\sideset{}{'}\sum$ denotes
the summation over strictly increasing multiindices. For $k\leq 2n$, we put
\begin{equation}  \label{e-gue160309qI}
[T]_k:=\sideset{}{'}\sum_{\left\vert J\right\vert=k}T_Je_J\quad(\in {\mathbb{C}}T^{*k}_0U).
\end{equation}
and a similar expression for $[T]$ (without the subscript $k$).  We identity $T$ with
$[T]$ without the danger of confusion.  We say that $\mathrm{ord\,}T\leq k$ if $T_J=0$, for all $\left\vert
J\right\vert>k$, and $\mathrm{ord\,}T=k$ if $\mathrm{ord\,}T\leq k$
and $[T]_k\neq0$.

A crucial result for our need here is an identity in Berenzin integral
(see~\cite[Proposition 3.21, Definition 3.4 and (1.28)] {BGV92})
which asserts that if $\mathrm{ord\,}%
T<2n$ then $\mathrm{STr\,}T=0$ (see \eqref{e-gue160114bI} for the definition of supertrace there) and
\begin{equation}  \label{e-gue160309qII}
\mathrm{STr\,}T=(-2i)^{2n}\mathrm{STr\,}T_{J_0}c(e_{J_0}),\ \
J_0=(1,2,\ldots,2n).
\end{equation}

Recall the identification $T^{*0,\bullet}_xU\cong T^{*0,\bullet}_0U$ just mentioned above the theorem, so that a
smooth function $F(x)\in (T^{*0,\bullet}U)\boxtimes (T^{*0,\bullet}U)^*|_{(0,x)}$
is identified with
a function $x\rightarrow F(x)\in\mathrm{End\,}(T^{*0,\bullet}_0U)$, giving a Taylor expansion
\begin{equation*}
F(x)=\sum_{\alpha\in\mathbb{N}^{2n}_0,\left\vert \alpha\right\vert\leq
P}x^\alpha F_\alpha+O(\left\vert x\right\vert^{P+1}),\ \ F_\alpha\in\mathrm{%
End\,}(T^{*0,\bullet}_0U).
\end{equation*}

We are ready to apply Getzler's rescaling technique to off-diagonal
estimates.  Consider $A_B(t,x,y)\equiv A_B(t,z,w):=A_{B,+}(t,z,w)\oplus
A_{B,-}(t,z,w)$ (cf. Proposition~\ref{t-gue150607}) and let $\chi\in C^\infty_0(U)$ with $\chi=1$ near $z=0$. Let
\begin{equation}  \label{e-gue160309qIII}
r(u,t,x):=\sum^{2n}_{k=1}u^{-\frac{k}{2}+n}[\chi(\sqrt{u}x)A_{B}(ut,0,\sqrt{u}%
x)]_k.
\end{equation}
Note that $A_B$ is actually identified with $[A_B]$ similar to the case of $T$ above, so that
the $k$-form part  ($k>n$) of \eqref{e-gue160309qIII} makes sense.

It is well-known that (see~\cite{BGV92}) $\lim_{u\rightarrow0}r(u,t,x)=g(t,x)
\in C^\infty(\mathbb{R}_+\times\mathbb{C}^n,{\mathbb{C}}T^{*\bullet}\mathbb{C}^n)$
in $C^{\infty}$-topology locally uniformly (${\mathbb{C}}T^{*\bullet}\mathbb{C}^n
=\oplus_{k=0}^{k=2n}\Lambda^k\mathbb{C}T^{*}\mathbb{C}^n$). 
In particular, $%
\lim_{u\rightarrow0}r(u,1,x)_{2n}=g(1,x)_{2n}$ in $C^\infty$-topology
locally uniformly, for their $2n$-form parts.

Let $b_s(z,w):=b^+_s(z,w)%
\oplus b^-_s(z,w)$, $s=n,n-1,\ldots$  (cf. \eqref{e-gue150608a}).  One sees
\begin{equation}  \label{e-gue160309qIV}
r(u,1,x)_{2n}=e^{-\frac{h_+(0,\sqrt{u}x)}{u}}\chi(\sqrt{u}x)\Bigr(%
u^{-n}[b_n(0,\sqrt{u}x)]_{2n}+u^{-n+1}[b_{n-1}(0,\sqrt{u}x)]_{2n}+\cdots%
\Bigr).
\end{equation}
Since $\lim_{u\rightarrow 0}e^{-\frac{h_+(0,\sqrt{u}x)}{u}}$ converges to a
smooth function in $C^\infty$-topology locally uniformly on $\mathbb{C}^n$ (see 
\eqref{e-gue150608a} in Proposition~\ref{t-gue150607}),
we deduce that
\begin{equation}  \label{e-gue160309qV}
\lim_{u\rightarrow0}\Bigr(u^{-n}[b_n(0,\sqrt{u}x)]_{2n}+u^{-n+1}[b_{n-1}(0,%
\sqrt{u}x)]_{2n}+\cdots\Bigr)=\hat g(x)\quad 
\in C^\infty(\mathbb{C}^n,{\mathbb{C}}T^{*\bullet}\mathbb{C}^n)
\end{equation}
in $C^\infty$-topology locally
uniformly. Fix $P\gg1$. Write
\begin{equation}  \label{e-gue160309y}
\hat g(x)=\sum_{\alpha\in\mathbb{N}^{2n}_0,\left\vert \alpha\right\vert\leq
P}\hat g_\alpha x^\alpha+O(\left\vert x\right\vert^{P+1}).
\end{equation}
and for each $s=n,n-1,\ldots$,
\begin{equation}  \label{e-gue160309yI}
b_s(0,x)=\sum_{\alpha\in\mathbb{N}^{2n}_0,\left\vert \alpha\right\vert\leq
P}b_{s,\alpha} x^\alpha+O(\left\vert x\right\vert^{P+1}).
\end{equation}
Hence
\begin{equation}  \label{e-gue160309yII}
[b_s(0,\sqrt{u}x)]_{2n}=\sum_{\alpha\in\mathbb{N}^{2n}_0,\left\vert
\alpha\right\vert\leq P}u^{\frac{\left\vert \alpha\right\vert}{2}%
}[b_{s,\alpha}]_{2n} x^\alpha+u^{\frac{P+1}{2}}O(\left\vert
x\right\vert^{P+1}),\ \ s=n,n-1,\ldots,
\end{equation}
and from \eqref{e-gue160309qV}, \eqref{e-gue160309y} and \eqref{e-gue160309yII}
it follows that for every $\alpha\in\mathbb{N}^{2n}_0$
\begin{equation}  \label{e-gue160309yIII}
\lim_{u\rightarrow0}\Bigr(u^{-n+\frac{\left\vert \alpha\right\vert}{2}%
}[b_{n,\alpha}]_{2n}+u^{-n+1+\frac{\left\vert \alpha\right\vert}{2}%
}[b_{n-1,\alpha}]_{2n}+\cdots\Bigr)=\hat g_\alpha.
\end{equation}
With \eqref{e-gue160309yIII} we conclude
\begin{equation}  \label{e-gue160309e}
[b_{n,\alpha}]_{2n}=0,\ \ \forall\left\vert \alpha\right\vert<2n.
\end{equation}
Combining \eqref{e-gue160309e} and \eqref{e-gue160309yI}, we see
\begin{equation*}
\mathrm{Tr_0\,}b^+_n(0,w)-\mathrm{Tr_0\,}b^-_n(0,w)=O(\left\vert
w\right\vert^{2n}).
\end{equation*}

We can repeat the method above for the second leading term, and deduce similarly
\begin{equation*}
\mathrm{Tr_0\,}b^+_s(0,w)-\mathrm{Tr_0\,}b^-_s(0,w)=O(\left\vert
w\right\vert^{2s}),\ \ s=n-1,n-2,\ldots,1.
\end{equation*}
The theorem follows.
\end{proof}

\subsection{Heat kernels of the modified Kohn Laplacians ($\mathrm{Spin}^c$ Kohn Laplacians)}

\label{s-gue150627I}
Based on Proposition~\ref{l-gue150606}, one is tempted to patch up 
the local heat kernels of the modified Kodaira Laplacian constructed in Propostion~\ref{t-gue150607} 
to form a global heat kernel for the modified Kohn
Laplacian.   This is no problem in the globally free case (of the $S^1$ action).  
In the locally free case, however, some delicate points 
arise as the relation of the two Laplacians given in the above proposition is, by nature, 
a local property, whereas the heat kernels are
global objects.   See discussions after proof of Proposition~\ref{l-gue150606} and of Theorem~\ref{t-gue150630I}
with \eqref{e-gue150630gI} for more.  

As remarked in Subsection~\ref{s-gue151025}, if we use the {\it adjoint} version of the
original heat equation, it becomes more effective to go over the desired process of patching up.
It is worth noting that an important role, mostly unseen traditionally, 
is played by the projection $Q_m^{\pm}$ in our situation.

Assume $X=D_1\bigcup
D_2\bigcup\cdots\bigcup D_N$ with each $B_j:=(D_j,(z,\theta),\varphi_j)$ a
BRT trivialization.  A slightly more complicated
set up is as follows.     Write for each $j$, $%
D_j=U_j\times]-2\delta_j,2\widetilde\delta_j[\subset\mathbb{C}^n\times%
\mathbb{R}$, $\delta_j>0$, $\widetilde\delta_j>0$, $U_j=\left\{z\in\mathbb{C}%
^n;\, \left\vert z\right\vert<\gamma_j\right\}$.  Put $\hat
D_j=\hat U_j\times]-\frac{\delta_j}{2},\frac{\widetilde\delta_j}{2}[$,
$\hat U_j=\left\{z\in\mathbb{C}^n;\, \left\vert z\right\vert<\frac{\gamma_j}{%
2}\right\}$.   We suppose $X=\hat D_1\bigcup\hat
D_2\bigcup\cdots\bigcup\hat D_N$.

Here are some cut-off functions; the choice is adapted to
BRT trivializations.   

i) $\chi_j(x)\in C^\infty_0(\hat D_j)$ with $\sum^N_{j=1}\chi_j=1$ on $X$.  Put
\begin{equation*}
A_j=\left\{z\in\hat U_j;\, \mbox{there is a
$\theta\in]-\frac{\delta_j}{2},\frac{\Td\delta_j}{2}[$ such that $\chi_j(z,\theta)\neq0$}%
\right\}.
\end{equation*}

ii) $\tau_j(z)\in C^\infty_0(\hat U_j)$ with $\tau_j\equiv1$ on some
neighborhood $W_j$ of $A_j$.

iii) $\sigma_j\in C^\infty_0(]-\frac{\delta_j}{2}%
,\frac{\widetilde\delta_j}{2}[)$ with $\int_{-\delta_j/2}^{\Td\delta_j/2}\sigma_j(\theta)d\theta=1$.

iv) $%
\hat\sigma_j\in C^\infty_0(]-\delta_j,\widetilde\delta_j[)$ such that $%
\hat\sigma_j=1$ on some neighbourhood of $\mathrm{Supp\,}\sigma_j$ and $%
\hat\sigma_j(\theta)=1$ if $(z,\theta)\in\mathrm{Supp\,}\chi_j$.

Write $x=(z,\theta)$, $y=(w,\eta)\in\mathbb{C}^{n}\times\mathbb{R}$.
We are going to lift many objects in the preceding subsection defined on $U_j$ 
to the ones defined on $\hat D_j$ via the above cut-off functions.  

Let $A_{B_j,+}(t,z,w)$, $K_{B_j,+}$, $h_{j,+}(z,w)$
and $b^+_{j,s}$ ($s=n,n-1,\ldots$) be as in Proposition~\ref%
{t-gue150607} and \eqref{e-gue150608a}.   Slightly tediously, 
we put
\begin{equation}  \label{e-gue150627f}
\begin{split}
&H_j(t,x,y)=H_{j, +}(t, x, y)=\chi_j(x)e^{-m\varphi_j(z)-im\theta}A_{B_j,+}(t,z,w)e^{m%
\varphi_j(w)+im\eta}\tau_j(w)\sigma_j(\eta)\\
&G_j(t,x,y)=G_{j, +}(t, x, y)=\chi_j(x)e^{-m\varphi_j(z)-im\theta}A_{B_j,+}(t,z,w)e^{m%
\varphi_j(w)+im\eta}\tau_j(w).
\end{split}%
\end{equation}

and (the last two equations are from \eqref{e-gue150608a}) 
\begin{equation}  \label{e-gue150901}
\begin{split}
&\hat K_{j,+}(t,x,y)=\hat
K_j(t,x,y)=\chi_j(x)e^{-m\varphi_j(z)-im\theta}K_{B_j,+}(t,z,w)e^{m%
\varphi_j(w)+im\eta}\tau_j(w)\sigma_j(\eta) \\
&\hat h_{j,+}(x,y)=\hat\sigma_j(\theta) h_{j,+}(z,w)\hat\sigma_j(\eta)\in
C^\infty_0(D_j),\ \ x=(z,\theta),\ \ y=(w,\eta) \\
&\hat
b^+_{j,s}(x,y)=\chi_j(x)e^{-m\varphi_j(z)-im\theta}b^+_{j,s}(z,w)e^{m%
\varphi_j(w)+im\eta}\tau_j(w)\sigma_j(\eta),\ \ s=n,n-1,\ldots\\
&\hat \beta^+_{j,s}(x,y)=\chi_j(x)e^{-m\varphi_j(z)-im\theta}b^+_{j,s}(z,w)e^{m%
\varphi_j(w)+im\eta}\tau_j(w),\ \ s=n,n-1,\ldots\\
&A_{B_j,+}(t,z,w)=e^{-\frac{h_+(z,w)}{t}}K_{B_j,+}(t,z,w) \\
&K_{B_j,+}(t,z,w)\sim
t^{-n}b^+_{j,n}(z,w)+t^{-n+1}b^+_{j,n-1}(z,w)+\cdots+b^+_{j,0}(z,w)+tb^+_{j,-1}(z,w)+%
\cdots\ \ \mbox{as $t\To0^+$}.
\end{split}%
\end{equation}

Remark that these expressions, apart from cut-off functions, are 
mainly motivated by the formulas \eqref{e-gue150606III},
\eqref{e-gue150606IV} of Proposition~\ref{l-gue150606}.  Let $%
H_j(t)$ be the continuous operator associated with
$H_j(t, x, y)$, for which we put down the expression for later
use (cf. \eqref{e-gue150626fIIII} and \eqref{e-gue150628gI})
\begin{equation}  \label{e-gue150627fI}
\begin{split}
H_j(t):\Omega^{0,+}(X,E)&\rightarrow\Omega^{0,+}(X,E), \\
u&\rightarrow\int
\chi_j(x)e^{-m\varphi_j(z)-im\theta}A_{B_j,+}(t,z,w)e^{m\varphi_j(w)+im\eta}%
\tau_j(w)\sigma_j(\eta)u(y)dv_X(y).
\end{split}%
\end{equation}

Consider the patched up kernel (recall $Q_m$ is the projection on the $m$-th
Fourier component, cf. \eqref{e-gue150525dhII})
\begin{equation}  \label{e-gue150627fII}
\Gamma(t):=\sum^N_{j=1}H_j(t)\circ
Q_m:\Omega^{0,+}(X,E)\rightarrow\Omega^{0,+}(X,E)
\end{equation}
and let $\Gamma(t,x,y)\in C^\infty(\mathbb{R}_+\times X\times X,
(T^{*0,+}X\otimes E)\boxtimes (T^{*0,+}X\otimes E)^*)$
be the distribution kernel of $\Gamma(t)$.

For an explicit expression,
one sees (using $Q_m$ of \eqref{e-gue151024b}) that
\begin{equation}  \label{e-gue150626fIII}
\begin{split}
\Gamma(t,x,y)&=\frac{1}{2\pi}\sum^N_{j=1}\int^\pi_{-\pi}H_j(t,x,e^{-iu}\circ
y)e^{-imu}du \\
&=\frac{1}{2\pi}\sum^N_{j=1}\int^\pi_{-\pi}e^{-\frac{\hat
h_{j,+}(x,e^{-iu}\circ y)}{t}}\hat K_j(t,x,e^{-iu}\circ y)e^{-imu}du \\
&\sim t^{-n}a^+_n(t,x,y)+t^{-n+1}a^+_{n-1}(t,x,y)+\cdots\ \ \mbox{as
$t\To0^+$},
\end{split}%
\end{equation}
where we have written
\begin{equation}  \label{e-gue150901a}
\begin{split}
&a^+_s(t,x,y)\,(=a_{s,m}^+(t,x,y)) \,\,(\hat b^+_{j, s}=\hat b_{j, s, m}^+)\\
&=\frac{1}{2\pi}\sum^N_{j=1}\int^\pi_{-\pi}e^{-\frac{\hat
h_{j,+}(x,e^{-iu}\circ y)}{t}}\hat b^+_{j,s}(x,e^{-iu}\circ y)e^{-imu}du,\ \
s=n,n-1,n-2,\ldots.
\end{split}
\end{equation}

For the initial condition of $\Gamma(t, x, y)$, one has the following.

\begin{lem}
\label{l-gue150626f} 
\begin{equation*}
\lim_{t\rightarrow0^+}\Gamma(t)u=Q_mu\ \ (\mbox{on $\mathscr
D'(X,T^{*0,+}X\otimes E)$})
\end{equation*}
for $u\in\Omega^{0,+}(X, E)$.
\end{lem}

\begin{proof}
For $u\in\Omega^{0,+}(X,E)$, $%
Q_mu\in \Omega^{0,+}_m(X, E)$, $Q_mu|_{D_j}$ can be expressed
as $e^{-im\eta}v_j(w)$ for some $v_j(w)\in\Omega^{0,+}(U_j,E)$. With %
\eqref{e-gue150627fI} we find (note $A_{B_j, +}(t)=I$ as $t\to 0$ and $dv_{D_j}=dv_{U_j}d\eta$ 
by \eqref{e-gue22a1})
\begin{equation}
\label{e-gue150626fIIII}
\begin{split}
&\lim_{t\rightarrow0^+}H_j(t)Q_mu \\
&=\lim_{t\rightarrow0^+}\int\chi_j(x)e^{-m\varphi_j(z)-im%
\theta}A_{B_j,+}(t,z,w)e^{m\varphi_j(w)+im\eta}\tau_j(w)\sigma_j(\eta)e^{-im%
\eta}v_j(w)dv_{U_j}(w)
d\eta \\
&=\lim_{t\rightarrow0^+}\int\chi_j(x)e^{-m\varphi_j(z)-im%
\theta}A_{B_j,+}(t,z,w)e^{m\varphi_j(w)}\tau_j(w)v_j(w)dv_{U_j}(w) \\
&=\chi_j(x)e^{-m\varphi_j(z)-im\theta}e^{m\varphi_j(z)}\tau_j(z)v_j(z) \\
&=\chi_j e^{-im\theta}v_j=\chi_jQ_mu.
\end{split}%
\end{equation}
With the above, the lemma follows from \eqref{e-gue150627fII} and $\sum_j\chi_j=1$.  
\end{proof}

$\Gamma(t)$ satisfies an {\it adjoint type} heat equation 
{\it asymptotically} in the following sense (cf. \cite[p. 96]{Du11}).

\begin{lem}
\label{l-gue150628} We consider $\widetilde{\Box}^+_{b,m}\circ Q_m$ still
denoted by $\widetilde{\Box}^+_{b,m}$.   $\Gamma(t, x, y)$ satisfies 
\begin{equation*}
\Gamma^{\prime }(t)u+\Gamma(t)\widetilde{\Box}^+_{b,m}u=R(t)u,\ \ \forall
u\in\Omega^{0,+}(X,E),
\end{equation*}
where $R(t):\Omega^{0,+}(X,E)\rightarrow\Omega^{0,+}(X,E)$ is the continuous
operator with distribution kernel $R(t,x,y)$ ($\in C^\infty(\mathbb{R}_+\times
X\times X, (T^{*0,+}X\otimes E)\boxtimes(T^{*0,+}X\otimes E)^*)$) 
which satisfies the following.  For every $\ell\in\mathbb{N}_0$, there exists an
$\varepsilon_0>0$, $C_\ell>0$ independent of $t$ such that
\begin{equation}  \label{e-gue150628g}
\left\Vert R(t,x,y)\right\Vert_{C^\ell(X\times X)}\leq C_{\ell}e^{-\frac{%
\varepsilon_0}{t}},\ \ \forall \,t\in\mathbb{R}_+.
\end{equation}
\end{lem}

\begin{proof}
As in the preceding lemma let $u\in\Omega^{0,+}_m(X,E)$ and write $%
u=e^{-im\eta}v_j(w)$ for some $v_j(w)\in\Omega^{0,+}(U_j,E)$ on $D_j$.  By this, 
\eqref{e-gue150606III}, \eqref{e-gue150607ab} and \eqref{e-gue150627fI} it is a bit tedious but straightforward 
to compute 
\begin{equation}  \label{e-gue150628gI}
\begin{split}
&H^{\prime }_j(t)u+H_j(t)\widetilde{\Box}^+_{b,m}u \\
&=\int\chi_j(x)e^{-m\varphi_j(z)-im\theta}A^{\prime
}_{B_j,+}(t,z,w)e^{m\varphi_j(w)+im\eta}\tau_j(w)\sigma_j(\eta)u(y)dv_X(y) \\
&\quad+\int\chi_j(x)e^{-m\varphi_j(z)-im\theta}A_{B_j,+}(t,z,w)e^{m%
\varphi_j(w)+im\eta}\tau_j(w)\sigma_j(\eta)(\widetilde{\Box}%
^+_{b,m}u)(y)dv_X(y) \\
&=\int\chi_j(x)e^{-m\varphi_j(z)-im\theta}A^{\prime
}_{B_j,+}(t,z,w)\tau_j(w)e^{m\varphi_j(w)}v_j(w)dv_{U_j}(w) \\
&\quad+\int\chi_j(x)e^{-m\varphi_j(z)-im\theta}A_{B_j,+}(t,z,w)\tau_j(w)(%
\widetilde{\Box}^+_{B_j,m}(e^{m\varphi_j}v_j))(w)dv_{U_j}(w) \\
&=\int\chi_j(x)e^{-m\varphi_j(z)-im\theta}A^{\prime
}_{B_j,+}(t,z,w)\tau_j(w)e^{m\varphi_j(w)}v_j(w)dv_{U_j}(w) \\
&\quad+\int\chi_j(x)e^{-m\varphi_j(z)-im\theta}A_{B_j,+}(t,z,w)(\widetilde{%
\Box}^+_{B_j,m}(\tau_je^{m\varphi_j}v_j))(w)dv_{U_j}(w) \\
&\quad+\int \chi_j (x)S_j(t,x,w)v_j(w)dv_{U_j}(w) \\
&=\int \chi_j(x)S_j(t,x,w)v_j(w)dv_{U_j}(w)=\int
\chi_j(x)S_j(t,x,w)e^{im\eta}\sigma_j(\eta)u(y)dv_X(y),
\end{split}%
\end{equation}
for some $S_j(t,x,w)\in C^\infty_0(\mathbb{R}_+\times D_j\times U_j,
(T^{*0,+}X\otimes E)\boxtimes(T^{*0,+}X\otimes E)^*)$. 

Note $\tau_j(z)=1$ for $%
(z,\theta)$ in some small neighborhood of $\mathrm{Supp\,}\chi_j$.
One sees that $%
S_j(t,x,w)=0$ if $(x,w)$ is in some small neighborhood of $(z,z)$.
Hence by using \eqref{e-gue150607abI} for \eqref{e-gue150628gI}
(on $|z-w|$ away from zero),
we conclude that for every $\ell\in%
\mathbb{N}_0$, there is an $\varepsilon>0$ independent of $t$ such that
\begin{equation}  \label{e-gue150628gII}
\left\Vert S_j(t,x,w)\right\Vert_{C^\ell(X\times X)}\leq C_{\ell}e^{-\frac{%
\varepsilon}{t}},\ \ \forall t\in\mathbb{R}_+.
\end{equation}

Put $\widetilde R(t,x,y):=\sum^N_{j=1}\chi_jS_j(t,x,w)e^{im\eta}\sigma_j(\eta)$\, $(\in
C^\infty(\mathbb{R}_+\times X\times X, (T^{*0,+}X\otimes E)\boxtimes (T^{*0,+}X\otimes E)^*))$
and set
\begin{equation*}
R(t,x,y)=\frac{1}{2\pi}\sum^N_{j=1}\int^\pi_{-\pi}\widetilde
R(t,x,e^{-iu}\circ y)e^{-imu}du.
\end{equation*}
Let $R(t):\Omega^{0,+}(X,E)\rightarrow\Omega^{0,+}(X,E)$ be the continuous
operator with distribution kernel $R(t,x,y)$. Note that
$R(t)=R(t)\circ Q_m$ (cf. \eqref{e-gue150626fIII} and \eqref{e-gue150627fII}).   By \eqref{e-gue150628gII}
$R(t,x,y)$ satisfies \eqref{e-gue150628g} and
by \eqref{e-gue150628gI} one sees $\Gamma^{\prime }(t)u+\Gamma(t)%
\widetilde{\Box}^+_{b,m}u=R(t)u$, $\forall u\in\Omega^{0,+}(X,E)$. The
lemma follows.
\end{proof}

To get back to the original heat equation from
its adjoint version, it suffices to take the adjoints
$\Gamma^*(t)$ of $\Gamma(t)$ and $%
R^*(t)$ of $R(t)$ (with respect to $(\,\cdot\,|\,\cdot\,)_E$)
because $\widetilde\Box^+_{b,m}$ is self-adjoint.  
Hence combining Lemma~\ref{l-gue150626f} and Lemma~\ref{l-gue150628}
one obtains the following (asymptotic) heat kernel.

\begin{thm}
\label{t-gue159628fg} With the notations above, we have
\begin{equation*}
\lim_{t\rightarrow0^+}\Gamma^*(t)u=Q_mu\ \ \mbox{on $\mathscr
D'(X,T^{*0,+}X\otimes E)$}
\end{equation*}
for $u\in\Omega^{0,+}(X,T^{*0,+}X\otimes E)$ and  ($\widetilde{\Box}^+_{b,m}\circ Q_m$ still
denoted by $\widetilde{\Box}^+_{b,m}$ below, which is self-adjoint) 
\begin{equation*}
\frac{\partial\Gamma^*(t)}{\partial t}u+\widetilde{\Box}^+_{b,m}%
\Gamma^*(t)u=R^*(t)u,\ \ \forall u\in\Omega^{0,+}(X,E)
\end{equation*}
where $R^*(t)$ is the
continuous operator with the distribution kernel $R^*(t,x,y)$ 
satisfying the estimate similar to Lemma~\ref{l-gue150628}.  
\end{thm}

Based on the above theorem, one way to solving our heat equation resorts to the method of {\it successive approximation}.
This part of reasoning is basically standard.   But because of the important role played by $Q_m^+$
in the final result (cf. {\it McKean-Singer (II)} in Corollary~\ref{t-gue150630w}), for the convenience of the reader
we sketch some details and refer the full details to, e.g. \cite[Section 2.4]{BGV92}. 

To start with, suppose $A(t)$,
$B(t)$ and $C(t)$\, $:\Omega^{0,+}(X,E)\rightarrow\Omega^{0,+}(X,E)$ are continuous operators
with distribution kernels $A(t,x,y)$, $B(t,x,y)$ and $C(t, x, y)$\, $(\in C^\infty(\mathbb{R}_+\times
X\times X, (T^{*0,+}X\otimes E)\boxtimes(T^{*0,+}X\otimes E)^*))$.
Define the (continuous) operator $%
(A\sharp B)(t):\Omega^{0,+}(X,E)\rightarrow\Omega^{0,+}(X,E)$ with distribution kernel
\begin{equation}\label{e-gue111}
\begin{split}
&(A\sharp B)(t,x,y):=\int^t_0\int_XA(t-s,x,z)B(s,z,y)dv_X(z)ds \\
&(\in C^\infty(\mathbb{R}_+\times X\times X, (T^{*0,+}X\otimes E)\boxtimes(T^{*0,+}X\otimes E)^*)).
\end{split}%
\end{equation}
It is standard that $((A\sharp B)\sharp C)(t)=(A\sharp (B\sharp
C))(t)$, denoted in common by 
\begin{equation*}
(A\sharp B\sharp C)(t). 
\end{equation*}
(The generalization to more operators is similar.)

The method of successive approximation results in a
solution (which is actually unique by Theorem~\ref{t-gue150630I} below) 
to our heat equation, as follows.  

\begin{prop}
\label{t-gue150630} i) {\rm (Existence)} Fix $\ell\in\mathbb{N}$, $\ell\geq2$.  There is an $%
\epsilon>0$ such that the sequence
\begin{equation}  \label{e-gue150630f}
\Lambda(t):=\Gamma^*(t)-(\Gamma^*\sharp R^*)(t)+(\Gamma^*\sharp R^*\sharp
R^*)(t)-\cdots
\end{equation}
converges in $C^\ell((0,\epsilon)\times X\times X, (T^{*0,+}X\otimes E)\boxtimes(T^{*0,+}X\otimes E)^*)$ 
and
\begin{equation}  \label{e-gue150630a}
\begin{split}
&\Lambda(t):\Omega^{0,+}(X,E)\rightarrow C^\ell(X,T^{*0,+}X\otimes E)\bigcap
L^{2,+}_m(X,E), \\
&\lim_{t\rightarrow0^+}\Lambda(t)u=Q_mu\ \ \mbox{on $\mathscr
D'(X,T^{*0,+}X\otimes E)$},\ \ \forall u\in\Omega^{0,+}(X,E), \\
&\Lambda^{\prime }(t)u+\widetilde{\Box}^+_{b,m}\Lambda(t)u=0,\ \ \forall
u\in\Omega^{0,+}(X,E).
\end{split}%
\end{equation}

ii) {\rm (Approximation)} Write $\Lambda(t,x,y)$, $(\in C^\ell((0,\epsilon)\times X\times X,
(T^{*0,+}X\otimes E)\boxtimes(T^{*0,+}X\otimes E)^*))$
for the distribution
kernel of $\Lambda(t)$. Then there exists an $\epsilon_0>0$ independent of $t$
such that
\begin{equation}  \label{e-gue150630}
\left\Vert \Lambda(t,x,y)-\Gamma^*(t,x,y)\right\Vert_{C^\ell(X\times X)}\leq
e^{-\frac{\epsilon_0}{t}},\ \ \forall\, t\in(0,\epsilon_0).
\end{equation}

With $\widetilde{\Box}^-_{b,m}$ in place of $\widetilde{\Box}^+_{b,m}$ in \eqref{e-gue150630a},
the corresponding statements for $\widetilde{\Box}^-_{b,m}$ (with $\Lambda_-, \Gamma_-$ etc.) 
hold true as well. 
\end{prop}

\begin{proof}
We sketche a proof of ii) and comment on i).  
For notational convenience
we set $Z=R^*$, $Z^2=R^*\sharp R^*$, $Z^3=R^*\sharp R^*\sharp R^*$ etc. as defined in 
\eqref{e-gue111} with
$||\cdot||_\ell$ as the $C^\ell$-norm on $X\times X$.
By using \eqref{e-gue150628g} (for $R^*$) one sees that there are $1>\delta_0, \delta_1>0$ such that
for all $t\in(0,\delta_0)$,
\begin{equation}  \label{e-gue150630IIy}
\left\Vert Z\right\Vert_{\ell}\leq\frac{1%
}{2}e^{-\frac{\delta_1}{t}},\ \ \left\Vert  Z^2\right\Vert_{\ell}\leq\frac{1}{2^2}e^{-\frac{\delta_1%
}{t}},\cdots.
\end{equation}
Similarly from the estimate of $\Gamma^*(t)$ (see \eqref{e-gue150607abI}) with the above
\eqref{e-gue150630IIy} we conclude that for all $t\in(0,\delta_0)$,
\begin{equation}  \label{e-gue150630II}
\begin{split}
&\left\Vert \Gamma^*\sharp Z\right\Vert_{\ell}\leq\frac{C_1}{2}e^{-\frac{\delta_1}{t}},\ \ \left\Vert \Gamma^*\sharp Z^2\right\Vert_{\ell}\leq\frac{C_1}{%
2^2}e^{-\frac{\delta_1}{t}},\cdots
\end{split}%
\end{equation}
where $C_1>0$ is some constant.  Hence the sequence \eqref{e-gue150630f} converges (in $C^\ell((0,\epsilon)%
\times X\times X,(T^{*0,+}X\otimes E)\boxtimes(T^{*0,+}X\otimes E)^*)$)
and  \eqref{e-gue150630} holds.

It takes slightly more work to verify \eqref{e-gue150630a} in i).   Let $q^k(t, x, y)$ 
be the $(k+1)$-th term in \eqref{e-gue150630f}.   One verifies
directly by computation of the convolution that 
\begin{equation}
\label{e-gue150630IIe1}
\partial_t q^k(t, x, y)+\widetilde{\Box}^+_{b,m}q^k(t, x, y)=
Z^k(t, x, y)+Z^{k+1}(t, x, y)
\end{equation}
(cf. \cite[(2) of Lemma 2.22]{BGV92}).  
Since $\Lambda(t, x, y)$ is the alternating sum of these $q^k$, 
by the good estimates \eqref{e-gue150630IIy} and \eqref{e-gue150630II}, 
one interchanges the order of the action of $(\partial_t+\widetilde{\Box}^+_{b,m})$ on $\Lambda(t, x, y)$ 
with the summation.  By telescoping with \eqref{e-gue150630IIe1}, one finds that the heat equation \eqref{e-gue150630a} of i) 
is satisfied (cf. \cite[Theorem 2.23]{BGV92}).  
\end{proof}

The uniqueness part of the above theorem is included in the following.
(Note $e^{-t\widetilde{\Box}^+_{b,m}}(x,y)$ is as in \eqref{e-gue150525fbI}.)

\begin{thm}
\label{t-gue150630I} i) {\rm (Uniqueness)}
We have $e^{-t\widetilde{\Box}%
^+_{b,m}}(x,y)=\Lambda(t,x,y)$ ($\in C^\ell((0,\epsilon)\times X\times X,
(T^{*0,+}X\otimes E)\boxtimes(T^{*0,+}X\otimes E)^*))$.  
Hence by \eqref{e-gue150630}, for every $\ell \in\mathbb{N}_0$ there exist $\epsilon_0>0$
and $\epsilon>0$ (independent of $t$) such that
\begin{equation}  \label{e-gue150630g}
\left\Vert e^{-t\widetilde{\Box}^+_{b,m}}(x,y)-\Gamma(t,x,y)\right%
\Vert_{C^\ell(X\times X)}\leq e^{-\frac{\epsilon_0}{t}},\ \ \forall
t\in(0,\epsilon).
\end{equation}
As a consequence  $e^{-t\widetilde{\Box}^+_{b,m}}(x,y)$ and $\Gamma(t,x,y)$ 
are the same in the sense of asymptotic expansion
(as defined in Definition~\ref{d-gue150608}).  

ii) {\rm (Asymptotic expansion)} More explicitly one has (cf. \eqref{e-gue150626fIII})
\begin{equation}  \label{e-gue150630gI}
\begin{split}
&e^{-t\widetilde{\Box}^+_{b,m}}(x,y)\sim
t^{-n}a^+_n(t,x,y)+t^{-n+1}a^+_{n-1}(t,x,y)+%
\cdots+a^+_0(t,x,y)+ta^+_{-1}(t,x,y)+\cdots\ \ \mbox{as $t\To0^+$}, \\
&a^+_s(t,x,y)\,(=a_{s,m}^+(t,x,y))=\frac{1}{2\pi}\sum^N_{j=1}\int^\pi_{-\pi}e^{-\frac{\hat
h_{j,+}(x,e^{-iu}\circ y)}{t}}\hat b^+_{j,s}(x,e^{-iu}\circ y)e^{-imu}du \\
&(\in C^\infty(\mathbb{R}_+\times X\times X, (T^{*0,+}X\otimes
E)\boxtimes(T^{*0,+}X\otimes E)^*)), 
\ \ s=n,n-1,n-2,\ldots.  
\end{split}%
\end{equation}

Similar statements hold for the case of $e^{-t\widetilde{\Box}^-_{b,m}}(x,y)$ as well.
\end{thm}

\begin{proof}
The argument for the uniqueness part is standard.
To sketch it, there is the following trick (cf. \cite[Lemma 2.16] {BGV92})
in which we shall use heat equations for 
both kernels (cf. \eqref{e-gue150525dhIII} and \eqref{e-gue150630a}).
For $0<t<\epsilon
$ ($\epsilon$ as in Proposition~\ref{t-gue150630}) and $f, g\in \Omega^{0,+}(X, E)$
\begin{equation*}  \label{e-gue150630gh}
\begin{split}
0&=\int^t_0\frac{\partial}{\partial s}\Bigr((\,\Lambda(t-s)f\,|\,e^{-s%
\widetilde{\Box}^+_{b,m}}g\,)_E\Bigr)ds \\
&=(\,Q_mf\,|\,e^{-t\widetilde{\Box}^+_{b,m}}g\,)_E-(\,\Lambda(t)f\,|\,Q_mg%
\,)_E \\
&=(\,f\,|\,e^{-t\widetilde{\Box}^+_{b,m}}g\,)_E-(\,\Lambda(t)f\,|\,g\,)_E \\
&=(\,e^{-t\widetilde{\Box}^+_{b,m}}f\,|\,g\,)_E-(\,\Lambda(t)f\,|\,g\,)_E,
\end{split}%
\end{equation*}
proving that  $e^{-t\widetilde{\Box}%
^+_{b,m}}(x,y)=\Lambda(t,x,y)$.

The estimates \eqref{e-gue150630g} and \eqref{e-gue150630gI} follow from
\eqref{e-gue150630} and \eqref{e-gue150626fIII} ($e^{-t\widetilde{\Box}^+_{b,m}}$ is
self-adjoint).  
\end{proof}

Remark that by Proposition~\ref{l-gue150606} it was tempting to speculate that 
the heat kernel for (modified/$\mathrm{Spin}^c$) Kohn Laplacian
might be (at least asymptotically) the (local) heat kernel for (modified/$\mathrm{Spin}^c$) Kodaira Laplacian.
This is however too much to be true as suggested by the above Theorem~\ref{t-gue150630I} 
because the asymptotic expansion of (modified) Kohn Laplacian involves a nontrivial $t$-dependence in $a_s(t, x, y)$
(cf. Remark~\ref{r-gue150508I} and Remark~\ref{r-gue160414a}).  

We are ready to establish a link between our index and
the heat kernel density of (modified) Kodaira Laplacian.
For $a^+_\ell(t,x,x)$ in \eqref{e-gue150630gI}, define 
$\mathrm{Tr\,}a^+_\ell(t,x,x)=\sum^d_{s=1}\langle\,a^+_\ell(t,x,x)e_s(x)\,|\,%
\,e_s(x)\,\rangle_{E}$ as usual,  where $\{e_s(x)\}_s$ an othonormal frame (of $T^{*0,+}_xX\otimes E_x$).  
(Similar notation and definition apply to the case of $a_\ell^-(t,x, x)$. )

To sum up from Corollary~\ref{t-gue150603} and \eqref{e-gue150630gI},
there is a second form of {\it McKean-Singer} type formula for the
index in our case (cf. Corollary~\ref{t-gue150603} for the first form).  

\begin{cor}
\label{t-gue150630w} {\rm (McKean-Singer (II))}  We have
\begin{equation}  \label{e-gue150630wII}
\sum^n_{j=0}(-1)^j\mathrm{dim\,}H^j_{b,m}(X,E)=\lim_{t\rightarrow0^+}\int_X%
\sum^n_{\ell=0}t^{-\ell}\Bigr(\mathrm{Tr\,}a^+_\ell(t,x,x)-\mathrm{Tr\,}%
a^-_\ell(t,x,x)\Bigr)dv_X(x).
\end{equation}
\end{cor}

By this result we are now reduced to computing $a_s(t, x,x)$ in the following 
section. 

\bigskip

\bigskip

\begin{center}
\large{Part II: Proofs of main theorems}
\end{center}

\section{Proofs of Theorems~\protect\ref{t-gue160114} and \ref{t-gue160307}}

\label{s-gue150702}

We are in a position to prove the main results of this paper.   A new ingredient
is the notion of ``distance function" $\hat d$ (see \eqref{e-gue160327aIe1}
for the definition and Theorem~\ref{t-gue160413} for its property). This function naturally appears when we compute $a_s(t, x,x)$ 
in the form of an integral \eqref{e-gue150630gI}.    
In the remaining part of this section  we prove that this ``distance function"
is equivalent to the ordinary distance function at least in the strongly pseudoconvex case 
(Theorem~\ref{t-gue160413}).

Theorem~\ref{t-gue160114} is proved in Theorem~\ref{t-gue160123}, Remark~\ref{r-gue160123}, 
Corollary~\ref{t-gue160224} together with Theorem~\ref{t-gue150630I};
Theorem~\ref{t-gue160307} proved in Theorem~\ref{t-gue150701} and in \eqref{e-gue150701a}.

In the same notations as before recall that $X_{p_\ell}=\left\{x\in X;\, \mbox{the period of $x$ is $\frac{2\pi}{p_\ell}$}%
\right\}$, $1\le \ell\le k$ with $p_1|p_\ell$ (all $\ell$) and $p=p_1$.   $X_p$ is open and dense in
$X$.  See the discussion preceding Theorem~\ref{t-gue160114} for more detail.

Let $G_j(t,x,y)$ be as in \eqref{e-gue150627f}.  (Notations set up in \eqref{e-gue150627f}--\eqref{e-gue150901a}
will be useful in what follows.)   By the construction of $G_j(t,x,y)$,
it is clear that
\begin{equation}  \label{e-gue160224I}
\begin{split}
&\frac{1}{2\pi}\sum^N_{j=1}G_j(t,x,x)\sim
t^{-n}\alpha^+_n(x)+t^{-n+1}\alpha^+_{n-1}(x)+\cdots\ \ \mbox{as $t\To0^+$},
\\
&\alpha^+_s(x)\,\,={\frac{1}{2\pi}}\sum_{j=1}^{N}\hat \beta^+_{j, s}(x, y)|_{y=x}\,\in C^\infty(X,\mathrm{End\,}(T^{*0,+}X\otimes E)),\ \
s=n,n-1,\ldots.
\end{split}%
\end{equation}

$\alpha_s^+(x)$ are independent of choice of BRT trivialization charts
(in view of Remark~\ref{r-gue150607}).   It is perhaps instructive to think of these as 
the data of the asymptotic expansion associated with the ``underlying
Kodaira Laplacian" (cf. {\it loc. cit.} and Proposition~\ref{l-gue150606})
regardless of the existence of a genuine ``underlying space".

Recall the asymptotic expansions of $\Gamma(t, x, y)$ and
$e^{-t\widetilde\Box^+_{b,m}}(x,y)$ (they coincide by Theorem~\ref{t-gue150630I}),
in which we have $a^+_s(t,x,y)$\, ($\in C^\infty(\mathbb{R}_+\times X\times
X,(T^{*0,+}X\otimes E)\boxtimes (T^{*0,+}X\otimes E)^*)$), $
s=n,n-1,\ldots$, cf. \eqref{e-gue150630gI} or \eqref{e-gue150626fIII}.
By the construction, $\Gamma(t, x, y)$ 
and $a_s(t, x, y)$ of \eqref{e-gue150630gI} depend on the choice of BRT charts.  
(The authors do not know whether there exists a canonical choice of $a_s(t, x, y)$ 
in this respect.)  

We are now ready to give a proof of the following.

\begin{thm}
\label{t-gue160123}  (cf. Theorem~\ref{t-gue160114})
For every $N_0\in\mathbb{N}$ with $N_0\ge N_0(n)$ for some $N_0(n)$, 
there exist $\varepsilon_0>0$, $\delta>0$ and $C_{N_0}>0$
such that
\begin{equation}  \label{e-gue160114r}
\begin{split}
&\left\vert
\sum^{N_0}_{j=0}t^{-n+j}a^+_{n-j}(t,x,x)-\Bigl(\sum\limits^{p_r}_{s=1}e^{\frac{%
2\pi(s-1)}{p_r}mi}\Bigr)\sum^{N_0}_{j=0}t^{-n+j}\alpha^+_{n-j}(x)\right\vert \\
&\leq C_{N_0}\Bigr(t^{-n+N_0+1}+t^{-n}e^{\frac{-\varepsilon_0\hat d(x,X^r_{%
\mathrm{sing\,}})^2}{t}}\Bigr),\ \  \forall\,x\in X_{p_r}\, (r=1,\ldots,k),\,\,\forall \,0<t<\delta.
\end{split}%
\end{equation}
\end{thm}

\begin{proof}
For simplicity, we only prove Theorem~\ref{t-gue160123} for $r=1$. The proof
for $r>1$ is similar.

As in the beginning of Section~\ref{s-gue150627I},
there are BRT trivializations $B_j:=(D_j,(z,\theta),\varphi_j)$, $%
j=1,\ldots,N$.  We write
\begin{equation*}
D_j=U_j\times]-2\delta_j,2\widetilde\delta_j[,\\ \quad
\hat D_j=\hat U_j\times]-\frac{\delta_j}{2},\frac{%
\widetilde\delta_j}{2}[\quad (\subset\mathbb{C}^n\times\mathbb{R}),
\end{equation*}
with $U_j=\left\{z\in\mathbb{C}^n;\, \left\vert
z\right\vert<\gamma_j\right\},
\hat U_j=\left\{z\in\mathbb{C}^n;\,
\left\vert z\right\vert<{\gamma_j}/{2}\right\}$
for some $\delta_j>0$, $
\widetilde\delta_j>0$, $\gamma_j>0$.  

Assume $X=\hat
D_1\bigcup\cdots\bigcup\hat D_N$.
In the following we let $%
\delta_j=\widetilde\delta_j=\zeta$ (all $j$),
$\zeta$ satisfy \eqref{e-gue160327}
with $4\left\vert
\zeta\right\vert<\frac{2\pi}{p}$.

It is easily verified that there is an
$\hat\varepsilon_0>0$ such that ($d(\cdot, \cdot)$ the ordinary distance function on $X$)
\begin{equation}  \label{e-gue160215}
\begin{split}
&\hat\varepsilon_0d((z_1,\theta_1),(z_2,\theta_1))\leq\left\vert
z_1-z_2\right\vert\leq\frac{1}{\hat\varepsilon_0}d((z_1,\theta_1),(z_2,%
\theta_1)),\quad \forall \,(z_1,\theta_1),(z_2,\theta_1)\in  D_j, \\
&\hat\varepsilon_0d((z_1,\theta_1),(z_2,\theta_1))^2\leq h_{j,+}(z_1,z_2)\leq%
\frac{1}{\hat\varepsilon_0}d((z_1,\theta_1),(z_2,\theta_1))^2, \quad \forall\,
(z_1,\theta_1),(z_2,\theta_1)\in  D_j,
\end{split}%
\end{equation}
where $h_{j,+}(z,w)$ is as in \eqref{e-gue150608a}.

Recall the modified distance $\hat d$ which is defined in \eqref{e-gue160327aIe1}.
We are going to compare $\hat d$ with \eqref{e-gue160215}.

Fix $x_0\in X_p$.  Suppose $x_0\in \hat D_j$ for some $j=1,2,\ldots,N$ and 
also suppose $x_0=(z,0)$ on $D_j$.  

Some crucial remarks are in order.

i) For  $0\le |u|\le 2\zeta$ the action of $e^{-iu}$ on $x_0$ is only moving along
the ``angle" direction (due to the assumption that a BRT trivialization $D_j$ is valid here),
i.e. the $z$-coordinates of $x_0$ and $e^{-iu}\circ x_0$ are the same.

ii) Let a $u_0\in [2\zeta,\frac{2\pi}{p}-2\zeta]$ be given. 
Assume that the action by $e^{-iu_0}$ on $x_0$ still belongs to $\hat D_j$
with a coordinate $e^{-iu_0}\circ x_0=(\tilde z, \tilde\eta)$.
Then it could happen that $\tilde z\ne z$
because the orbit $\{e^{-iv}\circ x_0\}$ for ${2\zeta\le v\le \frac{2\pi}{p}-2\zeta}$ may partly lie outside of
$D_j$.  We will show \eqref{e-gue160327I} below that indeed $\tilde z\ne z$ in 
this case.  

Remark that the above ii) is basically the reason responsible for why 
the contribution of our distance function $\hat d$ enters, as seen shortly.  
The question about whether the condition $e^{-iu_0}\circ x_0\in\hat D_j$
of ii) is vacuous or not will be discussed below
(equivalent to whether $J$ below is an empty set or not).    

We shall now formulate the above ii) more precisely. 
If $x_0\in \hat D_j$, we claim the following. 
\begin{equation}  \label{e-gue160327I}
\begin{split}
&\mbox{Suppose \,\,$e^{-i\theta}\circ x_0=(\Td z,\Td\eta)$\,\, also belongs to $\hat D_j$ for some
$\theta\in[2\zeta,\frac{2\pi}{p}-2\zeta]$}.\\
&\quad \mbox{Then\quad $\abs{z-\Td z}\geq\hat\varepsilon_0\,\hat d(x_0,X_{{\rm sing\,}})\,(>0)$}%
.
\end{split}%
\end{equation}

\begin{proof}[proof of claim]
By $(\tilde z,\tilde\eta)\in \hat D_j$ one has $e^{i\tilde\eta}\circ(\tilde z, \tilde\eta)=(\tilde z, 0)$
equivalently $e^{-i\tilde\eta}\circ (\tilde z, 0)=(\tilde z, \tilde\eta)$ (by the above i) as $|\tilde\eta|\le \frac{\zeta}{2}$
here).  One sees (by \eqref{e-gue160215} and isometry of $S^1$ action for the first inequality below)
\begin{equation}  \label{e-gue160327II}
\begin{split}
\left\vert \widetilde
z-z\right\vert
\geq\hat\varepsilon_0\,d(e^{i\tilde\eta}\circ(\tilde z, \tilde\eta), e^{i\tilde\eta}\circ(z, \tilde\eta))
&=\hat\varepsilon_0\,d(e^{i\tilde\eta}\circ (e^{-i\theta}\circ x_0), x_0)\\
&\geq\hat\varepsilon_0\inf\left\{d(e^{-iu}\circ
e^{-i\theta}\circ x_0,x_0);\, \left\vert u\right\vert\leq\frac{\zeta}{2}\right\} \\
&\geq\hat\varepsilon_0\inf\left\{d(e^{-i\hat \theta}\circ x_0,x_0);\,
\zeta\leq\hat \theta\leq\frac{2\pi}{p}-\zeta\right\}\\
&=\hat\varepsilon_0\,\hat
d(x_0,X_{\mathrm{sing\,}})
\end{split}
\end{equation}
(see \eqref{e-gue160327aIe1} for the definition of $\hat d$), as claimed.
\end{proof}

Remark that a sharp result in this direction  \eqref{e-gue160327I} is proved in 
Lemma~\ref{l-gue160417lem2}.  

We continue with the proof of the theorem.
We need to estimate $\Gamma(t)=\sum_{j=1}^NH_j(t)\circ Q_m$ for the first summation
to the left of \eqref{e-gue160114r}.  By definition (see \eqref{e-gue150626fIII})
this is in turn to estimate
\begin{equation}
\label{e-gue160327aI}
\frac{1}{2\pi}\int^{\pi}_{-\pi}H_j(t,x_0,e^{-iu}\circ x_0)e^{-imu}du
\end{equation}
and sum over $j=1,\ldots, N$.


We first assume that in \eqref{e-gue160327aI},  $x_0=(z, 0)$ in $\hat D_j$ and $x_0\not\in \hat D_k$ for
any other $k\ne j$.  

To work on \eqref{e-gue160327aI} we shall divide $[-\pi, \pi]$ in \eqref{e-gue160327aI} into two types.

The first type is to estimate $\frac{1}{2\pi}\int^{2\zeta}_{-2\zeta}H_j(t,x_0,e^{-iu}\circ x_0)e^{-imu}du$.
Note if $u\in [-2\zeta, 2\zeta]$ and $e^{-iu}\circ x_0=(z_u, \theta_u)$, then
$(z_u, \theta_u)=(z, u)$ (by i) above \eqref{e-gue160327I}), i.e. $e^{-iu}\circ (z, 0)=(z, u)$.
Hence, by \eqref{e-gue150627f}
the factor $e^{im\eta}$ in $H_j$ is going 
to be $e^{imu}$ and this is cancelling off the term $e^{-imu}$ in the integral 
\eqref{e-gue160327aI}.
By \eqref{e-gue150608a}, $h_+(z, z_u)=h_+(z, z)=0$ (also by \eqref{e-gue160215}) and
the factor $e^{-\frac{h_+}{t}}$ of $A_{B_j,+}$ in $H_j$ of \eqref{e-gue150627f} becomes $1$.
Finally we note for $\sigma_j$ in $H_j$ of \eqref{e-gue150627f}, $\int_I\sigma_j(\theta)d\theta=1$, $I=[-\frac{\zeta}{2},\frac{\zeta}{2}]$.

To sum up,  by \eqref{e-gue150627f} and \eqref{e-gue150608a} one obtains the following ($x_0\not\in \hat D_k$ for $k\ne j$)
\begin{equation}  \label{e-gue160125VI}
\begin{split}
&\frac{1}{2\pi}\int^{2\zeta}_{-2\zeta}H_j(t,x_0,e^{-iu}\circ x_0)e^{-imu}du
\\
&\mbox{$\sim \Bigr(t^{-n}\alpha^+_{n}(x_0)+t^{-(n-1)}\alpha^+_{n-1}(x_0)+%
\cdots\Bigr)$ as $t\To 0^+$},
\end{split}%
\end{equation}
where $\alpha^+_s(x)$, $s=n,n-1,\ldots$, are as in \eqref{e-gue160224I}.

For the second type suppose $u\in [2\zeta,\frac{2\pi}{p}-2\zeta]$.
Note the action by $e^{-iu}$ on $x_0$ may change the $z$ coordinate of $x_0$
by ii) above \eqref{e-gue160327I}.   We let
$J$ be the subset of those $u\in [2\zeta,\frac{2\pi}{p}-2\zeta]\equiv E$ that
$e^{-iu}\circ x_0=(z_u, \theta_u)$ belongs to $\hat D_j$ (then $z_u\neq z$ by \eqref{e-gue160327I}),
and $J'$ be the complement of $J$ in $E$.  
One finds, for some $\varepsilon_0>0, 
\delta>0$ and $C_0>0$ (independent of $j$, $x_0$), that
\begin{equation}
\label{e-gue160125V}
\begin{split}
&\left\vert \frac{1}{2\pi}\int_{u\in[2\zeta,\frac{2\pi}{p}-2\zeta]%
}H_j(t,x_0,e^{-iu}\circ x_0)e^{-imu}du\right\vert \\
&\leq  \frac{1}{2\pi}\int_{u\in J%
}\left\vert H_j(t,x_0,e^{-iu}\circ x_0)e^{-imu}\right\vert du +
 \frac{1}{2\pi}\int_{u\in J'%
}\left\vert H_j(t,x_0,e^{-iu}\circ x_0)e^{-imu}\right\vert du
\\
&\leq C_0t^{-n}e^{\frac{%
-\varepsilon_0\hat d(x_0,X_{\mathrm{sing\,}})^2}{t}}, \ \ \forall\, 0<t<\delta
\end{split}
\end{equation}
where the integral over $J'$ vanishes
because the cut-off function $\sigma_j$ in $H_j$ of \eqref{e-gue150627f} 
gives $\sigma_j(\eta(y))=0$ for $y=e^{-iu}\circ x_0\not\in \hat D_j$ 
as $\sigma_j=0$ outside $\hat D_j$ (see lines above \eqref{e-gue150627f}), and the second inequality arises
from applying \eqref{e-gue160215} and \eqref{e-gue160327I} to
$h_+(z, z_u)$ in $H_j$ (see \eqref{e-gue150627f} and \eqref{e-gue150608a}).

Is  $J$ an empty set?  We remark that the top term in \eqref{e-gue160125V}
is in general nonzero (by combining \eqref{e-gue160125VI} and 
Remark~\ref{r-gue160414a} for $p=1$).  Hence $J\ne \emptyset$
in general.   There is a geometrical way to see the claim that for some open subset $V$ of $X$, if $x_0\in V$,
then $J\ne \emptyset$.  
For simplicity assume $X=X_1\bigcup X_2$, i.e. $p_1=1$ and $p_2=2$.  
Choose $y\in X_2$.  Let $g=e^{-i\frac{2\pi}{p_2}}\in S^1$.    Fix a  
neighborhood $U\subset \hat D_j$ of $y$ in $X$.  Since $g\circ y=y$, by continuity argument 
there are neighborhoods $N_1$, $N_2$ of 
$y$, $g$ in $X$, $S^1$ respectively such that the action $h\circ x\in U$ if $(h,x)\in N_2\times N_1$.  
Choose $N_1\subset \hat D_j$, $N_2$ small and set $V\equiv N_1\setminus X_2$.
It follows that for these $x_0\in V$, 
$J\ne\emptyset$ since $N_2\subset J$.   This result also accounts for the necessity of the remark ii) 
above \eqref{e-gue160327I} and hence that of a certain extra contribution (e.g. $\hat d$) in estimates \eqref{e-gue160125V}.  

Suppose $p\,(=p_1)=1$.  
Then 
 \eqref{e-gue160125VI} and \eqref{e-gue160125V}, Definition~\ref{d-gue150608} for $\sim$
and Remark~\ref{r-gue150608} (by noting 
$\mathrm{dim}\, X=2n+1$, $\mathrm{dim}\,U_j=2n$, $M_0(m)=n+1$, $m=4n=2\beta$, $M_1(m,\ell)=n$ for $\ell=0$) 
immediately lead to
\begin{equation}  \label{e-gue160215VII}
\begin{split}
&\left\vert \Gamma(t,x_0,x_0)-\sum^{N_0}_{j=0}t^{-n+j}\alpha^+_{n-j}(x_0)\right\vert \\
&\leq C_{N_0}\Bigr(t^{-n+N_0}+C_0t^{-n}e^{\frac{-\varepsilon_0\hat
d(x_0,X_{\mathrm{sing\,}})^2}{t}}\Bigr),
\ \ \,N_0\ge N_0(n), \,\,\forall \,0<t<\delta.
\end{split}%
\end{equation}
Now adding $t^{-n+N_0+1}\alpha^+_{n-N_0-1}$ to \eqref{e-gue160215VII} and substracting 
it we improve $t^{-n+N_0}$ by $t^{-n+N_0+1}$.  Hence the estimate \eqref{e-gue160114r} of the theorem for $p=1$.

Suppose $p>1$.   Then one has the extra $p-1$ sectors in $[-\pi,\pi]$
(obtained by shifting the above first sector $s=1$ by a common $(s-1)\frac{2\pi}{p}$):
\begin{equation}
\label{e-gue160125v8}
[(s-1)\frac{2\pi}{p}-2\zeta, (s-1)\frac{2\pi}{p}+2\zeta], \,\,[(s-1)\frac{2\pi}{p}+2\zeta, s\frac{2\pi}{p}-2\zeta],
\quad s=1,\ldots, p
\end{equation}
($s=p+1$ identified with $s=1$) over which the integrals correspond to types I \eqref{e-gue160125VI} 
and II  \eqref{e-gue160125V} respectively.    One may check without difficulty that the version of the claim
\eqref{e-gue160327I} adapted to these sectors holds true as well. 
On each of these sectors, a simple (linear)  change of variable for $u$, which is to bring the intervals of the integration on 
these sectors back to those in \eqref{e-gue160125VI} and \eqref{e-gue160125V},
produces the extra numerical factor in sum (by $e^{-imu}du$ in \eqref{e-gue160327aI}) :
$ \sum\limits^{p}_{s=1}e^{\frac{2\pi(s-1)}{%
p}mi}$ as expressed in  \eqref{e-gue160114r}.

Finally, note that we have assumed $x_0=(z, 0)\in \hat D_j$.   In this case the above argument appears 
{\it symmetrical} in writing.    This (assumption) is however not essential.   Since we shall also adopt a similar 
assumption in Section~\ref{s-gue160416}, we give an outline about the 
{\it asymmetrical} way (i.e. $x_0=(z, \theta)$, $\theta\ne 0$) of the argument.  
By going over the same process, one sees the following.  
i) If $x=(z, v)$, with $0<v\le \frac{\zeta}{2}$, 
the intervals in \eqref{e-gue160125VI},
\eqref{e-gue160125V} shall be replaced by $[-2\zeta-v, 2\zeta-v]$, $[2\zeta-v, \frac{2\pi}{p}-2\zeta-v]$
(thought of as translated by a common $-v$) with the new integrals denoted by  \eqref{e-gue160125VI}',
\eqref{e-gue160125V}', respectively; 
ii) $[-2\zeta-v, 2\zeta-v]\supseteq [-\zeta,\zeta]$ hence $\int \sigma(u) du$ is still $1$ in \eqref{e-gue160125VI}' ;
iii) In the proof of claim \eqref{e-gue160327II}, $e^{i\tilde\eta}$ should be replaced by $e^{i\gamma}$ 
with $\gamma=\tilde\eta-v$, 
$(\tilde z, 0)$ by $(\tilde z, v)$ and $\theta\in  [2\zeta, \frac{2\pi}{p}-2\zeta]$ by $\theta\in  [2\zeta-v, \frac{2\pi}{p}-2\zeta-v]$
throughout \eqref{e-gue160327I} and \eqref{e-gue160327II}.  One can check that the the reasoning in \eqref{e-gue160327II} remains basically unchanged, and 
the conclusion of \eqref{e-gue160327II} holds true as well in this modified case;
iv) By the preceding ii) and iii), the results corresponding to 
\eqref{e-gue160125VI}' and \eqref{e-gue160125V}' hold true.  
Hence the asymmetrical way follows.  

We have also assumed $x_0\not\in \hat 
D_k$ for $k\ne j$.   This condition is unimportant if we take the preceding asymmetrical 
way into account (for $x_0\in\hat D_k$, $k\ne j$ in the general case), 
and note that there is a hidden partition of unity in $\{H_j\}_{j=1,\ldots,N}$.  

For an alternative to the above, it is to use the kernel 
$e^{-t\widetilde\Box_{b,m}^+}(x, y)$ in place of $\Gamma(t, x, y)$
and $a^+_s(t, x, y)$.    An advantage is that  $e^{-t\widetilde\Box_{b,m}^+}(x, y)$ is independent of BRT charts, 
so that for a given point $x_0$ we 
can take a covering of $X$ by convenient BRT charts for the previous 
special conditions to be satisfied (e.g. $x_0=(z, 0)$, $x_0\in \hat D_j$ for exactly one $j$ etc.).
By the asymptotic property between  $e^{-t\widetilde\Box_{b,m}^+}(x, y)$
and $\Gamma(t, x, y)$ \eqref{e-gue150630g}, this also leads to Theorem~\ref{t-gue160123}. 
\end{proof}

\begin{rem}
\label{r-gue160123} For the relation between $a^+_s(t,x,y)|_{y=x}$ and $\alpha_s(x)$ 
as stated in \eqref{e-gue160114} of Theorem~\ref{t-gue160114}, the method of the above proof 
works.   By using (setting $y=x$ below) 
\begin{equation}\label{e-gue160125re2}
a^+_s(t,x,y)=\frac{1}{2\pi}\sum^N_{j=1}\int^\pi_{-\pi}e^{-\frac{\hat
h_{j,+}(x,e^{-iu}\circ y)}{t}}\hat b^+_{j,s}(x,e^{-iu}\circ y)e^{-imu}du
\end{equation} (see \eqref{e-gue150630gI}) 
with the same reasoning as \eqref{e-gue160125VI}, \eqref{e-gue160125V} and \eqref{e-gue160125v8}, 
one obtains \eqref{e-gue160114} for $P_\ell=\mathrm{id}$.   For $\ell>0$,  \eqref{e-gue160114} follows by noting 
\begin{equation}\label{e-gue160123re1}
\partial_x e^{-\frac{x^2}{t}}
=-2t^{-\frac{1}{2}}(\frac{x^2}{t})^{1/2}e^{-\frac{x^2}{t}}=O(t^{-\frac{1}{2}})
\end{equation}
applied to \eqref{e-gue160125re2} and by noting Remark~\ref{r-gue150608} as in \eqref{e-gue160215VII}.   
Remark that if one extracts the corresponding coefficients of $t^{-s}$
in \eqref{e-gue160114r} of Theorem~\ref{t-gue160123} and uses the result \eqref{e-gue160114r}, 
the estimate appears to be $e^{-\frac{\varepsilon_0\hat d(x, X_{\mathrm{sing}})}{t}}+O(t^{\infty})$ 
which is slightly weaker than above (due to $O(t^{\infty})$).   
\end{rem}

For similar estimates with regard to $C^l$ topology we have the following (cf. \eqref{e-gue160123re1} and 
Remark~\ref{r-gue150608}).

\begin{cor}
\label{t-gue160224} In the same notation as above,
for any differential operator $P_\ell:C^\infty(X,T^{*0,+}X\otimes E)\rightarrow
C^\infty(X,T^{*0,+}X\otimes E)$ of order $\ell\in\mathbb{N}$ and every $N_0\ge N_0(n)$
for some $N_0(n)$, 
\begin{equation}
\begin{split}
&\left\vert P_\ell\Bigr(\sum^{N_0}_{j=0}t^{-n+j}a^+_{n-j}(t,x,x)-\bigl(\sum%
\limits^{p_r}_{s=1}e^{\frac{2\pi(s-1)}{p_r}mi}\bigr)\sum^{N_0}_{j=0}t^{-n+j}%
\alpha^+_{n-j}(x)\Bigr)\right\vert \\
&\leq C_{N_0}\Bigr(t^{-n+N_0+1-\frac{\ell}{2}}+t^{-n-\frac{\ell}{2}}e^{\frac{-\varepsilon_0\hat d(x,X^r_{%
\mathrm{sing\,}})^2}{t}}\Bigr),\ \ \forall \,0<t<\delta,\ \ \forall\, x\in
X_{p_r}
\end{split}%
\end{equation}
for some $\varepsilon_0>0$, $\delta>0$ and $C_{N_0}>0$ independent of $x$.
\end{cor}

Note the singular behavior $t^{-n}$ (in the term to the rightmost of \eqref{e-gue160114r}). 
So the estimate \eqref{e-gue160114r} is not directly applicable to 
the proof of our local index theorem.  That is, computation
involving $a_s$ cannot be automatically reduced to computation
involving $\alpha_s$ as soon as $x$\, $(\in X_{p_1})$ approaches $X_{\mathrm{sing}}$.
Intuitively $t^{-n}e^{\frac{-\varepsilon_0\hat d(x,X_{\mathrm{sing\,}})^2}{t}}$
goes to a kind of Dirac delta function (along $X_{\mathrm{sing}}$) as $t\to 0$ (apart from
a factor of the form $\frac{1}{t^\beta}$, some $\beta>0$).  
So after integrating \eqref{e-gue160114r}
over $X$, a nonzero contribution due to this term could appear or even blow up as $t\to 0$.  
A more precise analysis along this line will be taken up in the study 
of {\it trace integrals} in Section~\ref{s-gue160416}.  

Fortunately, the abovementioned singular behavior can be removed ($t^{-n}$ dropping out completely) 
after taking the {\it supertrace},
so that the index density for our need does exist.
(However, as far as the full kernel is concerned, a certain
estimate such as that in Theorem~\ref{t-gue160123} is unavoidable as discussed in Remark~\ref{r-gue160414ar}).

We shall now take up this improvement on \eqref{e-gue160114r} under supertrace.  
We formulate it as follows, whose proof is heavily based on 
the off-diagonal estimate obtained in Theorem~\ref{t-gue150627g}.  

\begin{thm} (cf. Theorem~\ref{t-gue160307}) 
\label{t-gue150701} With the notations above, for every $N_0\in\mathbb{N}, \, N_0\ge N_0(n)$ 
for some $N_0(n)$, there exist $\varepsilon_0>0$, $\delta>0$ and $C_{N_0}>0$ such that
\begin{equation}  \label{e-gue160226}
\begin{split}
&\left\vert \mathrm{Tr\,}e^{-t\widetilde\Box^+_{b,m}}(x,x)-\mathrm{Tr\,}%
e^{-t\widetilde\Box^-_{b,m}}(x,x)-\Big(\sum\limits^{p_r}_{s=1}e^{\frac{2\pi(s-1)}{%
p_r}mi}\Big)\sum^{N_0}_{j=0}t^{-n+j}\Bigr(\mathrm{Tr\,}\alpha^+_{n-j}(x)-\mathrm{%
Tr\,}\alpha^-_{n-j}(x)\Bigr)\right\vert \\
&\leq C_{N_0}\Bigr(t^{-n+N_0+1}+e^{-\frac{\varepsilon_0\hat d(x,X^r_{\mathrm{%
sing\,}})^2}{t}}\Bigr),\ \ \forall \,0<t<\delta,\ \ \forall x\in X_{p_r},
\end{split}%
\end{equation}
\end{thm}

The implication of Theorem~\ref{t-gue150701} yields a link between the two identities arising from
Corollary~\ref{t-gue150630w} and Proposition~\ref{t-gue150627} together with \eqref{e-gue150627}:
\begin{equation}\label{e-gue150701a}
\begin{split}
&\lim_{t\rightarrow0^+}\int_X\sum^n_{\ell=0}t^{-\ell}\Bigr(\mathrm{Tr\,}a^+_\ell(t,x,x)-\mathrm{Tr\,}%
a^-_\ell(t,x,x)\Bigr)dv_X(x)
=\sum^n_{j=0}(-1)^j\mathrm{dim\,}H^j_{b,m}(X,E)\\
&\lim_{t\to 0^+}\int_X\sum^n_{\ell=0}t^{-\ell}\Bigr(\mathrm{Tr\,}\alpha^+_\ell(x)-\mathrm{Tr\,}
\alpha^-_\ell(x)\Bigr)dv_X(x)\\
&\qquad\qquad=\frac{1}{2\pi}\int_X\mathrm{%
Td_b\,}(\nabla^{T^{1,0}X},T^{1,0}X)\wedge\mathrm{ch_b\,}(\nabla^{E},E)\wedge
e^{-m\frac{d\omega_0}{2\pi}}\wedge\omega_0(x).
\end{split}
\end{equation}
It follows that the two in \eqref{e-gue150701a} are equal
because in \eqref{e-gue160226}, $e^{-\frac{\varepsilon_0\hat d(x,X_{\mathrm{sing\,}})^2}{t}}\,(\le 1)\to 0$
in $L^1$ by Lebesgue's dominated convergence theorem as $t\to 0^+$ on $X_p$. We arrive now at an index theorem for our class of CR manifolds. 

\begin{cor} (cf. Corollary~\ref{c-gue150508I})
\label{t-gue160226a}
\begin{equation}  \label{e-gue160226I}
\begin{split}
&\sum^n_{j=0}(-1)^j\mathrm{dim\,}H^j_{b,m}(X,E)\\
&=\Bigl(\sum%
\limits^{p}_{s=1}e^{\frac{2\pi(s-1)}{p}mi}\Bigr)\frac{1}{2\pi}\int_X\left[\mathrm{%
Td_b\,}(\nabla^{T^{1,0}X},T^{1,0}X)\wedge\mathrm{ch_b\,}(\nabla^{E},E)\wedge
e^{-m\frac{d\omega_0}{2\pi}}\wedge\omega_0\right]_{2n+1}(x)
\end{split}%
\end{equation}
where $[\cdots]|_{2n+1}$ denotes the $(2n+1)$-form part.
\end{cor}

We turn now to the proof of Theorem~\ref{t-gue150701}.

\begin{proof}[proof of Theorem~\ref{t-gue150701}]
For simplicity, we only prove Theorem~\ref{t-gue150701} for $r=1$. The proof
for $r>1$ is similar.   Adopting the same notations as
in the proof of Theorem~\ref{t-gue160123} (e.g. $B_j$, $D_j$, $\hat D_j\cdots$), 
we shall follow a similar line of thought as in Theorem~\ref{t-gue160123}.

Fix $x_0\in X_p$.  As $e^{-t\widetilde\Box^+_{b,m}}(x, y)$ and 
$\Gamma(t, x, y)$ are asymptotically the same (Theorem~\ref{t-gue150630I}),
we also break the desired estimate at $x=x_0$ into two types of integrals 
corresponding to \eqref{e-gue160125VI} and \eqref{e-gue160125V}.  

One integral is over $I=[-2\zeta,2\zeta]$ and the other over $I'$, the complement of 
$I$ in $[-\pi,\pi]$.   The first type gives rise to
the first term to the right of \eqref{e-gue160226} almost the same way as \eqref{e-gue160125VI}.

The key of this proof lies in the second type which corresponds to \eqref{e-gue160125V}.
It is estimated over $I'$,
as in \eqref{e-gue160226e1} below.  (Here we rewrite $H_j$ in a convenient form, in terms of 
$\hat h_{j,+}$, $\hat K_{j,+}$ of \eqref{e-gue150901},
reminiscent of an analogous relation $A_{B,+}=e^{-\frac{h_+}{t}}K_{B,+}$ in \eqref{e-gue150608a}.) 
\begin{equation}
\label{e-gue160226e1}
\frac{1}{2\pi}\sum^N_{j=1}\left\vert \int_{u\in I'}
e^{-\frac{\hat h_{j,+}(x_0,e^{-iu}\circ x_0)}{t}}\Bigr(\mathrm{Tr\,}%
\hat K_{j,+}(t,x_0,e^{-iu}\circ x_0)-\mathrm{Tr\,}\hat K_{j,-}(t,x_0,e^{-iu}\circ x_0)%
\Bigr)du\right\vert. 
\end{equation}


We shall now show that there exist $\varepsilon_0>0$ and $C>0$ (independent
of $x_0$) such that \eqref{e-gue160226e1} is bounded above by
\begin{equation}
\label{e-gue160226e3}
Ce^{-\frac{\varepsilon_0\hat d(x_0,X_{\mathrm{sing\,}})^2}{t}}
\end{equation}
for small $t\in\mathbb{R}_+$.  

To see this we first note that for $k\ge 0$,
\begin{equation}\label{e-gue160226e2}
e^{-\varepsilon {x^2\over t}}\big({x^{2}\over t}\big)^k\le C_{k,\varepsilon}
e^{-\varepsilon {x^2\over 2t}}
\end{equation}
for some constant $C_{k, \varepsilon}$ independent of $x$ and $t>0$.  
Write $x_0=(z, \theta)$ and $e^{-iu}\circ x_0=(z_u,\theta_u)$ in BRT coordinates. 
Since $\hat h_{j,+}(x_0, e^{iu}\circ x_0)$ is essentially $h_+(z, z_u)\approx |z-z_u|^2$, we have 
\begin{equation}
\label{e-gue160226e4}
e^{-\frac{\hat h(x_0, e^{iu}\circ x_0)}{t}}\le e^{-2c_1\frac{|z-z_u|^2}{t}}\le 
e^{-c_1\frac{\hat\varepsilon_0\hat d(x_0,X_{\mathrm{sing\,}})^2}{t}}e^{-c_1\frac{|z-z_u|^2}{t}}
\end{equation}
for some constant $c_1>0$ by using  \eqref{e-gue160327I} for $\hat d$. 
By using the off-diagonal estimate of Theorme~\ref{t-gue150627g} 
and by \eqref{e-gue150608a}, \eqref{e-gue150901} for linking $b_{\bullet}$ with $K_{\bullet}$, 
one obtains the following estimate from \eqref{e-gue160226e4} 
\begin{equation}
\label{e-gue160226e5}\begin{split}
&\left\vert e^{-\frac{\hat h_{j,+}(x_0,e^{-iu}\circ x_0)}{t}}\mathrm{Tr\,}
\hat K_{j,+}(t,x_0,e^{-iu}\circ x_0)-\mathrm{Tr\,}\hat K_{j,-}(t,x_0,e^{-iu}\circ x_0)\right\vert\qquad\qquad \\
&\qquad \qquad \le e^{-c_1\frac{\hat\varepsilon_0\hat d(x_0,X_{\mathrm{sing\,}})^2}{t}} 
\sum_{k=0}^{n} \mbox{\,constants}\,\,\cdot e^{-c_1{ |z-z_u|^2\over t}}\big({|z-z_u|^{2k}\over t^k}+O(t)\big)
\end{split}\end{equation}
for \eqref{e-gue160226e1}.  
Now one readily obtains the bound \eqref{e-gue160226e3} from  \eqref{e-gue160226e5} and
\eqref{e-gue160226e2}.  

Combining the above estimates for integrals of the first type and second type \eqref{e-gue160226e1}, we obtain 
\eqref{e-gue160226} in the way similar to \eqref{e-gue160215VII} (with $t^{-n}$ dropping out of 
$t^{-n}e^{-\frac{{\varepsilon_0\hat d}^2}{t}}$).  
\end{proof}

In the remaining part of this section we give a geometric meaning for $\hat d(x,X^r_{\mathrm{%
sing\,}})$ (when $X$ is strongly pseudoconvex).
To this aim it is useful to use another equivalent form of the function $\hat d$, as follows
(without any pseudoconvexity condition on $X$).  

\begin{lem}
\label{t-gue160413a}
There exists a small constant $\varepsilon_0>0$ (satisfying \eqref{e-gue160327} at least)
with the following property.   Fix an $\varepsilon$ with $0<\varepsilon\leq \varepsilon_0$.
For $x\in X$ define
another ``distance function" $\hat d_2$ by (for a fixed $\ell$)
\begin{equation*}
\hat d_2(x, X_{\mathrm{sing}}^{{\ell}-1})=\inf\left\{d(x,e^{-i\theta}\circ x);\,
\frac{2\pi}{p_{\ell}}-\varepsilon\leq\theta\leq\frac{2\pi}{p_{\ell}}+\varepsilon\right\}
\end{equation*}
($X^{\ell-1}_{\mathrm{sing}}=X_{p_{\ell}}\bigcup X_{p_{\ell+1}}\cdots$).
Then $\hat d_2(x,  X_{\mathrm{sing}}^{\ell-1})$ is equivalent to
$\hat d(x,  X_{\mathrm{sing}}^{\ell-1})$.  (Namely
 $\frac{1}{C_{\ell,\varepsilon}}\hat d_2\leq \hat d\leq C_{\ell,\varepsilon}\hat d_2$
for some constant $C_{\ell,\varepsilon}$ independent of $x$).
\end{lem}

We postpone the proof of the lemma until after Theorem~\ref{t-gue160413}.

For technical reasons we impose a pseudoconvex condition on $X$ in the following
although the same result is expected to hold without this condition.   

\begin{thm}
\label{t-gue160413} With the notations above, assume that $X$ is
strongly pseudoconvex. Then there is a constant $C\geq1$ such that
\begin{equation*}
\frac{1}{C}d(x,X^r_{\mathrm{sing\,}})\leq\hat d(x,X^r_{\mathrm{sing\,}})\leq
Cd(x,X^r_{\mathrm{sing\,}}),\ \ \forall x\in X.
\end{equation*}
\end{thm}

\begin{proof}
For simplicity, we assume that $X=X_1\bigcup X_2$, i.e. $p_1=1$, $p_2=2$,
so that $X_{\mathrm{sing}}\equiv X_{\mathrm{sing}}^1=X_2\,(r=1)$ by definition.  For the general case, the
proof is essentially the same.   By Lemma~\ref{t-gue160413a}, for every (small and fixed) $%
\varepsilon>0$ we have
\begin{equation}  \label{e-gue160413a}
\hat d(x,X_{\mathrm{sing\,}})\approx\inf\left\{d(x,e^{-i\theta}\circ x);\,
\pi-\varepsilon\leq\theta\leq\pi+\varepsilon\right\}.
\end{equation}
Since $X$is strongly pseudoconvex, it is well-known that (see~\cite{HHL15})
there exists a CR embedding:
\begin{equation}  \label{e-gue160414}
\begin{split}
\Phi:X&\rightarrow\mathbb{C}^N, \\
x&\rightarrow(f_1(x),\ldots,f_N(x))
\end{split}%
\end{equation}
with $f_j\in H^0_{b,m_j}(X)$ for some $m_j\in\mathbb{N}$ ($j=1,\ldots,N$).

We assume that $m_1,\ldots,m_s$ are odd numbers and $%
m_{s+1},\ldots,m_N$ are even numbers.  By $p_1=1$ and $p_2=2$ one sees that (cf. \eqref{e-gue150704e1}) 
\begin{equation}  \label{e-gue160414I}
\mbox{$x\in X_{{\rm sing\,}}$ if and only if $f_1(x)=\cdots=f_{s}(x)=0$}
\end{equation}
so that
\begin{equation}  \label{e-gue160414II}
d(x,X_{\mathrm{sing\,}})^2\approx\sum^s_{j=1}\left\vert
f_j(x)\right\vert^2, \quad \forall x\in X. 
\end{equation}

Now, by using the embedding theorem \eqref{e-gue160414} (together 
with \eqref{e-gue150704e1}) we have
\begin{equation}  \label{e-gue160414III}
d(x,e^{-i\pi}\circ x)^2\approx\sum^N_{j=1}\left\vert
f_j(x)-f_j(e^{-i\pi}\circ x)\right\vert^2=4\sum^s_{j=1}\left\vert
f_j(x)\right\vert^2\approx d(x,X_{\mathrm{sing\,}})^2
\end{equation}
and hence for every $\pi-\varepsilon\leq\theta\leq\pi+\varepsilon$ ($\varepsilon>0$
small)
\begin{equation}  \label{e-gue160414IIIa}
d(x,e^{-i\theta}\circ x)^2\approx\sum^N_{j=1}\left\vert
f_j(x)-f_j(e^{-i\theta}\circ x)\right\vert^2\geq\sum^s_{j=1}\left\vert
(1-e^{-im_j\theta})f_j(x)\right\vert^2\approx\sum^s_{j=1}\left\vert
f_j(x)\right\vert^2\approx d(x,X_{\mathrm{sing\,}})^2.
\end{equation}
By \eqref{e-gue160414IIIa}
we conclude that
\begin{equation}  \label{e-gue160414a}
\inf\left\{d(x,e^{-i\theta}\circ x);\,
\pi-\varepsilon\leq\theta\leq\pi+\varepsilon\right\}\approx d(x,X_{\mathrm{%
sing\,}}).
\end{equation}
Combining \eqref{e-gue160413a} and \eqref{e-gue160414a} we have proved the theorem.
\end{proof}

We give now:

\begin{proof}[proof of Lemma~\ref{t-gue160413a}]
 In the following we write 
$\hat d(x)=\hat d(x, X^{\ell-1}_\mathrm{sing})$ and $\hat d_2(x)=\hat d_2(x, X^{\ell-1}_\mathrm{sing})$
for a fixed $\ell$.  
For an illustration we assume $X=X_1\bigcup X_2$, i.e. $p_1=1$, $p_2=2$,
and $x\in X_1$ ($\hat d_2=\hat d=0$ for $x\in X_2$.)  Write $I$ for the complement of $I'\equiv\,]\pi-\varepsilon, \pi+\varepsilon[$ in
$[\zeta, 2\pi-\zeta]\equiv K$ (where $\zeta$ satisfies \eqref{e-gue160327} and $\varepsilon>0$ a small constant
to be specified, cf. the line above \eqref{e-gue160444e2}).
By definition $\hat d_2\geq \hat d\,(=\hat d_{\zeta}) $ ($\hat d_2$ is to take inf over $I'$
while $\hat d$ over $K$, and $I'\subset K$).

It remains to see $\hat d_2\le C\hat d$ for some $C$.  
Put $$f_S(x)=\inf_{\theta\in S}\left\{d(x,e^{-i\theta}\circ x)\right\}$$ 
for a set $S$.
We claim that there exists a $c$, $1> c>0$,
\begin{equation}\label{e-gue160413ab}
  f_I(x)\geq c
\end{equation}
for all $x\in X_1$.
Indeed for each $x\in X$ and for any $\theta\in I=
[\zeta, \pi-\varepsilon]\cup [\pi+\varepsilon,
2\pi-\zeta]$ one sees $x\neq e^{-i\theta}\circ x$.
So \eqref{e-gue160413ab} follows by a compactness argument.
Let $M\geq 1$ be an upper bound of $\hat d_2$.
We claim 
\begin{equation}
\label{e-gue160413e1}
\hat d_2(x)\leq \frac{M}{c}\hat d(x),\quad x\in X.
\end{equation}
Note $\hat d(x)=f_K(x)=\min\{f_I(x), f_{I'}(x)\}$ and $\hat d_2=f_{I'}$. 
Suppose $f_K(y)<f_I(y)$.  Then $f_K(y)=f_{I'}(y)$, i.e. 
$\hat d(y)=\hat d_2(y)$ and \eqref{e-gue160413e1} holds
for these $y$ (as ${M\over c}>1$).  If  $f_K(y)\ge f_I(y)$ (for some $y\in X_1$), 
then $f_K(y)=f_I(y)$, giving $\hat d(y)\ge c$ by \eqref{e-gue160413ab}.   For these $y$,
\eqref{e-gue160413e1} still holds.   In any case we have proved \eqref{e-gue160413e1}
for $x\in X_1$, hence for $x\in X$ ($\hat d=\hat d_2=0$ at $x\in X_2$).

For another illustration, in the same notation as above except that say, $X=X_1\bigcup X_2\bigcup X_{4}$
(i.e. $p_3=4$).  We are going to prove the lemma for the case $\ell=2$ (with $x\in  X_1$, as 
$\hat d=\hat d_2=0$ at $x\not\in X_1$ for $\ell=2$).  

With the above $I, I'$ and $K$, let $J$ be the complement of
$J'\equiv\,]\frac{\pi}{2}-\varepsilon,\frac{\pi}{2}+\varepsilon[\,\bigcup
\,]\frac{3\pi}{2}-\varepsilon,\frac{3\pi}{2}+\varepsilon[$ in $I$.

It follows, similarly as \eqref{e-gue160413ab}, that there exists a $c_2$, $1> c_2>0$
such that 
\begin{equation}
\label{e-gue160413e2}
f_J(x)\ge c_2, \quad \forall\, x\in X.
\end{equation}

Let $\{W_\alpha\}_\alpha$ be the set of connected components of 
$X_4$.  Each $y\in X_4$ is a fixed
point of the subgroup ${\mathbb Z}_4=\{1, e^{i\frac{\pi}{2}}, e^{i\pi}, e^{i\frac{3\pi}{2}}\}$ of $S^1$; write 
$\lambda_{i,\alpha}(g)$ for all the eigenvalues of the isotropy (and isometric) action of $g\in \mathbb{Z}_4$ on 
$T_yX$ for $y\in W_\alpha$.  All of them are independent of the choice of $y\in W_\alpha$.  
Let $$C_M=\max_{1\ne g\in \mathbb{Z}_4, \lambda_{i,\alpha}(g)}\{|\lambda_{i,\alpha}(g)-1|\}>0; \quad 
c_m=\min_{1\ne g\in \mathbb{Z}_4, \lambda_{i,\alpha}(g)\ne 1}\{|\lambda_{i,\alpha}(g)-1|\}>0. $$

Let $B=\{x\in X;\,\hat d_2(x)\ge ({\frac{C_M}{c_m}}+1)\frac{M}{c_2}\hat d(x)>0\}$
($M\ge 1$ as above).   Clearly $B\subset X_1$ (zero distance
for $x\in X_2\cup X_4$).  We claim that
\begin{equation}
\label{e-gue160444}
\overline B\cap X_4=\emptyset.
\end{equation}

To see \eqref{e-gue160444} suppose otherwise.  Let $y_n\in B$
and $y_n\to y\in X_4$ as $n\to \infty$.  Observe that 
$f_K(y_n)\ne f_{I'}(y_n)$ for all $n$ because the equality $\hat d(y_n)=\hat d_2(y_n)$
(note $f_K=\hat d$ and $ f_{I'}=\hat d_2$) clearly contradicts the definition of $B$
with $y_n\in B$.  
By $K=I'\cup J'\cup J$, we are left with two possibilities for a $y_n$
\begin{equation}
\label{e-gue160444e4}
\begin{split}
&\mathrm{i)}\,\, f_K(y_n)=f_{J'}(y_n)\\
&\mathrm{ii)}\,\,  f_K(y_n)=f_{J}(y_n).
\end{split}
\end{equation}

Suppose i).   By examining the isotropy (and isometric) action of $\mathbb{Z}_4$ at $y\in X_4$, 
one sees that both ratios below 
\begin{equation}\label{e-gue160444e1}
\frac{d(y_n, e^{\i\pi}\circ y_n)}{d(y_n, e^{\i\frac{\pi}{2}}\circ y_n)},\quad
\frac{d(y_n, e^{\i\pi}\circ y_n)}{d(y_n, e^{\i\frac{3\pi}{2}}\circ y_n)}
\end{equation}
are bounded above by $\frac{C_M}{c_m}+\frac{1}{4}$ 
as $n\gg 1$.  
Since $I'$ and $J'$ are $\varepsilon$-neighborhoods around $\pi$ and $\{\frac{\pi}{2}, \frac{3\pi}{2}\}$
respectively, by choosing a sufficiently small $\varepsilon$ (say $\varepsilon\le \varepsilon_0$) 
one sees from \eqref{e-gue160444e1} 
\begin{equation}
\label{e-gue160444e2}
\frac{f_{I'}(y_n)}{f_{J'}(y_n)}\le {\frac{C_M}{c_m}+\frac{1}{2}}, \quad n\gg 1
\end{equation}
We claim that 
this contradicts $y_n\in B$.   Note $f_K=\hat d$
and $f_{I'}=\hat d_2$ so that the assmption i) $f_K(y_n)=f_{J'}(y_n)$
amounts to  $\hat d(y_n)=f_{J'}(y_n)$ and \eqref{e-gue160444e2} 
gives 
\begin{equation}
\label{e-gue160444e3}
\frac{\hat d_2(y_n)}{\hat d(y_n)}\le {\frac{C_M}{c_m}+\frac{1}{2}}, \quad n\gg 1.
\end{equation}
By $y_n\in B$, \eqref{e-gue160444e3} contradicts the definition of $B$.  

Suppose ii) of \eqref{e-gue160444e4}.  
By \eqref{e-gue160413e2}, $f_J(x)\geq c_2$ for all $x\in X$ 
hence by $f_K=\hat d$ and ii) of  \eqref{e-gue160444e4},  one obtains  $\hat d(y_n)\ge c_2$,
giving $\hat d_2(y_n)\ge ({\frac{C_M}{c_m}}+1)M$ by using $y_n\in B$,  
which is absurd since $\hat d_2\le M$ by assumption.  
The claim \eqref{e-gue160444} is proved by contradictions in i) and ii) of \eqref{e-gue160444e4}.  

Granting the claim \eqref{e-gue160444} we have $\overline B\subset X_1\cup X_2$ (which is open in $X$).
Since for $\theta\in I$ and $x\in X_1\cup X_2$ (in particular for $x\in \overline B$) $x\neq e^{-i\theta}\circ x$, 
by compactness there exists a $c_3$, $1> c_3>0$ satisfying (as in \eqref{e-gue160413ab})
\begin{equation}
\label{e-gue160444af1}
f_I(x)\ge c_3
\end{equation} 
for all $x\in \overline B$.  
One asserts that 
\begin{equation}
\label{e-gue160444f1}
\hat d_2(x)\le ({\frac{C_M}{c_m}}+1)\frac{M}{c_2c_3}\hat d(x),\quad \forall\,  x\in \overline B.
\end{equation}
The argument is similar.  By $f_K=\min\{f_I,f_{I'}\}$, a) $f_K(x)=f_{I'}(x)$ or b) $f_K(x)=f_I(x)$.
a) If $f_K(x)=f_{I'}(x)$, then by $f_K=\hat d$ and $f_{I'}=\hat d_2$, $\hat d(x)=\hat d_2(x)$; 
b) if $f_K(x)=f_I(x)$, then by \eqref{e-gue160444af1} and $f_K=\hat d$, $\frac{\hat d(x)}{c_3}\ge 1$
(for $x\in \overline B$).  
In both cases a) and b), \eqref{e-gue160444f1} holds 
(by $M\ge 1$ an upper bound of $\hat d_2$ and $\frac{1}{c_2},\frac{1}{c_3}\,>1$).  

Finally, Since the same inequalily of \eqref{e-gue160444f1} holds for all $x$ outside $B$ by definition
of $B$ (with $\hat d_2=\hat d=0$ for $x\in X_2\cup X_4$),
the equivalence between $\hat d$ and $\hat d_2$ (for all $x\in X$) is proved.

The proof for the general case clearly flows from the similar pattern as
above (although tedious).  We shall omit the detail.
\end{proof}

\section{Trace integrals and Proof of Theorem~\protect\ref{t-gue160416}}

\label{s-gue160416}

\subsection{A setup, including a comparison with recent developments}
\label{s-gue160416s0}
There is a vast literature about heat kernels on manifolds.  A comparison
between the previous results and those of ours in the present paper shall now be discussed
before we proceed further.  
A concise account of the (ordinary) heat kernel in 
diversified aspects is given in Richardson \cite{Ri10} and references therein.  A generalization of 
the heat kernel to orbit spaces of a group $\Gamma$ (of isometries) acting on a manifold $M$ 
dates back to the seminal work of H. Donnelly in late '70s \cite{Do76}, \cite{Do79}. 
Among others, Donnelly calculated the asymptotic expansion of 
the trace of the ordinary heat kernel on $M$ restricted to 
$\Gamma$-invariant functions (here $\Gamma$-action is assumed to be properly discontinous on $M$).  
Br\"uning and Heintze in '84 \cite{BH84} studied the equivariant trace with $\Gamma$ replaced by 
a compact group $G$ of isometries (including the trace 
restricted to $G$-invariant eigenfunctions).   A similar study (of trace) into the orbifold case has been made 
recently in \cite{Ri10} and \cite{DGGW08}.   In all of these works the asymptotic expansion of the (ordinary) heat kernel 
is more or less regarded as known.    The questions or techniques come down partly to
that used in Donnelly \cite{Do76} where the contributions to the trace integral are shown to be 
essentially supported on the fixed point set of the group action.  

In a closely related direction some authors consider the case of Riemannian foliations.
In this regard, if the orbits of a group acting by isometries are of the same dimension, 
this forms an example of a Riemannian foliation.   For a Riemannian foliation, one is usually 
restricted to the space of {\it basic functions} which are constant on leaves of 
the foliation.  Similar ideas apply to give {\it basic forms}.  The {\it basic Laplacian} and 
{\it basic heat kernel} $K_B(t, x, y)$ can then be defined.    Over decades there has been much study into 
the existence part of the basic heat kernel $K_B(t, x, y)$, which is finally 
proven in great generality by E. Park and K. Richardson in '96 \cite{PR96}.  Another proof 
on the existence is found in '98 \cite{Ri98}, which gives a specific formula 
for $K_B(t, x, y)$ and allows them to obtain an asymptotic formula for $K_B(t, x, x)$.
We denote the trace integral (on basic functions) by $\mathrm{Tr}\,e^{-t\triangle_B}$
(which is $\sum_m e^{-t\lambda_m}$ for certain eigenvalues with multiplicities).   In \cite{Ri10} and 
\cite{Ri98} the trace integral is also denoted by $K_B(t)$ which will be avoided here due to a possible confusion.
We shall dwell upon this important point after the next paragraph.  

Let's first pause for a moment for comparison.   {\it For the part of the trace integral}, the basic 
technique based on Donnelly is also employed here so that the extra contributions, if exist, 
are expected to be supported on the (lower dimensional) strata.   One of our features, however,  is 
Lemma~\ref{l-gue160417lem2} which leads to a precise information about 
the Gaussianlike term of the heat kernel and facilitates our ensuing asymptotic expansion (of the trace integral) 
in explicit expressions essentially based only on the data given by the ordinary (Kodaira) heat kernel
(hence computable in a sense, cf. Remarks~\ref{r-gue160416II}, ~\ref{r-gue160413a}).  
In the process we also need to sum over the group elements (Subsections~\ref{s-gue160416s1},
~\ref{s-gue160416s2}) and 
patch up these local sums over $X$ (Subsections~\ref{s-gue160416s3},~\ref{s-gue160416s4}).  
{\it For the part of the asymptotic expansion}, our present heat kernel by its very definition 
is similar to the $K_B$ above.  Yet objects beyond the basic forms, allowing a generalization in the equivariant sense, 
indexed by $m$($\in\mathbb{Z}$) in our notation (with $m=0$ corresponding to the case for $K_B$), 
with bundle-values, are considered here.   Since we 
allow CR nonK\"ahler case, suitable $\mathrm{Spin}^c$ structure in our CR version need be devised 
and equipped here in order for the rescaling technique of Getzler and our discovery of the off-diagonal 
estimate (Theorem~\ref{t-gue150627g}) to go through.    In this regard, it is not obvious 
at all (to us) whether the existence theory in the Riemannian case as above can be directly applied 
to our case.   Indeed, besides the need of the
$\mathrm{Spin}^c$ structure, our proof of the heat kernel is heavily based on the feature of 
the group action on CR manifolds, encoded by the BRT trivialization (Subsection~\ref{s-gue150514}),
through the use of the {\it adjoint} version of the original equation (Subsection~\ref{s-gue150920}).  
Above all, it lies in the following how our approach distinguishes itself from those of others.  

Notably, a seeming inconsistency could occur.  That is, a discovery in the works \cite{Ri10} and \cite{Ri98}
reveals that the so obtained asymptotic expansion for $K_B(t, x, x)$ there cannot be integrated (over $x$) to 
give the asymptotics of the trace (integral).   This perhaps takes one by surprise.  
See p. 2304 of \cite{Ri10} and Remark in p. 379 of \cite{Ri98}.  Despite this, the work \cite{Ri10}
manages to prove an asymptotic expansion for the trace integral (on basic functions) by using 
the work \cite{BH84} (rather than by integrating the asymptotics of $K_B(t, x, x)$ obtained therein).   
In this way, some nontrivial logarithmic terms are to appear unless they are proved to be vanishing.   
A conjecture has thus been introduced by K. Richardson in '10 \cite[Conjecture 2.5]{Ri10}
to the effect that in the Riemannian setting as above, 
for the (special) case of the isometric group action on a compact manifold, the logarithmic terms in the
asympototic expansion of the trace integral $\mathrm{Tr}\,e^{-t\triangle_B}$
must vanish and under a mild assumption (on orientation), 
there shall be no fractional powers in $t$ (except possibly an overall fractional power in $t$).
It is worth mentioning that the works \cite{Ri98} and 
\cite{Ri10} discuss a number of interesting examples pertinent to the aforementioned peculiar phenomenon.  
Despite that the seeming inconsistency is consistent with examples by explicit computations, it remains 
conceptually unclear how this phenomenon comes about.  

Our present work affirms the above conjecture of Richardson (with extension 
to the $S^1$-equivariant case) in the special case of CR manifolds studied here 
(see Theorems~\ref{t-gue160416d}, ~\ref{t-gue160416}).    One key point for 
all of this lies in \eqref{e-gue151108II} with $t$-dependent coefficients in 
$t$ powers, which is regarded as the asymptotic expansion one 
shall be dealing with in this paper, rather than a classical looking one \eqref{e-gue151108I}
(which is similar in nature to those proposed and studied in \cite{Ri98}, \cite{Ri10}).  
See also our Remarks~\ref{r-gue150508I}, ~\ref{r-gue160414a} and ~\ref{r-gue160414ar}, which are closely related 
to the above singular behavior of a classical formulation of asymptotic expansion. Put simply, 
the formulation \eqref{e-gue151108I} of an asymptotic expansion leads to 
certain discontinuities of the $t$-coefficients along the strata (cf. \cite[(4.7)]{Ri98} for a concrete example).
A remedy for \eqref{e-gue151108I} by \eqref{e-gue151108II} is mainly made via 
the introduction of a ``distance function" (see Theorem~\ref{t-gue160114}).  
Eventually, in this work we can restore the trace integral as the integration of our (unconventional) asymptotic 
expansion of the relevant heat kernel
(see Definition~\ref{d-gue150608} for the meaning of our asymptotic expansion).  
Thus, our trace integral and our asymptotic expansion of the heat kernel jointly 
clarify (with our class of manifolds) the somewhat undesirable phenomenon which is as mentioned above.  

To go from the trace integral to the index theorem (thought of as a supertrace integral) is usually not immediate.    
To the knowledge of the authors, the argument for the proof of index theorems by using trace integrals 
remains unclarified (cf. Remark~\ref{r-gue160416}).    Completely new ideas might be required; see \cite{BKR},
\cite{BKR08} for very interesting ideas.   In the present 
paper, we couldn't make our understanding of the (transversal) heat kernel (for our class of CR manifolds) 
complete without employing the rescaling technique of Getzler and the off-diagonal estimate 
(Theorem~\ref{t-gue150627g}) adapted to our setting.  
These results explore in depth the non-Gaussian terms of our (transversal) heat kernel, 
in contrast to the Gaussianlike term explored in the trace integral here.  
With these two parts together, our approach studies the meaningful separate aspects of 
the heat kernel in an unified manner, hence results in an (local) index theorem and the trace integral.   
These point to the differences between our approaches/results and those of the recent development.  

We turn now to our proof of the trace integral.  The line of thought in the proof involves four stages. 

In the first stage while the proof in the beginning echos that in 
last section, we shall make use of Lemma~\ref{t-gue160413a} and 
Theorem~\ref{t-gue160413} to handle the distance function $\hat d$.
(Here we assume the strongly pseudoconvex condition on $X$.) 
After this initial step,  we shall take a different approach that supersedes the previous one, which is more quantitative in nature without the strongly pseudoconvex condition on $X$ (hence without using Lemma~\ref{t-gue160413a} and 
Theorem~\ref{t-gue160413}).  This approach is partly based on 
the differential geometric information of the various isotropy actions
associated with the fixed point sets (strata) of the $S^1$ action.  
This allows us to 
learn more precise details about the heat kernel of Kohn Laplacian, hence 
to refine the computation in \eqref{e-gue160417aaI} which is basically qualitative. 
(See \eqref{e-gue160417aaI} for a kind of Dirac delta functions associated with the 
strata.)  Remark that one key point here is the notion of {\it type} which is initially designed 
for the need of computation.  In the fourth stage it is attached to the $S^1$ stratification closely.  

In the second, third and fourth stages, the treatment goes in line with that in the 
first stage and is mostly technical so as to 
integrate the results obtained in the first stage in a 
well organized manner.    The nonuniqueness way (subject to choice of 
BRT trivializations) of giving 
the asymptotic expansion of $e^{-t\widetilde\Box_{b,m}^+}(t, x, y)$ 
(cf.  Theorem~\ref{t-gue160123}) leaves us the freedom of choosing 
convenient BRT charts to work out some computations.   The salient fact that
$e^{-t\widetilde\Box_{b,m}^+}(t, x, y)$ is an intrinsic object (yet not directly computable), 
thus is independent of choice of BRT trivializations, is essential to 
giving intrinsic meanings to some BRT-dependent 
computations (cf. the contrast between Propositions~\ref{t-gue160416c}
and ~\ref{t-gue160416cc} on $\eta_s$-terms).   This conceptual understanding 
turns out to be crucial to our final result.   The extension 
of the previous notion ``type" to the $S^1$ stratification is the last conceptual 
step for the completion of the proof.  

As before, $X$ $(\mathrm{dim}\,X=2n+1)$ is a compact connected CR manifold with a transversal CR locally free 
$S^1$ action.   
To proceed with the proof of Theorem~\ref{t-gue160416},
assume $X=X_{p_1}\bigcup
X_{p_2}\bigcup\cdots \bigcup X_{p_k}$ where 
$X_{p_\ell}=\bigcup^{s_\ell}_{\gamma=1} X_{p_{\ell(\gamma)}}$ $(s_{\ell=1}=1)$
as a disjoint union of (connected) submanifolds $X_{p_{\ell(\gamma)}}$\,($\overline X_{p_{\ell}}$, being the fixed point set of 
an isometry $e^{-i\frac{2\pi}{p_\ell}}$, is a submanifold (possibly disconnected)).  

Write  $e_{\ell(\gamma)}$ for the (real) codimension of 
$X_{p_{\ell(\gamma)}}$ in $X$.  When there is no danger of confusion, 
we may drop $\gamma$ and write $e_{\ell}$ for $e_{\ell(\gamma)}$. 
Recall $X_{\mathrm{sing}}^{\ell-1}=X_{p_{\ell}}\bigcup X_{p_{\ell+1}}\cdots$.


We follow the notations in Subsection~\ref{s-gue150627I} and
the beginning of the last section.  Thus $B_j:=(D_j,(z,\theta),\varphi_j)$
($j=1, 2, \ldots, N$) with $D_j=U_j\times]-2%
\delta_j,2\widetilde\delta_j[$, $U_j=\left\{z\in\mathbb{C}^n;\, \left\vert z\right\vert<\gamma_j\right\}$
and similarly $\hat
D_j=\hat U_j\times]-\frac{\delta_j}{2},\frac{\widetilde\delta_j}{2}[$,
$\hat U_j=\left\{z\in\mathbb{C}^n;\, \left\vert z\right\vert<\frac{\gamma_j}{2%
}\right\}$. We let $\delta_j=\widetilde\delta_j=\zeta$, $%
j=1,2,\ldots,N$ and assume $%
X=\hat D_1\bigcup\cdots\bigcup\hat D_N$.
As before, we assume that $\zeta>0$ satisfies \eqref{e-gue160327}.

\subsection{Local angular integral}
\label{s-gue160416s1}

Recall $\hat h_{j,+}(x,y)$, $\hat b^+_{j,s}(x,y)$ of \eqref{e-gue150901}
(to be given below); $a_s^+(t, x, y)$ involves a certain integral over $[0,2\pi]$ (cf. \eqref{e-gue150901a}),
$s=n,n-1,\ldots$.   One key step is the following
local version.  That is, the (trace) integral of the form
\begin{equation}\label{e-gue160416be1}
I=I^{(j)}(p_\ell, g(x))\equiv \frac{1}{2\pi}\int^{\frac{2\pi}{p_\ell}+\varepsilon}_{\frac{2\pi}{p_\ell}%
-\varepsilon}\int_Xg(x)e^{-\frac{\hat h_{j,+}(x,e^{-iu}\circ x)}{t}}\mathrm{Tr}\,\hat
b^+_{j,s}(x,e^{-iu}\circ x)e^{-imu}dv_X(x)du.
\end{equation}
The trace ``Tr" here is actually well defined 
despite a slight abuse of notation about $\hat b^+_{j,s}(x, e^{-iu}\circ x)$ at the second 
variable (see the line above \eqref{e-gue150510f}).  

Recall the expressions in \eqref{e-gue150901} (to be used in what follows): 
\begin{equation}  \label{e-gue150901e1}
\begin{split}
&\hat h_{j,+}(x,y)=\hat\sigma_j(\theta) h_{j,+}(z,w)\hat\sigma_j(\eta)\in
C^\infty_0(D_j),\ \ x=(z,\theta),\ \ y=(w,\eta) \\
&\hat
b^+_{j,s}(x,y)=\chi_j(x)e^{-m\varphi_j(z)-im\theta}b^+_{j,s}(z,w)e^{m%
\varphi_j(w)+im\eta}\tau_j(w)\sigma_j(\eta),\ \ s=n,n-1,\ldots
\end{split}%
\end{equation}
with suitable cut-off functions $\chi_j$, $\tau_j$, $\sigma_j$ and $\hat \sigma_j$
defined there.  


There will be cases for the result \eqref{e-gue160416be1}.  
We need some preparations and notations.  

For $I=I(p_\ell, g)$ of \eqref{e-gue160416be1}, take a point $x_0\in \mathrm{Supp}\,g\cap\overline X_{p_\ell}$, 
then $x_0\in \overline X_{p_{\ell(\gamma_\ell)}}$ for a $\gamma_\ell=1, 
\ldots,s_\ell$.   Locally at $x_0$ there are higher dimensional strata 
\begin{equation*}
\overline X_{p_{i_1(\gamma_{i_1})}}=X\supsetneq\overline X_{p_{i_2(\gamma_{i_2})}}\supsetneq\ldots
\supsetneq\overline X_{p_{i_f(\gamma_{i_f})}}\supsetneq \overline X_{p_{i_{f+1}(\gamma_{i_{f+1}})}}=
\overline X_{p_{\ell(\gamma_\ell)}}
\end{equation*}
passing through $x_0$ where $i_1=1<i_2<\ldots<i_f<i_{f+1}=\ell,\,\in \{1, 2, ,\ldots, \ell-1,\ell\}$.  
Here (to be useful later) $p_{i_1}|p_{i_2}\cdots|p_{i_f}|p_\ell$ (by Remark~\ref{r-gue150804} similarly).  
We say 

\begin{defn}
\label{d-gue160416t1}
i) The {\it type} $\tau(I)$ of $I(p_\ell, g)$ is $\tau(I):=({i_1(\gamma_{i_1})}, {i_2(\gamma_{i_2})},\ldots,
{i_f(\gamma_{i_f})}, {i_{f+1}(\gamma_{i_{f+1}})})$
where $i_1=\gamma_{i_1}=1$ and $i_{f+1}=\ell$ always.   The length $l(\tau(I))$ of the type is $f+1$.  
$I(p_\ell, g)$ is said to be of {\it simple type} if in $\tau(I)$, $(i_1,i_2,\ldots,i_{f+1})=(1, 2, \ldots,\ell-1,\ell)$.  

ii) Two given types 
\begin{equation*}\begin{split}
&\tau(I(p_{\ell_1}, g_1))=({i_1(\gamma_{i_1})}, {i_2(\gamma_{i_2})},\ldots,{i_{f_1+1}(\gamma_{i_{f_1+1}})})\\
&\tau(I(p_{\ell_2}, g_2))=({j_1(\gamma'_{j_1})}, {j_2(\gamma'_{j
_2})},\ldots,{j_{f_2+1}(\gamma'_{j_{f_2+1}})})
\end{split}
\end{equation*} 
are said to be in the same {\it class} 
provided a) $f_1=f_2:=f$, $\ell_1=\ell_2$, $i_1=j_1, i_2=j_2,\ldots$, $i_f=j_f$ 
and b) the codimensions of the corresponding strata coincide: 
$e_{\ell_1(\gamma_{\ell_1})}=e_{\ell_2(\gamma'_{\ell_2})}$, 
$e_{i_1(\gamma_{i_1})}=e_{j_1(\gamma'_{j_1})}, 
e_{i_2(\gamma_{i_2})}=e_{j_2(\gamma'_{j_2})},\ldots, e_{i_{f}(\gamma_{i_{f}})}=e_{j_{f}(\gamma'_{j_{f}})}$.  

iii) As above $I=I(p_\ell, g)$, suppose $\mathrm{Supp}\,g\cap\overline X_{p_\ell}=\emptyset$, equivalently 
$\mathrm{Supp}\,g\subset \cup_{q=1}^{\ell-1}X_{p_{\ell-q}}$.  We say $\tau(I)$ is  
of {\it trivial type}.  
\end{defn} 
Remark that $g(x)$ will be chosen to be of very small support and the local nature of $I$, $\tau(I)$
will be obvious.   Namely, in this case $\tau(I)$ is independent of 
choice of $x_0\,(\in \mathrm{Supp}\,g\cap\overline X_{p_\ell})$.   
In the final subsection, the notion of ``type" will be naturally extended to 
each connected submanifold $X_{p_{\ell(\gamma)}}$ in the strata.  By this, the influence of the geometry of 
the $S^1$ stratification on the heat kernel trace integral will become more evident.  

Most numerical results in what follows will only depend on the equivalence classes 
of types.  But for the sake of notational convenience, we assume $I$ to be of simple type or trivial type in the 
proposition below.  The modification to the general type is basically only complicated in notation 
and will be treated later.  

\begin{prop}
\label{t-gue160416b} Suppose $x_0\in\hat D_j$.  Then there exist a
neighborhood $\tilde\Omega$ ($\Subset \hat D_j$) of $x_0$ and an $\tilde\varepsilon>0$ (depending 
on $x_0$) such that for every $\Omega\subset \tilde\Omega$, every $%
g(x)\in C^\infty_0(\Omega)$ we have the following for $I$ of \eqref{e-gue160416be1} with 
any $\varepsilon\le \tilde\varepsilon$.  
Note $I$ is assumed to be of simple type (if not of trivial type) as said prior to the proposition.  
(In the following, ii) and Case a) of iii) are basically of trivial type; i) and Case b) of iii) are of simple type.)  

i) $\ell=1$ ($p=p_1$).   For $x=(z, v)\in \hat D_j$ write $z(x)=z$ and $\theta(x)=v$.  
$$
I=e^{-\frac{2\pi i}{p}m}\frac{1}{2\pi}\int_{-\varepsilon}^{\varepsilon}
\int_Xg(x)\chi_j(x)\mathrm{Tr}\,b_{j,s}^+(z, z)\tau_j(z)\sigma_j(v+\psi) dv_X(x)d\psi.$$ 
In particular, $I$ is a constant independent of $t$.  (Note it is $b_{j, s}^+$ instead of $\hat b_{j, s}^+$
here; the same can be said with \eqref{e-gue160416bpe1} below.)  

ii) Suppose $e^{-i\frac{2\pi}{p_\ell}}\circ x_0\not\in\hat D_j$ (here $\ell=2,3,\ldots,k$).  
Then $I=0$. 

iii) Suppose $e^{-i\frac{2\pi}{p_\ell}}\circ x_0\in\hat D_j$ (here $\ell=2,3,\ldots,k$).

Case a) $x_0\in \bigcup_{q=1}^{q=\ell-1}X_{p_{\ell-q}}$.  Then $I\sim O(t^{\infty})$ as $t\To0^+$.   


Case b) $x_0\not\in \bigcup_{q=1}^{q=\ell-1}X_{p_{\ell-q}}$ and $x_0\in \overline X_{p_{\ell(\gamma_\ell)}}\subset
\overline X_{p_{\ell}}$.   
Take local coordinates ($e_\ell=e_{\ell(\gamma_\ell)}$ for some $\gamma_\ell=1, \ldots, s_\ell$) 
$$y=(y_1,\ldots,y_{2n+1})=(\hat y, Y)\,\, \mbox{with $\hat y=(y_1\ldots, y_{e_\ell})$ 
and $Y=(y_{e_{\ell}+1},\ldots,y_{2n+1})$}$$
defined on $\Omega$ such that 
$$
\overline X_{p_\ell}\cap\Omega=\left\{y\in\Omega;\, y_1=\cdots=y_{e_l}=0\right\}.$$
Assume (possibly after shrinking $\Omega$ about $x_0$) 
$\Omega=\bigcup_{j\in\{1,\cdots, k\}} (X_{p_{j(\gamma_j)}}\cap\Omega)$
(for some $\gamma_j=1,\cdots,s_j$) which is seen to be (by assumption of simple type) 
\begin{equation*}
(\overline X_{p_{\ell(\gamma_\ell)}}\cap\Omega)\bigcup_{q=1}^{\ell-1}
(X_{p_{\ell-q(\gamma_{\ell-q})}}\cap\Omega).
\end{equation*}
Write $e_{\ell-q+1}-e_{\ell-q}$ for the codimension 
of $X_{p_{\ell-q+1}}$ in $\overline X_{p_{\ell-q}}$ where $p_{\mu}=p_{\mu(\gamma_\mu)}$
for $\mu=\ell-q+1$ and $\mu=\ell-q$ respectively.  
If $y=(z, \theta)$ (in BRT coordinates), write $z(y)$ for $z$ and if 
$y=(0,Y)$, write $Y$ for $(0,Y)$ and $z(Y)$ for $z(y)$.  Similar notation for 
$\theta(Y)$ etc. 

Then $(e_\ell=e_{\ell(\gamma_\ell)})$
$$ I= b^{(j)}_{s,\frac{e_{\ell}}{2}}t^{\frac{%
e_{\ell}}{2}}+b^{(j)}_{s,\frac{e_{\ell}+1}{2}}t^{\frac{e_{\ell}+1}{2}}+\ldots
$$ where the first coefficient  $b^{(j)}_{s,\frac{e_{\ell}}{2}}$ is given by 
\begin{equation}
\label{e-gue160416bpe1}\begin{split}
b^{(j)}_{s,\frac{e_{\ell}}{2}}=
\pi^{\frac{e_{\ell}}{2}}{e^{-\frac{2\pi i}{p_{\ell}}m}}
\prod_{q=1}^{\ell-1}{\left\vert e^{\frac{i2\pi }{p_{\ell}}p_{\ell-q}}-1\right\vert^{-(e_{\ell-q+1}-e_{\ell-q})}}
\times\\
\frac{1}{2\pi}\int_{-\varepsilon}^{\varepsilon}\int_{\overline X_{p_{\ell(\gamma_\ell)}}}g(Y)\chi_j(Y)\mathrm {Tr}\,b_{j,s}^+(z(Y), z(Y))\tau_j(z(Y))\sigma_j(\theta(Y)+u)dv_{\overline X_{p_{\ell(\gamma_\ell)}}}(Y)du. 
\end{split}
\end{equation} 

In particular, for $s=n$ $(\mathrm{cf.}\,\mathrm{dim}X=2n+1)$, \eqref{e-gue160416bpe1} for 
$b^{(j)}_{n,\frac{e_{\ell}}{2}}$ simplifies by using $\mathrm{Tr}\,b^+_{j, n}(z,z)\equiv (2\pi)^{-n}$. 
\end{prop}

\begin{proof} Write $x_0=(z_0, \theta_0)$.  For simplicity, assume $\theta_0=0$ without
loss of generality (cf. the last three paragraphs of the proof of Theorem~\ref{t-gue160123}
for a similar situation).   Note that the existence of $\tilde\Omega$ and $\tilde\varepsilon$ 
in the statement above will be obvious from the proof below and we shall not refer to them explicitly.  

To see i), we note that $e^{-i\frac{2\pi}{p}}=\mathrm{id}$ ($p=p_1$) on $X$
(because it is so on $X_p$ by definition which is dense (and open) in $X$).
For $x=(z, v)$ lying in the BRT
neighborhood $\hat D_j$ and for $u=\frac{2\pi}{p}\pm\varepsilon$
such that $e^{-iu}\circ x=e^{\pm i\varepsilon}\circ x$ lies in $D_j$ ($\supset \hat D_j$), 
one has $e^{\pm i\varepsilon}\circ x=(z, v\mp\varepsilon)$ by construction of BRT charts $D_j$.
In this case $e^{-\frac{\hat
h_{j,+}(x,e^{-iu}\circ x)}{t}}\equiv 1$ since $\hat
h_{j,+}(x,e^{-iu}\circ x)=0$ by $h_{j,+}(z, z)=0$ of \eqref{e-gue150901e1}.
The same reasoning applies to $\mathrm{Tr}\,\hat b^+_{j,s}(x, e^{-iu}\circ x)$ to reach $\mathrm{Tr}\,b^+_{j,s}(z,z)$.  
Now choose a neighborhood $\Omega\Subset \hat D_j$ of 
$x_0$ then a small $\varepsilon>0$ (depending on $x_0$) such that 
$e^{\pm i\varepsilon}\circ x$ lies in $D_j$ for $x\in\Omega$.  
As $g\in C^\infty_0(\Omega)$, we can apply the above 
argument for these $x$ by making 
$\psi=u-\frac{2\pi}{p}$ ($|\psi|\le \varepsilon$) so that $e^{-iu}\circ x=e^{-i\psi}\circ x=(z, v+\psi)$.
In \eqref{e-gue150901e1} one thus has $\theta=v$, $w=z$ and $\eta=v+\psi$.
By these remarks, i) of the proposition follows from \eqref{e-gue160416be1} and \eqref{e-gue150901e1}.  

For ii) of the propostion, with $x_0\in \hat D_j$ yet $e^{-i\frac{2\pi}{p_\ell}}\circ x_0\not\in\hat D_j$,
by continuity of $S^1$ action there exist a
neighborhood $\Omega$ ($\Subset\hat D_j$) of $x_0$ and an $\varepsilon>0$ such that 
$e^{-iu}\circ x\not\in \hat D_j$ for $x\in \Omega$ and 
$u\in]\frac{2\pi}{p_\ell}-\varepsilon,\frac{2\pi}{p_\ell}+\varepsilon[$.  
Hence for these $x\in \Omega$, $\hat b_{j, +}(x, e^{-i\frac{2\pi}{p_{\ell}}}\circ x)=0$ 
by a cut-off function $\tau_j$ (of compact support in $\hat U_j\subset \hat D_j$) 
involved in $\hat b_{j, +}$ (see \eqref{e-gue150901e1}), giving $I=0$ in \eqref{e-gue160416be1}.  

For case a) of iii), the assumption gives 
$x_0\in X_{p_{\ell-q}}$, $q\in\{1, 2,\ldots,\ell-1\}$.  
Further, by assumption $e^{-i\frac{2\pi}{p_\ell}}\circ x_0\in\hat D_j$
we write $e^{-i\frac{2\pi}{p_\ell}%
}\circ x_0=(\widetilde z_0,\widetilde\theta_0)$ with  $\left\vert
\widetilde\theta_0\right\vert<\frac{\zeta}{2}$.  
We claim $\tilde z_0\ne z_0$.  The line of argument 
is slightly different from that in \eqref{e-gue160327I}.  
Suppose $\tilde z_0= z_0$.   Then
by $e^{i\widetilde\theta_0}\circ (e^{-i\frac{2\pi}{p_\ell}}\circ x_0)=e^{i\widetilde\theta_0}\circ
(\widetilde z_0,\widetilde\theta_0)=(\widetilde z_0, 0)=(z_0, 0)=x_0$ 
(recall $\theta_0=0$ in the beginning 
of the proof).     Hence, 
\begin{equation}
\label{e-gue160416biiia}
\frac{2\pi}{p_\ell}-\widetilde\theta_0=m\frac{2\pi}{p_{\ell-q}},\quad m\in \mathbb{Z}
\end{equation}
by assumption $x_0\in X_{p_{\ell-q}}$. But $%
\left\vert \widetilde\theta_0\right\vert<\frac{\zeta}{2}$ and $\zeta$ is assumed to
satisfy \eqref{e-gue160327}, so the above equality is absurd, proving 
the claim $\tilde z_0\ne z_0$ by contradiction. 

Now that  $\widetilde z_0\neq z_0$, 
there exists a neighborhood $\Omega$ of $x_0$ and an $\varepsilon>0$ (dependent on $x_0$) such that 
for $x\in\Omega$ and $\theta\in]\frac{2\pi}{p_\ell}-\varepsilon,%
\frac{2\pi}{p_\ell}+\varepsilon[$, writing $%
e^{-i\theta}\circ x=(\widetilde z,\tilde\theta)$ and $x=(z, \theta)$ one has $\left\vert \widetilde
z-z\right\vert\geq\frac{1}{2}\left\vert \widetilde z_0-z_0\right\vert\equiv \delta$
by using continuity of $S^1$ action at $x=x_0$ and $\theta=\frac{2\pi}{p_\ell}$.
From the property of $\hat h_{j,+}(x,y)$ (which is essentially $|z-w|^2$,
cf. \eqref{e-gue150608a} and \eqref{e-gue150901}) one sees that
$I$ of \eqref{e-gue160416be1} gives 
\begin{equation}  \label{e-gue160417J}
\frac{1}{2\pi}\int^{\frac{2\pi}{p_\ell}+\varepsilon}_{\frac{2\pi}{p_\ell}%
-\varepsilon}\int_Xg(x)e^{-\frac{\hat h_{j,+}(x,e^{-iu}\circ x)}{t}}\mathrm{Tr}\,\hat
b^+_{j,s}(x,e^{-iu}\circ x)e^{-imu}dv_X(x)du=O(t^\infty), \,\,\mbox{as $t\to 0^+$}
\end{equation} 
(for $g\in C^\infty_0(\Omega)$) simply because the exponential term in
\eqref{e-gue160417J} decays
rapidly if $\left\vert \widetilde
z-z\right\vert\geq \delta$ here,
proving case a) of iii) of the proposition. 

To prove case b) of iii), we first give an estimate under an additional 
assumption that $X$ is strongly pseudoconvex, 
then we will drop this assumption and carry out some refined computations
to complete the proof. 

Since $x_0$ is a fixed point of $e^{-\frac{2\pi i}{p_{\ell}}}$ by assumption, by continuity of $S^1$ 
action there exist an open subset $\Omega$ of $x_0$ and
a small constant $0<\varepsilon<\frac{\zeta}{2}$
such that $e^{-i\theta}\circ x\in\hat D_j$ for $x\in\Omega$
and $\theta\in]\frac{2\pi}{p_\ell}-\varepsilon,\frac{2\pi}{p_\ell}%
+\varepsilon[$.  We assume that $\Omega$ is small, say contained in the
BRT chart $\hat D_j$, and satisfies the local coordinates of case b) of iii). 
For $x=(z, v)\in \Omega$ and $\theta\in]\frac{%
2\pi}{p_\ell}-\varepsilon,\frac{2\pi}{p_\ell}+\varepsilon[$, 
write $e^{-i\theta}\circ x=(\widetilde z,\widetilde v)\in\hat D_j$. 

We claim that there exists a positive continuous function
$f_1(x)$ such that
\begin{equation}  \label{e-gue160417aII}
h_{j,+}(z,\widetilde z)\ge f_1(x)d(x,X_{p_\ell})^2, \qquad \forall\, x\in \Omega
\end{equation}
where $h_{j,+}$ is as in \eqref{e-gue150901e1} (cf. \eqref{e-gue150608a}).
(Here is the only place where we use the assumption $X$ is strongly pseudoconvex.) 

Granting the claim \eqref{e-gue160417aII}, 
with the local coordinates in iii), 
suppose $y(x_0)=Y(x_0)=0$.  Rewrite the quotient $\frac{d(y,X_{p_\ell})^2}
{(\left\vert y_1\right\vert^2+\cdots+\left\vert
y_{e_{\ell}}\right\vert^2)}$ as 
\begin{equation}  \label{e-gue160417aIII}
d(y,X_{p_\ell})^2=f_2(y)(\left\vert y_1\right\vert^2+\cdots+\left\vert
y_{e_{\ell}}\right\vert^2),\ \ \forall y\in \hat\Omega
\end{equation}
where $f_2(y)$ is a positive continuous function.  

With \eqref{e-gue160417aII} and \eqref{e-gue160417aIII}, we estimate
$\hat h_{j,+}$ below and have 
the following (see \eqref{e-gue150901e1} or \eqref{e-gue150901} and note $g(x)\in C^{\infty}_0(\Omega)$,
$\Omega$ small).  
\begin{equation}  \label{e-gue160417aaI}
\begin{split}
I&=\frac{1}{2\pi}\int^{\frac{2\pi}{p_\ell}+\varepsilon}_{\frac{2\pi}{p_\ell}%
-\varepsilon}\int_Xg(x)e^{-\frac{\hat h_{j,+}(x,e^{-iu}\circ x)}{t}}\mathrm{Tr}\,\hat
b^+_{j,s}(x,e^{-iu}\circ x)e^{-imu}dv_X(x)du \\
&\le \frac{1}{2\pi}\int^{\frac{2\pi}{p_\ell}+\varepsilon}_{\frac{2\pi}{p_\ell}%
-\varepsilon}\int_Xg(x)e^{-\frac{f_1(x)d(x,X_{p_\ell})^2}{t}}\mathrm{Tr}\,\hat
b^+_{j,s}(x,e^{-iu}\circ x)e^{-imu}dv_X(x)du \\
&=\frac{1}{2\pi}\int^{\frac{2\pi}{p_\ell}+\varepsilon}_{\frac{2\pi}{p_\ell}%
-\varepsilon}\int_Xg(y)e^{-\frac{f_1(y)f_2(y)(\left\vert y_1\right\vert^2+\cdots+\left\vert
y_{e_{\ell}}\right\vert^2)}{t}}\mathrm{Tr}\,\hat
b^+_{j,s}(y,e^{-iu}\circ y)e^{-imu}dv_X(y)du \\
&\sim c^{(j)}_{s,\frac{e_{\ell}}{2}}t^{\frac{e_{\ell}}{2}}+c^{(j)}_{s,\frac{e_{\ell}+1}{2}}t^{\frac{e_{\ell}+1}{2}%
}+\cdots\mbox{as $t\To0^+$},
\end{split}%
\end{equation}
where the last step is obtained by a change-of-variable 
(rescaling $y_i$ by $\sqrt{t}y_i$, $ i=1, \ldots e_{\ell}$, $e_{\ell}\geq 1$ as $\ell\ge 2$) and
$c^{(j)}_{s.k}\in\mathbb{R}$ is independent of $t$ ($k=\frac{e_{\ell}}{2},\frac{e_{\ell}+1}{2%
},\ldots$).   

We are left with the proof of the claim \eqref{e-gue160417aII}.
Part of the argument echos that for \eqref{e-gue160327I}.  
We first estimate $\left\vert z-\widetilde
z\right\vert^2$.   Without the danger of confusion we omit $``\circ"$ in what follows.  
By  $(z, 0)=e^{iv}x$ and $(\tilde z, 0)=e^{i\tilde v}(e^{-i\theta} x)$,
 $\left\vert z-\widetilde
z\right\vert$ is equivalent to $d(e^{iv}x, e^{i\tilde v}(e^{-i\theta} x))$
(cf. \eqref{e-gue160215}) 
which is the same as $d(x, e^{-iv}(e^{i\tilde v}(e^{-i\theta} x)))$.   As $(\tilde z, \tilde v), (z, v)\in \hat D_j$, 
one has $\tilde v, v\le \frac{\zeta}{2}$.  
By choosing  $\zeta$, $\varepsilon$ to be (much) less than the $\varepsilon_0$  
of Lemma~\ref{t-gue160413a}, one sees 
$d(x, e^{-iv}(e^{i\tilde v}(e^{-i\theta} x)))\ge 
\hat d_2(x, X_{\mathrm{sing}}^{{\ell}-1})$
of Lemma~\ref{t-gue160413a}.    By the same lemma
\begin{equation}
\label{e-gue160417e1}
\hat d_2(x, X_{\mathrm{sing}}^{{\ell}-1})\quad \mbox{is equivalent to}\quad 
\hat d(x, X_{\mathrm{sing}}^{{\ell}-1}).  
\end{equation} 
As $\left\vert z-\widetilde
z\right\vert^2$ is also equivalent to $h_{j,+}(z, \tilde z)$ of \eqref{e-gue160417aII}
(cf. \eqref{e-gue160215}) and \eqref{e-gue160417e1} is 
equivalent to $d(x, X_{\mathrm{sing}}^{{\ell}-1})$ by Theorem~\ref{t-gue160413}
with our assumption $X$ is strongly pseudoconvex, 
we have now shown 
\begin{equation}\label{e-gue160417e2}
h_{j,+}(z, \tilde z)\ge c\cdot d(x, X_{\mathrm{sing}}^{{\ell}-1})^2
\end{equation}
for some constant $c>0$.  
In view that $X_{\mathrm{sing}}^{{\ell}-1}=X_{p_{\ell}}\cup X_{p_{{{\ell}+1}}}\cdots $
and the assmption that $x_0\in \overline X_{p_{\ell}}$, 
one sees, possibly after shrinking $\Omega$,  
$d(x, X_{\mathrm{sing}}^{{\ell}-1})=d(x, \overline X_{p_{\ell}})=d(x, X_{p_{\ell}})$.
Hence we have reached 
\eqref{e-gue160417aII} from \eqref{e-gue160417e2}, as desired.  

Following \eqref{e-gue160417aaI} we shall now make some accurate computations
for case b) of iii) of this propostion.  

Henceforward we do not assume that $X$ is strongly pseudoconvex; we will not use 
Lemma~\ref{t-gue160413a} and Theorem~\ref{t-gue160413} as used above.  

Write $\frac{2\pi }{p_{\ell}}=\omega$ and $u=\psi+\omega$.  
In coordinates of case b) of iii), in view of \eqref{e-gue160417aaI} one seeks to identify, among others, 
\begin{equation}
\label{e-gue160417aae1}
\lim_{t\to 0^+}\frac{\hat h_{j,+}((\sqrt{t}\hat y, Y), e^{-iu}(\sqrt{t}\hat y, Y))}{t}
\end{equation}
where we have rescaled $\hat y\to \sqrt{t}\hat y$, and we omit $``\circ"$ for the $e^{-iu}\in S^1$ action.  

Since the fixed point set of an isometry is totally geodesic, we 
assume $Y$ to be a system of geodesic coordinates  at  $Y=0$ of $\overline X_{p_{\ell}}$, as 
$(\hat y, Y)$ the geodesic coordinates at $(0,0)$ of $X$.  
We choose $Y=0$ in \eqref{e-gue160417aae1} to simplify the notation.  Expressed in BRT coordinates,
$(\sqrt{t}\hat y, 0)=(z_0,v_0)$ and 
$e^{-i\omega}(\sqrt{t}\hat y, 0)=(z_1,v_1)$ 
(by continuity of $S^1$ action, for $t$ small 
$e^{-i\omega}(\sqrt{t}\hat y, 0)\in \hat D_j$ since $e^{-i\omega}(0, 0)=(0,0)$).  

One sees 
$e^{-iu}(\sqrt{t}\hat y, 0)=e^{-i\psi}(z_1, v_1)=(z_1, v_1+\psi)$ for $|\psi|\le \varepsilon$.  
By \eqref{e-gue150901e1} one sees 
\begin{equation}
\label{e-gue160417aae2}
\hat h_{j,+}((\sqrt{t}\hat y, 0), e^{-iu}(\sqrt{t}\hat y, 0))\,\,\mbox{ is 
now reduced to $h_{j,+}(z_0, z_1)$}
\end{equation} 
which is independent of $\psi$ for $u$ in 
the $\varepsilon$-neighborhood of $\omega$.   Namely $\hat h_{j,+}((\sqrt{t}\hat y, 0), e^{-iu}(\sqrt{t}\hat y, 0))
\equiv\hat h_{j,+}((\sqrt{t}\hat y, 0), e^{-i\omega}(\sqrt{t}\hat y, 0))$ for $u$ in the $\varepsilon$-neighborhood 
of $\omega$.  

Now $x_0=(0,0)$ is a fixed point of 
$e^{-i\omega}$ ($\omega=\frac{2\pi}{p_{\ell}}$).  
One can see that $T_{x_0}X$ under the isotropy action induced by $e^{-i\omega}$ decomposes as 
an orthogonal direct sum of eigenspaces (where $N(S/M)$ denotes the normal bundle of a
submanifold $S$ in an ambient manifold $M$ with $N_p(S/M)$ the fiber of $N(S/M)$ at $p$, 
and $\overline X_{p_\mu}=\overline X_{p_{\mu(\gamma_\mu)}}$ for some $\gamma_\mu=1,2,\cdots, s_{\mu}$;
$\mu=1,2,\cdots, \ell$)
$$T_{x_0}\overline X_{p_{\ell}}, 
N_{x_0} ({\overline X_{p_{\ell}}/\overline X_{p_{\ell-1}}}), N_{x_0}({\overline X_{p_{\ell-1}}/\overline X_{p_{\ell-2}}})\ldots, N_{x_0}({ \overline X_{p_2}/\overline X_{p_{1}}})$$
associated with eigenvalues 
\begin{equation}
\label{e-gue160417aae2a}
1, e^{i\omega p_{\ell-1}},  e^{i\omega p_{\ell-2}},
\ldots,  e^{i\omega p_{1}}\quad\mbox{respectively.}
\end{equation}
For instance, assume $\ell=2$ and take $g=e^\frac{2\pi i}{p_2}$.   Set $q=\frac{p_2}{p_1}\,(\in\mathbb{N})$.  
On $N_{x_0}(\overline X_{p_2}/\overline X_{p_1})$, $g\ne \mathrm{id}$ and 
$g^q=\mathrm{id}$.  Hence $v\in N_{x_0}(\overline X_{p_2}/\overline X_{p_1})$
is rotated by the angle $\frac{2\pi}{q}$ which is $\omega p_1$.  

The goal in what follows is to prove 
the claim that for $q=1, \ldots, \ell-1$,
\begin{equation}
\label{e-gue160417c1}
\lim_{t\To0^+}\frac{\hat h_{j,+}((\sqrt{t}\hat y, 0), e^{-i\omega}(\sqrt{t}\hat y, 0))}{t}=
\left\vert e^{i\omega p_{\ell-q}}-1\right\vert^2
||\hat y||^2\quad\mbox{for $\hat y\in 
N_{x_0}(\overline X_{\ell-q+1}/\overline X_{\ell-q})$}
\end{equation}
or equivalently, in the notation above (see \eqref{e-gue160417aae2}) 
\begin{equation}
\label{e-gue160417c2}
\lim_{t\To0^+}\frac{h_{j,+}(z_0, z_1)}{t}=\left\vert e^{i\omega p_{\ell-q}}-1\right\vert^2
||\hat y||^2 \quad\mbox{(for $\hat y\in 
N_{x_0}(\overline X_{\ell-q+1}/\overline X_{\ell-q})$)}
\end{equation}
where $||\cdot||$ denotes the norm with respect to
the metric tensor of $X$ at $x_0$. 

Our proof of claim \eqref{e-gue160417c1} is based on the following sequence of 
lemmas.  

\begin{lem}
\label{l-gue160417} In the notation above, for  $\hat y\in 
N_{x_0}(\overline X_{\ell-q+1}/\overline X_{\ell-q})$ we have 
$$\lim_{t\to 0^+}\frac{d_X^2((\sqrt{t}\hat y, 0), e^{-i\omega }(\sqrt{t}\hat y, 0))}{t}=\left\vert 
e^{i\omega p_{\ell-q}}-1\right\vert^2 ||\hat y||^2$$
where $d_X$ denotes the distance on $X$. 
\end{lem}

\begin{proof}    On $T_{(0,0)}X$ the action induced by $e^{-i\omega}$ 
rotates the tangent vector $\hat y$ by the angle $\omega p_{\ell-q}$.  
Hence 
the lemma follows from the well known fact that in a Riemannian manifold $(M,g)$, 
if $a$, $b$ in $M$ are the images of $A$ , $B$ in $T_pM$ 
by the exponential map at $p\in M$, then as $(a, b)\To(p,p)$
\begin{equation}
\label{e-gue160417lem1e1}\lim\frac{d_M(a, b)}{||A-B||}\To 1
\end{equation}
where $||\cdot||$ is $g$ at $p$ (cf. \cite[Proposition 9.10]{Hel}).  
\end{proof}

\begin{slem}\label{l-gue160417lem3}
Suppose $N$ is a Riemannian submanifold of a Riemannian manifold $M$.
Then the respective distance functions on $M$ and on $N$ are infinitesimally the same.
More precisely, suppose in $N$, $p_n\ne q_n$ for all $n\in \mathbb{N}$,
and $p_n, q_n\to p$ as $n\to \infty$ for a given point $p\in N$.  
Then $\lim_{n\to \infty}\frac {d_M(p_n,q_n)}{d_N(p_n,q_n)}=1$.   Moreover, suppose 
$t_{n, M}$ (resp. $t_{n,N}$) in $T_{p_n}M$ are the unit tangent vectors along which the minimal geodesics 
in $M$ (resp. $N$) join $p_n$ and $q_n$.   Then $\lim_n(t_{n, M}-t_{n, N})=0$.  
\end{slem}

\begin{proof}  Suppose the special case $p_n=p$ for all $n$.  
Let $\gamma_n$ be a geodesic (with unit speed) of $N$ joining $p$ and $q_n$, and 
$\beta_n=\mathrm{exp}_p^{-1}(\gamma_n)\subset T_pM$.  Write 
$l_n(t)$ for the length of (part of) $\beta_n$ (with the parameter going from $0$ to $t$) 
measured with the metric $g_{ij}=1+O(|x|^2)$ in geodesic coordinates (at $p$).  
Write $||v||$ for the Euclidean norm of a vector $v\in T_pM$ expressed in geodesic coordinates.  
Given a curve $\beta(t)\subset T_pM$, $\beta(0)=p$, $\dot\beta(0)\ne 0$, 
one sees the length function $l(t)=\int_0^t\sqrt{<\dot\beta(t), \dot\beta(t)>_{g_{ij}}}dt$
satisfies $\left\vert\frac{||\beta(t)||-l(t)}{||\beta(t)||}\right\vert=O(t)\le Ct$ for a (locally bounded) quantity $C$ 
which depends only, apart from $\beta$, on the
local geometry at $p$ (uniformly).    Clearly this implies the lemma if 
$q_n$ is assumed to approach $p$ along a given geodesic $\gamma$ of $N$.   If $q_n$ approaches
$p$ along different geodesics $\gamma_n$, since these geodesics can be uniformly controlled 
by the local geometry around $p$, the same results hold as well. 
For the general case where $p_n$ are different, the similar argument using the control by local geometry implies 
$\left\vert\frac {d_M(p_n,q_n)-d_N(p_n,q_n)}{d_M(p_n,q_n)}\right\vert\le C(d_M(p_n,q_n))$. 
The assertion about the unit tangent vectors can be proved similarly. 
\end{proof}

\begin{slem}\label{l-gue160417lem4}  Let $N$ be a differentiable manifold equipped with 
two Riemannian metrics $g$ and $h$, and $p\in N$.   Assume that $g(p)=h(p)$. 
Suppose that in $N$, $p_n\ne q_n$ for all $n\in \mathbb{N}$ such that $\lim_n p_n=\lim_n q_n=p$.  
Then $\lim_{n\to \infty}\frac {d_g(p_n,q_n)}{d_h(p_n,q_n)}=1$ where $d_{\bullet}$ 
denotes the (metric-dependent) distance function on $N$.   
\end{slem}

\begin{proof} The result is local; assume $N\subset\mathbb{R}^n$ as an open subset.  
By comparison to a fixed Euclidean metric, we assume $g$ is Euclidean inherited from $\mathbb{R}^n$.
By reasoning similar to the previous sublemma, it is seen that for $n>>1$, $d_h(p_n, q_n)$ is basically 
$(1\pm C\max\{||p_n-p||, ||q_n-p||\})d_g(p_n, q_n)$ (where $||\cdot||$ denotes the Euclidean norm)
with a uniform bound $C$.  The assertion follows.  
\end{proof}

The last lemma (as our main lemma) is as follows.  (This lemma can be viewed as  a sharp version 
of the important claim \eqref{e-gue160327I} in the proof of Theorem~\ref{t-gue160123}, which 
bears upon the reason why our distance function $\hat d$  arises.) 

\begin{lem}
\label{l-gue160417lem2}  In the previous notation, 
write $p=(z(p), \theta(p))$ and $q=(z(q), \theta(q))$ in $(z, \theta)$ coordinates on the BRT chart $\hat D_j=
\hat U_j\times]-\frac{\zeta}{2},\frac{\zeta}{2}[$.   We omit the subscript $j$ in what follows. 
 Let $S=S^1\circ p$ be the $S^1$-orbit of $p$ and $N(S/X)$ be the normal bundle of $S$ in $X$
identified with the orthogonal complement of $TS$ in $TX|_{S}$.  
Suppose $p_n\ne q_n$ for all $n$ and $p_n, q_n\to p\in X$ as $n\to \infty$ such that
$D_n=\mathrm{exp}_{X,p}^{-1}(p_n), A_n=\mathrm{exp}_{X,p}^{-1}(q_n)\in N_p(S/X)$. 
In the case where $p_n\ne p$ and $q_n\ne p$ (all $n$), 
suppose the angle at $p$ given by the vectors $D_n$ and $A_n$ are bounded away from $0$
as $n\to \infty$.  
Then (see \eqref{e-gue150823e1} for the metric on $U$ giving $d_U$ below) 
\begin{equation}
\label{e-gue160417lem2e1}
\lim_{n\to\infty}\frac{d_X(p_n,q_n)}{d_U(z(p_n),z(q_n))}=1.
\end{equation}
In particular $z(p_n)\ne z(q_n)$ for $n$ large.  
\end{lem}

\begin{proof}  In this proof we take the same notation $U=\hat U_j$ seated as an embedded submanifold of $X$ 
with $\theta=0$.   As in (the proof of) Sublemma~\ref{l-gue160417lem3}, we first assume $p_n=p$ for all $n$. 
By applying the $S^1$ isometries we assume $p=(z(p),0)\in U$. 
By using the construction of our rigid metric 
on $X$ (cf. \eqref{e-gue22a1} and \eqref{e-gue150823e1}) 
it may be assumed that at $p$, $N_p(S/X)=T_pU$.  
To see this, by the geometrical interpretation of BRT transformations
in (the proof of) Proposition~\ref{l-gue150524d} one can adjust the BRT coordinates 
such that $d\phi(p)=0$ (similar to the well known fact that for a hermitian metric $h$ of a holomorphic line bundle
$L$ on a complex manifold, at any given point $p$ one has $dh(p)=0$ up to a change of 
local frames of $L$).  This gives 
$T_p^{0,1}D=\{\frac{\partial}{\partial\overline z_j}+i\frac{\partial\phi}{
\partial\overline z_j}(z)\frac{\partial}{\partial\theta};\,
j=1,2,\ldots,n\}=\{\frac{\partial}{\partial\overline z_j};\,
j=1,2,\ldots,n\}$ (and $T^{1,0}D=\overline{T^{0,1}D}$)
(cf. {\it loc. cit.}).   It easily follows $N_p(S/X)=T_pU$ as claimed.  A word of caution is in order.   The (intrinsic) geometrical 
interpretation for BRT charts (cf. {\it loc. cit.} and remarks after \eqref{e-gue150823e1}) shows that the asserted 
\eqref{e-gue160417lem2e1} is independent of choice of 
BRT coordinates (on that particular BRT neighborhood).   On the other hand, $U$ considered as an embedded 
submanifold of $X$ as done above does depend on the choice of BRT coordinates.  

The first reduction step is as follows. 
$U$ is endowed with another (Riemannian) metric inherited, as a submanifold, from that of $X$.
This is different from the metric originally defined (on $U$, cf. \eqref{e-gue150823e1}),
yet the two metrics coincide at $p$ as can be seen above.   
By Sublemma~\ref{l-gue160417lem4}
it is enough to prove \eqref{e-gue160417lem2e1} with this inherited metric.    Without the danger of confusion
we shall adopt the same notation $d_U(\cdot, \cdot)$ for the new distance function in what follows. 

Fix an $n$ and set
$q=q_n$, $q'=e^{i\theta(q)}\circ q=(z(q),0)\in U$. 
Put $A=\mathrm{exp}_{X,p}^{-1}(q), B=\mathrm{exp}_{X,p}^{-1}(q')\in T_pX$,  
so $||A||=d_X(p, q), ||B||=d_X(p, q')$ (where $T_pX$ is equipped with the Euclidean metric $||\cdot||$ in 
geodesic coordinates). 
We are going to prove, as $n\to \infty$,
\begin{equation}
\label{e-gue160417lem2f1}
\lim\frac{||A||}{||B||}=1. 
\end{equation}  
One has $d_X(p, q')/d_U(p,q')\to 1$ by Sublemma~\ref{l-gue160417lem3}.  Hence, to prove \eqref{e-gue160417lem2e1}
for this special case $p_n=p$ is the same as to prove $d_X(p, q)/d_X(p, q')\to 1$ 
which is  \eqref{e-gue160417lem2f1} above.  

To see \eqref{e-gue160417lem2f1} (hence \eqref{e-gue160417lem2e1}), we first 
argue \eqref{e-gue1604172g1} below.  Let $L\subset T_pX$ be the line determined by $A, B$,
i.e. $L=\{A+t(B-A); t\in\mathbb{R}\}$.   Then 
\begin{equation}
\label{e-gue1604172g1}
\mbox{$L$ is approximately orthogonal to
$A$ and to $B$ (as $n$ large)}.
\end{equation}

Note that $A\in N_p(S/X)=T_pU$ from the condition of the lemma, and that $B$ is nearly lying 
on $T_pU$ (with a small angle between $B$ and $T_pU$) by using Sublemma~\ref{l-gue160417lem3}
on tangents.
Let $\Gamma_1=\{e^{i\theta}\circ q\}_{\theta\in [0, \theta(q)]}\subset X$ 
($\theta(q)\ge 0$, say) joining $q$ and $q'$
and $\Gamma=\mathrm{exp}_{X,p}^{-1}\Gamma_1\subset T_pX$ joining  $A$ and $B$. 
Recall that the vector field $T$ induced by the $S^1$ action is orthogonal to $U$ at $p$ as mentioned above, 
hence $T_{q'}\Gamma_1\perp T_{q'}U$ approximately (as $n>>1$).    On the other hand, 
by $T_pS\perp T_pU$ 
and $\Gamma_1\approx S$ (as $n\gg 1$), one sees by $A\in T_pU$ that 
$T_A\Gamma\perp T_pU$ approximately 
(as vector subspaces in $T_pX$).  
In sum,  if $q, q'$ are close to $p$ (so $T_{q'}U$ close to $T_pU$), 
then $T_A\Gamma\perp T_pU$, $T_B\Gamma\perp  T_pU$ approximately (cf. the foliation argument below);
for this we write $\Gamma\perp A, B$ approximately.  
We are ready to prove \eqref{e-gue1604172g1}.
Pulling back the $S^1$ foliation locally around $p$ via the $\mathrm{exp}_{X, p}$ in the same way as $\Gamma$ obtained by $\Gamma_1$, 
there is a foliation $\mathcal{F}$ around $p$ in $T_pX$ in which (part of) $\Gamma$ lies as a leaf.  
Write $p\in \Gamma_0\,(\subset \mathrm{exp}_{X, p}^{-1}S)\,\in \mathcal{F}$.    As $n\gg 1$, the line $L$ determined by $A, B\in \Gamma$ tends to 
the tangent line $(=T_pS)$ to $\Gamma_0$ at $p$ (since the leaf $\Gamma$ of $\mathcal{F}$
tends to the leaf $\Gamma_0$).  Hence by using the uniform continuity
for $\mathcal{F}$ around $p$, $L$ is close to lines $\tilde L$ tangent to leaves $\tilde \Gamma$ of $\mathcal{F}$ 
if $\tilde \Gamma$ are nearby $\Gamma_0$ and $\tilde L$ nearby $T_pS$.  
In particular, $L$ is close to the tangent lines
$T_A\Gamma, T_B\Gamma$ (as $n\gg 1$).  Now that $\Gamma\perp A, B$ approximately as 
just shown, giving readily $T_A\Gamma\perp A, T_B\Gamma\perp B$ approximately, this in turn yields 
$L\perp A, B$ approximately $(n\gg 1)$, proving \eqref{e-gue1604172g1}. 

For $q'$ close to $p$, by simple Euclidean geometry (on $T_pX$), $||A-B||$ is rather small in comparison to 
$||A||$ and $||B||$ by using $L\perp A, B$ approximately \eqref{e-gue1604172g1}, 
i.e. $||A-B||=o(||A||), o(||B||)$.  By using law of cosines, 
one can obtain \eqref{e-gue160417lem2f1}, yielding the special case $p_n=p$ of the lemma. 
As this step appears crucial and will be instrumental to the general case, we prefer to supply 
some details as follows. 

Take a triangle with vertices $T_i$ $(i=1,2,3)$, angles $\alpha_i$ at $T_i$ and $\delta_i$ the length 
of the side facing $T_i$.  Suppose $\alpha_2\le \alpha_3$ and both $\approx \frac{\pi}{2}$.  
Set $\alpha_2=\frac{\pi}{2}-\alpha$,
$0<\alpha\ll 1$.  Let $D$ sit on the line $L_1$ determined by $T_2$ and $T_3$ such that the line $L_2$ determined by 
$T_1$ and $D$ is perpendicular to $L_1$.  Assume first that $D$ sits between $T_2$ and $T_3$.  
Then $\delta_1=\delta_3\sin\theta_1+\delta_2\sin\theta_2$ where $\theta_1, \theta_2$ (with
$\theta_1+\theta_2=\alpha_1$) are angles given by 
$L_2$ and the two sides at $T_1$.   Thus $\frac{\delta_1}{\delta_3}\le 2\sin\alpha_1$ ($\frac{\delta_2}{\delta_3}\le 1$
by $\alpha_2\le \alpha_3$).   If $D$ sits outside the triangle, then $\delta_1=\delta_3\sin\alpha-\delta_2\sin\theta_3$,
$\theta_3=\alpha-\alpha_1$, 
so $\frac{\delta_1}{\delta_3}\le \sin\alpha$.  One obtains $\frac{\delta_1}{\delta_3}\to 0$ if both 
$\alpha_1\to 0$ and $\alpha\to 0$.  By $\delta_1^2=\delta_2^2+\delta_3^2-2\delta_2\delta_3\cos\alpha_1$,
one has $$(1-\frac{\delta_2}{\delta_3})^2
=(\frac{\delta_1}{\delta_3})^2-\frac{2\delta_2}{\delta_3}(1-\cos\alpha_1)\le (\frac{\delta_1}{\delta_3})^2\le \sin\alpha+2\sin\alpha_1$$
giving $\frac{\delta_2}{\delta_3}\to 1$ if both $\alpha_2, \alpha_3\approx \frac{\pi}{2}$ 
(hence $\alpha, \alpha_1\approx 0$).  As said, this yields \eqref{e-gue160417lem2f1}.  

We draw some consequences in order for the general case.
If $\alpha_1\to 0$ (by $\alpha_2, \alpha_3\to{\pi}/2$), then the two sides at $\alpha_1$ 
are close to each other, i.e. 
$\lim (A/||A||-B/||B||)=0$ (as $q\to p$).
One sees that if $C=\mathrm{exp}_{U,p}^{-1}(q')\in T_pU$,
then by using \eqref{e-gue160417lem2f1}
and Sublemma~\ref{l-gue160417lem3} on tangents via $B$, 
\begin{equation}
\label{e-gue160417lem2e2}
a)\,\,\lim ||A||/||C||=1, \quad b)\,\,\lim (A/||A||-C/||C||)=0.
\end{equation}

We are ready to prove the general case $p_n\ne p$.  Write 
$D=\mathrm{exp}_{X,p}^{-1}(p_n), F=\mathrm{exp}_{U,p}^{-1}(p'_n)$ 
in the same way as $A=\mathrm{exp}_{X,p}^{-1}(q_n), C=\mathrm{exp}_{U,p}^{-1}(q'_n)$ above. 
With $D,  F$ in place of $A, C$ in \eqref{e-gue160417lem2e2} one has the same results for $D, F$: 
\begin{equation}\label{e-gue160417lem2e3}
a)\,\,\lim ||D||/||F||=1, \quad b)\,\,\lim (D/||D||-F/||F||)=0.
\end{equation}

In view of \eqref{e-gue160417lem1e1} one has $||D-A||/d_X(p_n,q_n)\to 1, ||F-C||/d_U(p'_n, q'_n)\to 1$.  
Hence to prove \eqref{e-gue160417lem2e1}, i.e. $d_X(p_n,q_n)/d_U(p_n',q_n')\to 1$,  
is the same as to show $\lim||D-A||/||F-C||=1$.   This is intuitively 
clear by \eqref{e-gue160417lem2e2}, \eqref{e-gue160417lem2e3} (which alludes to $A\approx C$ and $D\approx F$) 
provided that the angle given by the two vectors $D$ and $A$ at $p$ (hence by $F$ and $C$ at $p$, cf. b) of 
\eqref{e-gue160417lem2e2} and \eqref{e-gue160417lem2e3}) is not approaching zero.  This is precisely the condition given in 
the lemma.    For the rigor of this argument one may use law of cosines without difficulty.  
Hence the lemma follows. 
\end{proof}

\begin{proof}[proof of claim \eqref{e-gue160417c2}]
By combining Lemma~\ref{l-gue160417lem2} and Lemma~\ref{l-gue160417}
we can finish the proof of the claim \eqref{e-gue160417c2} provided that 
$h_{j}(z_1,z_2)=d^2_{U_j}(z_1,z_2)$.  But this is a standard fact for the heat kernels of Dirac and Laplacian type
(see \cite[Theorem 2.29]{BGV92}); see also the famous result of S. R. S. Varadhan\cite{Var67} 
for a generalization in this regard.  
\end{proof}

\noindent 
{\it proof of Proposition~\ref{t-gue160416b} resumed}. 

We are now ready to prove case b) of iii) of Proposition~\ref{t-gue160416b}. 
To work on the integral $I$ of \eqref{e-gue160416be1} we are going to refine the 
computation contained in \eqref{e-gue160417aaI}.  
Indeed, case b) of iii) can be obtained if one notes the following $\alpha)-\epsilon)$
(part of them similar to the proof of i) of this proposition):

$\alpha$) $\int_{-\infty}^{\infty}e^{-a^2x^2}dx=\frac{\sqrt\pi}{|a|}$; 

$\beta$) 
using \eqref{e-gue160417c1} (with $q=1, 2, \ldots, \ell-1$) for $e^{-\frac{\hat h_{j,s}^+}{t}}$ 
in \eqref{e-gue160416be1};

$\gamma$) change of variable $u=\psi+\frac{2\pi}{p_{\ell}}$ in \eqref{e-gue160416be1};

$\delta$) in \eqref{e-gue160416be1}, by rescaling $\hat y\to \sqrt{t}\hat y$, $\hat b_{j,s}^+
((\sqrt{t}{\hat y}, Y),e^{-iu}\circ (\sqrt{t}{\hat y}, Y))$ replacing $\hat b_{j,s}^+(x, e^{-iu}\circ x)$,
tends to $\hat b_{j,s}^+
((0, Y),e^{-iu}\circ (0, Y))= \hat b_{j,s}^+
((0, Y),e^{-i\psi}\circ (0, Y))$ because $(0,Y)\in \overline X_{p_{\ell}}$.  But $|\psi|$ being small $(\le \varepsilon)$, $e^{-i\psi}\circ (0, Y)$
does not change the $z( (0, Y))$\,($=z(Y)$) coordinate in $\hat D_j$, giving 
$ \hat b^+_{j,s}
((0, Y),e^{-i\psi}\circ (0, Y))=b^+_{j,s}(z(Y), z(Y))$ (up to cut-off functions);

$\epsilon$) as $t\to 0$, $\sigma_j(\eta)=\sigma_j(\theta(e^{-iu}\circ(0,Y)))=\sigma_j(\theta(e^{-i\psi}\circ (0, Y)))
=\sigma_j(\theta(Y)+\psi)$.  With $\eta=\theta+\psi$ and $u=\psi+2\pi/p_\ell$ in \eqref{e-gue160416be1}
and \eqref{e-gue150901e1}, a cancellation occurs for the three 
exponentials there; eventually a numerical factor $e^{-i2\pi m/p_\ell}$ is pulled out.  And instead of 
$\mathrm{Tr}\,\hat b_{j, s}^+$ in $I$ of \eqref{e-gue160416be1}, we are reduced to 
$\chi_j\mathrm{Tr}\,b_{j, s}^+$ (no ``hat" on $b_{j, s}^+$ here) 
as put down in this proposition.  

The formula for the coefficient $b_{s, \frac{e_{\ell}}{2}}^{(j)}$ of 
\eqref{e-gue160416bpe1} follows from $\alpha)-\epsilon)$ above. 

Finally, for $s=n$, it is well-known that 
(dropping $j$ here) $\mathrm{Tr}\,b^+_n(z, z)$ in the integral \eqref{e-gue160416bpe1} being the leading coefficient term in 
the asymptotic expansion 
of the ($\mathrm{Spin}^c$) Kodaira heat kernel, is constant in $z$ and equals 
$(4\pi)^{-n}\cdot \big(\mathrm{rk}(\bigwedge T^{*0,1}(U))\big)=(2\pi)^{-n}$
 (\cite[(a) of Theorem 4.4.1]{Gi95}, cf. \cite[Theorem 2.41]{BGV92}).  

\end{proof}

\subsection{Global angular integral}
\label{s-gue160416s2}

To work out the global version (i.e. the integration on $[0,2\pi]$) it is natural to 
consider not only (an $\varepsilon$-neighborhood of) ${2\pi}/{p_\ell}$ but also 
all their multiples $s{2\pi}/{p_\ell}$, $s\in\mathbb{N}, s\le p_\ell$. 
The analysis will thus partly depend on whether $s{2\pi}/{p_\ell}=s'{2\pi}/{p_{\ell'}}$ for some $s, s', p_{\ell}, p_{\ell'}$
or not.  One needs a systematic control of the behavior in this regard.  
Further, since the result will appear as a sum over these $\varepsilon$-neighborhoods, 
to organize this sum in a manageable way
is also desirable.    We shall now mainly deal with these issues in this subsection.  

There are minor duplication and perhaps discrepancy in notation between this subsection and the preceding one.
But, it would have appeared cumbersome if we had set up this generality right in the preceding subsection 
(as the resulting proof would have become much less illuminating).  

Recall the $S^1$-period of $X$ denoted by $\frac{2\pi}{p_1}>\frac{2\pi}{p_2}>\ldots >\frac{2\pi}{p_k}$
with $p_1|p_\ell$, $1\le \ell\le k$ (cf. Remark~\ref{r-gue150804}) and the stratum $X_{p_\ell}$ (the set of points of 
period $\frac{2\pi}{p_\ell}$) is a disjoint union of connected submanifolds 
$\bigcup_{1\le \gamma_\ell\le s_\ell} X_{p_{\ell(\gamma_\ell)}}$.  

\begin{defn}
\label{d-gue160416d2c}  Fix a smooth function $g(x)\not\equiv 0$ on $X$. 
We say that $g$ is of {\it small support} if the following conditions are satisfied.  
\begin{equation}
\label{e-gue160416dce1}\begin{split}
i)\,\,&\mathrm{Supp}\,g\subset\, X_{p_{i_1(\gamma_{i_1})}}\cup X_{p_{i_2(\gamma_{i_2})}}\cup\
\ldots\cup X_{p_{i_t(\gamma_{i_t})}}\cup X_{p_{i_{t+1}(\gamma_{i_{t+1}})}},\,\,1=i_1<i_2<\dots<i_{t+1}\le k \\
ii)\,\,&\mathrm{Supp}\,g\cap X_{p_{i_s(\gamma_{i_s})}}\ne \emptyset, \quad \forall s,\,\,1\le s\le t+1\\
iii)\,\,&\overline X_{p_{i_1(\gamma_{i_1})}}\supsetneq \overline X_{p_{i_2(\gamma_{i_2})}}\supsetneq
\ldots\supsetneq \overline X_{p_{i_t(\gamma_{i_t})}}\supsetneq \overline X_{p_{i_{t+1}(\gamma_{i_{t+1}})}}.
\end{split}
\end{equation}
\end{defn}

Obviously, given any $x_0\in X$ there exists a neighborhood $\Omega\ni x_0$ such that 
every nontrivial $g(x)\in C^{\infty}_0(\Omega)$ is of small support in the sense above. 

\begin{defn}
\label{d-gue160416d2a} Let $g(x)$ be a smooth function on $X$ of small support in the sense above, 
\eqref{e-gue160416dce1}.  Let $c\in \mathbb{N}$.    We define a number $i(c, g)=i(c)$ associated with $c$ and $g$ as follows.  

i) $i(c, g):=\ell\ge 2$ if the following is satisfied a) $c|p_\ell$, $\ell=i_s$ for some 
$s$, $2\le s\le {t+1}$ and b) $c\not|p_{i_{s'}}$ for all $s'<s$.  ii) $i(c, g):=1$ if $c|p_1$ $(p_1=p_{i_1})$.  
This is independent of $g$.  iii) $i(c, g):=\infty$ if $c\not|p_{i_s}$ for each $s$ with $1\le s\le t+1$.  
\end{defn}

It is easily seen that $p_{i(c)}|p_{i_s}$ for each $i_s$ with $i(c)\le i_s\le i_{t+1}$ if $i(c)\ne \infty$.
Indeed, $p_{i_1}|p_{i_2}|\ldots|p_{i_{t+1}}$ (cf. Remark~\ref{r-gue150804} for a similar case).  

By the above definition, one sees 
\begin{lem}
\label{l-gue160412d2a} Let $x\in \mathrm{Supp}\,g$, $i(c)\ne \infty$ and 
$h\in \mathbb{N}$ with  $(h, c)=1$.   It holds 
$e^{-i\frac{2\pi h}{c}}\circ x=x$ if and 
only if $x\in \overline X_{p_{i(c)}}$.  
\end{lem}

Let $h\in \mathbb{N}$ with  $(h, c)=1$ and $h<c$.  We 
consider the integral similar to \eqref{e-gue160416be1} for $i(c)\ne\infty$: 
\begin{equation}\label{e-gue160416d2e1}
I=I^{(j)}(p_{i(c)}, g(x), \frac{h}{c})\equiv \frac{1}{2\pi}\int^{\frac{2\pi h}{c}+\varepsilon}_{\frac{2\pi h}{c}%
-\varepsilon}\int_Xg(x)e^{-\frac{\hat h_{j,+}(x,e^{-iu}\circ x)}{t}}\mathrm{Tr}\,\hat
b^+_{j,s}(x,e^{-iu}\circ x)e^{-imu}dv_X(x)du.
\end{equation}
The above extends to the case $i(c)=\infty$ simply by formally setting $I^{(j)}(p_{i(c)=\infty}, g(x), \frac{h}{c})$
to be the integral to the right of \eqref{e-gue160416d2e1}.  


\begin{defn}
\label{d-gue160416bc}
i) Set $i(c)$ of Definition~\ref{d-gue160416d2a} to be $\ell$.   Assume $\ell\ne \infty$.  
Define the type $\tau(I(p_{i(c)}, g(x), \frac{h}{c}))$ to be $\tau(I(p_\ell, g(x)))$ where 
$\tau(I(p_\ell, g))$ is given in i) of Definition~\ref{d-gue160416t1}.  
ii) The notions of {\it simple type} and {\it class} are defined similarly.  
(Note there is no trivial type for $\ell\ne \infty$.) 
iii) If $i(c)=\infty$, then $\mathrm{Supp}\,g\cap X_{a}=\emptyset$
where $X_a$ is the fixed point set of $a=e^{-i2\pi/c}\in S^1$ (whether $X_a$ is empty or not depends on $c$).  
In this case we define it to be the {\it trivial type} (cf. iii) of Definition~\ref{d-gue160416t1}).  
\end{defn}

\begin{rem}
\label{r-gue160416bcr}
The notion of type concerns only 
the local stratificaton (of the $S^1$ action) at $x_0$ around which $\Omega\supset\mathrm{Supp}\,g$ is a small 
neighborhood.   Thus, with a small open subset $\Omega\subset X$ one may associate 
the type $\tau(\Omega)$ without referring to any kind of integral.    
In the fourth subsection, we will basically adopt this viewpoint for our purpose.  
\end{rem}

First, one examines the case $i(c, g)=\infty$, which turns out to be inessential.  

\begin{lem}
\label{l-gue160416rl}  Let $c\in \mathbb{N}$.  
There exists a (finite) covering of BRT trivializations on $X$ with the following property. 
For any smooth function $g$ of small compact support (cf. Definition~\ref{d-gue160416d2c}), 
suppose $i(c, g)=\infty$ and $x_0$ with $g(x_0)\ne 0$, is given. Then there exist a small open set $\Omega\ni x_0$
and a small $\tilde\varepsilon>0$ such that for any $\chi \in C^{\infty}_0(\Omega)$, if one replaces $g$ by $\chi g$ in 
the integral $I$ of \eqref{e-gue160416d2e1}, this $I$ with any $0<\varepsilon\le\tilde\varepsilon$, equals $0$, or $O(t^{\infty})$
as $t\to 0^+$.
\end{lem}

\begin{proof}
By the definition of $i(c, g)=\infty$, if the fixed point set $X_a$ of $a=e^{-i2\pi/c}$ is 
empty, one chooses a covering of BRT trivializations $\{D_j\}_j$ (cf. lines above Subsection~\ref{s-gue160416s1})
such that 
$\hat D_j\cap a^h(\hat D_j)=\emptyset$ for all $j$ and $h$ with $(h, c)=1$.   
In this case the remaining argument by using the continuity of the 
$S^1$ action, is almost the same as ii) of Proposition~\ref{t-gue160416b}, yielding $I=0$.   If $X_a$ is not empty, 
one chooses any (finite) covering of BRT trivializations (such as the one given prior to Subsection~\ref{s-gue160416s1}). 
Then it may occur the extra case $a^h\circ x_0\in \hat D_j$ for some $\hat D_j\ni x_0$ if the choice of $g$ is such that 
$x_0$ is very near $X_a$.   In this case the remaining argument is essentially similar to Case a) of iii) of 
Proposition~\ref{t-gue160416b}, giving rise to $I=O(t^{\infty})$ as $t\to 0^+$.    
\end{proof}


To compute $I$ of \eqref{e-gue160416d2e1},
we assume the {\it simple type} condition for $I$ (when it is not of trivial type) as given in Definition~\ref{d-gue160416bc}.  
Combining Lemma~\ref{l-gue160416rl} we have the following corollary, as a generalization of Proposition~\ref{t-gue160416b}.  

\begin{cor}
\label{t-gue160416bc} Notations and the simple type condition as above.    Assume that the covering by BRT trivializations
satisfies Lemma~\ref{l-gue160416rl}.   Let $c\in \mathbb{N}$, $x_0\in X$, $\Omega\ni x_0$ an open subset 
and $g\in C^{\infty}_0(\Omega)$ (of small support as above).  Then the $\varepsilon>0$ (in $I$) and $\Omega$ 
can be chosen to satisfy the following.  
a) The same results (for computing $I=I^{(j)}(p_{i(c)}, g(x), \frac{1}{c})$ of \eqref{e-gue160416d2e1}) hold true 
as in Proposition~\ref{t-gue160416b} provided that one adopts the replacement of $e^{-i{2\pi}/{p}}, e^{-i{2\pi}/{p_\ell}}$ by 
$e^{-i{2\pi}/{c}}$ in i), ii) and iii) of the statement, 
and $\ell\mbox{\,(not the one in $\frac{2\pi}{p_\ell}$})$ by $
i(c)$ in Cases a), b) of iii) throughout (so $p_{\ell-q}\to p_{i(c)-q}$, $e_\ell\to e_{i(c)}, e_{\ell-q+1}\to e_{i(c)-q+1},  
e_{\ell-q}\to e_{i(c)-q}$ and $\gamma_\ell\to \gamma_{i(c)}, X_{\ell}\to X_{i(c)}, X_{\ell-q}\to X_{i(c)-q}$ etc.
in Case b)).  Note that after the replacement, $i(c)=\infty$ in Case a) and $i(c)\ne \infty$ in Case b), of iii). 

b) More generally, for $h\in \mathbb{N}$, $(h, c)=1$ and $h<c$, 
with the replacement $\frac{2\pi}{c}\to \frac{2\pi h}{c}$ (and $\ell$ not the one in $\frac{2\pi}{p_\ell}$
by $i(c)$ in Cases a), b) of iii)), the same results (for computing $I=I^{(j)}(p_{i(c)}, g(x), \frac{h}{c})$ of \eqref{e-gue160416d2e1}) hold true as well. 
\end{cor}

\begin{proof}[proof of Corollary~\ref{t-gue160416bc}] One sees that with the replacement of 
$\frac{2\pi}{p_{\ell}}$ by $\frac{2\pi}{c}$ or $\frac{2\pi h}{c}$, the condition on $c$ (in Definition~\ref{d-gue160416d2a}) renders the argument 
in proof of Proposition~\ref{t-gue160416b} essentially unchanged.  For instance, 
with \eqref{e-gue160416biiia} replaced by 
$\frac{2\pi h}{c}-\widetilde\theta_0=m\frac{2\pi}{p_{\ell-q}}$, taking $\zeta$ smaller does the job.
Further, with substitution of 
$\ell$ (not in $\frac{2\pi}{p_\ell}$) by $i(c)$, 
the distinct eigenvalues $\{1,e^{i\frac{2\pi}{p_{\ell}}p_{\ell-1}}, e^{i\frac{2\pi}{p_{\ell}}p_{\ell-2}},\ldots
e^{i\frac{2\pi}{p_{\ell}}p_1}\}$ of the isotropy action 
(of $e^{-i\frac{2\pi}{p_\ell}}$ at $x_0$) (cf. \eqref{e-gue160417aae2a}) are changed to 
$\{1,e^{i\frac{2\pi h}{c}p_{i(c)-1}}, e^{i\frac{2\pi  h}{c}p_{i(c)-2}},\ldots
e^{i\frac{2\pi  h}{c}p_1}\}$ (of $e^{-i\frac{2\pi h}{c}}$ at $x_0$) (which remain distinct).  
\end{proof}

Let $c\in \mathbb{N}$ and $g$ a smooth function on $X$ of 
small support as above, with $i(c)\,(=i(c, g))$ in Definition~\ref{d-gue160416d2a}.  We are going to associate 
certain numerical factors.   For a contrast, we will give them for cases of 
the simple type and the general type separately (cf. Definitions~\ref{d-gue160416t1} and 
~\ref{d-gue160416bc}).   

For the simple type, the numerical factor $d_c=d_{c, g, m}$ is set to be 
\begin{equation}\label{e-gue160416p2e1}\begin{split}
&i)\,\,\, i(c)=1\\
  \qquad\qquad \mbox{if\,\,\,}c>1, \quad d_c\,(=d_{c, g, m})\,&:=\sum_{h\in\mathbb{N}, h<c, (h, c)=1}e^{-\frac{2\pi i}{c}{hm}}; \qquad  \mbox{if\,\,\,} c=1, \,\,d_c:=1.\qquad \qquad \qquad \qquad\\
&ii) \,\,\,\infty>i(c)\ge 2\\
\qquad d_c\,(=d_{c, g, m})\,&:=(\sqrt{\pi})^{e_{i(c)}}\sum_{h\in\mathbb{N}, h<c, (h, c)=1}\frac{e^{-i\frac{2\pi h}{c}{m}}}
{\prod_{q=1}^{i(c)-1}\left\vert e^{i\frac{2\pi h}{c}{p_{i(c)-q}}}-1\right\vert^{e_{i(c)-q+1}-e_{i(c)-q}}} \qquad \qquad 
 \qquad \qquad\\
&iii)\,\,\, i(c)=\infty\\
\qquad d_c\,(=d_{c, g, m})\,&:=1.
\end{split}
\end{equation}

\begin{rem}
\label{r-gue160416p2r}  
For the general type, the $d_{c, g, m}$ for $i(c)\ge 2$ should be modified as follows. 

In notation of Definition~\ref{d-gue160416bc}, let it be given 
$\tau(I(p_{i(c)}, g, \frac{h}{c}))=({i_1(\gamma_{i_1})}, {i_2(\gamma_{i_2})},\ldots, i_f(\gamma_{i_f}),i_{f+1}(\gamma_{i_{f+1}}))$ where $i_1=\gamma_{i_1}=1, i_{f+1}=i(c)\ne\infty$, say.  (In the previous Definition~\ref{d-gue160416t1}, $i_{f+1}=\ell$.   Here we have $i_{f+1}=i(c)$.) 

One sees that the eigenvalues of the isotropy action of $e^{-\frac{2\pi i}{c}{h}}$ 
(at $x_0\in \overline X_{p_{i(c)(\gamma_{i(c)})}}$) 
are: 
\begin{equation}
\label{e-gue160416p2re1}
e^{i\frac{2\pi h}{c}{p_{i_1}}},\, e^{i\frac{2\pi h}{c}{p_{i_2}}}, \ldots,\, e^{i\frac{2\pi h}{c}{p_{i_f}}},\, 1 
\end{equation}
(because by writing $e^{-\frac{2\pi i}{c}{h}}= (e^{-\frac{2\pi i}{p_\ell}})^j$ if $h/c=j/p_\ell$ with $\ell=i(c)$, the
eigenvalues are $\lambda^j$ where 
$\lambda=e^{i\omega p_{i_1}}, e^{i\omega p_{i_2}},\ldots$, $\omega=\frac{2\pi }{p_\ell}$, cf. \eqref{e-gue160417aae2a}) 
with multiplicities (where $e_{i_1(\gamma_{i_1})}=0$) 
\begin{equation}
\label{e-gue160416p2re1}
 e_{i_2(\gamma_{i_2})}-e_{i_1(\gamma_{i_1})}, \ldots, \,
e_{i_f(\gamma_{i_f})}-e_{i_{f-1}(\gamma_{i_{f-1}})},\,e_{i(c)(\gamma_{i(c)})}-e_{i_f(\gamma_{i_f})},\,
\mathrm{dim}\,X-e_{i(c)(\gamma_{i(c)})}. 
\end{equation}

Write, for $1\le r\le f$,
\begin{equation}\label{e-gue160416re2}
\mbox{$\lambda_r:=e^{i\frac{2\pi h}{c}{p_{i_r}}}$;  \quad
$m_r:=e_{i_{r+1}(\gamma_{i_{r+1}})}-e_{i_{r}(\gamma_{i_{r}})}$\, (with 
$i_{f+1}:=i(c)$)}.
\end{equation}

{\bf Numerical factors $d_c$ for the general type.} Given 
$\tau(I)=({i_1(\gamma_{i_1})}, {i_2(\gamma_{i_2})},\ldots, i_f(\gamma_{i_f}),  i_{f+1}(\gamma_{i_{f+1}}))$ of general type, $I=I(p_\ell, g(x), \frac{h}{c})$ with $i(c)=\ell$, the numerical factors $d_c$ similar to \eqref{e-gue160416p2e1} 
are defined as follows.  
   
If $c$ is of $i(c)=1$ or $\infty$, then $d_c\,(=d_{c,g, m,I})$ is the same as $d_c$ in i), iii) of \eqref{e-gue160416p2e1}. 
For $c$ with $\infty>i(c)\ge 2$, 
\begin{equation}
\label{e-gue160416p2re2} \begin{split}&\infty>i(c)\,(=i(c, g))\,\ge 2, \quad 
d_c\,(=d_{c,g, m, I}=d_{c,g, m, \tau(I)}=d_{c,g, m,\tau(\Omega)})\\
&\qquad\qquad\qquad\qquad \qquad\quad\,\,\,
:=(\sqrt{\pi})^{e_{i(c)(\gamma_{i(c)})}}\sum_{h\in\mathbb{N}, h<c, (h, c)=1}\frac{e^{-i\frac{2\pi h}{c}{m}}}
{\prod_{r=1}^{i_f}\left\vert \lambda_r-1\right\vert^{m_r}}
\end{split}
\end{equation}
where we recall $e_{i(c)(\gamma_{i(c)})}=\mathrm{codim}\,X_{p_{i(c)(\gamma_{i(c)})}}$ 
and we refer $\tau(\Omega)$ to Remark~\ref{r-gue160416bcr}. 

Note that the $\Omega$ $(\Supset\mathrm{Supp}\,g)$ is chosen to be small enough so that
there is no mixing of types.   The factor $d_c$ of \eqref{e-gue160416p2re2} is well defined. 

\end{rem}

We turn now to the global version (over the entire $[0, 2\pi]$) of the integral $I$ (of \eqref{e-gue160416be1} 
or \eqref{e-gue160416d2e1}).

Let 
\begin{equation}\label{e-gue160416ce1}
J_s=J_s^{(j)}=J_{s,m}^{(j)}(g(x))\equiv \frac{1}{2\pi}\int_Xg(x)e^{-\frac{\hat h_{j,+}(x,e^{-iu}\circ x)}{t}}
\mathrm{Tr}\,\hat b^+_{j,s}(x,e^{-iu}\circ x)e^{-imu}dv_X(x)
\end{equation}
where $g(x)\in C^{\infty}_0(\Omega)$ with $\Omega$ satisfying Proposition~\ref{t-gue160416b}.
We assume $g(x)$ is of small support in the sense of Definition~\ref{d-gue160416d2c}. 

We are going to compute $\int_0^{2\pi}J_sdu$, a $[0,2\pi]$-integral version 
of the integral $I$ in \eqref{e-gue160416be1}.  

To this aim, the main tool is the corollary below which is a reformulation of Corollary~\ref{t-gue160416bc}.   
But prior to this, let's 
write, for $x_0\in X, \,g(x)\in C_0^{\infty}(\Omega)$ with $\Omega$ a small neighborhood at $x_0$, 
\begin{equation}\label{e-gue160416bce2}\begin{split}
&\mbox{$\ell\ge 1$,}\quad I^{+,(j)}_{\ell(\gamma_\ell),s}\,(=I_{\overline X_{p_{\ell(\gamma_\ell)}},s,m}^{+,(j)}(g(x)))\,
\quad\mbox{($z(Y)=z(x)$ below; $b^+_{j, s}=b^+_{j, s, m}$ without ``hat" over it)}\\
&\qquad\quad \equiv 
\frac{1}{2\pi}\int_{-\varepsilon}^{\varepsilon}\int_{\overline X_{p_{\ell(\gamma_\ell)}}}
g(Y)\chi_j(Y)\mathrm{Tr}\,b_{j,s}^+(z(Y), z(Y))\tau_j(z(Y))\sigma_j(\theta(Y)+u)
dv_{\overline X_{p_{\ell(\gamma_\ell)}}}(Y)du\\
&\mbox{$\ell=1$,}\quad
I^{+,(j)}_{1,s}\,(=I^{+,(j)}_{X,s, m})\,=I_{\overline X_{p_{1}},s}^{+,(j)}(g(x))=I_{X,s}^{+,(j)}(g(x))
\\
&\mbox{$d_c=d_{c,g,m,\tau(\Omega)}$ cf. \eqref{e-gue160416p2re2}},
\quad\mbox{also written as $d_{c,g,m,I}$ with $I=I_{i(c)(\gamma_{i(c)}), s}^{+,(j)}$.}
\end{split}
\end{equation}

{\it The proof of the previous results in the case of simple type remains basically unchanged
for the case of general type}.   With the numerical factor $d_c$ introduced 
in \eqref{e-gue160416p2e1} and \eqref{e-gue160416p2re2}, one sees
the following.  

\begin{cor}
\label{t-gue160416bc2}
Notations as above with the general type $\tau(I)$ allowed (cf. Definitions~\ref{d-gue160416t1}
and~\ref{d-gue160416bc}).  Assume that the covering by BRT trivializations $\{\hat D_j\}_j$ satisfies
Lemma~\ref{l-gue160416rl}.  Let $c\in\mathbb{N}$ and $x_0\in X$.  For an $\varepsilon$, write $\lambda_c:= 
\bigcup_{h\in\mathbb{N}, h<c, (h, c)=1}]\frac{2\pi h}{c}-\varepsilon, \frac{2\pi h}{c}+\varepsilon[$
for $c> 1$ and $\lambda_1:=]-\varepsilon, \varepsilon[$ for $c=1$.   Write $\Omega\subset X$
for an open subset with $x_0\in \Omega$.   Then the
$\varepsilon>0$ and $\Omega$ can be chosen to satisfy the following.   (Recall that with respect to $\{\hat D_j\}_j$ we write 
$J_s^{(j)}=J_s^{(j)}(g)$ for any given $g\in C^{\infty}_0(\Omega)$  of small support in the sense 
of Definition~\ref{d-gue160416d2c}.)  

i) Suppose $x_0\in X_{p_1}$.  Then 
\begin{equation*}
\begin{split}
&\mbox{if \,\, $i(c)=1$,}\quad \int_{\lambda_c}J^{(j)}_sdu
=d_cI^{+,(j)}_{X, s};\quad 
\\
&\mbox{if \,\, $i(c)\ge 2$ \,(giving $i(c)=\infty$ here),}\quad \int_{\lambda_c}J^{(j)}_sdu= 0 \mbox{\,\, or \,$\sim O(t^{\infty})$}
\quad \mbox{(as 
$t\To0^+$)}.
\end{split}
\end{equation*}

The $\ell\ge 2$ in the following ii) and iii) is such that 
$x_0\in X_{p_\ell}$ (so $x_0\in X_{p_{\ell(\gamma_\ell)}}$, 
for some $\gamma_\ell$).  

ii) Assume $e^{-\frac{i2\pi h}{c}}\circ x_0\not\in \hat D_j$ (giving $i(c)=\infty$).  Then 
$\int_{\lambda_c}J^{(j)}_sdu=0$. 

iii) Assume $e^{-\frac{i2\pi h}{c}}\circ x_0\in \hat D_j$.
Then (as $t\To0^+$)
\begin{equation*}
\begin{split}  &\mbox{if \,\, $i(c)\ge \ell+1$ \,(giving $i(c)=\infty$ here)},\quad
\int_{\lambda_c}J^{(j)}_sdu=0 \mbox{\,\, or \,$\sim O(t^{\infty})$};\\
& \mbox{if \,\, $2\le i(c)\le \ell$,}\quad
\int_{\lambda_c}J^{(j)}_sdu\sim d_cI^{+,(j)}_{i(c)(\gamma_{i(c)}),s}
\sqrt{t}^{e_{i(c)(\gamma_{i(c)})}}+O(\sqrt{t}^{e_{i(c)(\gamma_{i(c)})}+1});\\
&\mbox{if \,\, $i(c)=1$,}\quad
\int_{\lambda_c}J^{(j)}_sdu=d_cI^{+,(j)}_{X,s},
\end{split}
\end{equation*}
where we note $d_c=d_{c,g,m,I}$ in \eqref{e-gue160416bce2} with $I=I_{i(c)(\gamma_{i(c)}), s}^{+,(j)}$.  
\end{cor}

We are now ready to compute $\int_0^{2\pi}J_sdu$.  

Let $\lambda_c$ be as in Corollary~\ref{t-gue160416bc2} (with a given $g(x)$ of small support). 
\begin{equation}\label{e-gue160416de1} \begin{split}
&S_1= \{c\in\mathbb{N};  \,\,i(c, g)=1\}, \qquad S_2=\{c\in\mathbb{N};\,\, \mbox{$\infty\ge i(c, g)\ge 2$ and $c|p_\ell$ for 
some $\ell$, $2\le \ell\le k$}\}\\ 
&\Lambda=\Lambda_1\bigcup\Lambda_2\quad\mbox{with}\quad
\Lambda_1=\bigcup_{\scriptstyle c\in S_1\atop\scriptstyle  }\lambda_c,\,\,
\Lambda_2=\bigcup_{\scriptstyle c\in S_2 \atop\scriptstyle }\lambda_c; 
\qquad N=[0,2\pi]\setminus \Lambda.  
\end{split}
\end{equation}


By using \eqref{e-gue160416de1}
and corollaries above one has 
the following.   


\begin{prop}
\label{t-gue160416c}  Suppose we are given any $\delta>0$ and an $s=n, n-1,\ldots$. 
Then there exists a (finite) covering $\{\hat D_j\}_j$ of $X$ by BRT charts that satisfy Lemma~\ref{l-gue160416rl}
and the following.   Suppose $x_0\in X$ and $\Omega$ a neighborhood of $x_0$
with $\Omega\Subset\hat D_j$ for every $\hat D_j\ni x_0$.  
Then there exist an $\varepsilon>0$ and a choice of $\Omega$ above such that for $
g(x)\in C^\infty_0(\Omega)$ (of small support in the sense of Definition~\ref{d-gue160416d2c}), 
writing $J^{(j)}_s=J^{(j)}_{s,m}(g(x))$ (for any $\hat D_j\ni x_0$, $s=n, n-1,\ldots$ as above) 
one has the following.  

i) $\int_{\Lambda_1}J_s^{(j)}du=\big(\sum_{q=1}^{p_1} e^{-\frac{i2\pi qm}{p_1}}\big)I^{+,(j)}_{X,s}$.

ii) Case a) If $x_0\in X_{p_1}$, $\int_{\Lambda_2}J^{(j)}_sdu\sim  0$ or $O(t^{\infty})$ (as $t\To0^+$). 

Case b)  If $x_0\in X_{p_\ell}$ with $\ell\ge 2$, 
\begin{equation*}
\int_{\Lambda_2}J^{(j)}_sdu\sim\big(\sum_{\scriptstyle c\atop\scriptstyle 2\le i(c)\le \ell}
(d_cI^{+,(j)}_{i(c)(\gamma_{i(c)}), s}\sqrt{t}^{e_{i(c)(\gamma_{i(c)})}}+O(\sqrt{t}^{e_{i(c)(\gamma_{i(c)})}+1}))
\big)+O(t^{\infty})\,\,\mbox{(as $t\To0^+$)}
\end{equation*}
where $d_c=d_{c,g,m,I}, I=I_{i(c)(\gamma_{i(c)}), s}^{+,(j)}$, cf. \eqref{e-gue160416p2re2} and \eqref{e-gue160416bce2}.  

iii) For the part $\int_NJ^{(j)}_s(g)du\equiv \eta_s^{(j)}(g)$, the estimate holds $|\eta_s^{(j)}(g)|<\delta\cdot \int_X|g(x)|dv_X(x)$. 
\end{prop}

\begin{proof}  To see i),  by i) and iii) of Corollary~\ref{t-gue160416bc2} 
it suffices to note $\sum_{c, i(c)=1}d_c=\sum_{q=1}^{p_1} e^{-\frac{i2\pi qm}{p_1}}$
by \eqref{e-gue160416p2e1}.   Case a) of ii)  follows directly from i) of {\it loc. cit.} whereas 
Case b) from iii) of {\it loc. cit.} by writing $\int_{\Lambda_2}J_sdu
=\sum_{2\le i(c)\le \ell}\int_{\lambda_c}J_sdu+\sum_{i(c)\ge \ell+1}\int_{\lambda_c}J_sdu$. 

To see iii), note the following.  First, the action by $e^{-i\theta}$ with $\theta\in N$ 
is fixed point free on $X$.    Each point in $X$ has a distinguished neighborhood $\hat\Omega$ 
such that if $x\in \hat\Omega$ and $\theta\in N$,  $e^{-i\theta}\circ x \not\in \hat\Omega$
by using continuity of the $S^1$ action.  
By compactness,  we can now assume (reset) a 
(finite) covering of BRT charts consisting of these $\hat\Omega$ and satisfying Lemma~\ref{l-gue160416rl}
(as there are only finitely many $c$ here).      
Given $x_0\in X$, we take a small neighborhood $\Omega\Subset\hat D_j$ for every $\hat D_j\ni x_0$.  
One sees that by the cut-off function $\tau_j$ involved in the integrand of $J^{(j)}_s$, 
$J_s^{(j)}\equiv 0$ identically (since $\tau_j({e^{-i\theta}\circ x})=0$ if $e^{-i\theta}\circ x\not\in \hat D_j$ 
cf. \eqref{e-gue150901}, which holds for $x\in \hat\Omega=\hat D_j$ and $\theta\in N$).     

However, a word of warning is in order.  Given the preceding covering of certain BRT charts, our integral $J_s^{(j)}$ 
shall be computed with respect to this new covering (because as said above, the covering has been ``reset")
which we just did.   However, it is not necessarily true that the original $\varepsilon$ 
(of Corollary~\ref{t-gue160416bc2} used above) remains 
altogether applicable in the new setting.   Namely, 
under the new BRT covering we need to choose 
an $\varepsilon_1$, possibly smaller than $\varepsilon$, to ensure that the original argument of the proof 
(of Proposition~\ref{t-gue160416b}) go through well.  
To see what the above means, let's make the dependence on the parameter $\varepsilon$ explicit 
and denote by $\Lambda(\varepsilon)$, 
$N(\varepsilon)$ etc.   The $N$ in the last paragraph will be denoted by $N(\varepsilon)$.
Only with the replacement by $\Lambda_1(\varepsilon_1), \Lambda_2(\varepsilon_1), \Lambda(\varepsilon_1)$ ($\subset \Lambda_1(\varepsilon), \Lambda_2(\varepsilon), \Lambda(\varepsilon)$ respectively), 
$N(\varepsilon_1)\,(\supset N(\varepsilon))$ and the new BRT covering throughout the present 
proposition, can we obtain the corresponding $\varepsilon_1$-versions of i) and ii) of this proposition.     And iii) will 
correspondingly be replaced by $\int_{N(\varepsilon_1)}J^{(j)}_sdu$.   But as just shown in the last paragraph, 
$\int_{N(\varepsilon)}J^{(j)}_sdu=0$, it follows $\int_{N(\varepsilon_1)}J^{(j)}_sdu=\int_{N(\varepsilon_1)\setminus N(\varepsilon)}J^{(j)}_sdu$.   Now the measure of $N(\varepsilon_1)\setminus N(\varepsilon)$ is controlled 
by $C\cdot(\varepsilon-\varepsilon_1)$ for a fixed constant $C$; the integrand of $J_s^{(j)}$ can be 
bounded in a way independent of $\varepsilon, \varepsilon_1$ and the BRT covering.  By a choice of 
a sufficiently small $\varepsilon$ beforehand,  iii) (for $\varepsilon_1$-version) follows.  We have proved 
the proposition with $\varepsilon_1$ (in place of the previous $\varepsilon$).   
\end{proof}

\subsection{Patching up angular integrals over $X$; proof for the simple type}
\label{s-gue160416s3}
We are going to study the main issue 
\begin{equation}
\label{e-gue160416ce1}
\int_X\mathrm{Tr}\,a^+_{s}(t, x, x)dv_X(x)
\end{equation}
where we recall (by \eqref{e-gue150901a}) 
\begin{equation}  \label{e-gue150901s3e1}
\begin{split}
&a^+_s(t,x,y)\,(=a_{s,m}^+(t,x,y)) \,\,\,(\hat b^+_{j, s}=\hat b^+_{j, s, m}) \\
&=\frac{1}{2\pi}\sum^N_{j=1}\int^\pi_{-\pi}e^{-\frac{\hat
h_{j,+}(x,e^{-iu}\circ y)}{t}}\hat b^+_{j,s}(x,e^{-iu}\circ y)e^{-imu}du,\ \
s=n,n-1,n-2,\ldots.
\end{split}
\end{equation}
For this, one would like to patch up those integrals $\int_0^{2\pi}J_s^{(j)}du$ of the last subsection over $j$.  
However, $a_s^+(t, x, y)$ is not canonically defined by our method
and is in fact dependent on the choice of BRT charts.  
A direct study of it appears inefficient (unless one sticks to a fixed covering of BRT charts). 

It turns out to be more effective if instead, one studies its equivalence (cf. \eqref{e-gue150630gI} in the asymptotic sense): 
\begin{equation}
\label{e-gue160416cde1}
\int_X\mathrm{Tr}\,e^{-t\widetilde\Box^+_{b, m}}( x, x)dv_X(x)
\end{equation}
in which $e^{-t\widetilde\Box^+_{b, m}}(x, y)$ is of course independent of choice of BRT charts.  

Suppose a $\delta>0$ and an $s=n, n-1,\ldots$ are given.  Assume that the BRT covering $\{\hat D_j\}_j$ 
satisfies Proposition~\ref{t-gue160416c} in which 
by using compactness, one can find a (finite) covering $\{\Omega_\alpha\}_\alpha$ of $X$, 
$\Omega_\alpha\Subset \hat D_j$ if $\hat D_j\cap\Omega_\alpha\ne \emptyset$, and 
an $\varepsilon>0$ such that the conclusion i), ii) and iii) of that proposition
hold with each of these
 $\Omega_\alpha$ and this $\varepsilon$. 
As indicated in Proposition~\ref{t-gue160416b}, whenever necessary, one can shrink the size of $\Omega_\alpha$ without 
changing $\varepsilon$.    For $\rho=\frac{\varepsilon}{2}$, 
we assume (possibly after shrinking $\Omega_\alpha$ and using compactness) for each $\alpha, j$, 
and for some (possibly big) $m>1$, 
\begin{equation}
\label{e-gue160416cde3}
\mbox{$\theta$-coordinates of $\Omega_\alpha,\,\,\hat D_j$
lie inside of \,$[-\rho, \rho], \,\,[-m\rho, m\rho]$\, respectively.} 
\end{equation}

Let $\{g_\alpha(x)\}_\alpha$ be a partition of unity {\it subordinate} to this covering 
(i.e. $\mathrm{Supp}\,g_\alpha\Subset\Omega_\alpha$).  We further assume each $g_\alpha$ is 
of small support in the sense of Definition~\ref{d-gue160416d2c}.  
One sees that as $t\To0^+$ 
\begin{equation}
\label{e-gue160416cde4}
\int_X g_\alpha(x)\mathrm{Tr}\,e^{-t\widetilde\Box^+_{b, m}}(x, x)dv_X(x)\sim \sum_j\sum_{s=n, n-1,\cdots}
t^{-s}\int_0^{2\pi} J_s^{(j)}(g_\alpha(x))du
\end{equation}
where the term to the right is computed with respect to any given BRT covering of $X$, 
including but not restricted to, the previous $\{\hat D_j\}_j$.   
Hence at each stage of the computation we may choose convenient BRT charts
for the need (as far as the asymptotic expansion as $t\to 0^+$ is concerned).   

By Proposition~\ref{t-gue160416c}, \eqref{e-gue160416cde4} is reduced to 
computing $I_{\ell, s}^{(j)}(g_\alpha)$ (see \eqref{e-gue160416bce2})
(for a fixed $g_\alpha$).  

Henceforth, in the following we fix an (arbitrarily given) $\alpha$.  
As aforementioned, we are free to reset the BRT charts $\{\hat D_j\}_j$ (with certain cut-off functions). 
To do so, we make the following definition for convenience. 

\begin{defn}
\label{d-gue160416}
Fix an $x_0\in X$.   $\{\hat D_j\}_j$ ($\hat D_j\subset D_j$ etc. notations as in the beginning of this section) a (finite) covering of $X$, is said to be a covering by {\it distinguished BRT charts at $x_0$} provided 
that $x_0\in \hat D_j$ for some $j$ and $x_0\not\in \hat D_k$ for all $k\ne j$.  
\end{defn}

Now, we can further assume that for the above fixed $\alpha$ and for an $x\in \Omega_\alpha$, $\{\hat D_j\}_j$ 
is distinguished at $x$ in the sense of Definition~\ref{d-gue160416}.  
In fact we can assume a little more that $\Omega_\alpha\Subset\hat D_{j_0}$ and $\Omega_\alpha\cap \hat D_k=\emptyset$
for $k\ne j_0$; namely $\{\hat D_j\}_j$ is distinguished at $x$ for each $x\in \Omega_{\alpha}$.   Also, 
we assume that \eqref{e-gue160416cde3} is satisfied.  

We shall now choose the cut-off function $\sigma_{j_0}$, in notation of \eqref{e-gue160416bce2}, that 
satisfies (see lines above \eqref{e-gue150627f})
\begin{equation}
\label{e-gue160416cde6}
\int_{-m\rho}^{m\rho}\sigma_{j_0}(u)du=\int_{-\rho}^{\rho}\sigma_{j_0}(u)du=1,\qquad
\mathrm{Supp}\,\sigma_{j_0}\,\subset\,\,\, ]-\rho, \rho[
\end{equation}
and choose $\chi_{j_0}\equiv 1$, so $\tau_{j_0}\equiv 1$, on $\overline\Omega_\alpha$
(see {\it loc. cit.}).  

With the above setup, some simplifications for \eqref{e-gue160416cde4} occur.   Firstly, 
\begin{equation}\label{e-gue160416cde4a}
\int_X g_\alpha(x)\mathrm{Tr}\,e^{-t\widetilde\Box^+_{b, m}}(x, x)dv_X(x)\sim \sum_{s=n, n-1,\cdots}
t^{-s}\int_0^{2\pi} J_s^{(j_0)}(g_\alpha(x))du. 
\end{equation}
We are reduced, by Proposition~\ref{t-gue160416c}, to computing the integrals in \eqref{e-gue160416bce2}.

Secondly, in notation of \eqref{e-gue160416bce2} there is an angular integral 
\begin{equation}
\label{e-gue160416cde5}
\int_{-\varepsilon}^{\varepsilon}\sigma_{j_0}(\theta(Y)+u)du.
\end{equation}
For a fixed $Y\in \Omega_\alpha$, $\theta(Y)\in [-\rho, \rho]$ by \eqref{e-gue160416cde3}, and for 
$u$ going through $[-\varepsilon, \varepsilon]=[-2\rho, 2\rho]$ of \eqref{e-gue160416cde5}, one sees that 
$\theta(Y)+u$ covers $[-\rho, \rho]$, it follows from \eqref{e-gue160416cde6}
that the angular integral \eqref{e-gue160416cde5} is $1$. 

Thirdly, by the above condition on $\chi_{j_0}$ and $\tau_{j_0}$ one obtains, with \eqref{e-gue160416cde5} $\equiv 1$,
the following for \eqref{e-gue160416bce2}. 
\begin{equation}
\label{e-gue160416bce7}
I_{\ell(\gamma_\ell),s}^{+,(j_0)}(g_\alpha(x))\,(=I_{\overline X_{p_{\ell(\gamma_\ell)}},s}^{+,(j_0)}(g_\alpha(x)))\,=
\frac{1}{2\pi}\int_{\overline X_{p_{\ell(\gamma_\ell)}}}g_\alpha(Y)\mathrm{Tr}\,b_{j_0,s}^+(z(Y), z(Y))
dv_{\overline X_{p_{\ell(\gamma_\ell)}}}(Y)
\end{equation}
$(\ell\ge 1; \,s=n,n-1,\ldots)$.  

Finally, recall $\alpha^+_s(x)$ as in our main result Theorem~\ref{t-gue160114}
(cf. Theorem~\ref{t-gue160123}) defined in \eqref{e-gue160224I} which is 
independent of choice of BRT charts and cut-off functions (cf. Remarks~\ref{r-gue150508I} and \ref{r-gue150607}).  
Indeed one sees, for $x\in \Omega_\alpha$, 
\begin{equation}
\label{e-gue160416bce7a}
\frac{1}{2\pi}b^+_{j_0, s}(z(x), z(x))=\alpha^+_s(x)\,\,(=\alpha^+_{s, m}(x))
\end{equation}
by \eqref{e-gue160224I} and the choice of distinguished BRT charts here. 

In sum, since the above applies to each $\Omega_\alpha$ in the covering $\{\Omega_\alpha\}_\alpha$, 
by \eqref{e-gue160416bce7} and \eqref{e-gue160416bce7a} we have reached 
the following {\it invariant expressions} (independent of choice of BRT coverings) 
\begin{equation}
\label{e-gue160416cde8}\begin{split}
 (k\ge \ell\ge 1)\quad S^+_{\ell(\gamma_\ell),s}(g_\alpha)\,(=S^+_{X_{p_{\ell(\gamma_\ell)}},s, m}(g_\alpha))\equiv 
S_{\ell(\gamma_\ell),s}^{+,(j_0)}(g_\alpha)\,&=
\int_{\overline X_{p_{\ell(\gamma_\ell)}}}g_\alpha(Y)\mathrm{Tr}\,\alpha_{s, m}^+(Y, Y)
dv_{\overline X_{p_{\ell(\gamma_\ell)}}}(Y)\qquad \\
 S^+_{{\ell(\gamma_\ell)},s}\,(=S^+_{X_{p_{\ell(\gamma_\ell)}},s, m})\,\equiv
\sum_\alpha S^+_{{\ell(\gamma_\ell)},s}(g_\alpha)&=
\int_{\overline X_{p_{\ell(\gamma_\ell)}}}\mathrm{Tr}\,\alpha_{s, m}^+(Y, Y)dv_{\overline X_{p_{\ell(\gamma_\ell)}}}(Y).\qquad 
\end{split}
\end{equation}

Now \eqref{e-gue160416cde4} and \eqref{e-gue160416cde4a} can be given, 
by using \eqref{e-gue160416cde8} and Proposition~\ref{t-gue160416c}, as follows.  

First, we classify the set $\{\Omega_\alpha\}_\alpha$ by writing 
\begin{equation}
\label{e-gue160416cde10}
\chi(\alpha)\,(=\chi(\Omega_\alpha))\,=\ell\,\,\,\mbox{if \,\,$\Omega_\alpha\bigcap X_{p_\ell}\ne \emptyset$\,\,and\,\,for any $\ell'>\ell$, \,$\Omega_\alpha\bigcap X_{p_{\ell'}}
=\emptyset$},\,\,\, 1\le\ell, \ell'\le k. 
\end{equation}

Note $d_c$ below (with a specific $c$) is given without ambiguity (cf. \eqref{e-gue160416p2re2})
by the local nature of $\Omega_\alpha$ and $g_\alpha$. 
\begin{prop}
\label{t-gue160416cc} In the notation above and in terms of the functions of \eqref{e-gue160416cde8}, 
let $\alpha$, $\delta>0$ and any $m\in \mathbb{N}$ be given.  
Then we have, as $t\to 0^+$, for $g_\alpha$ with $\mathrm{Supp}\,g_\alpha\Subset \Omega_\alpha$
such that $\chi(\alpha)=\ell$, 
\begin{equation}\label{e-gue160416cce1}\begin{split}
&\int_X g_\alpha(x) \mathrm{Tr}e^{-t\widetilde\Box^+_{b, m}}(x, x)dv_X(x)\sim 
\sum_{s=n, n-1, \ldots}t^{-s}A_s(g_\alpha)\quad\mbox{where $A_s(g_\alpha)$ is given by}\\ 
&A_s(g_\alpha)=\big(\sum_{q=1}^{p_1} e^{-\frac{i2\pi qm}{p_1}}\big)S^{+}_{X,s}(g_\alpha)+
\sum_{\scriptstyle c\atop\scriptstyle 2\le i(c)\le \ell}(d_cS^{+}_{i(c)(\gamma_{i(c)}), s}(g_\alpha)\sqrt{t}^{e_{i(c)(\gamma_{i(c)})}}+O(\sqrt{t}^{e_{i(c)(\gamma_{i(c)})}+1}))\\
&\qquad\qquad +\eta_s(g_\alpha)
\end{split}
\end{equation}
where $i(c)=i(c, g_\alpha)$, $d_c=d_{c,g_\alpha,m,I}, I=S^{+}_{i(c)(\gamma_{i(c)}), s}(g_\alpha)$ by 
\eqref{e-gue160416bce2} and \eqref{e-gue160416p2re2}, and $\eta_s(g_\alpha)$ (which equals 
$\eta_s^{(j_0)}(g_\alpha)$ in notation of Proposition~\ref{t-gue160416c} and 
distinguished BRT charts at $x_0$) satisfies the estimate 
$$|\eta_s(g_\alpha)|\le \delta\cdot\int_Xg_\alpha dv_X$$
for $s=n, n-1, n-2,\ldots, -m$.  
\end{prop}


In the remaining of this subsection, to streamline the argument we assume the simplest case 
that 
\begin{equation*} \begin{split}
&i) \quad \mbox{each $X_{p_\ell}$, $1\le \ell\le k$, is {\it connected}}\\ 
&ii) \,\,\,\,X=\overline X_{p_1}\supsetneq \overline X_{p_2}\supsetneq \overline X_{p_3}\ldots\supsetneq \overline X_{p_k}.
\end{split}
\end{equation*}
(We postpone the general case to the next subsection.)   One sees $p_1|p_2|\ldots|p_k$.  

Hence all types reduce to simple types (cf. Definition~\ref{d-gue160416bc}).  

{\bf In this case, the subscript $\gamma_\ell$ in $\ell(\gamma_\ell)$ will henceforth be dropped
throughout the remaining of this subsection}.   

It will take a bit more work to sum \eqref{e-gue160416cce1} over $\alpha$.  
The numerical factor $d_c$ (cf. \eqref{e-gue160416p2e1}) in this simplified case satisfies the following.  
For smooth functions $g, g'$ of small support 
(Definition~\ref{d-gue160416d2c}),
if $i(c, g)=i(c, g')\,(\le \infty)$, then $d_{c, g, m}=d_{c, g', m}$.  
It is useful to set,  for $g=g_\alpha$ with $\chi(\alpha)\ge \ell$ in ii) below ($\chi(\alpha)$ 
as in \eqref{e-gue160416cde10}), 
\begin{equation}
\label{e-gue160416cde11}\begin{split}
i)\,D_{1, g}\,(=D_{1,g, m})\,\equiv\sum_{c,\, i(c, g)=1}d_{c,g, m}=\sum_{q=1}^{p_1}e^{-\frac{i2\pi qm}{p_1}};\quad 
ii)\,(k\ge\ell\ge 2)\quad D_{\ell, g}\,(=D_{\ell,g, m})\,\equiv\sum_{c,\, i(c, g)=\ell}d_{c,g, m}.
\end{split}
\end{equation}


Suppose $\alpha, \beta\in \bigcup_{\ell'\ge \ell}\chi^{-1}(\ell')$.   
As said above, one sees $D_{\ell, g_\alpha}=D_{\ell, g_\beta}$
for $1\le \ell\le k$ (because $i(c, g_\alpha)=\ell$ if and only if $i(c, g_\beta)=\ell$ here).  
We write 

\begin{defn}
\label{d-gue160416d2d} 
$D_\ell\,(=D_{\ell,m})\,:=D_{\ell, g_\alpha}\,(=D_{\ell, g_\alpha,m})$
for any $\alpha$ with $\chi(\alpha)\ge \ell$, $1\le \ell\le k$. 
\end{defn} 

By using \eqref{e-gue160416cce1} of Proposition~\ref{t-gue160416cc} 
and Definition~\ref{d-gue160416d2d}, 
one sees, for $\alpha\in \chi^{-1}(\ell)$, 
\begin{equation}
\label{e-gue160416cde10e1}\begin{split} &\alpha\in \chi^{-1}(\ell),\quad
\int_X g_\alpha(x) \mathrm{Tr}\,e^{-t\widetilde\Box^+_{b, m}}(x, x)dv_X(x)\sim\sum_{s=n, n-1,\dots}t^{-s}\times\\
&\eta_s(g_\alpha)+D_1S_{1,s}^+(g_\alpha)+\sum_{c, i(c)=2}\big(d_cS_{2,s}^+(g_\alpha)\sqrt{t}^{e_2}+O(\sqrt{t}^{e_2+1})\big)
+\ldots +\sum_{c, i(c)=\ell}\big(d_cS_{\ell,s}^+(g_\alpha)\sqrt{t}^{e_\ell}+O(\sqrt{t}^{e_\ell+1})\big)\\
&=\sum_{s}t^{-s}\Big(\eta_s(g_\alpha)+D_1S_{1,s}^+(g_\alpha)+\big(D_2S_{2,s}^+(g_\alpha)\sqrt{t}^{e_2}+O(\sqrt{t}^{e_2+1})\big)
+\ldots+\big(D_\ell S_{\ell,s}^+(g_\alpha)\sqrt{t}^{e_\ell}+O(\sqrt{t}^{e_\ell+1})\big)\Big).  
\end{split}
\end{equation}

Combining \eqref{e-gue160416cde10} and \eqref{e-gue160416cde10e1}
yields the following as $t\To0^+$ (where $\{\alpha\}_\alpha=\chi^{-1}(1)\cup\chi^{-1}(2)\cup\chi^{-1}(3)\ldots$):
\begin{equation}
\label{e-gue160416cde9}\begin{split}
&\int_X \mathrm{Tr}\,e^{-t\widetilde\Box^+_{b, m}}(x, x)dv_X(x)
\,\,\,\Big(=\sum_\alpha\int_X g_\alpha(x)\mathrm{Tr}e^{-t\widetilde\Box^+_{b, m}}(x, x)dv_X(x)\Big)\\
&\sim 
\sum_{s=n, n-1,\ldots}t^{-s}\Big(\big(\sum_{\alpha\in \chi^{-1}(1)}D_1S_{1, s}^{+}(g_\alpha)\big)
+\big((\sum_{\alpha\in \chi^{-1}(2)}D_1S_{1, s}^{+}(g_\alpha))+(\sum_{\alpha\in \chi^{-1}(2)}D_2S_{2, s}^{+}(g_\alpha)\sqrt{t}^{e_2}+O(\sqrt{t}^{e_2+1}))\big)\\
&\qquad +\big((\sum_{\alpha\in \chi^{-1}(3)}D_1S_{1, s}^{+}(g_\alpha))+(\sum_{\alpha\in \chi^{-1}(3)}D_2S_{2, s}^{+}(g_\alpha)\sqrt{t}^{e_2}+O(\sqrt{t}^{e_2+1}))\\
&\qquad\qquad\qquad\qquad \qquad\qquad\qquad \qquad+(\sum_{\alpha\in \chi^{-1}(3)}D_3S_{3, s}^{+}(g_\alpha)\sqrt{t}^{e_3}+O(\sqrt{t}^{e_3+1}))\big)+\ldots+\sum_\alpha\eta_s(g_\alpha)\Big).
\end{split}
\end{equation}

We rearrange \eqref{e-gue160416cde9} as (only keeping terms in leading order)
\begin{equation}
\label{e-gue160416cde12}\begin{split}
&\sum_st^{-s}\times\Big(\sum_\alpha\eta_s(g_\alpha)+\big(\sum_{\alpha\in\chi^{-1}(1)}D_{1}S_{1, s}^{+}(g_\alpha)+
\sum_{\alpha\in\chi^{-1}(2)}D_{1}S_{1, s}^{+}(g_\alpha)+\sum_{\alpha\in\chi^{-1}(3)}D_{1}S_{1, s}^{+}(g_\alpha)+\dots\big)\\
&+\big((\sum_{\alpha\in\chi^{-1}(2)}D_{2}S_{2, s}^{+}(g_\alpha)\sqrt{t}^{e_2}+\ldots)+(\sum_{\alpha\in\chi^{-1}(3)}D_{2}S_{2, s}^{+}(g_\alpha)\sqrt{t}^{e_2}+\ldots)+(\sum_{\alpha\in\chi^{-1}(4)}\ldots)+\ldots\big)\\
&+\big((\sum_{\alpha\in\chi^{-1}(3)}D_{3}S_{3, s}^{+}(g_\alpha)\sqrt{t}^{e_3}+\ldots)+
(\sum_{\alpha\in\chi^{-1}(4)}D_{3}S_{3, s}^{+}(g_\alpha)\sqrt{t}^{e_3}+\ldots)+(\sum_{\alpha\in\chi^{-1}(5)}\ldots)+\ldots\big)+\ldots\Big)\\
\end{split}
\end{equation}
which equals (by \eqref{e-gue160416cde8} and $S^+_{\ell, s}(g_\alpha)=0$ for $\alpha\in\chi^{-1}(\ell')$ with $\ell'<\ell$), 
\begin{equation}
\label{e-gue160416cde13}\begin{split}
&\sum_{s=n,n-1\ldots}t^{-s}\times\Big(\sum_\alpha\eta_s(g_\alpha)+D_{1}S^+_{1, s}+\big(D_2S^+_{2, s}\sqrt{t}^{e_2}+O(\sqrt{t}^{e_2+1})\big)\\
&\qquad\qquad\qquad\qquad\qquad\qquad +\big(D_3S^+_{3, s}\sqrt{t}^{e_3}+O(\sqrt{t}^{e_3+1})\big)+\ldots\big)+\ldots\Big). 
\end{split}
\end{equation}

For $s=n$, we have $S^+_{\ell, n}=\frac{1}{2\pi}(2\pi)^{-n}\mathrm{vol}(X_{p_\ell})$ 
(see iii) of Proposition~\ref{t-gue160416b}).  

For the term given by the sum $\sum_\alpha\eta_s(g_\alpha)$ in \eqref{e-gue160416cde13}, 
by Proposition~\ref{t-gue160416cc} 
we obtain $\sum_\alpha|\eta_s(g_\alpha)|\le \delta\cdot\mathrm{vol}(X)$ 
(as $\sum_\alpha\int_X g_\alpha dv_X=\mathrm{vol}(X)$), 
$s=n, n-1, \ldots, -m$, where $\delta>0$ and 
$m\in \mathbb{N}$ are arbitrarily prescribed.    
By the definition of asymptotic expansion (cf. Definition~\ref{d-gue150608})
and the fact that \eqref{e-gue160416cde13} has been an asymptotic expansion, one sees 
the term $\sum_\alpha\eta_s(g_\alpha)$ becomes immaterial to the exact form of the asymptotic expansion. 

Further, the asymptotic expansion of $\int_X \mathrm{Tr}\,a_s^+(t, x, x) dv_X(x)$ of \eqref{e-gue160416ce1}
basically follows from that of $\int_X \mathrm{Tr}\,e^{-t\widetilde\Box^+_{b, m}}(x, x)dv_X(x)$.  

We have now proved (part of) the main result of this section.  

\begin{thm} (cf. Theorem~\ref{t-gue160416})
\label{t-gue160416d}  Suppose $X=\overline X_{p_1}\supsetneq \overline X_{p_2}
\supsetneq\cdots \supsetneq X_{p_k}=\overline X_{p_k}$ with each stratum $X_{p_\ell}$ a connected submanifold.   
Let $a^+_s(t,x,y)\,(=a_{s,m}^+(t, x, y))$, $s=n,n-1,\ldots$, be as in 
\eqref{e-gue150630gI}.  Write $e_2$ for the (real) codimension 
of $X_{p_2}$ (which is an even number, cf. Remark~\ref{r-gue160416r2} below). 
(Recall the numerical factors $D_{\ell,m}$ as given in Definition~\ref{d-gue160416d2d} 
and the integrals $S^+_{\ell,s}\,(=S^+_{\ell, s, m})$ in \eqref{e-gue160416cde8} with subscripts simplified 
in the present case.)  Then the following holds.  

i) As $t\To0^+$,
\begin{equation}  \label{e-gue160417wII}\begin{split}
\int_X\mathrm{Tr}\,&e^{-t\widetilde\Box^+_{b,m}}(x,x)dv_X(x) \\
&\sim D_{1, m}\big((2\pi)^{-1}(2\pi t)^{-n}\mathrm{vol}(X)
+t^{-n+1}S^+_{1, n-1}+t^{-n+2}S^+_{1, n-2}+\ldots\big)\\ 
&\qquad+(2\pi)^{-(n+1)}D_{2, m}\mathrm{vol}(X_{p_2})t^{-n+\frac{e_2}{2}}+O(t^{-n+\frac{e_2+1}{2}}).
\end{split}
\end{equation}
In particular, by $\sum_{q=1}^{p_{1}}e^{-{\frac{2\pi i}{p_{1}}}qm}=p_{1}$ 
for $p_{1}|m$ and $0$ otherwise, one has $D_{1, m}=p_1$ if $p_1|m$.    
If $p_2|m$ (thus $p_1|m$ too), then $D_{1,m}, D_{2, m}>0$.    

ii) In the asymptotic expansion \eqref{e-gue160417wII},
all the coefficients of $t^j$ for $j$ being half-integral, vanish.  


iii) As a consequence of \eqref{e-gue160416cde13} and ii) 
\begin{equation}\label{e-gue160416wIIe1}
\int_X \mathrm{Tr}\,a_{s,m}^+(t, x, x) dv_X(x)\sim
D_{1,m}S^+_{1, s}+D_{2,m}S^+_{2, s}{t}^{\frac{e_2}{2}}+O({t}^{\frac{e_2}{2}+1}),\,\,\mbox{ as\, $t\to 0^+$}.
\end{equation}

The similar results hold true for the case $\int_X\mathrm{Tr}\,e^{-t\widetilde\Box^-_{b,m}}(x,x)dv_X(x)$ 
and $\int_X\mathrm{Tr}\,a_{s,m}^-(t, x, x))dv_X(x)$ as well. 
\end{thm}
 
\begin{proof} It remains to prove ii) of the theorem.  
Recall the last two paragraphs of the proof of Proposition~\ref{t-gue160416b}, 
especially the item $\delta$) there.  
In the present case, by scaling $(\hat y\to \sqrt{t}\hat y)$ and using \eqref{e-gue160417c1}, it reduces to 
computing the expansion (in $\sqrt{t}$) of 
\begin{equation}\label{e-gue160416te1}
\begin{split} & a)\quad  \hat b_{j,s}^+
((\sqrt{t}{\hat y}, Y),e^{-iu}\circ (\sqrt{t}{\hat y}, Y))\\
&b) \quad dv_X(x)
\end{split}
\end{equation}
for a fixed $u$.   

Write $g_u(x)=e^{-iu}\circ x$ and 
$\hat b_{j,s}^+
((\sqrt{t}{\hat y}, Y),e^{-iu}\circ (\sqrt{t}{\hat y}, Y))
=(\hat b_{j, s}\circ (\mathrm{id},g_u))((\sqrt{t}{\hat y}, Y), (\sqrt{t}{\hat y}, Y))$.  
By $\delta$) mentioned above, $g_u(0, Y)$ is only away from $(0, Y)$ 
by a small difference $(\le \varepsilon)$ in their $\theta$-coordinates, hence by continuity,
$g_u((\sqrt{t}{\hat y}, Y))$ lies in an $O(2\varepsilon)$-small neighborhood 
of $(0,Y)$ (as $t\to 0^+$), giving that   
the Taylor expansion of $\hat b_{j, s}(x, y)$, $x=(\sqrt{t}{\hat y}, Y), y=e^{-iu}\circ x$, 
around $x=y=(0,Y)\equiv Y\equiv 0$ can be done in terms of integral powers of $\sqrt{t}\hat y_i$ where 
$\hat y_i$ is in $\hat y$.   
Hence the coefficients of the $t^{j}$ for $j$ being half-integral must involve 
an odd power of some variable $\hat y_i$ in $\hat y$.  Since $\hat y$ sits in an even dimensional space (cf.
i) of Remark~\ref{r-gue160416r2} below), $dv_X(x)$ is of integral power in $t$.   
With a), b) of \eqref{e-gue160416te1}, by using i) the claim \eqref{e-gue160417c1}, 
ii) $\int_{-\infty}^{\infty}e^{-\hat y_i^2}\hat y_i^{n}d\hat 
y_i=0$ for an 
odd number $n$ and iii) for a polynomial $P(x)$, 
$\int_{1/{\sqrt{t}}}^{\infty}e^{-x^2}P(x)dx\sim O(t^{\infty})$ (as $t\to 0^+$), 
our assertion about the asymptotic expansion in ii) of the theorem follows.  
\end{proof}
\begin{rem}
\label{r-gue160416r1}
One may think of the second line in \eqref{e-gue160417wII} as 
the {\it main terms} which remind one of the close relation between 
the Kodaira Laplacian and Kohn Laplacian (cf. Proposition~\ref{l-gue150606}).
However, one key point in this paper is the idea that 
if the $S^1$ action is locally free (but not globally free) on $X$, 
then this relation cannot be altogether extended to 
their heat kernels.  In this regard the {\it correction terms} exist, and consist in 
the third line of \eqref{e-gue160417wII} 
linked up with the higher strata of the (locally free) $S^1$ action beyond the principal stratum.

\end{rem}

\begin{rem}
\label{r-gue160416r2} i) We argue $e_{\ell}$ is even.  $X_{p_\ell}$ is $S^1$ invariant; 
$TX_{p_\ell}=(\mathbb{R}T^{\perp}\cap TX_{p_\ell})\oplus \mathbb{R}T|_{X_{p_\ell}}$ 
where $\mathbb{R}T$
is the line subbundle of $TX$ generated by $\partial/{\partial \theta}$ such that $ \mathbb{R}T\oplus \mathbb{R}T^\perp=TX$.
In a BRT chart $U\times ]\varepsilon, \varepsilon[$ we denote $U\times \{0\}$ ($\subset X$) by $\tilde U$.  Write 
$(\mathbb{R}T^{\perp}\cap TX_{p_\ell})|_{\tilde U\cap X_{p_\ell}}\equiv E$.   For any given $p\in U$, 
with $\tilde p=p\times \{0\}$, 
one may choose a BRT coordinate such that $E_{\tilde p}\subset T_{\tilde p}\tilde U$ 
(see the proof of Lemma~\ref{l-gue160417lem2} where $T_{\tilde p}\tilde U=\mathbb{R}T^{\perp}|_{\tilde p}$). 
Without loss of generality we may assume $\tilde p\in X_{p_\ell}$.  
Write $g=e^{i2\pi/p_\ell}\,(\in S^1)$ which is CR and an isometry with the fixed point set $\overline X_{p_\ell}$.   
By $dg\circ J=J\circ dg=J$ on $E_{\tilde p}$ with $J$ the complex structure of $T_{\tilde p}{\tilde U}$, 
$E_{\tilde p}$ is invariant under $J$.   It follows $E$ is of even dimension, so 
$X_{p_\ell}$ is of odd dimension, i.e. $e_{\ell}$ is even.  
ii) For $1\le \ell\le k$ we can write $X_{p_\ell}\to M_{p_\ell}$ for a complex manifold $M_{p_\ell}$, 
as an $S^1$ fiber bundle.   The quantities $\alpha_s^{\pm}(x)$ by the construction (see 
\eqref{e-gue160416bce7}-\eqref{e-gue160416cde8}) is $S^1$ invariant 
hence descend to $M_{p_\ell}$.   One sees $S_{X_\ell, s}^{\pm}\,(\equiv S_{\ell,s}^{\pm})\,=
\frac{2\pi}{p_\ell}S^{\pm}_{M_{p_\ell,s}}(S^{\pm}_{M_{p_\ell,s}}=\int_{M_{p_\ell}}\mathrm{Tr}\,\alpha^{\pm}_sdv_{M_{p_\ell}})$. 
Here, the metric on $M_{p_\ell}$ (cf. $dv_{M_{p_\ell}}$) is defined in a way similar to that given in \eqref{e-gue150823e1}. 
This suggests a question of how the heat kernels (for the locally free $S^1$ action) of the present paper
may be connected with (certain suitably defined) heat kernels in the orbifold base $X/S^1$.  
In a certain Riemannian setting, some work in a similar direction has been done (cf. \cite[Theorems 3.5, 3.6]{Ri10}).  
\end{rem}

\subsection{Types for $S^1$ stratifications; proof for the general type}
\label{s-gue160416s4}
Lastly, to modify the above reasoning to the case beyond the simple type is essentially not difficult.  
Suppose, say, $X_{p_2}$ has several connected components $Y_i$ such that the simple type condition 
is assumed along each component $Y_i$.   Then, clearly the above argument applies to the individual $Y_i$
and the result is just to sum up over $i$.   
Without assuming the simple type condition on $Y_i$, say, inside some $Y_i$ the next stratum $X_{p_3}$ 
has seated several components $Z_j$ or some components $Z_j$ 
are seated even outside of each $Y_i$.  Then by localization argument
along each $Z_j$ just as done above, one repeats the pattern similarly.   The process continues.  

We are now motivated to transplant the notion of ``type", ``class" in Definition~\ref{d-gue160416t1} for the integral 
$I$ of \eqref{e-gue160416be1} into the geometry of the stratification of the $S^1$ action.   

For a connected component $X_{p_{\ell(\gamma_\ell)}}\subset X_{p_\ell}$,
$\gamma_\ell=1,\ldots, s_\ell$, contained in 
the higher dimensional connected components of the strata 
\begin{equation*}
(X=)\overline X_{p_{i_1(\gamma_{i_1})}}\supsetneq \overline X_{p_{i_2(\gamma_{i_2})}} 
\ldots  \supsetneq \overline X_{p_{i_f(\gamma_{i_f})}}\supsetneq \overline X_{p_{i_{f+1}(\gamma_{i_{f+1}})}}
=\overline X_{p_{\ell(\gamma_\ell)}}
\end{equation*}
where $i_1=1<i_2<\ldots<i_f<i_{f+1}=\ell\in \{1, 2, ,\ldots, \ell-1,\ell\}$,
we define its type $\tau(X_{p_{\ell(\gamma_\ell)}})$ by 
\begin{equation}
\label{e-gue160416r2e1}
\tau(X_{p_{\ell(\gamma_\ell)}})\equiv\tau(\ell(\gamma_\ell)):=({i_1(\gamma_{i_1})}, {i_2(\gamma_{i_2})},\ldots, i_f(\gamma_{i_f}), i_{f+1}(\gamma_{i_{f+1}})),\quad
i_1=\gamma_{i_1}=s_{1}=1; i_{f+1}=\ell. 
\end{equation}
One has $p_{i_1}|p_{i_2}|\ldots|p_{i_{f+1}}$ (cf. Remark~\ref{r-gue150804} for a similar case). 

The notions such as {\it simple type}, {\it class} and {\it length} $l(\tau)$
are defined similarly, cf. Definition~\ref{d-gue160416t1}.   (No definition of 
{\it trivial type} is given here.    See iii) of Definition~\ref{d-gue160416d2} below in which 
$i(c, [\tau])=\infty$ corresponds to the trivial type, cf. iii) of Definition~\ref{d-gue160416bc}.)  

Recall if $M\subset N$ is a finite disjoint union of 
submanifolds $M_j$, then the dimension of $M$ is $\max_j\{\mathrm{dim}_{\mathbb{R}}\,M_j\}$
and the codimension of $M$ is $\mathrm{dim}_{\mathbb{R}}\, N-\mathrm{dim}_{\mathbb{R}}\,M$. 

The following definition, which is bit tedious yet bears a great similarity as previously,
is set up for the immediate use in the general situation.  
\begin{defn}
\label{d-gue160416d2} 
i) Write $\nu_\ell=\{[\tau]; \tau=\tau(X_{p_{\ell(\gamma_\ell)}}), \gamma_\ell=1, 2, \ldots ,s_\ell\}$ 
for the set of equivalence classes of types $\tau=\tau(X_{p_{\ell(\gamma_\ell)}})$ of 
connected components $X_{p_{\ell(\gamma_\ell)}}$ in $X_{p_\ell}$.

ii) Write (similar to \eqref{e-gue160416cde8}, \eqref{e-gue160416bce7a}) \begin{equation*}
S^+_{\ell(\gamma_\ell),s}\,(=S^+_{X_{p_{\ell(\gamma_\ell)}},s, m})\,=
\int_{\overline X_{p_{\ell(\gamma_\ell)}}}\mathrm{Tr}\,\alpha_{s}^+(Y, Y)dv_{\overline X_{p_{\ell(\gamma_\ell)}}}(Y)
\qquad(\alpha^+_s=\alpha_{s,m}^+)
\end{equation*}
associated with $X_{p_{\ell(\gamma_\ell)}}$.  

iii) Let $[\tau]=[({i_1(\gamma_{i_1})}, {i_2(\gamma_{i_2})},\ldots, i_f(\gamma_{i_f}), i_{f+1}(\gamma_{i_{f+1}}))]$ be given.  If $c|p_1$, define $i(c, [\tau])=1$.  (Hence it is independent of $[\tau]$.)  If  $c\not|p_1$ and 
$c|p_\ell$, $\ell=i_s$ for some $s$, $2\le s\le {f+1}$, such that $c\not|i_{s'}$ for all $s'<s$.
Then $i(c, [\tau]):=\ell\ge 2$.    If $c\not|p_{i_s}$ for $1\le s\le t+1$, $i(c, [\tau]):=\infty$. 
We may write $i(c)$ for $i(c, [\tau])$.  

iv) For $i(c, [\tau])\ge 2$, define the numerical factors 
 $d_{c,m,[\tau]}$ correspondingly as in \eqref{e-gue160416p2re2}.  For $i(c)=1$, 
define $d_{c,m,[\tau]}=d_c$
(which is independent of $\tau$) as in \eqref{e-gue160416p2e1}.  

v)  For a given $[\tau]$, if $i(c)=1$, then 
define the weight factors $D_1$ as in \eqref{e-gue160416cde11} (which is independent of $\tau$)
and if $i(c)\ne 1, \infty$, define $D_{\ell,[\tau]}\,(=D_{\ell,m, [\tau]})\,\equiv\sum_{c,\, i(c)=\ell}d_{c,m, [\tau]}$.  

vi) Write $e_{i_q, [\tau]}\equiv e_{i_q(\gamma_{i_q})}$ with $\tau(\ell(\gamma_\ell))=(i_1(\gamma_{i_1}), \ldots,i_q(\gamma_{i_q}),\ldots,  i_{f+1}(\gamma_{i_{f+1}}))$ for the codimension of 
$X_{p_{i_q}(\gamma_{i_q})}$.  Obviously, $e_{i_1, [\tau]}<e_{i_2, [\tau]}<\cdots$.
For $[\tau]\in \nu_\ell$, write $e_{[\tau]}\equiv  e_{i_{f+1}(\gamma_{i_{f+1}})}$, i.e. $e_{\ell(\gamma_\ell)}$.  

vii) Write $e
=\min_{\scriptstyle [\tau]\in \nu_\ell\atop \scriptstyle 2\le \ell\le k}e_{[\tau]}\,
(=\min_{\scriptstyle \tau, l(\tau)= 2}e_{[\tau]}\mbox{\,\, by vi) above})$ 
and for $\ell\ge 2$, 
\begin{equation*}
\hat\nu_\ell=\{[\hat\tau_\ell]\in \nu_\ell; \,\,\mbox{$\hat\tau_\ell=(1, \ell(\gamma_\ell))$ of length two 
such that  $e_{\ell(\gamma_\ell)}=e$, i.e. $e_{[\hat \tau_\ell]}=e$}\}\subset \nu_\ell.
\end{equation*}
Of course, it is not ruled out that for some values of $\ell$, $\hat\nu_\ell$ could be an empty set.   
Intuitively, one thinks of $e$ as the minimal codimension among those connected components $X_{p_{\ell(\gamma_\ell)}}$
such that if $\overline X_{p_{\ell'}}\supsetneq  X_{p_{\ell(\gamma_\ell)}}$, then $ X_{p_{\ell'}}=X_{p_1}$ the 
principal stratum.  

viii) For a fixed $[\kappa]\in \nu_\ell$ in i), write (see \eqref{e-gue160416r2e1} for $\tau(\ell(\gamma_\ell))\equiv \tau(X_{p_{\ell(\gamma_\ell)}})$) $$Z_{[\kappa], s}=
\sum_{\scriptstyle \gamma_\ell, \,[\tau(\ell(\gamma_\ell))]=[\kappa]\atop\scriptstyle 1\le \gamma_\ell\le s_\ell }S^+_{\ell(\gamma_\ell), s}.$$
\end{defn}

For the case of general type, we can obtain results corresponding to 
\eqref{e-gue160416cde9}, \eqref{e-gue160416cde12} and \eqref{e-gue160416cde13}
(yet complicated in expressions here).  
We are content with summarizing the final result as follows.   
\begin{equation}\label{e-gue160416r2e2}
\begin{split}
&\int_X \mathrm{Tr}e^{-t\widetilde\Box^+_{b, m}}(x, x)dv_X(x)\\
&\sim\sum_{s=n,n-1\ldots}t^{-s}\times\Big(D_{1}S^+_{1, s}+
\sum_{[\kappa_2]\in \nu_2}\big(D_{2,[\kappa_2]}Z_{[\kappa_2], s}
\sqrt{t}^{e_{[\kappa_2]}}+O(\sqrt{t}^{e_{[\kappa_2]}+1})
\big)\\
&\qquad\qquad\qquad\qquad +
\sum_{[\kappa_3]\in \nu_3}\big(D_{3,[\kappa_3]}Z_{[\kappa_3], s}
\sqrt{t}^{e_{[\kappa_3]}}+O(\sqrt{t}^{e_{[\kappa_3]}+1})
\big)+\ldots\Big). 
\end{split}
\end{equation}



The following main result of this subsection parallels Theorem~\ref{t-gue160416d} in the last subsection.
By comparison, to collect the coefficients for the next leading order in $t$ or $\sqrt{t}$ in \eqref{e-gue160416r2e2} here, we 
have a slightly more complicated summation (regarded as part of {\it corrections} as indicated 
in Remark~\ref{r-gue160416r1}) in the formula below.   Note that the conversion to the stated form of 
Theorem~\ref{t-gue160416} is nothing but a direct consequence of an examination (slightly tedious) of the 
various definitions here.  

\begin{thm} (cf. Theorem~\ref{t-gue160416})
\label{t-gue160416dc}  Notations as in Theorem~\ref{t-gue160416d}
without assuming the conditions of connectedness and simple-type there.   
The weight  factors $D_{\ell,m,[\tau]}$,
the integrals $S^+_{\ell(\gamma_\ell),s}\,(=S^+_{\ell(\gamma_\ell),s,m})$, $e$, $\hat \tau_\ell$ etc. are just given above.  
One has the following. 

i) As $t\To0^+$,
\begin{equation}  \label{e-gue160417wII1}\begin{split}
\int_X\mathrm{Tr}\,&e^{-t\widetilde\Box^+_{b,m}}(x,x)dv_X(x) \\
&\sim D_{1, m}\big((2\pi)^{-1}(2\pi t)^{-n}\mathrm{vol}(X)
+t^{-n+1}S^+_{1, n-1}+t^{-n+2}S^+_{1, n-2}+\ldots\big)\\ 
&\qquad+t^{-n+\frac{e}{2}}
(\sum_{\scriptstyle [\hat\tau_\ell]\in \hat\nu_\ell\atop\scriptstyle 2\le\ell\le k}D_{\ell, m, [\hat\tau_\ell]}Z_{[\hat\tau_\ell],n})
+O(t^{-n+\frac{e+1}{2}})
\end{split}
\end{equation}
(where recall $Z_{[\hat\tau_\ell],n}=(2\pi)^{-(n+1)}\sum_{\scriptstyle \gamma_\ell,\,1\le \gamma_\ell\le s_\ell
\atop \scriptstyle \mathrm{codim}\,X_{p_{\ell(\gamma_\ell)}}=e}
\mathrm{vol}(X_{p_{\ell(\gamma_\ell)}})>0$ in the locally free case of the $S^1$ action).  
If $p_\ell|m$ (thus $p_1|m$ too), then $D_{1,m}, D_{\ell,m,[\hat\tau_\ell]}\,>0$.  

ii) In the asymptotic expansion \eqref{e-gue160417wII1},
all the coefficients of $t^j$ for $j$ being half-integral, vanish.  


iii) As a consequence of \eqref{e-gue160416r2e2} and ii),
\begin{equation}\label{e-gue160416wIIe1}
\int_X \mathrm{Tr}\,a_{s,m}^+(t, x, x) dv_X(x)\sim
D_{1,m}S^+_{1, s}+
t^{\frac{e}{2}}
(\sum_{\scriptstyle [\hat\tau_\ell]\in \hat\nu_\ell\atop\scriptstyle 2\le\ell\le k}
D_{\ell,m,[\hat\tau_\ell]}Z_{[\hat\tau_\ell], s})+O(t^{\frac{e}{2}+1})
\end{equation}
(where $Z_{[\hat\tau_\ell], s}=\sum_{\scriptstyle \gamma_\ell,\,1\le \gamma_\ell\le s_\ell
\atop \scriptstyle  \mathrm{codim}\,X_{p_{\ell(\gamma_\ell)}}=e}S^+_{\ell(\gamma_\ell), s}$).

The similar results hold for $\int_X\mathrm{Tr}\,e^{-t\widetilde\Box^-_{b,m}}(x, x))dv_X(x)$ 
and $\int_X \mathrm{Tr}\,a_{s,m}^-(t, x, x)dv_X(x)$.  
\end{thm}

\begin{rem}
\label{r-gue160416II}  The quantities involved above are computable 
in the sense that they are basically reduced to those involved in the (ordinary) Kodaira heat kernel 
by ii) of Definition~\ref{d-gue160416d2}, cf. Remark~\ref{r-gue160413a}.
\end{rem}

\begin{rem}
\label{r-gue160416}
It is not obvious how one can compute the {\it supertrace} integral, hence our index
theorem~\ref{t-gue150701}, Corollary~\ref{t-gue160226a} solely by techniques
similar to those derived in \eqref{e-gue160417wII1}, partly because here we are not using 
the {\it off-diagonal estimate} of Theorem~\ref{t-gue150627g} which 
is partly based on a cancellation result in Berenzin integral (see the proof 
of that theorem).    These results (of estimate and cancellation) appear to lie beyond what the geometry 
of the stratifications can reveal as done in this Section~\ref{s-gue160416}.  
 
\end{rem}

 \bigskip

{\emph{\ \textbf{Acknowledgements.} The first named author would like to
thank the Ministry of Science and Technology of Taiwan for the support of
the project: 104-2115-M-001-011-MY2. The second named author was partially
supported by the Ministry of Science and Technology of Taiwan for the
project: 104-2628-M-001-003-MY2 and the Golden-Jade fellowship of Kenda
Foundation. We would also like to thank Professor Paul Yang for his interest
and discussion in this work. The second named author would like to express
his gratitude to Professor Rung-Tzung Huang for useful discussion in this
work. }}

\end{document}